\documentclass{amsart}
\usepackage[normalem]{ulem}
\usepackage{aligned-overset}
\usepackage{enumitem}
\usepackage{amsmath}
\usepackage{amsfonts, dsfont}
\usepackage{amssymb}
\usepackage{color}
\usepackage{graphicx,mathtools}
\usepackage{multirow}
\usepackage{float}
\usepackage{soul}
\usepackage[all, knot, color]{xy}
\usepackage{latexsym}
\usepackage{hyperref}
\usepackage{comment}
\setlist[enumerate,1]{label={\rm (\roman*)}, topsep=0pt,itemsep=3pt,parsep=0pt,partopsep=0pt, align=left, leftmargin=20pt, itemindent=3pt, labelsep=5pt, labelwidth=1em}
\allowdisplaybreaks

\numberwithin{equation}{section}

\theoremstyle{plain}
\newtheorem{thm}{Theorem}[section]
\newtheorem{lem}[thm]{Lemma}
\newtheorem{prop}[thm]{Proposition}
\newtheorem{cor}[thm]{Corollary}
\newtheorem{conjecture}[thm]{Conjecture}

\theoremstyle{definition}

\newtheorem{deff}[thm]{Definition}
\newtheorem{rem}[thm]{Remark}
\newtheorem{remark}[thm]{Remark}
\newtheorem{example}[thm]{Example}

\newcommand{\kk}{\Bbbk}

\theoremstyle{remark}

\newcommand{\qbin}[2]{{\begin{bmatrix} #1\\ #2\end{bmatrix}}_q}
\newcommand{\qqbin}[3]{{\begin{bmatrix} #1\\ #2\end{bmatrix}}_{#3}}
\newcommand{\Mod}[1]{#1\mbox{-mod}}
\def\Z{\mathbb{Z}}
\def\F{\mathbb{F}}

\def\BQ{\mathbb{Q}}
\def\BG{\mathbb{G}}
\def\ZD{{\Z_d\!\setminus\!\{\ol0\}}}

\def\k{\Bbbk}
\def\bk{{\mathbf{k}}}

\def\bo{{\pmb{0}}}

\newcommand{\J}[2]{{J_{#1}^{[#2]}}}

\newcommand{\JI}[2]{{\ol{J}^{\hspace{1pt}\pmb{#2}}_{#1}}}
\def\tJ{{\tilde{J}}}
\newcommand{\Jt}[2]{{\tJ^{\hspace{1pt}\pmb{#2}}_{#1}}}
\newcommand{\JC}[2]{{J^{\hspace{1pt}\pmb{#2}}_{#1}}}
\newcommand{\ZC}[2]{{\mathds{J}^{#2}_{#1}}}
\newcommand{\Zt}[2]{{\tilde{\mathds{J}}^{#2}_{#1}}}
\newcommand{\SJ}[1]{{\mathds{S}_J(#1)}}

\def\N{\mathcal{N}}
\def\CC{\mathcal{C}}
\def\HH{\mathcal{H}}
\def\DD{\mathcal{D}}

\renewcommand{\1}{{\mathds{1}}}
\def\st{{\mbox{\rm Tw}}}
\def\3ust{{\underline{\underline{\mathbf{Tw}}}}}
\def\2ust{\underline{\mathbf{Tw}}}

\DeclareMathOperator{\OutTw}{Out_{\rm Tw}}
\newcommand{\OutTwBr}{\operatorname{Out}_{\rm Tw}^{\rm br}}
\newcommand{\OutTwZ}{\operatorname{Out}_{\rm Tw}^{Z}}
\newcommand{\OutTwBrZ}{\operatorname{Out}_{\rm Tw}^{{\rm br},Z}}
\def\ost{{\overline{\st}}}

\def\a{\alpha}
\def\ld{{\lambda}}

\def\g{\gamma}
\def\o{\otimes}

\def\to{\rightarrow}
\def\ol{\overline}
\def\ok{{\ol{k}}}
\def\oi{{\ol{i}}}

\def\ou{{\ol{u}}}
\def\ov{{\ol{v}}}
\def\ow{{\ol{w}}}
\def\eo{{e_{\ol{0}}}}
\def\eu{{e_{\ol{u}}}}
\def\ev{{e_{\ol{v}}}}
\def\ao{{a_{\ol{0}}}}
\def\au{{a_{\ol{u}}}}
\def\av{{a_{\ol{v}}}}
\def\aw{{a_{\ol{w}}}}
\newcommand{\kg}[1]{{\k\langle #1 \rangle}}
\newcommand{\kgt}[1]{{{\k\langle #1 \rangle}^\times_1}}

\def\ul{\underline}
\def\inv{^{-1}}

\def\Dic{{\mathbb{D}}}
\def\rev{{\operatorname{rev}}}

\def\tr{\operatorname{tr}}
\def\co{\operatorname{co}}
\def\ord{\operatorname{ord}}

\def\id{\operatorname{id}}

\def\coev{\operatorname{coev}}
\def\ev{\operatorname{ev}}
\def\TPsi{\underline{\underline{\operatorname{\Psi}}}}
\def\OPsi{\underline{\operatorname{\Psi}}}
\def\Ttc{\underline{\underline{\operatorname{\mathbf{IntTC}}}}}

\def\Aut{\operatorname{Aut}_\o}

\def\OAut{\underline{\operatorname{Aut}}_\o}
\def\oAut{\ol{\operatorname{Aut}}_\o}

\def\Hom{\operatorname{Hom}}
\def\irr{\operatorname{Irr}}

\def\ad{\operatorname{ad}}

\def\Inv{\operatorname{Inv}}

\def\brho{{\operatorname{\overline{\rho}}}}

\def\se{{\operatorname{e}}}
\def\vp{{\varphi}}
\def\tvp{{\tilde{\varphi}}}
\def\ve{{\varepsilon}}
\newcommand\cb[1]{{(I_{#1}, \partial_{#1})}}
\newcommand\C[1]{{#1\mbox{-mod}}}

\newcommand{\YD}[1]{{{}^{#1}_{#1} \mathcal{YD}}}

\def\b{\beta}

\def\jacd{{\genfrac(){.5pt}{2}{-1}{d}}}

\def\A{{A^\times_1}}
\def\B{{B^\times_1}}

\usepackage{microtype}
\begin{document}
\nocite{*}
\title{On stable twist 2-groups of quasi-Hopf algebras}

\author[Ng]{Siu-Hung Ng}
\address{Department of Mathematics, Louisiana State University, Baton Rouge, LA 70803, USA} 
\email{rng@math.lsu.edu}

\author[Wang]{Xingting Wang}
\address{Department of Mathematics, Louisiana State University, Baton Rouge, LA 70803, USA}
\email{xingtingwang@math.lsu.edu}


\date{}

\begin{abstract}
For any finite-dimensional quasi-Hopf algebra $A$, we define the notion of a stable twist of $A$. We show that stable twists form a strict 2-group $\2ust(A)$, which is equivalent to $\OAut(\C{A})^{\rev}$. Given a group homomorphism $\rho: G \to \ost(A)$, we determine the obstruction in
$H^3(G,Z_G(A))$
to lift $\rho$ to a $G$-action on $\C{A}$. We also establish a bijection between such liftings and cocentral cleft extensions of $\k[G]$ by $A$.

The main application of $\2ust(A)$ is the computation of $\OAut(\C{A})$, together with an explicit description of the corresponding tensor autoequivalences. We apply this framework to pointed Hopf algebras $A$ that arise as central extensions of Taft algebras by cyclic groups, and we explicitly determine all of their tensor autoequivalences. The equivariantizations of $\C{A}$ by $\ost(A)$ produce a new family of Hopf algebras $\N(n,\xi)$  of dimension $2nm$, where $\xi$ is a scalar of order $m\mid n$ such that neither $\N(n,\xi)$ nor its dual has the Chevalley property.

In dimension 24, we prove that $\N(6,-1)$ and its dual are precisely the Hopf algebras $H$ for which neither $H$ nor $H^*$ has the Chevalley property. This completes the classification of nonsemisimple Hopf algebras of dimension 24.
\end{abstract}

\maketitle

\tableofcontents

\section{Introduction}
Tensor categories are a central organizing framework across representation theory, quantum topology, and mathematical physics, arising as representation categories of quantum groups \cite{Drin90}, vertex operator algebras, and subfactors, and underpinning topological quantum field theory, topological phases of matter, and quantum invariants of links and $3$-manifolds (see, e.g., \cite{Kas, EGNO}). The symmetries of a tensor category $\CC$ are its \emph{tensor autoequivalences}: they assemble into a strict $2$-group $\OAut(\CC)$ whose group of isomorphism classes $\oAut(\CC)$ governs group actions on $\CC$ and the resulting equivariantizations and graded extensions \cite{ENO2010}. Computing this group is a basic but generally difficult problem. We solve it algebraically for the large class of \emph{integral} tensor categories --- those whose Frobenius--Perron dimensions are integers --- by exploiting the theorem of Etingof and Ostrik that every integral finite tensor category is equivalent to $\C{A}$ for some finite-dimensional quasi-Hopf algebra $A$ \cite[Proposition 6.1.14]{EGNO}. This identification replaces categorical autoequivalences by concrete algebraic data, the \emph{stable twists} of $A$.

A stable twist for a quasi-Hopf algebra $A$ is a pair $(\varphi,J)$ consisting of a gauge transformation $J$ on $A$ and a quasi-bialgebra isomorphism $\varphi\colon A\to A_J$ to the Drinfeld twist of $A$. Given such a pair, the functor ${}_\varphi(-)$ with monoidal structure defined by $J$ yields a tensor autoequivalence of $\C{A}$; conversely, by results of the first author and Schauenburg \cite{NgSchaunburgCentral} (see also \cite{Schauenburg1996, EG2002THA}), every tensor autoequivalence of $\C{A}$ arises in this way. The notion was introduced by Davydov \cite{Davy2010} for Hopf algebras under the name \emph{twisted automorphism}; we extend it here to the full quasi-Hopf setting and show that the set of stable twists carries a natural strict $2$-group structure $\2ust(A)$ equivalent to $\OAut(\C{A})$ (Corollary~\ref{p:eqv}). More broadly, varying $A$ we organize the theory in $3$-groupoid language: stable twists assemble into a strict $3$-groupoid $\3ust$ equivalent to a $3$-groupoid $\Ttc$ of integral tensor categories (Theorem~\ref{thm:31eqv-multi}), with $\2ust(A)$ the automorphism $2$-group at $A$. The group of isomorphism classes, $\ost(A):=\st(A)/\st_0(A)$, is the algebraic counterpart of $\oAut(\C{A})$; it is finite when $A$ is semisimple (Corollary~\ref{p:eqv}(iv)) but may be infinite otherwise, e.g.\ for Taft algebras (\S\ref{sss:A0-stable}).

\smallskip
For the group algebra $\k[G]$ of a finite group $G$ over an algebraically closed field of characteristic zero, Davydov \cite{DavydovGA} showed that $\ost(\k[G])\cong\operatorname{Out}(G)\ltimes\operatorname{Aut}^{\rm anti}_{\rm Tw}(\k[G])$ when $|G|$ is coprime to $6$; here $\operatorname{Aut}^{\rm anti}_{\rm Tw}(\k[G])$ is the group of antipode-preserving twisted automorphisms, and the coprimality condition ensures the vanishing of certain obstruction classes in the Schur multiplier. The lazy cohomology $H^2_\ell(\k^G)$ --- the kernel of the projection to the automorphism component --- was computed by Guillot--Kassel \cite{GuillotKassel2010} in terms of class-preserving outer automorphisms and abelian normal subgroups equipped with invariant alternating forms (Example~\ref{ex:group-algebra}). Cocycle deformations of group algebras were further studied by Galindo--Natale \cite{GalindoNatale2007}, with the semisimple Hopf-algebra case and its Galois objects treated by Mej\'{i}a Casta\~{n}o--Montgomery--Natale--Vega--Walton \cite{CocycleDeform}, and Davydov--Etingof--Nikshych \cite{DEN2018} computed the group of braided tensor autoequivalences for representation categories of small quantum groups at roots of unity.

For a finite-dimensional Hopf algebra $A$, our main computational tool is the short exact sequence
\[
1\longrightarrow H^2_\ell(A^\ast)\longrightarrow \ost(A)\xrightarrow{\;\Pi\;} \OutTw(A)\longrightarrow 1
\]
established in Theorem~\ref{t:general-exact-stable}. Here $H^2_\ell(A^\ast)$ is the lazy cohomology group of $A^\ast$ in the sense of Bichon--Carnovale \cite{BichonCarnovale2006}, capturing the stable twists that are invisible to the algebra-automorphism projection $\Pi$, while $\OutTw(A)$ parametrizes those algebra-automorphism classes that can be lifted to a compatible gauge transformation. This decomposition reduces the computation of $\ost(A)$ to two independent problems --- one cohomological and one automorphism-theoretic --- and provides the organizing principle for our main calculations.

The group $\ost(A)$ is the gateway to constructing $G$-actions on $\C{A}$. Any $G$-action induces a group homomorphism $\brho\colon G\to \ost(A)$, and we construct an obstruction class $[\brho]\in H^3(G,Z_G(A))$, valued in the abelian group $Z_G(A)$ of central grouplike elements of $A$. We prove that $\brho$ lifts to a $G$-action on $\C{A}$ if and only if $[\brho]$ vanishes (Theorem~\ref{LiftGactionTw}); when it does, the equivalence classes of liftings form a torsor over $H^2(G,Z_G(A))$. The methodology here parallels that of Etingof, Nikshych, and Ostrik~\cite{ENO2010} for $G$-extensions of fusion categories, but applied to a different $3$-groupoid target: the algebraic $2$-group $\2ust(A)$ in place of the Brauer--Picard $3$-groupoid (cf.~\S\ref{SS:31cat}).

On the algebraic side, each $G$-action on $\C{A}$ is encoded by concrete data $(\rho,\sigma)$ --- a group homomorphism $\rho\colon G\to\ost(A)$ together with a normalized 2-cocycle $\sigma\colon G\times G\to A$ --- from which we build a quasi-Hopf algebra $A\#_{(\rho,\sigma)}G$ whose representation category recovers the equivariantization $(\C{A})^G$. We show that $A\#_{(\rho,\sigma)}G$ sits in a cocentral cleft extension $1\to A\to A\#_{(\rho,\sigma)}G\to \k[G]\to 1$, and establish a bijection between $G$-equivariantizations on $\C{A}$ and cocentral cleft extensions of $\k[G]$ by $A$, up to equivalence. The notion of cleft extension may be viewed as a Hopf-algebraic realization of various categorical extension constructions \cite{BN2011, BN2014, EG2017, CGO2001}. When $A$ is a Hopf algebra, $A\#_{(\rho,\sigma)}G$ is again an ordinary Hopf algebra, furnishing a systematic source of new examples. In particular, for $G=C_2$, we show that every extension of $\k[C_2]$ by $A$ is automatically cocentral cleft, and nontrivial $C_2$-equivariantizations of $\C{A}$ are in bijection with elements of order two in $\ost(A)$, provided $H^i(C_2,Z_G(A))=1$ for $i=2,3$.

Beyond the Hopf setting we push the framework in two further directions. First, for a quasi-Hopf algebra with a central sub-bialgebra $Z\cong\kk^G$, we prove a \emph{reduction theorem} (Theorem~\ref{t:reduction-general}): provided the twist-liftable outer automorphisms of $(A,\Phi)$ and $(A,\Phi_\omega)$ act trivially on $H^3(G,\kk^\times)$, modifying the associator by a central normalized $3$-cocycle $\phi_\omega\in Z^{\otimes 3}$ --- giving $\Phi_\omega:=\Phi\,\phi_\omega$ --- leaves the stable twist group unchanged, $\ost(A,\Phi_\omega)\cong\ost(A)$. Applied to $A_1(md,\xi)$ this produces an explicit family of genuinely quasi-Hopf $C_2$-equivariantizations $\N^\omega(md,\xi)$ (Definition~\ref{def:N-omega-family}). Second, for a quasitriangular quasi-Hopf algebra $(A,R)$ we introduce the \emph{braided lazy cohomology} $H^2_{\ell,{\rm br}}(A^*)$ as the kernel of a braided refinement of the above exact sequence (Theorem~\ref{t:braided-SES}). This frames two known phenomena uniformly and Hopf-algebraically: for the Drinfeld double $D(G)$ the kernel recovers, as explicit gauge transformations $(\id,K)\in\st^{\rm br}(D(G),R)$, Davydov's identification of the \emph{soft symmetries} of $\C{D(G)}$ \cite{SoftSym2025} with the Bogomolov multiplier $B_0(G)$ \cite{Davydov2014} (Theorem~\ref{t:B0=lazy-br}), while the Etingof--Nikshych--Ostrik correspondence makes the braided stable twist group compute Brauer--Picard groups (Proposition~\ref{p:Tw-BrPic}); related constructions of non-semisimple modular categories via relative monoidal centers appear in work of Laugwitz--Walton \cite{LaugwitzWalton2021}.

We apply these results to the classification of nonsemisimple Hopf algebras. The program of classifying Hopf algebras over an algebraically closed field $\k$ of characteristic zero, organized by dimension, has seen significant progress (cf.\ \cite{YN2011, BG2013, Na, Ng09, Ng2005, Zhu94}): the number of isomorphism classes in every dimension less than 32, except 24, is known to be finite. In dimension 24, semisimple Hopf algebras with $C^*$-structure were classified in \cite{IK}, and nonsemisimple Hopf algebras with the Chevalley property --- meaning the coradical is a Hopf subalgebra, which includes all pointed ones --- were classified in \cite{BGN}. It had remained an open problem for over a decade whether a 24-dimensional Hopf algebra $H$ exists such that neither $H$ nor $H^*$ has the Chevalley property. We resolve this question: the Hopf algebra $\N(6,-1)$ (the member of the family defined below) and its dual are the only two Hopf algebras of dimension 24 lacking the Chevalley property. As a consequence, the classification of nonsemisimple Hopf algebras of dimension 24 is now complete, and the number of isomorphism classes of Hopf algebras in every dimension less than 32 is finite. By contrast, there are infinitely many in dimension 32 \cite{B2000, Gra2000b}.

\medskip
\noindent\textit{The family $\N(md,\xi)$.}\ For coprime integers $d,m>1$ and a primitive $m$-th root of unity $\xi$, we introduce a Hopf algebra $\N(md,\xi)$ of dimension $2m^2 d$ over any field containing $\BQ(\sqrt[m]{-1})$. It is generated by a grouplike $g$, a skew-primitive $x$, and a third generator $u$ with $u^2\in\langle g\rangle$, and realizes the $C_2$-equivariantization of $\C{A_1(md,\xi)}$, where $A_1(md,\xi)$ is the pointed Hopf algebra classified in \cite{AN01}. The full presentation, coalgebra structure, and antipode are given in Definition~\ref{def:N-family}, where $\N(md,\xi)$ arises as the nontrivial extension $A[u^{\pm 1};\tilde\varphi,\tilde J,z]$ from the unique nontrivial class in $\ost(A_1(md,\xi))\cong C_2$.

The computation of $\ost(A_\delta(n,\xi))$ uses the exact sequence framework above. For $A_0(n,\xi)$ (the case $\delta=0$), the sequence identifies $\ost(A_0(n,\xi))$ as a semidirect product of $(\k,+)$ by $\k^\times/\BG_m$ when $n=m$ (Taft case), and as $\k^\times/\BG_m$ when $n>m$; in particular every $2$-cocycle on the Taft algebra $T_{m^2}(\xi)$ arises from a stable twist. For $A_1(md,\xi)$ with $d>1$ (the case $\delta=1$), both $H^2_\ell(A_1^*)$ and $\OutTw(A_1)$ are computable, yielding $\ost(A_1)\cong C_2$ when $\gcd(m,d)=1$ and trivial otherwise. Since $Z_G(A_1)\cong C_d$, the group cohomology $H^i(C_2,Z_G(A_1))$ vanishes for $d$ odd; combined with our main results, this yields a unique nontrivial $C_2$-action on $\C{A_1(md,\xi)}$ up to equivalence, whose equivariantization is precisely $\N(md,\xi)$. The family $\N(md,\xi)$ lacks both the Chevalley and the dual Chevalley properties. We propose the following.
\begin{conjecture}
    Let $p$ be an odd prime and $\k$ be an algebraically closed field of characteristic zero. The Hopf algebra $\N(2p,-1)$ and its dual are the only Hopf algebras in dimension $8p$ without the Chevalley property. 
\end{conjecture}

Together with \cite{BGN}, this conjecture would complete the classification of all nonsemisimple Hopf algebras of dimension $8p$. In Section~\ref{S:clas} we verify it for $p=3$ and prove several general results toward the general case; in particular, the conjecture holds whenever $|G(H)|=2p$. Assuming $|G(H)|=2p$ with both $H$ and $H^*$ nonpointed, \cite{BGN} shows that such $H$ fits into an exact sequence $1\to A_\delta(2p,-1)\to H\to \k[C_2]\to 1$ of Hopf algebras; our bijection between nontrivial extensions of $\k[C_2]$ by $A$ and elements of order two in $\ost(A)$ then yields $H\cong \N(2p,-1)$. The family $\N(2p,-1)$ thereby establishes the existence, in dimension $8p$, of Hopf algebras lacking both the Chevalley and dual Chevalley properties, and leads to a complete classification of nonsemisimple Hopf algebras of dimension $24$.

The paper is organized as follows. In Section~\ref{S:twistG}, we define stable twists for finite-dimensional quasi-Hopf algebras, establish the $2$-group equivalence $\2ust(A)\simeq \OAut(\C{A})$, and prove the general exact sequence \eqref{eq:general-stable-SES} together with its braided counterpart. In Section~\ref{S:Equiv}, we develop the obstruction theory for lifting group homomorphisms $G\to\ost(A)$ to $G$-actions on $\C{A}$, and realize $G$-equivariantizations as cocentral cleft extensions of $\k[G]$ by $A$. In Section~\ref{S:StGrH}, we compute the stable twist groups of the pointed Hopf algebras $A_\delta(n,\xi)$ for $\delta\in\{0,1\}$, and construct the family $\N(md,\xi)$ when $\gcd(m,d)=1$ and $d$ is odd. In Section~\ref{S:StGrQH}, we extend the framework to quasi-Hopf algebras: we prove a reduction theorem for central $3$-cocycle modifications of the associator, study braided stable twists, and connect the kernel of the braided exact sequence to soft symmetries, the Bogomolov multiplier, and the Brauer--Picard group. Finally, in Section~\ref{S:clas}, we classify all $8p$-dimensional nonsemisimple Hopf algebras with $|G(H)|=2p$ for an odd prime $p$, thereby completing the classification of nonsemisimple Hopf algebras of dimension $24$.

\medskip
\noindent\textit{Conventions.}\ Throughout this paper, $\kk$ denotes an algebraically closed field of characteristic zero unless stated otherwise. All algebras and (quasi-)Hopf algebras are  finite-dimensional over $\kk$; all groups are finite; all $n$-cocycles are normalized.  For any finite-dimensional algebra $A$ over $\k$,  $\C{A}$ denotes the category of $A$-modules which are finite-dimensional over $\k$. We adopt the notation and definition of finite tensor categories from \cite{EGNO} and all tensor categories throughout this paper are finite tensor categories over $\k$.


\section{Stable twists of quasi-Hopf algebras}\label{S:twistG}
In this section, we characterize tensor autoequivalences of an integral tensor category $\C{A}$, where $A$ is a finite-dimensional quasi-Hopf algebra, in terms of gauge transformations of $A$ together with compatible quasi-bialgebra isomorphisms --- the \emph{stable twists} of $A$. In \S\ref{SS:31cat} we organize stable twists into a strict $3$-groupoid $\3ust$ and prove its equivalence with a $3$-groupoid of integral tensor categories, so that the automorphism $2$-group at $A$ recovers $\OAut(\C{A})$. In \S\ref{SS:general-SES} we derive, for a finite-dimensional Hopf algebra $A$, the exact sequence computing $\ost(A)$ from the lazy cohomology of $A^*$ and the liftable outer automorphisms of $A$, together with its braided refinement for quasitriangular $A$.

A \emph{quasi-bialgebra} $(A,\Delta,\ve,\Phi)$ is an associative unital $\kk$-algebra $A$ equipped with unital algebra maps $\Delta: A\to A\otimes A$ and $\ve: A\to \kk$, and an invertible element $\Phi\in A\otimes A\otimes A$ satisfying the identities below:
\begin{align}
  (\id \otimes \Delta)\Delta(a)&\,=\Phi\,\left((\Delta\otimes \id)\Delta(a)\right)\,\Phi^{-1},\quad a\in A,\label{coasso1}\\
  (\id \otimes \id \otimes \Delta)(\Phi)(\Delta\otimes \id\otimes \id)(\Phi)&\,=(1\otimes \Phi)(\id \otimes \Delta\otimes \id)(\Phi)(\Phi\otimes 1),\label{coasso2}\\
  (\ve\otimes \id)(\Delta(a))&\, =a=(\id \otimes \ve)(\Delta(a)),\quad a\in A,\label{coasso3} \\
  (\id \otimes \ve\otimes \id)(\Phi)&\,=1\otimes 1.\label{coasso4}
\end{align}
We call $\Delta$ the \emph{comultiplication}, $\ve$ the \emph{counit}, and $\Phi$ the \emph{associator} of $A$. For any $a \in A$, we will use Sweedler's notation $\Delta(a)=\sum_{(a)} a_1\otimes a_2$ or simply $\Delta(a)=a_1\otimes a_2$ with the summation notation suppressed. The \emph{augmentation ideal} of $A$ is defined to be $A^+:={\rm ker}\, \ve$. A quasi-bialgebra map from $(A,\Delta,\ve,\Phi)$ to another quasi-bialgebra $(A',\Delta',\ve',\Phi')$ is an algebra map $\varphi: A\to A'$ such that $\Delta' \circ \varphi=(\varphi\otimes \varphi)\circ \Delta$, $\ve'\circ \varphi=\ve$, and $(\varphi\otimes \varphi\otimes\varphi)(\Phi)=\Phi'$. A quasi-bialgebra $A$ is called a \emph{quasi-Hopf algebra} if there exist an anti-algebra automorphism $S$ of $A$ and elements $\alpha,\beta\in A$ such that 
\begin{align}
  \sum_{(a)}  S(a_{1})\,\alpha\, a_{2}\,=\ve(a)\alpha,\quad 
  \sum_{(a)} a_{1}\,\beta\, S(a_{2})\,=\ve(a)\beta,\quad a\in A, \label{EquasiS}\\
  \sum_i  \Phi_i^1\,\beta\,S(\Phi_i^2)\,\alpha\,\Phi_i^3\,=1,\quad \sum_i S(\overline{\Phi}_i^1)\,\alpha\, \overline{\Phi}_i^2\,\beta\, S(\overline{\Phi}_i^3)\,=1,\label{EAssociator}
\end{align}
where we write the associator as $\Phi=\sum_i\Phi^1_i\otimes \Phi^2_i\otimes \Phi^3_i$ with inverse $\Phi^{-1}=\sum_i \overline{\Phi}^1_i\otimes \overline{\Phi}_i^{2}\otimes \overline{\Phi}_i^3$. A nonzero element $g$ of a quasi-Hopf algebra $A$ is called {\it central grouplike} if $g$ is central in $A$ and $\Delta(g)=g\otimes g$. We denote by $Z_G(A)$ the set of all central group-like elements in $A$. Let $g\in Z_G(A)$ be central grouplike.  By \eqref{EquasiS}, we have $S(g)\alpha g=\ve(g)\alpha=\alpha$. Since both $g$ and $S(g)$ are central, we obtain that 
\begin{align*}
    1\overset{\eqref{EAssociator}}&{=}\sum_i  \Phi_i^1\,\beta\,S(\Phi_i^2)\,\alpha\,\Phi_i^3=\sum_i\Phi_i^1\,\beta\,S(\Phi_i^2)\,(S(g)\alpha g)\,\Phi_i^3\\
&= S(g)\left(\sum_i\Phi_i^1\,\beta\,S(\Phi_i^2)\, \alpha \,\Phi_i^3\right)\,g\overset{\eqref{EAssociator}}{=}S(g)g.
\end{align*}
Similarly, $gS(g)=1$, and hence $g$ is invertible with inverse given by $S(g)$. Moreover, $\Delta(S(g))=\Delta(g^{-1})=\Delta(g)^{-1}=(g\otimes g)^{-1}=S(g)\otimes S(g)$. So $S(g)\in Z_G(A)$. Therefore, $Z_G(A)$ forms an abelian group under the multiplication of $A$ with inverse given by $S$. For fundamental results on quasi-Hopf algebras, the reader is directed to \cite{Drin90} and \cite{Kas}. Specifically, when the associator $\Phi=1\otimes 1\otimes 1$ and $\alpha=\beta=1$ are trivial, any quasi-bialgebra or quasi-Hopf algebra simplifies to a bialgebra or Hopf algebra.

Let $(A,\Delta,\ve,\Phi,S,\alpha,\beta)$ be a quasi-Hopf algebra. A \emph{gauge transformation} on $A$ is an invertible element $J\in A\otimes A$ such that
\begin{equation}\label{twist}
(\id \otimes \ve)(J)=1=(\ve\otimes \id)(J).
\end{equation}
Using any gauge transformation $J$ of $A$, we can define a new quasi-Hopf algebra $A_J$, called the {\it Drinfeld twist} of $A$ by $J$, which, as a quasi-bialgebra, has the same algebra structure and counit $\ve$ as $A$, but new comultiplication 
$$
\Delta_J(x):=J \Delta(x) J\inv\  \text{ for } x \in A\,
$$
and associator
\begin{align}\label{TAss}
\Phi_J=(1\otimes J)(\id \otimes \Delta)(J)\, \Phi\, (\Delta\otimes \id)(J^{-1})(J^{-1}\otimes 1).
\end{align}
The formula~\eqref{TAss} is precisely what is needed to ensure the pentagon axiom holds for $A_J$ with the twisted coproduct $\Delta_J$. To describe the quasi-Hopf algebra structure, write $J=\sum_i f_i\otimes g_i$ with inverse $J^{-1}=\sum_i \ol f_i\otimes \ol g_i$ and set
\begin{align}\label{ABJ}
\alpha_J=\sum_i\, S(\ol f_i)\, \a\, \ol g_i \text{\quad    and\quad  } \b_J=\sum_i\, f_i\,\b\, S(g_i).
\end{align}
Then $A_J=(A,\Delta_J,\ve,\Phi_J,S,\alpha_J,\beta_J)$ becomes a quasi-Hopf algebra. In particular, when $A$ is a Hopf algebra with $\Phi=1\otimes 1\otimes 1$ and $\alpha=\beta=1$, a gauge transformation $J$ is called a \emph{2-cocycle} if additionally $\Phi_J=1\otimes 1\otimes 1$ or 
\begin{equation}\label{2-cocycleJ}
    (1\otimes J)(\id \otimes \Delta)(J)=(J\otimes 1)(\Delta\otimes \id)(J). 
\end{equation}
In this case, $A_J=(A,\Delta_J,\ve,S_J)$ is again a Hopf algebra with the antipode $S_J$ given by
\begin{align}
    S_J(x)=\beta_J\, S(x)\, \alpha_J \text{ for } x\in A\, ,
\end{align}
where 
\begin{align}\label{HABJ}
\alpha_J=\sum_i\, S(\ol f_i)\, \ol g_i \text{\quad    and\quad  } \b_J=\sum_i\, f_i\, S(g_i)
\end{align}
such that $\alpha_J\beta_J=1=\beta_J\alpha_J$. 

Let $\A$ denote the multiplicative group of all invertible elements $x$ in $A$ such that $\ve(x) = 1$. Then, $Z_G(A)$ is a central subgroup of $\A$. Take any $x\in \A$. For any gauge transformation $J$ of $A$,  we denote by $J_x:=(x\otimes x)J\Delta(x)^{-1}$, which is another gauge transformation of $A$. We say that two gauge transformations $J, J'$ of $A$ are \emph{equivalent} if $J'=J_x$ for some $x\in \A$. In particular, for any $x \in \A$, the element $\partial_x := (1 \otimes 1)_x = (x \otimes x) \Delta(x)^{-1}$ is called an \emph{inner gauge transformation} of $A$. Furthermore, we denote by $I_x$ the \emph{inner automorphism} of $A$ defined by $I_x(a) = x a x^{-1}$ for all $a \in A$.

\begin{deff}\label{def:twisted-iso}
Let $A,B$ be two quasi-Hopf algebras.
\begin{enumerate}
    \item A {\it twisted isomorphism} from $A$ to $B$ is a pair $(\varphi, J)$, where $J$ is a gauge transformation of $B$ and $\varphi: A\to B_J$ is a quasi-bialgebra isomorphism. When $A=B$, we call such a twisted isomorphism $(\varphi,J)$ a \emph{stable twist} for $A$.    
     \item We say two twisted isomorphisms $(\varphi, J)$ and $(\varphi', J')$ are {\it equivalent} if there exists $x \in \B$ such that $J'=J_x$ and $\varphi'=I_x \circ \varphi$, and we will abbreviate $I_x \circ \varphi$ as $\varphi_x$.
     \item A twisted isomorphism $(\varphi, J)$ is said to be {\it trivial} if  $J=\partial_x$ for some $x \in \B$.
     \item[(iv)] For any algebra isomorphism $\varphi: A\to B$, and $v \in A^{\o n}$, we simply write $\varphi(v)$ for $\varphi^{\o n}(v)$ in $B^{\otimes n}$.
\end{enumerate}
\end{deff}

Explicitly, a stable twist $(\varphi,J)$ of a quasi-Hopf algebra $(A,\Delta,\ve,\Phi)$ satisfies, for all $a\in A$:
\begin{align}\label{def:st}
\ve(\varphi(a))&=\ve(a)\\
J\Delta(\varphi(a))&=(\varphi\otimes \varphi)\Delta(a)J\\
\varphi(\Phi)&=\Phi_J\, .
\end{align}

\subsection{\texorpdfstring{The $3$-groupoids of finite-dimensional quasi-Hopf algebras and integral tensor categories}{The 3-groupoid Tw of quasi-Hopf algebras and integral tensor categories}}\label{SS:31cat}

In this subsection, we define a strict $3$-groupoid $\3ust$ whose $0$-cells are finite-dimensional quasi-Hopf algebras, with twisted isomorphisms as $1$-cells, intertwiners as $2$-cells, and central grouplike adjustments as $3$-cells. We prove that $\3ust$ is equivalent to a $3$-groupoid $\Ttc$ of integral tensor categories (Theorem~\ref{thm:31eqv-multi}). From that, we derive the strict $2$-group $\2ust(A)$ of stable twists as the automorphism $2$-group of $A$, which is equivalent to $\OAut(\C{A})$ (Corollary~\ref{p:eqv}).

\subsubsection{$3$-groupoid preliminaries}\label{sss:nr-prelim}
We adopt the \emph{$(n,r)$-category} conventions of \cite[\S 2.3]{HTT}; see also \cite[\S 2.2]{ENO2010}. Briefly, an $(n,r)$-category is an $n$-category whose $k$-cells are equivalences for $k>r$, so $(n,0)$-categories are \emph{$n$-groupoids}. The only case we need is a \emph{$3$-groupoid} = $(3,0)$-category $\CC$, in which all $1$-, $2$-, and $3$-cells are invertible. We write $1$-cells as $f\colon C\to C'$, $2$-cells between parallel $1$-cells as $\eta\colon f\Rightarrow f'$, and $3$-cells between parallel $2$-cells as $\gamma\colon\eta\Rrightarrow\eta'$. We refer the reader to \cite[\S 2.3]{HTT} for the full theory and use without further comment the standard notions of cells, compositions, identities, whiskerings, and the reverse $\CC^{\rev}$ (reversing only the $1$-cell direction; cf.~\cite{ENO2010}).

A \emph{strict $(n,r)$-category} is one in which all coherence isomorphisms (associators, unitors, interchange, etc.) are identities; equivalently, compositions of $k$-cells are associative and unital on the nose at every level $k\le n$. A \emph{strict $3$-functor} $F\colon\CC\to\DD$ between strict $3$-groupoids is an assignment on cells of every dimension that respects compositions and identities \emph{on the nose}: $F(\id_C)=\id_{F(C)}$ and $F(g\circ f)=F(g)\circ F(f)$ as equalities (not merely up to invertible higher cells) for $1$-cells, and likewise $F$ preserves vertical and horizontal compositions of $2$- and $3$-cells strictly. This is in contrast with the more general notion of a weak (pseudo) $3$-functor, in which composition is preserved only up to coherent invertible higher cells; we will not need that level of generality, as our $3$-functor $\TPsi$ of \S\ref{sss:31functor} is strict.

A strict $3$-functor $F\colon\CC\to\DD$ is an \emph{equivalence of $3$-groupoids} if it satisfies the following four-level essential surjectivity and full-faithfulness conditions \cite[Def.~2.3.4.1]{HTT}:
\begin{itemize}
\item[\textup{(E0)}] for every $D\in\operatorname{Ob}\DD$, there exists $C\in\operatorname{Ob}\CC$ with $F(C)$ equivalent to $D$ via an invertible $1$-cell;
\item[\textup{(E1)}] for every $1$-cell $f\colon F(C)\to F(C')$ in $\DD$, there is a $1$-cell $\hat f\colon C\to C'$ in $\CC$ with $F(\hat f)$ invertible-$2$-cell isomorphic to $f$;
\item[\textup{(E2)}] for every parallel pair of $1$-cells $f,f'\colon C\to C'$, every $2$-cell $F(f)\Rightarrow F(f')$ in $\DD$ is in the image of $F$ up to a $3$-cell isomorphism;
\item[\textup{(E3)}] for every parallel pair of $2$-cells, $F$ induces a bijection on $3$-cells.
\end{itemize}

For an $(n,r)$-category $\CC$ and an integer $0\le k\le n$, the \emph{$k$-truncation} $\tau_{\le k}\CC$ is the $(k,\min(r,k))$-category obtained in two steps: (i) discard all $j$-cells for $j\ge k+1$; (ii) identify two parallel $k$-cells $f,g$ whenever there exists an invertible $(k+1)$-cell $f\Rightarrow g$ in $\CC$. For a $3$-groupoid $\CC$:
\begin{itemize}
\item $\tau_{\le 0}\CC$ \textup{(}a set\textup{)}: $0$-cells of $\CC$ modulo the equivalence relation generated by invertible $1$-cells; we write $\pi_0\CC:=\tau_{\le 0}\CC$.
\item $\tau_{\le 1}\CC$ \textup{(}an ordinary category\textup{)}: $0$-cells of $\CC$, with morphisms the $2$-isomorphism classes of $1$-cells.
\item $\tau_{\le 2}\CC$ \textup{(}a strict $2$-groupoid\textup{)}: $0$- and $1$-cells of $\CC$, with $2$-cells modulo invertible $3$-cells.
\item $\tau_{\le 3}\CC=\CC$.
\end{itemize}

A \emph{pointed} $(n,r)$-category is a pair $(\CC,*)$ with $*\in\operatorname{Ob}\CC$. The \emph{looping at $*$} is the pointed $(n-1,\max(r-1,0))$-category $\Omega_*\CC$ whose $0$-cells are the $1$-cells $*\to*$ of $\CC$, and whose $k$-cells for $1\le k\le n-1$ are the $(k+1)$-cells of $\CC$ between parallel $k$-cells, with compositions inherited from $\CC$; we take the identity $1$-cell $\id_*$ as the canonical basepoint of $\Omega_*\CC$. Iterating gives $\Omega_*^k\CC$ for $0\le k\le n$. The \emph{homotopy groups} of $(\CC,*)$ are
\[
\pi_0(\CC,*):=\pi_0\CC,\qquad \pi_k(\CC,*):=\pi_0(\Omega_*^k\CC)\quad\text{for }1\le k\le n.
\]
For $k\ge 1$, vertical composition of $k$-cells descends to a group structure on $\pi_k(\CC,*)$.

\begin{deff}\label{def:3ust-multi}
We define a strict $3$-groupoid $\3ust$ as follows:
\begin{itemize}
    \item[\textup{($0$-cells)}] all finite-dimensional quasi-Hopf algebras $(A,\Delta,\ve,\Phi,S,\alpha,\beta)$ over $\kk$;
    \item[\textup{($1$-cells)}] for $A,B$ in $\3ust$, a $1$-cell $A\to B$ is a twisted isomorphism $(\vp,J)\colon A\to B$. The composition of $(\vp,J)\colon A\to B$ and $(\vp',J')\colon B\to C$ is
    \[
    (\vp',J')\circ(\vp,J)\;:=\;\bigl(\vp'\vp,\;\vp'(J)\,J'\bigr)\colon A\to C,
    \]
    and the identity at $A$ is $(\id_A,1\otimes 1)$;
    \item[\textup{($2$-cells)}] for two twisted isomorphisms $(\vp,J),(\vp',J')\colon A\to B$, a $2$-cell $(\vp,J)\Rightarrow(\vp',J')$ is an element $x\in\B$ such that
    \[
    (\vp_x,J_x)\;=\;(I_x\circ\vp,\,(x\otimes x)J\Delta(x)^{-1})\;=\;(\vp',J').
    \]
    We denote this $2$-cell by $\eta_x$. Vertical composition $\eta_{x'}\circ\eta_x:=\eta_{x'x}$ and horizontal composition (whiskering) are induced  by $\eta_x*\eta_y=\eta_{x\vp(y)}$ for two $\eta_x: (\vp,J)\Rightarrow(\vp',J')$ and $\eta_y: (\psi,K)\Rightarrow(\psi',K')$; the identity is $\eta_1$;
    \item[\textup{($3$-cells)}] for two parallel $2$-cells $\eta_x,\eta_{x'}\colon(\vp,J)\Rightarrow(\vp',J')$, a $3$-cell $\eta_x\Rrightarrow\eta_{x'}$ is an element $g\in Z_G(B)$ such that $x'=g\,x$. All three composition operations on $3$-cells are induced by multiplication in $Z_G(B)$; the identity is $1\in Z_G(B)$.
\end{itemize}
\end{deff}

\subsubsection{The strict $2$-group $\2ust(A)$}\label{sss:2group-loop}

Following \cite[\S 2]{ENO2010}, to a strict $2$-group (equivalently, a categorical group) $\mathcal G$ --- a monoidal groupoid in which all objects are invertible --- we attach two homotopy groups:
\[
\pi_0(\mathcal G)\ :=\ \{\text{isomorphism classes of objects}\},\qquad
\pi_1(\mathcal G)\ :=\ \operatorname{Aut}_{\mathcal G}(\1),
\]
the group of isomorphism classes of objects and the automorphism group of the unit object $\1$, respectively; $\pi_1(\mathcal G)$ is abelian by the Eckmann--Hilton argument \cite{EckmannHilton1962,Leinster2004} and carries a natural $\pi_0(\mathcal G)$-action. The \emph{classifying space} $B\mathcal G$ is then a connected $2$-type with
\[
\pi_1(B\mathcal G)=\pi_0(\mathcal G),\qquad \pi_2(B\mathcal G)=\pi_1(\mathcal G),\qquad \pi_k(B\mathcal G)=0\ \ (k\ge 3)
\]
\cite[\S 2.3.5]{HTT}. For an abstract group $G$ and $n\ge1$, the \emph{Eilenberg--MacLane space} $K(G,n)$ is the pointed connected space with $\pi_n(K(G,n))=G$ and $\pi_k(K(G,n))=0$ for $k\ge1$, $k\ne n$.

Fix a finite-dimensional quasi-Hopf algebra $A$, regarded as a $0$-cell of $\3ust$. Following \cite[\S 2.2]{ENO2010}, we organize the stable twists of $A$ into a strict $2$-group (categorical group) $\2ust(A)$, the automorphism $2$-group of $A$, whose objects are the $1$-cells $(\vp,J)\colon A\to A$ \textup{(}stable twists; their set is denoted $\st(A)$\textup{)} and whose morphisms are the $2$-cells $\eta_x$ of $\3ust$. Unpacking Definition~\ref{def:3ust-multi} at the single object $A$:
\begin{itemize}
\item[\textup{(O)}] objects: $(\vp,J)\in\st(A)$, monoidal unit $(\id_A,1\otimes 1)$;
\item[\textup{(M)}] morphisms $(\vp,J)\to(\vp',J')$: elements $x\in\A$ with $(\vp_x,J_x)=(\vp',J')$, composing by multiplication in $\A$.
\end{itemize}
Equipping this with the composition of $1$-cells in $\3ust$ as monoidal product yields a strict monoidal groupoid --- the \emph{strict $2$-group of stable twists of $A$} --- denoted $\2ust(A)$. The grouplike adjustments $g\in Z_G(A)$ \textup{(}the $3$-cells of $\3ust$\textup{)} are recorded not as $2$-morphisms of $\2ust(A)$ but, after delooping, as $\pi_2$ of its classifying space (Proposition~\ref{prop:pin}). Concretely, the composition of $1$-cells in $\3ust$ specializes at $A$ to
\begin{align}\label{GroupTW}
 (\varphi, J)\o (\varphi', J')~:=~\left(\varphi\varphi',\, \varphi(J')\,J\,\right),
\end{align}
which is again a stable twist of $A$. Every object $(\varphi,J)\in\st(A)$ is invertible with $(\varphi,J)^{-1}=(\varphi^{-1},\varphi^{-1}(J^{-1}))$. The $1$-morphisms compose by multiplication in $\A$ (with $\id_{(\vp,J)}:=1$), and for any $x\in\Hom_{\2ust(A)}((\varphi,J),(\varphi',J'))$,
$$
(\vp, J) \o x' :=\vp(x')  \text{ and } x \o (\vp'', J'')  := x
$$
gives the bifunctor $-\o -: \2ust(A) \times  \2ust(A) \to \2ust(A)$.

Let $\st_0(A)$ denote the set of stable twists of the form $(I_x,\partial_x)$ with $x\in\A$; these are precisely the stable twists isomorphic to the unit $(\id_A,1\otimes 1)$ in $\2ust(A)$. Under \eqref{GroupTW}, $\st(A)$ is a group with normal subgroup $\st_0(A)$, and the quotient
\[
\ost(A)\;:=\;\st(A)/\st_0(A)\;=\;\pi_0(\2ust(A))=\pi_1(B\2ust(A))
\]
is the \emph{group of stable twists of $A$}, naturally identified with the set of isomorphism classes of objects of $\2ust(A)$. We further set $\st_t(A)\subseteq\st(A)$ to be the subgroup of \emph{trivial} stable twists \textup{(}those with $J=\partial_x$ for some $x\in\A$\textup{)}, and $\ol\st_t(A):=\st_t(A)/\st_0(A)$. Write $\ol{(\varphi,J)}$ for the image of $(\varphi,J)$ in $\ost(A)$. By induction on \eqref{GroupTW},
\begin{align}\label{powerstable}
    (\varphi, J)^n~~=~\left(\varphi^n, \varphi^{n-1}(J)\cdots \varphi(J)J\right),
\end{align}
and the \emph{order} of $\ol{(\varphi, J)}$ in $\ost(A)$ is the smallest $n\ge 1$ with $(\varphi^n,\varphi^{n-1}(J)\cdots\varphi(J)J)\in\st_0(A)$.

\subsubsection{The $3$-groupoid $\Ttc$ and the strict $3$-functor $\TPsi$}\label{sss:31functor}
We now introduce the $3$-groupoid $\Ttc$ of integral tensor categories and construct a strict $3$-functor $\TPsi\colon\3ust^{\rev}\to\Ttc$. We begin by recalling the standard notions of tensor functor and monoidal natural isomorphism. A {\it tensor functor} from a $\k$-linear tensor category $(\mathcal C,\otimes,a,\1,l,r)$ to another $(\mathcal C',\otimes',a',\1',l',r')$ is a pair $(F,\xi)$ consisting of a $\k$-linear functor $F: \mathcal C\to \mathcal C'$ with $F(\1) \cong \1'$ together with a natural isomorphism
\begin{align}\label{tensorSt}
\xi_{V,W}: F(V\otimes W)~\xrightarrow{\sim}~F(V)\otimes' F(W),\quad \text{for any}\quad V,W\in\mathcal C,
\end{align}
called the \emph{monoidal structure} of $F$, such that the hexagon diagram commutes for all $U,V,W\in \mathcal C$ (cf.~\cite[Definition 2.4.1]{EGNO}). A tensor functor whose underlying $F$ is a $\k$-linear equivalence is called a \emph{tensor equivalence}. Two tensor functors $(F,\xi)$ and $(F',\xi')$ are \emph{monoidally naturally isomorphic} if there is a natural isomorphism $\eta: F\xRightarrow{\sim} F'$ such that
$$
\xymatrix{F(V \o W) \ar[r]^-{\xi_{V,W}}\ar[d]_-{\eta_{V \o W}} & F(V) \o' F(W) \ar[d]^-{\eta_V \o \eta_W}\\
F'(V \o W) \ar[r]^-{\xi'_{V, W}}& F'(V) \o' F'(W)
}
$$
commutes for any $V, W \in \CC$. A \emph{monoidal natural automorphism} of $(F,\xi)$ is a monoidal natural isomorphism $\zeta:(F,\xi)\Rightarrow(F,\xi)$. Up to monoidal natural isomorphism we may always normalize $F(\1)=\1'$ and $\eta_\1=\id_{\1'}$.

Recall further that a finite tensor category $\CC$ is \emph{integral} if the Frobenius--Perron dimensions of all of its objects are integers; equivalently, $\CC$ is tensor equivalent to $\C{A}$ for some finite-dimensional quasi-Hopf algebra $A$ \textup{(}\cite[Proposition 6.1.14]{EGNO}\textup{)}. 

\begin{deff}\label{def:Ttc}
Define the $3$-groupoid $\Ttc$ of integral tensor categories over $\kk$ by:
\begin{itemize}
    \item[\textup{($0$-cells)}] integral tensor categories $\CC$ over $\kk$;
    \item[\textup{($1$-cells)}] tensor equivalences $(F,\xi)\colon\DD\to\CC$;
    \item[\textup{($2$-cells)}] monoidal natural isomorphisms $\eta\colon(F,\xi)\Rightarrow(F',\xi')$;
    \item[\textup{($3$-cells)}] for two parallel $2$-cells $\eta,\eta'\colon(F,\xi)\Rightarrow(F',\xi')$, a $3$-cell $\eta\Rrightarrow\eta'$ is a monoidal natural automorphism $\zeta\colon F\Rightarrow F$ such that $\eta'_V=\eta_V\circ\zeta_V$ for every object $V$ (i.e., $\eta'=\eta\circ\zeta$ as vertical composition of natural transformations).
\end{itemize}
Compositions are the usual horizontal and vertical compositions of tensor equivalences and monoidal natural transformations. By definition, all cells at every level are invertible, so $\Ttc$ is a strict $3$-groupoid.
\end{deff}

For an integral tensor category $\CC$, the \emph{strict $2$-group of tensor autoequivalences} $\OAut(\CC)$ is the categorical group with objects the tensor autoequivalences of $\CC$ and morphisms the monoidal natural isomorphisms. We denote by $\oAut(\CC)$ the set of isomorphism classes of $1$-morphisms of $\OAut(\CC)$, which is an ordinary group. The group $\oAut(\CC)$ is also denoted $\mathrm{Aut}^\otimes(\CC)$ or $\mathrm{Eq}^\otimes(\CC)$ in the literature \cite{ENO2010,DavydovNikshych2013}, and called the \emph{group of tensor autoequivalences} of $\CC$; the strict $2$-group $\OAut(\CC)$ itself is sometimes written $\underline{\mathrm{Aut}}^\otimes(\CC)$ or $\underline{\mathrm{Eq}}^\otimes(\CC)$.

We now construct the strict $3$-functor $\TPsi\colon\3ust^{\rev}\to\Ttc$ cell by cell.

\smallskip\noindent\textit{The construction $({}_\vp(-),\xi^J)$.}
For any twisted isomorphism $(\vp,J)\colon A\to B$, we associate a tensor equivalence
\[
({}_\vp(-),\xi^J)\colon \C{B}\longrightarrow\C{A}
\]
as follows. For $V\in\C{B}$, the underlying vector space ${}_\vp V$ equals $V$ but the $A$-action is pulled back along $\vp$: $a\cdot_\vp v:=\vp(a)\,v$. Any $B$-module map $f\colon V\to W$ then automatically becomes an $A$-module map $\!_\vp(f):=f$. The monoidal structure is multiplication by $J$,
\[
\xi^J_{V,W}\;:=\;\bigl({}_\vp(V\otimes W)\xrightarrow{\,J\cdot\,}{}_\vp V\otimes {}_\vp W\bigr).
\]
The hexagon axiom for $({}_\vp(-),\xi^J)$ reduces to $\vp(\Phi^A)=(\Phi^B)_J$, which is exactly the stability condition.

\smallskip\noindent\textit{The construction $\psi_x$ and $\zeta_g$.}
For a $2$-cell $\eta_x\colon (\vp,J)\Rightarrow(\vp',J')$ in $\3ust$ (with $x\in\B$ and $(\vp_x,J_x)=(\vp',J')$), define a monoidal natural isomorphism
\[
\psi_x\colon ({}_\varphi(-),\xi^J)\xRightarrow{\sim}({}_{\varphi'}(-),\xi^{J'}),\qquad (\psi_x)_V(v):=xv.
\]
Naturality and the monoidal compatibility are checked exactly as before:
\[
\psi_x(a\cdot_\varphi v)=x\varphi(a)v=\varphi_x(a)xv=a\cdot_{\varphi'}\psi_x(v),
\]
and
\[
\left((\psi_x)_V\otimes (\psi_x)_W\right)\circ \xi_{V,W}^J=(x\otimes x)J=J_x\Delta(x)=\xi_{V,W}^{J'}\circ (\psi_x)_{V\otimes W}.
\]
For a $3$-cell $g\in Z_G(B)$ (with $x'=gx$), define a monoidal natural automorphism $\zeta_g$ of $({}_\vp(-),\xi^J)$ by $(\zeta_g)_V(v):=gv$ for $V\in\C{B}$, $v\in V$. The assignments
\[
A\mapsto\C{A},\quad (\vp,J)\mapsto({}_\vp(-),\xi^J),\quad \eta_x\mapsto\psi_x,\quad g\mapsto\zeta_g
\]
are compatible with composition at all four levels \emph{strictly}, that is, as honest equalities rather than up to invertible higher cells:
\begin{align}\label{eq:strict-comp}
({}_{\vp'\vp}(-),\xi^{\vp'(J)J'})\;&=\;({}_\vp(-),\xi^J)\circ({}_{\vp'}(-),\xi^{J'}),\\
\psi_{x'x}=\psi_{x'}\circ\psi_x,\qquad &\zeta_{gg'}=\zeta_g\circ \zeta_{g'}.\notag
\end{align}

\begin{lem}\label{NatST}
Let $A,B$ be finite-dimensional quasi-Hopf algebras, and let
\[
(\vp,J),(\vp',J')\colon A\longrightarrow B
\]
be two twisted isomorphisms. Write $F:=({}_\vp(-),\xi^J)$ and $F':=({}_{\vp'}(-),\xi^{J'})$ for the associated tensor equivalences $\C{B}\to\C{A}$. For any monoidal natural isomorphism $\eta\colon F\Rightarrow F'$, there exists $x\in\B$ with $\vp'=\vp_x$ and $J'=J_x$ such that $\eta=\psi_x$. Moreover, if $\psi_{x'}$ is another such isomorphism $F\Rightarrow F'$, then $x'=gx$ for a unique $g\in Z_G(B)$. In particular, the assignment $\eta_x\mapsto\psi_x$ defines a bijection
\[
\bigl\{\eta_x\colon(\vp,J)\Rightarrow(\vp',J')\bigr\}\;\xrightarrow{\;\sim\;}\;\bigl\{\eta\colon F\Rightarrow F'\bigr\}
\]
which is $Z_G(B)$-equivariant \textup{(}via $3$-cells $g\colon\eta_x\Rrightarrow\eta_{gx}$ on the left and $g\colon\eta\mapsto\eta\cdot\zeta_g$ on the right\textup{)}.
\end{lem}
\begin{proof}
Denote by $r_b\colon B\to B$ right multiplication by $b\in B$. Then $r_b$ is a left $B$-module map, hence a morphism in $\C{B}$, and naturality of $\eta$ on $r_b$ yields
\[
\eta_B(b)\;=\;\eta_B\circ{}_\varphi(r_b)(1)\;=\;{}_{\varphi'}(r_b)\circ\eta_B(1)\;=\;\eta_B(1)\,b.
\]
Set $x:=\eta_B(1)\in B^\times$. For any $V\in\C{B}$ and $v\in V$, the map $f_v\colon B\to V$, $b\mapsto bv$, is left $B$-linear, and naturality ${}_{\varphi'}(f_v)\circ\eta_B=\eta_V\circ{}_\varphi(f_v)$ gives $\eta_V(v)=xv$. Hence $\eta=\psi_x$.

Each $\eta_V$ is a morphism in $\C{A}$, i.e.\ $A$-linear with source action via $\vp$ and target via $\vp'$: $\eta_V(\vp(a)\,v)=\vp'(a)\,\eta_V(v)$. Taking $V=B$, $v=1$ gives $x\,\vp(a)=\vp'(a)\,x$, so $\vp'=I_x\circ\vp=\vp_x$. Since $\eta$ is monoidal, the diagram
\[
\xymatrix{\!_\varphi(V\otimes W) \ar[d]_-{\eta_{V\o W}}\ar[r]^-{J\cdot}&  \!_\varphi V \o \!_\varphi W\ar[d]^-{\eta_V \o \eta_W}\\
~\!_{\varphi'}(V\otimes W)\ar[r]^-{J'\cdot}&\!_{\varphi'} V\otimes \!_{\varphi'} W
}
\]
commutes for all $V,W\in\C{B}$. Taking $V=W=B$ and evaluating both compositions on $1\otimes 1$ yields
\[
J'\,\Delta(x)\;=\;(x\otimes x)\,J\qquad\text{in }B\otimes B,
\]
i.e.\ $J'=(x\otimes x)J\Delta(x)^{-1}=J_x$. Applying $(\ve\otimes\id)$ gives $\ve(x)=1$ and $x\in\B$.

Let $\psi_{x'}\colon F\Rightarrow F'$ be another such monoidal natural isomorphism. By applying previous argument to $x'$, we have $\vp_{x'}=\vp'=\vp_x$ and $J_{x'}=J'=J_x$ with $x'\in\B$. Set $g:=x'x^{-1}\in\B$. From $\vp_{x'}=\vp_x$, i.e.\ $I_{x'}\circ\vp=I_x\circ\vp$, we get $I_g=\id$ on $\vp(A)$, i.e.\ $g$ centralizes $\vp(A)$. Since $\vp\colon A\to B$ is bijective, hence $g\in Z(B)$. Next, $x'\otimes x'=(g\otimes g)(x\otimes x)$ in $B\otimes B$, so the identity $J'\Delta(x')=(x'\otimes x')J$ becomes
\[
J'\,\Delta(x')\;=\;(g\otimes g)(x\otimes x)\,J\;=\;(g\otimes g)\,J'\,\Delta(x).
\]
Since $g\in Z(B)$, we deduce that 
\[
\Delta(x')\,\Delta(x)^{-1}\;=\;g\otimes g.
\]
But $\Delta(x')=\Delta(gx)=\Delta(g)\,\Delta(x)$, so $\Delta(g)=g\otimes g$ and hence $g\in Z_G(B)$. Conversely, for any $g\in Z_G(B)$, $x':=gx$ satisfies $\vp_{x'}=\vp_x$ and $J_{x'}=J_x$, so $\psi_{gx}\colon F\Rightarrow F'$ is another monoidal natural isomorphism of the same type.

The $Z_G(B)$-equivariance of the bijection $\eta_x\mapsto\psi_x$ is then $\psi_{gx}=\zeta_g\circ\psi_x$.
\end{proof}

\begin{cor}\label{lem:nat-auto-ZG-multi}
Let $A,B$ be finite-dimensional quasi-Hopf algebras and $(\vp,J)\colon A\to B$ a twisted isomorphism. The group of monoidal natural automorphisms of the tensor equivalence $F=({}_\vp(-),\xi^J)\colon\C{B}\to\C{A}$ is canonically isomorphic to $Z_G(B)$:
\[
Z_G(B)\;\xrightarrow{\;\sim\;}\;\Aut(F),\qquad z\longmapsto\zeta_z\quad\text{where }(\zeta_z)_V(v):=zv.
\]
\end{cor}
\begin{proof}
Apply Lemma~\ref{NatST} with $(\vp',J')=(\vp,J)$ and basepoint $x=1$, which trivially satisfies $\vp_1=\vp$ and $J_1=J$. The uniqueness clause then shows every $x'\in\B$ with $\vp_{x'}=\vp$ and $J_{x'}=J$ is of the form $x'=g\cdot 1=g$ for some $g\in Z_G(B)$, and the bijection $\eta_x\mapsto\psi_x$ specializes to $z\mapsto\zeta_z$. The group law on $\Aut(F)$ corresponds to multiplication in $Z_G(B)\subseteq B$.
\end{proof}

\begin{thm}\label{thm:31eqv-multi}
The strict $3$-functor 
\[
\TPsi\colon \3ust^{\rev}\xrightarrow{\;\sim\;}\Ttc
\]
defined on cells by
\begin{align*}
A\;&\longmapsto\;\C{A}, & (\vp,J)\colon A\to B\;&\longmapsto\;\bigl({}_\vp(-),\xi^J\bigr)\colon\C{B}\to\C{A},\\
\eta_x\;&\longmapsto\;\psi_x, & g\in Z_G(B)\;&\longmapsto\;\zeta_g,
\end{align*}
is an equivalence of $3$-groupoids.
\end{thm}
\begin{proof}
Strictness of $\TPsi$ is the strict compatibility~\eqref{eq:strict-comp}: composition of $1$-, $2$-, and $3$-cells is preserved on the nose, and identities go to identities. We verify the four-level essential-surjectivity / full-faithfulness conditions (E0)--(E3) of \S\ref{sss:nr-prelim} in turn.

\begin{enumerate}
\item[\textup{($0$-level)}] $\TPsi$ is essentially surjective on objects: every integral tensor category is equivalent to $\C{A}$ for some finite-dimensional quasi-Hopf algebra $A$ \textup{(}\cite[Proposition~6.1.14]{EGNO}\textup{)}.
\item[\textup{($1$-level)}] $\TPsi$ is essentially surjective on $1$-cells modulo $2$-cells: given a tensor equivalence $(F,\xi)\colon\C{B}\to\C{A}$, \cite[Theorem~2.2]{NgSchaunburgCentral} produces a twisted isomorphism $(\vp,J)\colon A\to B$ with $F\cong{}_\vp(-)$ and $\xi$ corresponding to $\xi^J$ under this isomorphism (see also \cite{Schauenburg1996, EG2002THA}).
\item[\textup{($2$-level)}] $\TPsi$ is fully faithful on $2$-cells modulo $3$-cells by Lemma~\ref{NatST}.
\item[\textup{($3$-level)}] $\TPsi$ is fully faithful on $3$-cells by Corollary~\ref{lem:nat-auto-ZG-multi}, which matches $g\in Z_G(B)$ with $\zeta_g\in\mathrm{Aut}_\otimes({}_\vp(-))$.
\end{enumerate}
These four levels of compatible bijections constitute an equivalence of $3$-groupoids in the sense of \cite[Ch.~2]{HTT}.
\end{proof}

\begin{remark}\label{rem:honest-31cat}
Since the $1$-cells of $\3ust$ and $\Ttc$ are twisted \emph{iso}morphisms and tensor \emph{equivalences} respectively, both are $3$-groupoids. There is a natural enlargement of both to honest $(3,1)$-categories with non-invertible $1$-cells, for which Theorem~\ref{thm:31eqv-multi} continues to hold:
\begin{itemize}
\item Let $\3ust^{\rm hom}$ be the $(3,1)$-category with the same $0$-cells as $\3ust$ and with $1$-cells the \emph{twisted homomorphisms} $(\vp,J)\colon A\to B$, where $J$ is a gauge transformation of $B$ and $\vp\colon A\to B_J$ is a quasi-bialgebra \emph{homomorphism} (not necessarily an isomorphism); $2$- and $3$-cells are defined exactly as in Definition~\ref{def:3ust-multi}. Composition and identities are given by the same formulas as in $\3ust$.
\item Let $\Ttc^{\rm hom}$ be the $(3,1)$-category with the same $0$-cells as $\Ttc$ and with $1$-cells \emph{all tensor functors} $(F,\xi)\colon\DD\to\CC$ in the sense of \cite[Def.~4.2.5]{EGNO} \textup{(}i.e., $\kk$-linear faithful exact monoidal functors with $F(\1)\cong\1'$\textup{)}, not only tensor equivalences; $2$- and $3$-cells are unchanged.
\end{itemize}
The cell-by-cell construction $(\vp,J)\mapsto({}_\vp(-),\xi^J)$, $\eta_x\mapsto\psi_x$, $g\mapsto\zeta_g$ of \S\ref{sss:31functor} extends verbatim to $\3ust^{\rm hom}$, since the hexagon axiom for $({}_\vp(-),\xi^J)$ only uses $\vp(\Phi^A)=(\Phi^B)_J$ and does not require $\vp$ to be bijective; faithfulness and exactness of ${}_\vp(-)$ are automatic from the fact that the underlying $\kk$-linear functor is restriction of scalars along $\vp$. The reconstruction theorem of the first author and Schauenburg \cite[Thm.~2.2]{NgSchaunburgCentral} produces, from any tensor functor $(F,\xi)\colon\C{B}\to\C{A}$ between module categories of finite-dimensional quasi-Hopf algebras, a twisted homomorphism $(\vp,J)\colon A\to B$ with $F\cong{}_\vp(-)$ and $\xi$ matching $\xi^J$; their argument constructs $\vp$ and $J$ directly from $F$ and $\xi$ without using bijectivity of $\vp$, so it applies in the enlarged setting. Consequently the strict $3$-functor extends to an equivalence of honest $(3,1)$-categories \textup{(}with $\3ust^{\rm hom}$ on the source side and $\Ttc^{\rm hom}$ on the target side\textup{)}
\[
\TPsi^{\rm hom}\colon (\3ust^{\rm hom})^{\rev}\;\xrightarrow{\;\sim\;}\;\Ttc^{\rm hom}\, .
\]
\end{remark}

\subsubsection{The strict $2$-group equivalence at a single $A$}\label{sss:peqv-recover}
The equivalence of $3$-groupoids of Theorem~\ref{thm:31eqv-multi} descends through every truncation \cite[Cor.~2.3.4.5]{HTT}; specializing to a single $0$-cell $A$ recovers the strict $2$-group equivalence $\2ust(A)\simeq\OAut(\C{A})$ of Corollary~\ref{p:eqv}.

The space relevant for the classification of $G$-actions is the classifying space $B\2ust(A)$ of the strict $2$-group $\2ust(A)$ \textup{(}\S\ref{sss:nr-prelim}\textup{)}, whose homotopy groups we now read off.

\begin{prop}\label{prop:pin}
Let $A$ be a finite-dimensional quasi-Hopf algebra. The classifying space $B\2ust(A)$ of the strict $2$-group of stable twists is a connected $2$-type with
\begin{align*}
\pi_1\bigl(B\2ust(A)\bigr)&\;\cong\;\ost(A)\;\cong\;\oAut(\C{A})\;\cong\;\pi_1(B\OAut(\C{A})),\\
\pi_2\bigl(B\2ust(A)\bigr)&\;\cong\;Z_G(A)\; \cong\; \pi_2(B\OAut(\C{A})),\\
\pi_n\bigl(B\2ust(A)\bigr)&\;\cong\;  \pi_n(B\OAut(\C{A}))\; \cong \;1\qquad(n\ge 3),
\end{align*}
and the action of $\pi_1$ on $\pi_2$ is $\overline{(\vp,J)}\cdot z=\vp(z)$ for $z\in Z_G(A)$. 
\end{prop}
\begin{proof}
Since $\2ust(A)$ is a categorical group, $B\2ust(A)$ is a connected $2$-type with $\pi_1=\pi_0(\2ust(A))$ and $\pi_2=\pi_1(\2ust(A))$. By the notation of \S\ref{sss:notation}, $\pi_0(\2ust(A))=\st(A)/\st_0(A)=\ost(A)$, identified with $\oAut(\C{A})$ via Corollary~\ref{p:eqv}(i). For $\pi_1(\2ust(A))$: by Corollary~\ref{lem:nat-auto-ZG-multi} applied to the unit object $(\id_A,1\otimes 1)$, its automorphism group consists of the $x\in\A$ with $(I_x,\partial_x)=(\id_A,1\otimes 1)$, i.e.\ $x\in Z(A)\cap G(A)=Z_G(A)$; as $\2ust(A)$ has no nonidentity $2$-morphisms, these automorphisms are not further identified, so $\pi_1(\2ust(A))=Z_G(A)$. The $\pi_1$-action on $\pi_2$ is conjugation through the algebra component, $\overline{(\vp,J)}\cdot z=\vp(z)$. Finally $B\2ust(A)\simeq B\OAut(\C{A})$ since $\2ust(A)\simeq\OAut(\C{A})$ (Corollary~\ref{p:eqv}), which gives the displayed comparisons.
\end{proof}

Concretely, there is the strict monoidal functor
\[\underline{\Psi}\colon \2ust(A)^{\rev}~ \overset{\otimes}{\cong}~ \OAut(\C{A}),\]
sending $(\varphi,J)\in\st(A)$ to $({}_\varphi(-),\xi^J)$ and a $\2ust(A)$-morphism $x\colon(\vp,J)\to(\vp',J')$ to $\psi_x$. Tracing the data, $\OPsi$ is monoidal: for $(\varphi,J),(\varphi',J')\in\st(A)$,
\begin{align*}
\OPsi((\varphi,J)\otimes (\varphi',J')) &\,=\OPsi\left(\,\varphi\varphi',\varphi(J')J\right)
\,=\left(\,\!_{\varphi\varphi'}(-),\xi^{\varphi(J')J}\right)\\
&\,=\left(\,\!_{\varphi'}(\!_{\varphi}(-)),\xi^{J'}\varphi(\xi^J)\right)
\,=\OPsi(\varphi',J')\circ \OPsi(\varphi,J),
\end{align*}
and $\OPsi(\id_A,1\otimes 1)=\id_{\C{A}}$.

\begin{cor}\label{p:eqv}
Let $A$ be any finite-dimensional quasi-Hopf algebra. The strict monoidal functor
\begin{align*}
\OPsi\colon \2ust(A)^\rev~\overset{\otimes}{\cong}~ \OAut(\C{A})
\end{align*}
is an equivalence of strict $2$-groups. Moreover:
\begin{enumerate}
\item There is an induced group isomorphism $\overline{\Psi}\colon \ol\st(A)^{op}\xrightarrow{\sim}\oAut(\C{A})$.
\item $\ol\st(A)\cong \ol\st(A^{op})\cong \ol\st(A^{cop})$ as groups.
\item For any gauge transformation $J$ of $A$, $\ol\st(A)\cong \ol\st(A_J)$ as groups.
\item If $A$ is semisimple, then $\ol\st(A)$ is a finite group.
\end{enumerate}
\end{cor}
\begin{proof}
The strict $2$-equivalence is induced by $\TPsi$ (Theorem~\ref{thm:31eqv-multi}) on the automorphism $2$-groups at $A$ and $\C{A}$: an equivalence of $3$-groupoids induces a strict monoidal equivalence of automorphism $2$-groups \cite[Cor.~2.3.4.5]{HTT}. (i) is the induced isomorphism on $\pi_0$ of the $\Hom$-$2$-categories, namely $\ost(A)^{op}\cong\oAut(\C{A})$. (ii) and (iii) follow from (i) and the elementary tensor equivalences $\C{A}\cong\C{A^{op,cop}}$ and $\C{A_J}\cong\C{A}$. (iv) is the finiteness of tensor functors (up to equivalence) between fusion categories \cite[Theorem 2.31]{ENO}.
\end{proof}

\subsubsection{Notation}\label{sss:notation}
We collect the notation here for $\3ust$, the strict $2$-group $\2ust(A)$, and the integral tensor category side.
\begin{center}
\small
\begin{tabular}{@{}ll@{}}
$\3ust$ (two underlines) & strict 3-groupoid of finite-dimensional quasi-Hopf \\ & algebras over a fixed field $\k$\\
$\2ust(A)$ (one underline) & strict $2$-group (categorical group) of stable twists \\ & of $A$: objects $\st(A)$, morphisms $\eta_x$\\
$B\2ust(A)$ & classifying space of $\2ust(A)$, a connected $2$-type\\
$\ost(A)=\pi_0(\2ust(A))$      & the group of stable twists $\st(A)/\st_0(A)=\pi_1(B\2ust(A))$\\
$Z_G(A)=\pi_1(\2ust(A))$ & central grouplikes $=\pi_2(B\2ust(A))$\, .
\end{tabular}
\end{center}
Likewise for integral tensor categories:
\begin{center}
\begin{tabular}{ll}
$\Ttc$ &3-groupoid of integral tensor categories\\
$\OAut(\CC)$ &  strict $2$-group of tensor autoequivalences of $\CC$\\
$B\OAut(\CC)$ & classifying space of $\OAut(\CC)$, a connected $2$-type\\
$\oAut(\CC)=\pi_0(\OAut(\CC))$ & the group of equivalence classes of \\
& tensor autoequivalences of $\CC$\, . 
\end{tabular}
\end{center}

\subsubsection{The braided sub-$3$-groupoid $\3ust^{\rm br}$}\label{sss:braided}
When the quasi-Hopf algebras involved carry quasitriangular structure, a natural sub-$3$-groupoid of $\3ust$ records only those $1$-cells that preserve the braiding. This is the algebraic counterpart of restricting from tensor functors to braided tensor functors. A finite-dimensional \emph{quasitriangular} (QT) quasi-Hopf algebra is a finite-dimensional quasi-Hopf algebra $A=(A,\Delta,\ve,\Phi,S,\a,\beta)$ equipped with an invertible $R$-matrix $R=R^1\otimes R^2\in A\otimes A$ satisfying
\begin{align}
(\Delta\otimes\id)(R)&=\Phi_{231}\,R_{13}\,\overline{\Phi}_{132}\,R_{23}\,\Phi_{123}, \label{eq:QT1}\\
(\id\otimes\Delta)(R)&=\overline{\Phi}_{312}\,R_{13}\,\Phi_{213}\,R_{12}\,\overline{\Phi}_{123}, \label{eq:QT2}\\
\Delta^{\rm op}(a)&=R\,\Delta(a)\,R^{-1}\quad\text{for all }a\in A, \label{eq:QT3}
\end{align}
where $\Phi_{ijk}$ denotes the appropriate permutation of the tensor factors of $\Phi$ (and $\overline{\Phi}=\Phi^{-1}$, as in \eqref{EAssociator}). These identities imply 
\begin{align}
(\ve \otimes \id)(R)=1=(\id \otimes \ve)(R).
\end{align}
The category $\C{A}$ is then a braided integral tensor category, with braiding $c_{V,W}(v\otimes w)=R^2\cdot w\otimes R^1\cdot v$ for any two $A$-modules $v\in V,w\in W$. When $A$ is a Hopf algebra ($\Phi=1\otimes 1\otimes 1$), relations \eqref{eq:QT1}--\eqref{eq:QT3} imply that the $R$-matrix satisfies the quantum Yang--Baxter equation \cite[Thm.~VIII.2.4]{Kas}.

\begin{deff}\label{def:braided-twisted-hom}
Let $(A,R_A)$ and $(B,R_B)$ be two finite-dimensional quasitriangular quasi-Hopf algebras. A twisted isomorphism $(\vp,J)\colon (A,R_A)\to(B,R_B)$ is \emph{braided} if it preserves the $R$-matrix up to gauge, in the sense
\begin{equation}\label{eq:R-cond-multi}
(\vp\otimes\vp)(R_A)\;=\;J_{21}^{\phantom{1}}\,R_B\,J^{-1}\;\in\;B\otimes B,
\end{equation}
where $J_{21}=\tau(J)$ with $\tau$ the tensor flip. When $(A,R_A)=(B,R_B)$, a braided twisted isomorphism is said to be a \emph{braided stable twist}.
\end{deff}

\begin{lem}\label{lem:braided-closed}
Braided twisted isomorphisms are closed under composition and contain all identities; moreover, the condition \eqref{eq:R-cond-multi} is preserved under $2$-cell equivalence and $3$-cell equivalence in $\3ust$.
\end{lem}
\begin{proof}
Take $(\vp,J)\colon (A,R)\to(A',R')$ and $(\vp',J')\colon (A',R')\to(A'',R'')$ both braided, and set $(\vp'',J''):=(\vp',J')\circ(\vp,J)=(\vp'\vp,\vp'(J)J')$. Then
\begin{align*}
(\vp'\vp\otimes\vp'\vp)(R)
&=(\vp'\otimes\vp')\bigl(J_{21}R'J^{-1}\bigr)\\
&=\vp'(J)_{21}\bigl(J'_{21}R''(J')^{-1}\bigr)\vp'(J^{-1})\\
&=(\vp'(J)J')_{21}R''(\vp'(J)J')^{-1},
\end{align*}
verifying \eqref{eq:R-cond-multi} for $(\vp'',J'')$. Moreover, the identity $(\id_A,1\otimes 1)\colon(A,R)\to(A,R)$ trivially satisfies \eqref{eq:R-cond-multi}.

If $(\vp,J)$ is braided and $(\vp',J')=(\vp_x,J_x)$ for $x\in(A')^\times_1$, a routine computation gives
\begin{align*}
(\vp_x\otimes\vp_x)(R)
&=(x\otimes x)(\vp\otimes\vp)(R)(x\otimes x)^{-1}\\
&=(x\otimes x)J_{21}R'J^{-1}(x\otimes x)^{-1}\\
&=(J_x)_{21}(\Delta')^{\rm op}(x)R'\Delta'(x^{-1})J_x^{-1}\\
&=(J_x)_{21}R'J_x^{-1},
\end{align*}
where the last step uses $\Delta'^{\rm op}(x^{-1})=R'\Delta'(x^{-1})(R')^{-1}$ from~\eqref{eq:QT3}. The $3$-cell datum $g\in Z_G(A')$ relates $x'=gx$, so $(\vp_{x'},J_{x'})=(\vp_{gx},J_{gx})$; the same computation, using $g$ central grouplike and $R'$-invariant, gives \eqref{eq:R-cond-multi} for $(\vp_{x'},J_{x'})$.
\end{proof}

Let $(A,R)$ be a finite-dimensional QT quasi-Hopf algebra. The set of \emph{braided stable twists} on $(A,R)$ is $\st^{\rm br}(A,R):=\{(\vp,J)\in\st(A)\mid (\vp\otimes\vp)(R)=J_{21}\,R\,J^{-1}\}$. By Lemma~\ref{lem:braided-closed}, $\st^{\rm br}(A,R)$ is a subgroup of $\st(A)$ containing $\st_0(A)$. We define the \emph{braided stable twist group}\label{d:braided-twist-group} as the quotient group 
\[\ost^{\rm br}(A,R)\,:=\,\st^{\rm br}(A,R)\big/\st_0(A)\, .\]

\begin{deff}\label{def:3ust-br}
The \emph{braided sub-$3$-groupoid} $\3ust^{\rm br}\subseteq\3ust$ has
\begin{itemize}
\item[\textup{($0$-cells)}] All finite-dimensional quasitriangular quasi-Hopf algebras $(A,R)$ over $\kk$;
\item[\textup{($1$-cells)}] braided twisted isomorphisms;
\item[\textup{($2,3$-cells)}] the same $2$- and $3$-cells as in $\3ust$.
\end{itemize}
By Lemma~\ref{lem:braided-closed}, $\3ust^{\rm br}$ is a strict sub-$3$-groupoid of $\3ust$ \textup{(}stable under all compositions and higher cells\textup{)}, whose single-object specialization at $(A,R)$ has $\pi_1$-group equal to the braided stable twist group $\ost^{\rm br}(A,R)$.
\end{deff}

On the categorical side, let $\Ttc^{\rm br}\subseteq\Ttc$ be the corresponding sub-$3$-groupoid whose $0$-cells are braided integral tensor categories, whose $1$-cells are the braided tensor equivalences, and whose $2$- and $3$-cells are unchanged.

The equivalence conditions (E0)--(E3) of \S\ref{sss:nr-prelim} restrict to the braided setting: every braided integral tensor category arises from some QT $(A,R)$, and every braided tensor equivalence $\C{A'}\to\C{A}$ comes from a braided twisted isomorphism by tracking the $R$-matrix through the proof of \cite[Theorem~2.2]{NgSchaunburgCentral}, while the $2$- and $3$-cell levels are unchanged from $\3ust$ and $\Ttc$. Hence the strict $3$-functor $\TPsi$ of Theorem~\ref{thm:31eqv-multi} restricts to a strict $3$-equivalence
\[
\TPsi^{\rm br}\colon\bigl(\3ust^{\rm br}\bigr)^{\rev}\;\xrightarrow{\;\sim\;}\;\Ttc^{\rm br}
\]
between the braided sub-$3$-groupoids. Specializing to a single $0$-cell $(A,R)$ recovers, on the categorical side, the strict $2$-group $\OAut^{\rm br}(\C{A})$ of \emph{braided} tensor autoequivalences of $\C{A}$ (cf.~\cite[\S 2.2]{ENO2010}); this gives the braided analogues of Proposition~\ref{prop:pin} and Corollary~\ref{p:eqv} recorded next.

\begin{prop}\label{br:prop:pin}
Let $(A,R)$ be any finite-dimensional quasitriangular quasi-Hopf algebra. The classifying space $B \2ust^{\rm br}(A,R)$ of the strict $2$-group of braided stable twists is a connected $2$-type with
\begin{align*}
\pi_1\bigl(B \2ust^{\rm br}(A,R)\bigr)&\;\cong\;\ost^{\rm br}(A,R)\;\cong\;\oAut^{\rm br}(\C{A})\;\cong\;\pi_1(B\OAut^{\rm br}(\C{A})),\\
\pi_2\bigl(B \2ust^{\rm br}(A,R)\bigr)&\;\cong\;Z_G(A)\; \cong\; \pi_2(B\OAut^{\rm br}(\C{A})),\\
\pi_n\bigl(B \2ust^{\rm br}(A,R)\bigr)&\;\cong\;  \pi_n(B\OAut^{\rm br}(\C{A}))\; \cong \;1\qquad(n\ge 3),
\end{align*}
and the action of $\pi_1$ on $\pi_2$ is $\overline{(\vp,J)}\cdot z=\vp(z)$ for $z\in Z_G(A)$.
\end{prop}
\begin{proof}
Identical to Proposition~\ref{prop:pin}, with $\2ust(A)$ and $\ost(A)$ replaced by $\2ust^{\rm br}(A,R)$ and $\ost^{\rm br}(A,R)$: the automorphisms of the unit object in $\2ust^{\rm br}(A,R)$ are again $Z_G(A)$ (Corollary~\ref{lem:nat-auto-ZG-multi}), and $B\2ust^{\rm br}(A,R)\simeq B\OAut^{\rm br}(\C{A})$ by Corollary~\ref{c:peqv-br}.
\end{proof}

\begin{cor}\label{c:peqv-br}
Let $(A,R)$ be any finite-dimensional quasitriangular quasi-Hopf algebra. The strict monoidal functor
\[
\OPsi^{\rm br}\colon \2ust^{\rm br}(A,R)^{\rev}\;\overset{\otimes}{\cong}\;\OAut^{\rm br}(\C{A})
\]
is an equivalence of strict $2$-groups. Moreover:
\begin{enumerate}
\item There is an induced group isomorphism $\ol\Psi^{\rm br}\colon\ost^{\rm br}(A,R)^{op}\xrightarrow{\sim}\oAut^{\rm br}(\C{A})$.
\item $\ost^{\rm br}(A,R)\cong\ost^{\rm br}(A^{op,\rm cop}, R^{op,\rm cop})$, where $R^{op,\rm cop}:=R$ as an element of $A^{op}\otimes A^{op}$.
\item For any gauge transformation $J$ of $A$, the twisted QT structure $(A_J,R_J)$ with $R_J:=J_{21}\,R\,J^{-1}$ satisfies $\ost^{\rm br}(A,R)\cong\ost^{\rm br}(A_J,R_J)$ as groups.
\item If $A$ is semisimple, then $\ost^{\rm br}(A,R)$ is a finite group.
\end{enumerate}
\end{cor}
\begin{proof}
This is the $2$-truncation of the braided $3$-equivalence $\TPsi^{\rm br}$ at $(A,R)$; items (i)--(iv) follow exactly as in Corollary~\ref{p:eqv}, using Lemma~\ref{lem:braided-closed} to remain within braided cells.
\end{proof}

\subsection{A general exact sequence for stable twists}\label{SS:general-SES}
In this subsection we establish a short exact sequence that decomposes the stable twist group $\ost(A)$ of a finite-dimensional quasi-Hopf algebra $A$ with central associator $\Phi_A$ into two computable pieces. Theorem~\ref{t:general-exact-stable} exhibits $\ost(A)$ as an extension of the group $\OutTw(A)$ of \emph{twist-liftable} outer automorphisms by the lazy cohomology $H^2_\ell(A^*)$ of \cite{BichonCarnovale2006}. 

Let $A$ be a finite-dimensional quasi-Hopf algebra $(A,\Delta,\ve,\Phi_A,S,\a,\beta)$ over $\kk$ whose associator $\Phi_A$ is central in $A^{\otimes 3}$.  Ordinary Hopf algebras are the special case $\Phi_A=1\otimes 1\otimes 1$; more generally, this includes quasi-Hopf algebras whose associator $\Phi_A$ comes from a central sub-bialgebra of $A$, as treated in Section~\ref{S:StGrQH}. Write the map
\[
\Pi_A:\ost(A)\to \frac{\operatorname{Aut}_{\rm alg}(A)}{\operatorname{Inn}(A)}=:\operatorname{Out}_{\rm alg}(A)
\]
induced by $(\varphi,J)\mapsto \varphi$.

For such an $A$, an \emph{invariant twist} is an invertible element $J\in A\otimes A$ satisfying the counit normalization $(\ve\otimes\id)(J)=(\id\otimes\ve)(J)=1$, the $2$-cocycle condition
\[
d^2(J)\;:=\;(1\otimes J)(\id\otimes\Delta)(J)(\Delta\otimes\id)(J^{-1})(J^{-1}\otimes 1)\;=\;1\otimes 1\otimes 1
\]
and the invariance condition
\[
J\Delta(a)=\Delta(a)J \qquad \text{for all } a\in A.
\]
We denote by $H^2_\ell(A^\ast)$ the group of all invariant twists in $A\otimes A$ under tensor multiplication, modulo the coboundaries $\partial_z=(z\otimes z)\Delta(z^{-1})$ with $z\in Z(A)\cap A_1^\times$. The invariant-twist condition and its coboundaries involve only the coalgebra data $(A,\Delta,\ve)$, so this group is defined for \emph{every} finite-dimensional quasi-Hopf algebra $A$ and is independent of the associator $\Phi_A$. When $A$ is an ordinary Hopf algebra, $A^*$ is a bialgebra and $H^2_\ell(A^*)$ is its lazy cohomology in the sense of \cite{BichonCarnovale2006}, which accounts for the notation; for a general quasi-Hopf algebra $A^*$ need not be a bialgebra, and we take the group above as the definition of $H^2_\ell(A^*)$.

\begin{deff}\label{def:OutTw}
The \emph{group of twist-liftable outer automorphisms of $A$} is the image
\[
\OutTw(A)\;:=\;\operatorname{Im}\bigl(\Pi_A\colon\ost(A)\to\operatorname{Out}_{\rm alg}(A)\bigr)\;\subseteq\;\operatorname{Out}_{\rm alg}(A).
\]
Equivalently, $\OutTw(A)$ consists of those classes $[\vp]\in\operatorname{Out}_{\rm alg}(A)$ for which there exists a gauge transformation $J\in B\otimes B$ making $(\vp,J)\in\st(A)$ a stable twist; we then say that $\vp$ is \emph{twist-liftable} and that $J$ is a \emph{compatible gauge} for $\vp$.
\end{deff}

Our main result reduces the computation of $\ost(A)$ to two independent problems --- one cohomological, one automorphism-theoretic --- and is the organizing tool for our calculations throughout the paper, for $A_\delta(n,\xi)$ in \S\ref{S:StGrH} and for the modified quasi-Hopf algebras $(A,\Phi)$ in \S\ref{S:StGrQH}.

\begin{thm}\label{t:general-exact-stable}
Let $A$ be a finite-dimensional quasi-Hopf algebra with central associator $\Phi_A$. The map $\Pi_A$ is well-defined and fits into a short exact sequence
\begin{equation}\label{eq:general-stable-SES}
1\longrightarrow H^2_\ell(A^\ast)\longrightarrow \ost(A)\xrightarrow{\Pi_A} \OutTw(A)\longrightarrow 1.
\end{equation}
In particular, $H^2_\ell(A^\ast)$ is independent of the central associator $\Phi_A$.
\end{thm}
\begin{proof}
The map $\Pi_A$ is well-defined because equivalent stable twists differ by an element of $A_1^\times$. More precisely, if $(\varphi',J')=(I_u,\partial_u)(\varphi,J)$ for some $u\in A_1^\times$, then $\varphi'=I_u\circ \varphi$, so $\varphi$ and $\varphi'$ define the same class in $\operatorname{Aut}_{\rm alg}(A)/\operatorname{Inn}(A)$.

Surjectivity onto $\OutTw(A)$ is tautological. It remains to identify the kernel. Note that one has $\operatorname{Inn}_{A_1^\times}(A)=\operatorname{Inn}(A)$. Indeed, if $u\in A^\times$, then $\ve(u)\in \kk^\times$ since $\ve$ is an algebra map, and $I_u=I_{\ve(u)^{-1}u}$ with $\ve(\ve(u)^{-1}u)=1$. Let $[(\varphi,J)]\in \ker(\Pi_A)$. Then $\varphi$ is inner, so there exists $u\in A_1^\times$ such that $\varphi=I_u$. Replacing $(\varphi,J)$ by an equivalent stable twist, we may assume
\[
(\varphi,J)\sim (\id_A,K),\qquad K=(u^{-1}\otimes u^{-1})J\Delta(u).
\]
Since the first component is now $\id_A$, the intertwiner condition for a stable twist gives $K\Delta(a)=\Delta(a)K$ for all $a\in A$, and the associator condition $\id^{\otimes 3}(\Phi_A)=\Phi_A=(\Phi_A)_K=d^2(K)\Phi_A$ together with the centrality of $\Phi_A$ in $A^{\otimes 3}$ forces $d^2(K)=1\otimes 1\otimes 1$. Hence $K$ is an invariant twist. Conversely, every invariant twist $K$ gives a stable twist of the form $(\id_A,K)$.

Finally, two invariant twists $(\id_A,K)$ and $(\id_A,K')$ define the same class in $\ost(A)$ if and only if they differ by a coboundary coming from some $z\in Z(A)\cap A_1^\times$, namely $K'=\partial_z K$. Thus $\ker(\Pi_A)$ is exactly $H^2_\ell(A^\ast)$, which proves \eqref{eq:general-stable-SES}. Since neither the invariant-twist condition nor the coboundary relation involves $\Phi_A$, the kernel $H^2_\ell(A^*)$ is independent of the central associator.
\end{proof}

The following recovers the classification of Etingof--Nikshych--Ostrik for $\operatorname{Vec}_G^\omega$ as a special case. 

\begin{example}[Pointed fusion categories]\label{ex:pointed-SES}
Let $A=\kk^G$ with associator $\Phi_\omega=\sum_{a,b,c\in G}\omega(a,b,c)\,\delta_a\otimes\delta_b\otimes\delta_c$ for a normalized $3$-cocycle $\omega$ on $G$. The representation category of $A$ is the pointed fusion category $\C{A}\simeq\operatorname{Vec}_G^\omega$. Every algebra automorphism of $\kk^G$ comes from a unique group automorphism of $G$, and the twist-liftability condition forces this group automorphism to preserve $[\omega]\in H^3(G,\kk^\times)$. Theorem~\ref{t:general-exact-stable} then specializes to
\begin{equation}\label{eq:VecGom-SES-intro}
1\;\longrightarrow\;H^2(G,\kk^\times)\;\longrightarrow\;\ost(\kk^G,\Phi_\omega)\;\xrightarrow{\;\Pi\;}\;\operatorname{Aut}(G;[\omega])\;\longrightarrow\;1,
\end{equation}
recovering the categorical sequence of \cite{ENO2010,Davy2010} mentioned above. 
\end{example}

Let $\operatorname{Aut}_{\rm qH}(A)$ denote the group of quasi-Hopf algebra automorphisms of $A$, i.e., algebra automorphisms $\varphi$ such that $(\varphi\otimes\varphi)\Delta=\Delta\varphi$, $\ve\circ\varphi=\ve$, $\varphi^{\otimes 3}(\Phi_A)=\Phi_A$. When $A$ is an ordinary Hopf algebra ($\Phi_A=1$), $\operatorname{Aut}_{\rm qH}(A)=\operatorname{Aut}_{\rm Hopf}(A)$.

\begin{prop}\label{p:general-splitting}
The map $[\varphi]\mapsto \overline{(\varphi,1\otimes 1)}$ defines an injective group homomorphism
\[
s\colon \operatorname{Aut}_{\rm qH}(A)/\operatorname{Inn}_{G(A)}(A)\hookrightarrow \ost(A).
\]
In particular, the image of $\operatorname{Aut}_{\rm qH}(A)$ under the composition $\Pi_A\circ s$ is a subgroup of $\OutTw(A)$, whose kernel is
\[
\frac{\operatorname{Aut}_{\rm qH}(A)\cap \operatorname{Inn}(A)}{\operatorname{Inn}_{G(A)}(A)}\, . 
\]
\end{prop}
\begin{proof}
If $\varphi\in\operatorname{Aut}_{\rm qH}(A)$, then $(\varphi,1\otimes 1)$ is a stable twist of $A$. If $\varphi'=I_g\circ \varphi$ with $g\in G(A)$, then
\[
\partial_g=(g\otimes g)\Delta(g)^{-1}=1\otimes 1,
\]
so $(\varphi',1\otimes 1)$ and $(\varphi,1\otimes 1)$ define the same element of $\ost(A)$, proving well-definedness. Conversely, if $\overline{(\varphi,1\otimes 1)}=\overline{(\varphi',1\otimes 1)}$ in $\ost(A)$, then $\varphi'=I_u\circ\varphi$ for some $u\in A_1^\times$ with $\partial_u=1\otimes 1$, i.e., $\Delta(u)=u\otimes u$. Hence $u\in G(A)$ and $[\varphi]=[\varphi']$ in $\operatorname{Aut}_{\rm qH}(A)/\operatorname{Inn}_{G(A)}(A)$, proving injectivity. Finally, the kernel of $\Pi_A\circ s$ is clearly from the definition. 
\end{proof}

\begin{example}\label{ex:group-algebra}
Let $G$ be a finite group of order coprime to $6$ and $A=\k[G]$. Then $\ost(\k[G])$ is finite (Corollary~\ref{p:eqv}(iv)), and Theorem~\ref{t:general-exact-stable} reads
\begin{equation}\label{eq:group-algebra-SES}
1\longrightarrow H^2_\ell(\k^G)\longrightarrow \ost(\k[G])\xrightarrow{\;\Pi\;} \OutTw(\k[G])\longrightarrow 1\,.
\end{equation}
Davydov \cite[Theorem~6.5]{DavydovGA} computed the total group as a semidirect product
\begin{equation}\label{eq:group-algebra-sdp}
\ost(\k[G])\;\cong\; \operatorname{Out}(G)\ltimes \operatorname{Aut}^{\rm anti}_{\rm Tw}(\k[G]),
\end{equation}
where $\operatorname{Aut}^{\rm anti}_{\rm Tw}(\k[G])$ is the abelian group of \emph{anti-symmetric invariant twists}. By \cite[Theorem~2.4]{DavydovGA}, the latter is parametrized by pairs $(B,b)$ with $B\trianglelefteq G$ an abelian normal subgroup and $b$ a $G$-invariant nondegenerate alternating bimultiplicative form $B\times B\to\k^\times$.

The two flanking terms of \eqref{eq:group-algebra-SES} are then read off \eqref{eq:group-algebra-sdp}. Write $\operatorname{Out}_c(G):=\operatorname{Aut}_c(G)/\operatorname{Inn}(G)$ for the group of class-preserving outer automorphisms; an automorphism $\varphi\in\operatorname{Aut}(G)$ becomes inner in $\k[G]$ iff $\varphi$ is class-preserving \cite[(2.5)]{GuillotKassel2010}. Hence
\[
\OutTw(\k[G])\;\cong\; \operatorname{Out}(G)/\operatorname{Out}_c(G),\qquad H^2_\ell(\k^G)\;\cong\; \operatorname{Out}_c(G)\ltimes \operatorname{Aut}^{\rm anti}_{\rm Tw}(\k[G]).
\]
Dually, Guillot--Kassel \cite[Theorem~4.5]{GuillotKassel2010} parametrize $H^2_\ell(\k[G])$ by the same pairs $(B,b)$ modulo $\operatorname{Out}_c(G)$. The sequence \eqref{eq:group-algebra-SES} splits iff $\operatorname{Out}(G)\twoheadrightarrow\operatorname{Out}(G)/\operatorname{Out}_c(G)$ has a section.

For $G$ a nonabelian simple group of order coprime to $6$, $\operatorname{Aut}^{\rm anti}_{\rm Tw}(\k[G])=1$ (no nontrivial abelian normal subgroup) and $\operatorname{Out}_c(G)=1$ (Burnside; \cite[\S 4]{DavydovGA}), so $\ost(\k[G])\cong\operatorname{Out}(G)$.
\end{example}

\subsubsection{The braided variant}\label{SS:braided-SES}

Let $(A,R)$ be a finite-dimensional QT quasi-Hopf algebra, with braided stable twist group $\ost^{\rm br}(A,R)\subseteq\ost(A)$ (\S\ref{SS:31cat}). Write $\Pi^{\rm br}$ for the restriction of $\Pi_A$ to $\ost^{\rm br}(A,R)$, and set $\OutTwBr(A,R):=\operatorname{Im}(\Pi^{\rm br})\subseteq\operatorname{Out}_{\rm alg}(A)$. Its kernel is the \emph{braided lazy cohomology} $H^2_{\ell,{\rm br}}(A^*)$, consisting of the classes in $H^2_\ell(A^*)$ of invariant twists $J$ that additionally satisfy $J_{21}RJ\inv=R$. Like $H^2_\ell(A^*)$, this group is defined for \emph{every} finite-dimensional QT quasi-Hopf algebra $(A,R)$ and is independent of the associator $\Phi_A$.

\begin{thm}\label{t:braided-SES}
For a finite-dimensional QT quasi-Hopf algebra $(A,R)$ with central associator $\Phi_A$, the map $\Pi^{\rm br}$ fits into a short exact sequence
\begin{equation}\label{eq:braided-SES}
1\longrightarrow H^2_{\ell,{\rm br}}(A^*)\longrightarrow\ost^{\rm br}(A,R)\xrightarrow{\Pi^{\rm br}}\OutTwBr(A,R)\longrightarrow 1\,.
\end{equation}
\end{thm}
\begin{proof}
Identical to Theorem~\ref{t:general-exact-stable}, with the braided compatibility $(\vp\o\vp)(R)=J_{21}RJ\inv$ imposed at each step. The kernel consists of stable twists $(\id,K)$ with $K$ an invariant $2$-cocycle satisfying $K_{21}RK\inv=R$.
\end{proof}

The kernel $H^2_{\ell,{\rm br}}(A^*)$ is the main invariant of our braided theory. In \S\ref{SS:soft} we relate it to the soft symmetries of $\C{A}$: for $A=D(G)$ it recovers Davydov's identification with the Bogomolov multiplier $B_0(G)$ \cite{Davydov2014} (Theorem~\ref{t:B0=lazy-br}), and through the Etingof--Nikshych--Ostrik correspondence it computes Brauer--Picard groups (Proposition~\ref{p:Tw-BrPic}).

\section{$G$-equivariantizations and cocentral extensions}\label{S:Equiv}

For a finite-dimensional quasi-Hopf algebra $A$ and a finite group $G$, we classify $G$-actions on $\C{A}$ via the Postnikov classification of such actions in terms of the strict $2$-group $\2ust(A)$ of stable twists (\S\ref{SS:Gaction}), realize the resulting equivariantization $(\C{A})^G$ as the representation category of an explicit quasi-Hopf algebra $A\#_{(\rho,\sigma)}G$ (\S\ref{SS:G-equiv-qHopf}), and identify these equivariantizations with cocentral cleft extensions of $\k[G]$ by $A$ (\S\ref{SS:CCE}). \S\ref{SS:Examples} works out the cyclic case in detail, and \S\ref{SS:braided-equiv} collects the braided counterparts when $A$ is quasitriangular.

\subsection{$G$-actions via the strict $2$-group of stable twists}\label{SS:Gaction}

\subsubsection{$G$-actions on monoidal categories and the Postnikov classification}\label{sss:Gaction-formalism}
We recall the classical notion of a $G$-action on a monoidal category and the homotopy-theoretic Postnikov classification of such actions, following \cite{ENO2010,Brown1982,HTT}, then realize it for $\CC=\C{A}$ at the level of stable twists.

Let $G$ be a finite group. Denote by $\ul G$ the discrete monoidal category with object set $G$, only identity morphisms, monoidal product $g\otimes h:=gh$, and unit $e\in G$; viewed as a strict $2$-group, its classifying space is $BG$, the Eilenberg--MacLane space $K(G,1)$. For a monoidal category $\CC$, a \emph{$G$-action} on $\CC$ is a monoidal functor
\[
\zeta\colon\ul G\;\longrightarrow\;\OAut(\CC),
\]
and two $G$-actions are \emph{equivalent} if their underlying monoidal functors are monoidally naturally isomorphic. Passing to classifying spaces, equivalence classes of $G$-actions on $\CC$ correspond to pointed homotopy classes of maps $BG\to B\OAut(\CC)$.

For a connected pointed space $X$, pointed homotopy classes of maps $BG\to X$ are classified by Postnikov obstruction theory. When $X$ is a $2$-type with $\pi_1(X)=\Gamma$ and $\pi_2(X)=M$ \textup{(}a $\Gamma$-module\textup{)}, the homotopy type of $X$ is determined by the Postnikov $k$-invariant $k_X\in H^3(\Gamma,M)$, and the classification reads:
\begin{itemize}
\item[\textup{(i)}] the homotopy class induces a group homomorphism $\brho\colon G\to\Gamma$ on $\pi_1$;
\item[\textup{(ii)}] $\brho$ lifts to a pointed homotopy class $BG\to X$ iff the pullback $\brho^*k_X\in H^3(G,M)$ vanishes;
\item[\textup{(iii)}] when $\brho^*k_X=0$, equivalence classes of lifts form an $H^2(G,M)$-torsor.
\end{itemize}
See \cite[Ch.~IV.6]{Brown1982} for the group-cohomological proof. 

The Postnikov classification above applies to $G$-actions on any finite tensor category $\CC$, with $X=B\OAut(\CC)$ as the target. We restrict our attention to integral $\CC$, since by \cite[Proposition~6.1.14]{EGNO} these are precisely the categories $\C{A}$ for a finite-dimensional quasi-Hopf algebra $A$, and our equivalence of $3$-groupoids (Theorem~\ref{thm:31eqv-multi}, Corollary~\ref{p:eqv}) replaces $\OAut(\C{A})$ by the strict $2$-group $\2ust(A)$ of stable twists of \S\ref{sss:2group-loop}. By Proposition~\ref{prop:pin}, the classifying space $B\2ust(A)$ has homotopy invariants
\begin{equation}\label{eq:pi-2ust}
\pi_1\bigl(B\2ust(A)\bigr)\;=\;\ost(A),\qquad \pi_2\bigl(B\2ust(A)\bigr)\;=\;Z_G(A),
\end{equation}
with $\pi_1$ acting on $\pi_2$ by $\overline{(\vp,J)}\cdot z=\vp(z)$. Applying the Postnikov classification to $X=B\2ust(A)$ with $\Gamma=\ost(A)$ and $M=Z_G(A)$ then classifies $G$-actions on $\C{A}$ --- i.e., pointed maps $BG\to B\OAut(\C{A})\simeq B\2ust(A)$ (Corollary~\ref{p:eqv}) --- by a homomorphism $\brho\colon G\to\ost(A)$, vanishing of the obstruction $\brho^*k_{B\2ust(A)}\in H^3(G,Z_G(A))$, and an $H^2(G,Z_G(A))$-torsor of lifts.

\subsubsection{Cocycle realization}\label{sss:Gaction-cocycle}
We now realize the Postnikov classification of \S\ref{sss:Gaction-formalism} at the level of algebraic data on $A$. By Corollary~\ref{p:eqv}, a $G$-action on $\C{A}$ --- a monoidal functor $\ul G\to\OAut(\C{A})$ --- translates into the intrinsic datum of a monoidal functor $\ul G\to\2ust(A)$, made precise by the commutative diagram
\[
\xymatrix{
\underline{G}\ar[r]\ar[d]_-{\Inv} & \OAut(\C{A})\\
\underline{G}^\rev\ar[r]^-{\widetilde{\rho}} & \2ust(A)^\rev\ar[u]_-{\OPsi}
}\, ,
\]
where the strict monoidal functor $\Inv: \ul G \to \ul G^\rev$ assigns each object of $\ul G$ to its inverse. We summarize the statement as the following lemma.

\begin{lem}\label{GactionTw}
  Any $G$-action on $\C{A}$ is isomorphic to
  \[
  \OPsi\,\tilde\rho\, \Inv: \ul G \to \OAut(\C{A})
  \]
  for some monoidal functor $\rho: \ul G \to \2ust(A)$, where $\tilde\rho: \ul G^\rev \to \2ust(A)^\rev$ is the reverse functor induced from $\rho$.
  In particular, two $G$-actions on $\C{A}$ are equivalent if and only if their corresponding monoidal functors from $\underline{G}$ to $\2ust(A)$ are isomorphic.  
\end{lem}

By definition, any monoidal functor $(\rho,\sigma): \underline{G}\to \2ust(A)$ consists of the following data:
\begin{itemize}
    \item[(a)] $\rho:\underline{G}\to \2ust(A)$ is a functor such that for each $g\in G$, $\rho(g)=(\varphi^g,J^g)$ for some stable twist $(\varphi^g,J^g)$ of $A$;
    \item[(b)] $\rho(e)=(\id_A,1\otimes 1)$;
    \item[(c)] For each pair $g,h\in G$, a morphism $\sigma(g,h)\in \Hom_{\2ust(A)}(\rho(gh) ,\rho(g) \otimes \rho(h))$ with $\sigma(g,h)\in \A$ such that the following diagrams
\begin{align}\label{Gaction}
\xymatrix{
\rho(ghk)\ar[d]_-{\sigma(gh,k)}\ar[r]^-{\sigma(g,hk)} & \rho(g)\otimes \rho(hk)\ar[d]^-{\rho(g)\otimes \sigma(h,k)}\\
\rho(gh)\otimes \rho(k)\ar[r]_-{\sigma(g,h)\otimes \rho(k)} & \rho(g)\otimes \rho(h)\otimes \rho(k)
}
\quad\text{and}\quad 
\xymatrix{
\rho(g)\ar[r]^-{\sigma(e,g)}\ar[d]_-{\sigma(g,e)}\ar[dr]^-{\id} & \rho(e)\otimes \rho(g)\ar[d]^-{=}\\
\rho(g)\otimes \rho(e)\ar[r]_-{=}& \rho(g)
}
\end{align}
commute for all $g,h,k\in G$.
\end{itemize}

\begin{prop}\label{propG}
A monoidal functor $(\rho,\sigma): \underline{G}\to \2ust(A)$ is determined by a collection of stable twists $\{(\varphi^g,J^g)\,|\, g\in G\}$ with $(\varphi^e,J^e)=(\id_A,1\otimes 1)$ and a convolution invertible map $\sigma: G\times G\to A$ such that 
 \begin{itemize}
     \item[(1)] $G$ measures $A$ via $g\cdot a=\varphi^g(a)$ such that 
\begin{align}\label{twistG}
g\cdot (h\cdot a)=\sigma(g,h)\, \left((gh)\cdot a\right)\, \sigma^{-1}(g,h)
\end{align} 
for all $g,h\in G$ and $a\in A$.
     \item[(2)] $\sigma$ is a 2-cocycle such that $\sigma(g,e)=\sigma(e,g)=1$ and 
          \begin{align}\label{2-cocycle}
 [g\cdot \sigma(h,k)]\, \sigma(g,hk)=\sigma(g,h)\,\sigma(gh,k)   \end{align}
for all $g,h,k\in G$.
     \item[(3)] 
     Additionally, the following identity  
\begin{align}\label{Gnat2}
(\sigma(g,h)\otimes \sigma(g,h))\,J^{gh}\,\Delta(\sigma(g,h))^{-1}~=~\varphi^g(J^h)J^g
\end{align}
 holds for all $g,h\in G$.
 \end{itemize}
\end{prop}
\begin{proof}
By the discussion above, a monoidal functor $\underline{G}\to \2ust(A)$ is exactly such data $(\rho,\sigma)$. Since
\begin{align*}
  \sigma(g,h): \rho(gh)&\,=(\varphi^{gh},J^{gh})\xrightarrow{\sim} \rho(g)\otimes \rho(h)\\
  &\,=(\varphi^g,J^g)\otimes (\varphi^h,J^h)=(\varphi^g\varphi^h,\varphi^g(J^h)J^g)
\end{align*}
is an isomorphism in $\2ust(A)$, \eqref{twistG} and \eqref{Gnat2} hold with $\sigma(g,h)\in\A$, and (2) is the commutativity of \eqref{Gaction}.

Conversely, assume (1)--(3). As $(\varphi^g,J^g)$ is a stable twist, $\varphi^g: A\to A_{J^g}$ is a quasi-bialgebra isomorphism, so $\ve\circ\varphi^g=\ve$ and $(\ve\otimes 1)(J^g)=1$; applying $\ve\otimes 1$ to \eqref{Gnat2} gives $\ve(\sigma(g,h))=1$, and convolution-invertibility puts the range of $\sigma$ in $\A$. It is then routine that $(\rho,\sigma)$ is a well-defined monoidal functor $\underline{G}\to\2ust(A)$.
\end{proof}

Let $\brho: G\to \ost(A)$ be a group homomorphism. For each $g\in G$, we can write $\brho(g)=\overline{(\varphi^g,J^g)}$ for some representative stable twist $(\varphi^g,J^g)$ of $A$ with $(\varphi^e,J^e)=(\id_A,1\otimes 1)$.  We observe that $\brho$ induces a left $G$-module structure on $Z_G(A)$ via $g\cdot \alpha=\varphi^g(\alpha)$ for any $g\in G$ and $\alpha\in Z_G(A)$. Since $\varphi^g: A\to A_{J^g}$ is a quasi-bialgebra isomorphism, $\varphi^g(\alpha)$ is central in $A$. Moreover, we have
\begin{align*}
  \Delta(\varphi^g(\alpha))&\,= (J^g)^{-1}\Delta_{J^g}(\varphi^g(\alpha))J^g=(J^g)^{-1}(\varphi^g\otimes \varphi^g)\Delta(\alpha)J^g\\
  &\,=(J^g)^{-1}(\varphi^g(\alpha)\otimes \varphi^g(\alpha))J^g=\varphi^g(\alpha)\otimes \varphi^g(\alpha).  
\end{align*}
Hence, the abelian group $Z_G(A)$ is stable under $\varphi^g$ and is indeed a $G$-module, since $\varphi^g \varphi^h(\alpha) = I_{\sigma(g,h)} \circ \varphi^{gh}(\alpha) = \varphi^{gh}(\alpha)$ for all $g, h \in G$, with $\varphi^e = \mathrm{id}_A$. Note that the $G$-action on $Z_G(A)$ does not depend on the choice of the representatives $(\varphi^g,J^g)$ and hence is uniquely determined by the group homomorphism $\overline \rho: G\to \ost(A)$.

We can always lift $\brho$ to a functor $\rho: \ul G\to \2ust(A)$ via $\rho(g)=(\varphi^g,J^g)$. Since $\brho(gh)=\brho(g)\brho(h)$ in $\ost(A)$, we have $\rho(gh)\cong \rho(g)\otimes \rho(h)$ in $\2ust(A)$  for all $g,h\in G$. So there is some map $\sigma: G\times G\to \A$ such that $\sigma(g,h)\in \Hom_{\2ust(A)}(\rho(gh),\rho(g)\otimes \rho(h))$ or equivalently
\begin{align*}
    (I_{\sigma(g,h)},\partial_{\sigma(g,h)})(\varphi^{gh},J^{gh})=(\varphi^g,J^g)(\varphi^h,J^h)=(\varphi^g\varphi^h,\varphi^g(J^h)J^g)
\end{align*}
with $\sigma(g,e)=\sigma(e,g)=1$. Whether such a $\sigma$ can be chosen to satisfy the strict $2$-cocycle identity \eqref{2-cocycle} --- i.e., whether $\rho\colon\ul G\to\2ust(A)$ promotes to a monoidal functor --- is controlled by the cocycle representative of the Postnikov $k$-invariant of $B\2ust(A)$ pulled back along $B\brho$, as recorded in the following lemma.

\begin{lem}\label{lem:3co}
Retain the notations above. There is a normalized $3$-cocycle $\tau\in Z^3(G,Z_G(A))$ such that
\begin{equation}\label{eq:tau-cocycle}
\tau(g,h,k)\left([g\cdot \sigma(h,k)]\, \sigma(g,hk)\right)=\sigma(g,h)\,\sigma(gh,k)\quad\text{for all }g,h,k\in G,
\end{equation}
and the cohomology class $[\tau]\in H^3(G,Z_G(A))$ depends only on $\brho$.
\end{lem}
\begin{proof}
Both $\sigma(g,h)\sigma(gh,k)$ and $[g\cdot\sigma(h,k)]\sigma(g,hk)$ are morphisms $\rho(ghk)\to\rho(g)\otimes\rho(h)\otimes\rho(k)$ in $\2ust(A)$, so by Corollary~\ref{lem:nat-auto-ZG-multi} they differ by an element of $Z_G(A)$; thus
\[
\tau(g,h,k)\;:=\;\sigma(g,h)\,\sigma(gh,k)\bigl([g\cdot\sigma(h,k)]\,\sigma(g,hk)\bigr)^{-1}\;\in\;Z_G(A)
\]
is well-defined and satisfies \eqref{eq:tau-cocycle}, with $\tau(g,h,e)=\tau(g,e,k)=\tau(e,h,k)=1$ from $\sigma(g,e)=\sigma(e,g)=1$ and $G$-action $g\cdot\alpha:=\vp^g(\alpha)$. The identity $\delta\tau=1$ is the pentagon coherence for the $k$-invariant of a $2$-type: the five bracketings of $\rho(g)\rho(h)\rho(k)\rho(l)$ are the vertices of the Stasheff pentagon, and strict associativity of $\otimes$ forces
\[
\bigl[g\cdot\tau(h,k,l)\bigr]\cdot\tau(gh,k,l)^{-1}\cdot\tau(g,hk,l)\cdot\tau(g,h,kl)^{-1}\cdot\tau(g,h,k)\;=\;1.
\]

Finally, $[\tau]$ is independent of the choices: a second admissible $\sigma'$ differs from $\sigma$ by $\mu:=\sigma'\sigma^{-1}$, valued in $Z_G(A)$ (Corollary~\ref{lem:nat-auto-ZG-multi}, both implementing $\rho(gh)\xrightarrow{\sim}\rho(g)\rho(h)$), whence $\tau'=\tau\cdot\delta\mu$; and rescaling the section $(\vp^g,J^g)$ by some $x_g\in A_1^\times$ shifts $\tau$ by a coboundary. So $[\tau]\in H^3(G,Z_G(A))$ depends only on $\brho$.
\end{proof}

The class $[\tau]$ is the pullback of the $k$-invariant of $B\2ust(A)$. We denote it by $[\brho]=\brho^*k_{B\2ust(A)}$ and call it the \emph{obstruction class} of $\brho$. The cocycle classification of $G$-actions on $\C{A}$ then reads as follows.

\begin{thm}\label{LiftGactionTw}
Let $A$ be a finite-dimensional quasi-Hopf algebra over $\kk$ and $\brho\colon G\to\ost(A)$ a group homomorphism. Then $\brho$ lifts to a $G$-action on $\C{A}$ if and only if its obstruction class $[\brho]$ is trivial in $H^3(G,Z_G(A))$. When the obstruction vanishes, the equivalence classes of liftings form a torsor over $H^2(G,Z_G(A))$.
\end{thm}
\begin{proof}
By Lemma~\ref{GactionTw}, $G$-actions on $\C{A}$ correspond bijectively to monoidal functors $\ul G\to\2ust(A)$, hence to pointed homotopy classes of maps $BG\to B\2ust(A)$. Apply the Postnikov classification of \S\ref{sss:Gaction-formalism} to the $2$-type $X=B\2ust(A)$ with $\pi_1=\ost(A)$, $\pi_2=Z_G(A)$ \eqref{eq:pi-2ust}, and identify the pullback obstruction $\brho^*k_X$ with the explicit cocycle class $[\tau]$ of Lemma~\ref{lem:3co}. Both claims follow: $\brho$ lifts iff $[\brho]=[\tau]=0$ \textup{(}equivalently, $\sigma$ can be rechosen to satisfy the strict $2$-cocycle identity \eqref{2-cocycle}\textup{)}, and the lifts then form an $H^2(G,Z_G(A))$-torsor via the freedom $\sigma\mapsto\mu\sigma$, $\mu\in Z^2(G,Z_G(A))$.
\end{proof}

\subsection{\texorpdfstring{$G$-equivariantization as a quasi-Hopf algebra}{G-equivariantization as a quasi-Hopf algebra}}\label{SS:G-equiv-qHopf}
We first recall the standard notion of $G$-equivariantization of a monoidal category (cf.~\cite[\S 4.15]{EGNO}). Suppose a group $G$ acts on a monoidal category $\CC$ via a monoidal functor $(\zeta,\zeta_2,\zeta_0): \underline{G}\to \OAut(\CC)$. In detail, for each $g\in G$, $\zeta(g)=(\zeta^g,\zeta^g_2,\zeta^g_0)$ is a tensor autoequivalence of $\CC$; as usual, we assume $\zeta_0: \zeta(e)=\id_\CC$ and $\zeta^g_0: \zeta^g(\1)=\1$ are identities. A {\it $G$-equivariant} object in $\CC$ is a pair $(X,p)$, where $X$ is an object of $\CC$ and $p=(p^g)_{g\in G}$ is a collection of morphisms $p^g: X\to \zeta^g(X)$ in $\CC$ such that the diagrams
\begin{align}\label{obGe}
\xymatrix{
X\ar[r]^-{p^g}\ar[d]_-{p^{gh}} & \zeta^gX\ar[d]^-{\zeta^g(p^h)}\\
\zeta^{gh}X\ar[r]^-{(\zeta_2^{g,h})_X} & \zeta^g\zeta^hX
}
\quad\text{ and }\quad
\xymatrix{
X\ar[rr]^-{p^e}\ar[dr]_-{\id_X} && \zeta^eX\ar[dl]^-{\zeta_0}\\
&X&.
}
\end{align}
commute for all $g,h\in G$. Consequently, $p^e=\id$ and the morphisms $p^g$ are isomorphisms in any $G$-equivariant object $(X,p)$. A {\it $G$-equivariant } morphism $f: (X,p)\to (Y,q)$ between two $G$-equivariant objects is a morphism $f: X\to Y$ in $\CC$ such that $q^gf=\zeta^g(f)p^g$ for all $g\in G$.  

The {\it $G$-equivariantization} of $\mathcal{C}$, denoted by $\mathcal{C}^G$, is the category whose objects are $G$-equivariant objects and whose morphisms are $G$-equivariant morphisms. It is again a monoidal category, with monoidal product of two $G$-equivariant objects $(X,p)$ and $(Y,q)$ defined by
\begin{align}\label{tensorGE}
(X,p)\otimes (Y,q)~&:=~(X\otimes Y,w),\quad \text{where}\\
w^g~&:=~\left(X\otimes Y\xrightarrow{p^g\otimes q^g}\zeta^gX\otimes \zeta^gY\xrightarrow{(\zeta^g_2)_{X,Y}^{-1}} \zeta^g(X\otimes Y)\right)\notag
\end{align}
together with the unit object $(\1,\id)$ and the associativity constraint inherited from $\CC$. 

By Lemma \ref{GactionTw} and Proposition~\ref{propG}, any $G$-action on $\C{A}$ is given by $\OPsi\widetilde{\rho}\Inv$ for a monoidal functor $(\rho,\sigma): \underline{G}\to \2ust(A)$. We write $\rho(g)=(\varphi^g,J^g)$ with $\rho(e)=(\id_A,1\otimes 1)$, and the convolution-invertible map $\sigma: G\times G\to A$ satisfies \eqref{twistG}, \eqref{2-cocycle}, and \eqref{Gnat2}. For any $g \in G$, one can verify that $\varphi^g$ is a quasi-bialgebra isomorphism from $(A, \Delta, \ve, \Phi, S, \alpha, \beta)$ to each of the following quasi-Hopf algebras: \[(A,\Delta_{J^g},\ve,\Phi_{J^g},\varphi^g\circ S\circ (\varphi^g)^{-1},\varphi^g(\alpha),\varphi^g(\beta))\quad \text{ and }\quad  (A,\Delta_{J^g},\ve,\Phi_{J^g},S,\alpha_{J^g},\beta_{J^g})\, .\]
 By \cite[Proposition 1.1]{Drin90}, there is a unique $\nu^g\in \A$ such that
\begin{align}\label{nuS}
\varphi^gS(a)=\nu^gS(\varphi^g(a))(\nu^g)^{-1},\quad \varphi^g(\alpha)=\nu^g\alpha_{J^g},\quad  \varphi^g(\beta)=\beta_{J^g}(\nu^g)^{-1} \end{align}
for any $a\in A$. Again by applying \cite[Proposition 1.1]{Drin90} to $\varphi^g\varphi^h$ as a quasi-bialgebra isomorphism from $(A,\Delta,\ve,\Phi,S,\alpha,\beta)$ to each of:
\begin{gather*}
\left(A,\,\Delta_{\varphi^g(J^h)J^g},\,\ve,\Phi_{\varphi^g(J^h)J^g},\,\varphi^g
\varphi^h\circ S\circ (\varphi^g\varphi^h)^{-1},\,\varphi^g\varphi^h(\alpha),\,\varphi^g\varphi^h(\beta)\right)\quad \text{and}\\
\left(A,\,\Delta_{\varphi^g(J^h)J^g},\,\ve,\,\Phi_{\varphi^g(J^h)J^g},\,S,\,\alpha_{\varphi^g(J^h)J^g},\,\beta_{\varphi^g(J^h)J^g}\right)\, , 
\end{gather*}
we obtain 
   \begin{align}\label{uniqueNV}
    \sigma(g,h)\,\nu^{gh}\, S(\sigma(g,h))~=~\varphi^g(\nu^h)\nu^g\quad \text{for all $g,h\in G$.} 
\end{align}

 We now construct a new quasi-Hopf algebra arising from the $G$-equivariantization of the integral tensor category $\C{A}$. We denote the corresponding $\k$-basis of $\k[G]$ by $\{u^g \mid g \in G\}$, with multiplication given by $u^g u^h = u^{gh}$ for all $g, h \in G$.

\begin{deff}\label{DefGE}
Let $(\rho,\sigma): \underline{G}\to \2ust(A)$ be any monoidal functor. We define the {\it G-equivariantization} of the quasi-Hopf algebra $(A,\Delta,\ve,\Phi,S,\alpha,\beta)$ in terms of $G$ and $(\rho,\sigma)$ as the following quasi-Hopf algebra $A\#_{(\rho,\sigma)} G$ such that 
\begin{itemize}
    \item[(i)] $A\#_{(\rho,\sigma)} G$ is the crossed product algebra $A\#_\sigma \k[G]=A\otimes \k[G]$, and we use the identifications $a = a \# u^e$ and $u^g= 1_A \# u^g$ for any $a \in A$ and $g \in G$. Thus, the multiplication in $A\#_{(\rho,\sigma)} G$ is given by
\[
(a u^g)(bu^h)~=~a\, \varphi^g(b)\, \sigma(g,h) u^{gh}
\]
for all $g,h\in G$ and $a,b\in A$.
\item[(ii)] The quasi-Hopf algebra structure on $A \#_{(\rho,\sigma)} G$ is defined as follows: For any $a \in A$ and $g \in G$, we set $\Delta(u^g) = (J^g)^{-1}(u^g\otimes u^g)$ and
$$
\Delta(a u^g)=\Delta(a)\Delta(u^g),\quad \ve(a u^g)=\ve(a),\quad S(au^g)=(u^g)^{-1}\nu^gS(a),
$$
where the associator $\Phi$ and $\alpha,\beta$ are inherited from $A$.
\end{itemize} 
\end{deff}

We sketch the verification that $A\#_{(\rho,\sigma)} G$ is a well-defined quasi-Hopf algebra; the substantive checks invoke conditions \eqref{twistG}--\eqref{Gnat2} of Proposition~\ref{propG}, and the remaining details are routine.

The crossed-product algebra structure on $A\#_\sigma\k[G]$ is well-defined by \cite[Lemma 7.1.2]{Mont98}, using \eqref{twistG} and \eqref{2-cocycle}; \eqref{Gnat2} forces $\sigma(g,h)\in A_1^\times$. The counit $\ve(au^g):=\ve(a)$ is an algebra map by $\ve\circ\varphi^g=\ve$, and $(\ve\otimes\id)\Delta(u^g)=(\id\otimes\ve)\Delta(u^g)=u^g$ follows from the counit-normalization of $J^g$. The key step is that $\Delta$ respects multiplication; the substantive computation uses \eqref{Gnat2}:
\begin{align*}
\Delta(u^gu^h)&=\Delta(\sigma(g,h))(J^{gh})^{-1}(u^{gh}\otimes u^{gh})\\
\overset{\eqref{Gnat2}}&{=}(J^g)^{-1}\varphi^g(J^h)^{-1}(u^gu^h\otimes u^gu^h)
\;=\;\Delta(u^g)\Delta(u^h).
\end{align*}
The relation $\Delta(u^gb)=\Delta(u^g)\Delta(b)$ follows from the intertwiner property $(\varphi^g\otimes\varphi^g)\Delta=\Delta_{J^g}\circ\varphi^g$, and quasi-coassociativity $(\id\otimes\Delta)\Delta(u^g)=\Phi\,(\Delta\otimes\id)\Delta(u^g)\,\Phi^{-1}$ follows from the associator condition $\Phi_{J^g}=\varphi^g(\Phi)$ together with centrality of $\Phi$ in the calculation. Hence $(A\#_{(\rho,\sigma)}G,\Delta,\ve,\Phi)$ is a quasi-bialgebra.

For the antipode, $u^gu^{g^{-1}}=\sigma(g,g^{-1})$ gives $(u^g)^{-1}=u^{g^{-1}}\sigma(g,g^{-1})^{-1}$, so $S$ is bijective $\k$-linear with $S^{-1}(au^g)=(u^g)^{-1}S^{-1}(a\nu^g)$. That $S$ is an anti-algebra map reduces, via \eqref{uniqueNV} and \eqref{nuS}, to the identity $S((au^g)(bu^h))=S(bu^h)S(au^g)$; the quasi-antipode axioms \eqref{EquasiS} reduce, using \eqref{nuS} and the formulas $\alpha_{J^g}=\sum_i S(\ol f_i)\alpha\ol g_i$ and $\beta_{J^g}=\sum_i f_i\beta S(g_i)$ (with $(J^g)^{\pm 1}=\sum f_i\otimes g_i,\sum\ol f_i\otimes\ol g_i$), to the corresponding axioms in $A$. The associator axiom \eqref{EAssociator} is inherited directly from $A$. Hence $A\#_{(\rho,\sigma)}G$ is a quasi-Hopf algebra in which $A$ is a quasi-Hopf subalgebra, and $\C{A\#_{(\rho,\sigma)}G}$ is an integral tensor category which we identify next with the $G$-equivariantization of $\C{A}$. 

\begin{thm}\label{HRepGE}
Let $G$ be a finite group and $A$ a finite-dimensional quasi-Hopf algebra.  Suppose $G$ acts on $\C{A}$ via a monoidal functor $(\rho,\sigma): \underline{G}\to \2ust(A)$. Then there is a natural tensor equivalence 
   \[
   (\C{A})^G~\stackrel{\otimes}{\cong}~\C{A\#_{(\rho,\sigma)} G}. 
   \]
\end{thm}
\begin{proof}
By Lemma \ref{GactionTw}, the corresponding monoidal functor $(\zeta,\zeta_2)=\OPsi\,\widetilde{\rho}\,\Inv: \underline{G}\to \OAut(\C{A})$ satisfies
\[
\zeta^g=\OPsi(\varphi^{g^{-1}},J^{g^{-1}})\ \text{ and }\ \zeta_2(g,h)=\sigma(h^{-1},g^{-1}): \zeta^{gh}\xrightarrow{\sim} \zeta^g\zeta^h \text{ for all $g,h\in G$}.
\]
Let $F:(\C{A})^G \to \C{A}$ be the forgetful functor $(X,p)\mapsto X$. For each $g\in G$, $p^g: X\to \zeta^g(X)=\!_{\varphi^{g^{-1}}}X$ is a morphism in $\C{A}$, i.e.\ $p^g: X\to X$ is $\k$-linear with $p^g(ax)=\varphi^{g^{-1}}(a)p^g(x)$ for $a\in A$, $x\in X$. Note that $A\#_{(\rho,\sigma)} G$, as an algebra, is generated by the subalgebra $A$ and the set of generators $\{u^g\}_{g\in G}$ subject to the relations 
\begin{align}
u^ga&\,=\varphi^g(a)u^g, \label{actionA}\\
u^gu^h&\,=\sigma(g,h)u^{gh},\label{Geqaction}\\
u^e&\,=1 \label{ueident} 
\end{align}
for all $g,h\in G$ and $a\in A$. Hence $X=F(X,p)$ becomes a left $A\#_{(\rho,\sigma)} G$-module with $u^g$ acting as $p^{g^{-1}}$: \eqref{actionA} says $p^g: X\to \!_{\varphi^{g^{-1}}}X$ is an $A$-module map, and \eqref{Geqaction}--\eqref{ueident} say diagram \eqref{obGe} commutes.

A $G$-equivariant morphism $(X,p)\to(Y,q)$ is exactly a $\k$-linear map commuting with the $A$- and all $u^g$-actions, i.e.\ an $A\#_{(\rho,\sigma)}G$-module map; so $F$ lifts to a fully faithful functor $(\C{A})^G\to\C{A\#_{(\rho,\sigma)}G}$. Take the tensor product $(X,p)\otimes (Y,q)=(X\otimes Y,w)$ where $w$ is given by \eqref{tensorGE}. Since $\zeta^{g^{-1}}=\OPsi(\varphi^g,J^g)$, the monoidal structure $(\zeta^{g^{-1}}_2)_{X,Y}: \zeta^{g^{-1}}(X\otimes Y)\to \zeta^{g^{-1}} X\otimes \zeta^{g^{-1}} Y$ is multiplication by $J^g$. Then, the corresponding $A\#_{(\rho,\sigma)}G$-module structure on $F(X\otimes Y,w)=X\otimes Y$ is determined by its comultiplication such that  
\begin{align*}
    u^g(X\otimes Y)&\,=w^{g^{-1}}(X\otimes Y)=(\zeta^{g^{-1}}_2)_{X,Y}^{-1}(p^{g^{-1}}\otimes q^{g^{-1}})(X\otimes Y)\\
    &\,=(J^g)^{-1}(u^g\otimes u^g)(X\otimes Y)=\Delta(u^g)(X\otimes Y).
\end{align*}
Note that the associativity constraint on $(\C{A})^G$ is inherited from $\C{A}$, which is given by the associator $\Phi$ of $A$. Consequently, we can extend $F$ to a tensor functor 
\[(F,\xi):(\C{A})^G\xrightarrow{\otimes} \C{A\#_{(\rho,\sigma)}G}
\]
by setting the monoidal structure 
\[\xi_{(X,p),(Y,q)}: F((X,p)\otimes (Y,q))=F(X\otimes Y,w)=X\otimes Y\xrightarrow{\id} X\otimes Y=F(X,p)\otimes F(Y,q)
\]
as identity. 

Finally, $F$ is essentially surjective: any $A\#_{(\rho,\sigma)}G$-module $X$ equals $F(X,p)$ for $p^g:=u^{g^{-1}}: X\to \!_{\varphi^{g^{-1}}}X=\zeta^g(X)$, which satisfies \eqref{obGe}.
\end{proof}

\subsection{Cocentral cleft extensions}\label{SS:CCE}

Let $K$ be a Hopf algebra. For two quasi-Hopf algebras 
\[(A,\Delta_A,\ve_A,\Phi_A,\alpha_A,\beta_A,S_A) \text{ and } (H,\Delta_H,\ve_H,\Phi_H,\alpha_H,\beta_H,S_H),\] we say $H$ is an \emph{extension of $K$ by $A$} if there is a sequence of quasi-Hopf algebras
\begin{equation}\label{E:SES}
1\to A\xrightarrow{\iota} H\xrightarrow{\pi} K\to 1
\end{equation}
consisting of 
\begin{itemize}
    \item[(1)] $\iota$ is an injective quasi-bialgebra map.
    \item[(2)] $\pi$ is a surjective quasi-bialgebra map. 
    \item[(3)] $\ker \pi=H\iota(A^+)=\iota(A^+)H$.
    \item[(4)] $A=H^{co \pi}=\!^{co \pi}H$.
    \end{itemize}
In such an extension, we can identify $A$ with its image $\iota(A)$ in $H$. Then since $A$ is a quasi-Hopf subalgebra of $H$, we have $\Delta_A=\Delta_H|_A$, $\ve_A=\ve_H|_A$ and $\Phi_H=\Phi_A$. Also we can identify the Hopf algebra $K$ with $\pi(H/A^+H)=\pi(H/HA^+)$ together with the trivial associator $\pi(\Phi_H)=1\otimes 1\otimes 1$. Moreover, two extensions $H$ and $H'$ of $K$ by $A$ are \emph{equivalent} if there is a quasi-bialgebra isomorphism $f: H\to H'$ which induces the identity maps on $A$ and $K$. Moreover, we say the extension \eqref{E:SES} is \emph{cleft} if    
    \begin{itemize}
       \item[(5)] $\pi$ has a section $\gamma: K\to H$ that is a convolution invertible right $K$-comodule map. 
    \end{itemize}
Finally, we say a cleft extension is \emph{cocentral} if, additionally, we have
\begin{itemize}
  \item[(6)] $\sum_{(h)}h_1\otimes \pi(h_2)=\sum_{(h)} h_2\otimes \pi(h_1)$ for any $h\in H$. 
\end{itemize}
When $A$ and $H$ are both Hopf algebras, conditions (1)-(4) coincide with the notion of short exact sequence of Hopf algebras in the sense of \cite[Definition 1.3]{Ma94}. It is well-known that any extension of finite-dimensional Hopf algebras is automatically cleft.

  \begin{thm}\label{2dimE}
       Let $H$ be a quasi-Hopf algebra with a quasi-Hopf subalgebra $A$. Then, for any group $G$, $H$ arises as a $G$-equivariantization of $A$ if and only if $H$ can be realized as a cocentral cleft extension of $\k[G]$ by $A$. Consequently, equivalence classes of $G$-equivariantizations of $\C{A}$ are in one-to-one correspondence with equivalence classes of cocentral cleft extensions of $\k[G]$ by $A$.
   \end{thm}
\begin{proof}
Suppose $H=A\#_{(\rho,\sigma)} G$ for some monoidal functor $(\rho,\sigma): \underline{G}\to \2ust(A)$. Consider the quasi-bialgebra surjection $\pi: H \to \k[G]$ defined by $\pi(a u^g) =\ve(a)g$ for any $a \in A$ and $g \in G$. It is immediate to check that conditions (1)-(6) hold for the sequence   
$$
1\to A\to H\xrightarrow{\pi} \k[G]\to 1.
$$
with a section $\gamma: \k[G]\to H$ of $\pi$ given by $\gamma(g)=u^g$ that is a convolution invertible right $\k[G]$-comodule map.

Conversely, let $1\to A\to H\xrightarrow{\pi} \k[G]\to 1$ be a cocentral cleft extension. We will construct a monoidal functor $(\rho,\sigma): \underline{G}\to \2ust(A)$ such that $H\cong A\#_{(\rho,\sigma)} G$ as quasi-bialgebras. By (5), let $\gamma$ be a section of $\pi$, which is a convolution invertible right $\k[G]$-comodule map. Let $u^g=\gamma(g)$. One can rescale $\gamma$ so that $\gamma(1)=1$ and $\ve(u^g)=1$. Since $(\gamma*\gamma^{-1})(g)=\gamma(g)\gamma^{-1}(g) = \ve(g)1=1$, we have $\gamma^{-1}(g) = (u^g)^{-1}$ for all $g\in G$.  As a $\k[G]$-cleft extension, $\k[G]$ measures $A$ via 
\[
g\cdot a:=\gamma(g)\,a\,\gamma^{-1}(g)=u^g\,a\,(u^g)^{-1}
\]
and there is a 2-cocycle $\sigma: G\times G\to A$ via 
\[
\sigma(g,h):=\gamma(g)\gamma(h)\gamma^{-1}(gh)=u^gu^h(u^{gh})^{-1}\,
\]
such that $A\#_\sigma \k[G]\to H$, $a\#g \mapsto au^g$ defines an isomorphism of $\k$-algebras, which is also a left $A$-module and right $\k[G]$-comodule map.  Consequently, there is a decomposition of left $A$-modules $H=\bigoplus_{g\in G} Au^g$ where $Au^g=\{h \in H\mid (\id \o \pi)\Delta(h) = h \o g\}$. Since $\pi(u^g)=g$ and $\pi$ is cocentral, 
$\Delta(u^g)\in Au^g\otimes Au^g$ or $\Delta(u^g)(u^g \o u^g)^{-1} \in (A \o A)^\times$ since $u^g \in H^\times_1$. Let $J^g =\partial_{u^g}= (u^g \o u^g)\Delta(u^g)^{-1} \in (A \o A)^\times$. Then we have $\Delta(u^g) = (J^g)^{-1}(u^g \o u^g)$ where $J^g$ is a gauge transformation of $A$. Let $\varphi^g=I_{u^g}|_A$ be the measuring $\varphi^g(a)=u^g a (u^g)^{-1}$ above, an algebra automorphism of $A$ with $\ve\circ\varphi^g=\ve$. So the trivial stable twist $(\varphi^g,J^g)=(I_{u^g},\partial_{u^g})$ of $H$ restricts to a stable twist of $A$. After equipping $A\#_\sigma \k[G]$ with a quasi-bialgebra structure given by
\[
\Delta(a u^g)=\Delta(a)(J^g)^{-1}(u^g\otimes u^g),\qquad \ve(au^g)=\ve(a)
\]
for all $a \in A$ and $g \in G$, the previous map $A\#_\sigma \k[G]\to H$ becomes an isomorphism of quasi-bialgebras. It remains to show that the pair $(\rho,\sigma)$ with $\rho(g)=(\varphi^g,J^g)$ satisfies Proposition \ref{propG}.
This is again a routine check by applying \cite[Lemma 7.1.2]{Mont98} and the relations
\[
\Delta(u^g u^h)=\Delta(\sigma(g,h)u^{gh})
\]
for all $g,h\in G$.

Finally, by Lemma \ref{GactionTw} it remains to match isomorphisms $(\rho,\sigma)\cong(\rho',\sigma')$ with equivalences of the extensions $H=A\#_{(\rho,\sigma)}G$ and $H'=A\#_{(\rho',\sigma')}G$. Suppose $f: H\to H'$ is an equivalence of extensions, with cocentral maps $\pi,\pi'$ to $\k[G]$. We have $H=\bigoplus_{g\in G} Au^g$ and $H'=\bigoplus_{g\in G} Au'^g$. Since $\pi'\circ f=\pi$, we must have $f(Au^g)=Au'^g$ or $f(u^g)=\phi(g)u'^g$ for some map $\phi: G\to A$. Since $1=\ve(f(u^g))=\ve(\phi(g)u'^g)=\ve(\phi(g))$ and $f$ is bijective, we have $\phi: G\to A_1^\times$. Note that $f|_A=\id$. So $f(u^ga)=f(\varphi^g(a)u^g)$ and $\Delta'(f(u^g))=(f\otimes f)\Delta(u^g)$ implies that
 \begin{align}\label{prof.natco}
 (I_{\phi(x)},\partial_{\phi(g)})(\varphi'^g,J'^g)=(\varphi^g,J^g)
 \end{align}
 for all $g\in G$. This implies that $\phi(g): \rho'(g)=(\varphi'^g,J'^g)\to \rho(g)=(\varphi^g,J^g)$ is an isomorphism in $\2ust(A)$. Furthermore, $f(u^gu^h)=f(\sigma(g,h)u^{gh})$ implies that
 \begin{align}\label{prof.natiso}
      \phi(g)\varphi'^{g}(\phi(h))\sigma'(g,h)=\sigma(g,h)\phi(gh)
 \end{align}
 for all $g,h\in G$. This means that the following diagram
 \[
 \xymatrix{
\rho'(gh)\ar[r]^-{\sigma'(g,h)}\ar[d]_-{\phi(gh)}  &\rho'(g)\otimes \rho'(h)\ar[d]^-{\phi(g)\otimes \phi(h)}\\
\rho(gh)\ar[r]^-{\sigma(g,h)}& \rho(g)\otimes \rho(h)
 }
 \]
 commutes for all $g,h\in G$. Consequently $\phi: (\rho',\sigma')\to (\rho,\sigma)$ is an isomorphism of monoidal functors, so the corresponding $G$-actions are equivalent. Conversely, such a $\phi$ yields the equivalence $f: H\to H'$, $f(au^g)=a\phi(g)u'^g$, by the same computation.
 \end{proof}

\begin{thm}\label{CCEGE}
    Let $A$ be a finite-dimensional quasi-Hopf algebra and $G$ a finite group. Suppose $H^2(G,Z_G(A))$ is trivial for any group homomorphism $\brho: G\to \ost(A)$. The following are in one-to-one correspondence:
    \begin{enumerate}
        \item $G$-actions on $\C{A}$ up to equivalence;
        \item monoidal functors from $\underline{G}\to \2ust(A)$ up to isomorphism;
\item group homomorphisms $\brho: G\to \ost(A)$ with trivial $[\brho]$ in $H^3(G,Z_G(A))$;
    \item cocentral cleft extensions of $\k[G]$ by $A$ up to equivalence.
    \end{enumerate} 
\end{thm}
\begin{proof}
  (i)$\leftrightarrow$(ii) is Lemma~\ref{GactionTw}, (ii)$\leftrightarrow$(iii) is Theorem~\ref{LiftGactionTw}, and (i)$\leftrightarrow$(iv) is Theorem~\ref{2dimE}.
\end{proof}

When $A$ is an ordinary Hopf algebra ($\Phi=1\otimes 1\otimes 1$, $\alpha=\beta=1$), the equivariantization $A\#_{(\rho,\sigma)}G$ is again a Hopf algebra with antipode $S(u^g)=(u^g)^{-1}\beta_{J^g}$. As every extension of finite-dimensional Hopf algebras is automatically cleft, Theorem~\ref{CCEGE} then holds verbatim with ``cocentral cleft extensions'' replaced by ``cocentral extensions'' of $\k[G]$ by $A$.

\subsection{\texorpdfstring{Examples of cyclic $G$-equivariantizations}{Examples of cyclic G-equivariantizations}}\label{SS:Examples}
We specialize \S\S\ref{SS:Gaction}--\ref{SS:CCE} to the case where $G$ is a cyclic group, recording explicit formulas for the $3$-cocycle obstruction and the resulting cocentral cleft extensions.

\subsubsection{$C_n$-equivariantizations on finite-dimensional quasi-Hopf algebras}

Let $G=C_n$ be the cyclic group of order $n$ with a fixed generator $g$  such that $g^n=1$. Let $\brho: C_n\to \ost(A)$ be a group homomorphism and write $\rho(g)=\overline{(\varphi,J)}$ for some stable twist $(\varphi,J)$ of $A$. Without loss of generality, we can lift $\brho$ to a functor $\rho: \underline{C_n}\to \2ust(A)$ via $\rho(g^i)=(\varphi,J)^i$ for $0\le i\le n-1$. Since $g^n=1$, there exists some $z\in A_1^\times$ such that $(\varphi,J)^n=(I_z,\partial_z)$. 

By abuse of notation, we will write $g^i$ where the index $i$ is taken modulo $n$ in $C_n$. Then it is clear to check that there are isomorphisms 
\[
\sigma(i,j):\rho(i+j)\xrightarrow{} \rho(i)\otimes \rho(j)=
\begin{cases}
 1 & i+j<n\\
 z & i+j\ge n 
\end{cases}
\text{ for $0\le i,j\le n-1$}. 
\] 
Note that $Z_G(A)$ is a $C_n$-module via $g\cdot a=\varphi(a)$ for any $a\in Z_G(A)$. Moreover, a straightforward computation gives the associated normalized 3-cocycle $\tau\in Z^3(C_n,Z_G(A))$ as 
\[
\tau(i,j,k)=
\begin{cases}
    z\varphi^i(z)^{-1} &  i+j\ge n \text{ and }j+k\ge n\\
 1 &  i+j< n \text{ or } j+k<n
\end{cases}
\text{ for $0\le i,j,k\le n-1$}. 
\]

\begin{lem}\label{lem-Caction}
The obstruction class $[\brho]$ is trivial if and only if there exists $z'=az$ for some $a\in Z_G(A)$ such that $\varphi(z')=z'$. Moreover, up to equivalence, there are $|Z_G(A)^{\varphi}/(1+\varphi+\cdots+\varphi^{n-1})\cdot Z_G(A)|$ distinct liftings of $\brho$ from $\underline{G}$ to $\2ust(A)$. 
 \end{lem} 
 \begin{proof}
Suppose $[\brho]$ is trivial. Then there is some normalized map $\gamma: C_n\times C_n\to Z_G(A)$ such that
\[
\tau(i,j,k)=\frac{\varphi^i(\gamma(j,k))\, \gamma(i,j+k)}{\gamma(i+j,k)\,\gamma(i,j)}\quad \forall\, 0\leq i,j,k\leq n-1. 
\]
By taking $i=k=1$, we get $\varphi(\gamma(n-1,1))/\gamma(1,n-1)=z\varphi(z)^{-1}$ and \[\varphi(\gamma(j,1))\gamma(1,j+1)=\gamma(1+j,1)\gamma(1,j)\] for $1\le j< n-1$. Multiplying them all together and letting $a=\gamma(n-1,1)\gamma(n-2,1)\cdots \gamma(1,1)$, we get $\varphi(a)/a=z\varphi(z)^{-1}$, or equivalently $\varphi(z')=z'$ with $z'=az$. Conversely, suppose $\varphi(z')=z'$ with $z'=az$ for some $a\in Z_G(A)$. It is clear that $(\varphi,J)^n=(I_{z'},\partial_{z'})$. Then we can reset isomorphisms $\sigma(i,j):\rho(i+j)\to \rho(i)\otimes \rho(j)$ by replacing $z$ with $z'$. Since Lemma \ref{lem:3co} shows that $[\brho]$ depends only on $\brho$ and not on the choice of $\sigma$, $[\brho]$ is trivial by the formula for $\tau$.

The last statement follows from Lemma \ref{GactionTw} and the well-known fact about $H^2(C_n,Z_G(A))=Z_G(A)^{g}/(1+g+\cdots+g^{n-1})\cdot Z_G(A)$ \cite[Theorem 6.2.2]{Weibel1994}.
 \end{proof}

Suppose $(\varphi,J)$ is a stable twist for $A$ satisfying $(\varphi,J)^n=(I_a,\partial_a)^{-1}$ for some positive integer $n$ and $a\in A_1^\times$ such that $\varphi(a)=a$. Then one can define a quasi-Hopf algebra $A[u^{\pm 1};\varphi,J,a,n]$, generated as a $\k$-algebra by $A$ and $u^{\pm 1}$ subject to the relations
\begin{gather*}
xu=u\varphi(x)\quad \text{for all $x\in A$},\quad  u^n=a.
\end{gather*}
Its quasi-Hopf algebra structures are inherited from $A$ together with  
\begin{gather*}
\Delta(u)=(u\otimes u)J,\quad \ve(u)=1, 
\quad S(u)= \nu^{-1} u^{-1},
\end{gather*}
where $\nu$ is the unique invertible element of $A$ determined by
\[
\varphi S(x)=\nu S(\varphi(x))\nu^{-1},\quad \varphi(\alpha)=\nu \alpha_J,\quad \varphi(\beta)=\beta_J\nu^{-1}.
\] 
In particular, when $A$ is a Hopf algebra with $\Phi=1\otimes 1\otimes 1$ and $\alpha=\beta=1$, $A[u^{\pm 1};\varphi,J,a,n]$ is again a Hopf algebra with antipode $S(u)=\alpha_Ju^{-1}$. When $n=2$ we abbreviate $A[u^{\pm 1};\varphi,J,a]:=A[u^{\pm 1};\varphi,J,a,2]$.

\begin{cor}\label{cor:anotEG}
   Suppose $\brho: C_n\to \ost(A)$ has a lifting via $(\rho,\sigma)$. Then there is a quasi-bialgebra isomorphism $f: A[u^{\pm 1};\varphi,J,a,n]\cong A\#_{(\rho,\sigma)}C_n$ such that $f|_A=\id$ and $f(u)=(u^g)^{-1}$ with $a=(u^{g})^{-n}\in A_1^\times$. Moreover, two extensions $A[u^{\pm 1};\varphi,J,a,n]$ and $A[u^{\pm 1};\varphi',J',a',n]$ of $\k[C_n]$ by $A$ are equivalent if and only if there exists $b\in A_1^\times$ such that
\begin{enumerate}
    \item $(I_b,\partial_b)(\varphi,J)=(\varphi',J')$ and 
\item $a=a'\varphi'^{(n-1)}(b)\cdots \varphi'(b)b$.    
\end{enumerate}
\end{cor}
\begin{proof}
   It is straightforward to check that $f$ is a bijective quasi-bialgebra map from $A[u^{\pm 1};\varphi,J,a,n]$ to $A\#_{(\rho,\sigma)}C_n$.

   Suppose two extensions $A[u^{\pm 1}; \varphi,J,a,n]$ and $A[u^{\pm 1};\varphi',J',a',n]$ are equivalent via some quasi-bialgebra isomorphism $g$. By definition, we have $g|_A=\id$ and $g(u)=ub$ for some $b\in \A$. Hence, one can check that for any $x\in A$, $u\varphi'(x)b=xub=g(x)g(u)=g(xu)=g(u\varphi(x))=ub\varphi(x)$ which yields that $\varphi'=I_b\circ \varphi$. Moreover, $(u\otimes u)(b\otimes b)J=(g\otimes g)\Delta(u)=\Delta(g(u))=\Delta(ub)=(u\otimes u)J\Delta(b)$. This implies that $(b\otimes b)J\Delta(b)^{-1}=J'$. We get (i). And (ii) follows from $a=g(u^n)=g(u)^n=(ub)^n=u^n\varphi'^{(n-1)}(b)\cdots \varphi'(b)b=a'\varphi'^{(n-1)}(b)\cdots \varphi'(b)b$. The converse statement can be argued in a similar manner. 
\end{proof}

\subsubsection{$C_2$-equivariantizations on finite-dimensional Hopf algebras}

In the following, we focus on the case where $G = C_2$. It is well known in group theory that if a group (possibly infinite) contains a normal subgroup of order two, that subgroup must lie in the center of the group. We extend this result to any Hopf algebra $H$ in Lemma~\ref{subnormal} below. Recall the \emph{left adjoint action} of $H$ on itself is given by
\[
(\ad_lh)(a)=\sum h_1aS(h_2)
\]
for any $h,a\in H$. And the \emph{right adjoint action} of $H$ on itself is given by
\[
(\ad_rh)(a)=\sum S(h_1)ah_2.
\]

\begin{lem}\label{subnormal}
Suppose $H$ is a Hopf algebra with a two-dimensional separable subalgebra $A$. Then the following are equivalent:
\begin{enumerate}
    \item $A$ is a left coideal subalgebra of $H$ and $(\ad_lH)(A)\subseteq A$;
    \item $A$ is a right coideal subalgebra of $H$ and $(\ad_rH)(A)\subseteq A$;
    \item $A$ is a central Hopf subalgebra of $H$.
\end{enumerate}
As a consequence, any finite-dimensional Hopf extension $1\to A\to H\to\k[C_2]\to 1$ is cocentral.
\end{lem}
\begin{proof}
It is obvious that (iii)$\Rightarrow$(i),(ii). We prove (i)$\Rightarrow$(iii); the implication (ii)$\Rightarrow$(iii) is similar. After a base-field extension if necessary, we may assume $A\cong \k \times \k$ as algebras. By choosing $a$ to be one of the indecomposable idempotents of $A$, we can further assume $A$ is spanned by two elements $\{1,a\}$ satisfying $a^2=a$ with $\ve(a)=0$. Since $\Delta(A)\subseteq H\otimes A$, $\Delta(a)=a\otimes 1+g\otimes a$ for some group-like $g \in G(H)$. Invariance of $A$ under its left adjoint action lets us write
\begin{align}
(\ad_lh)(a)=\sum h_1aS(h_2)~=~\alpha(h)1+\beta(h)a
\end{align}
for $h\in H$, with $\alpha,\beta$ in the finite dual Hopf algebra $H^0$; since $\ad_l(hh')=\ad_l(h)\ad_l(h')$ and $\ad_l(1)=1$, $\beta\in G(H^0)$ and $\alpha$ is $(\beta,\ve)$-primitive. Applying $\Delta$ to $(\ad_l h)(a)$ yields
\begin{align*}
&\quad\Delta((\ad_l h)(a))\\
&=(\alpha(h)1+\beta(h)a)\otimes 1+\beta(h)g\otimes a\\
&=\sum h_1a_1S(h_4)\otimes h_2a_2S(h_3)\\
&=\sum h_1aS(h_2)\otimes 1+\sum h_1gS(h_4)\otimes h_2aS(h_3)\\
&=(\alpha(h)1+\beta(h)a+\sum h_1g\alpha(h_2)S(h_3))\otimes 1+(\sum h_1g\beta(h_2)S(h_3))\otimes a.
\end{align*}
Hence $\sum h_1g\alpha(h_2)S(h_3)=0$ for all $h\in H$. This implies $\alpha=0$ by applying $\ve$ to the equation. Consideration of $\Delta(a^2)=\Delta(a)$ yields $g=1-a-gag^{-1}=1-(1+\beta(g))a\in A$. Thus, $A$ is indeed a Hopf subalgebra of $H$. Next, we apply $\ad_l(h)$ to $a^2=a$, which implies that 
\begin{align*}
\beta(h)a=(\ad_l h)(a^2)=\sum (\ad_l h_1)(a)(\ad_l h_2)(a)=\sum (\beta(h_1)a)(\beta(h_2)a)=\beta^2(h)a.
\end{align*}
Therefore, $\beta^2=\beta$ and $\beta=\ve$ in $H^0$ for $\beta$ is group-like.  Our result now follows from the well-known fact that the left adjoint action of $H$ on $A$ is trivial if and only if $A$ lies in the center of $H$. 

The final statement follows from duality since in the dual exact sequence $1\to \k^{C_2}\to H^*\to A^*\to 1$, $\k^{C_2}$ is a separable normal Hopf subalgebra of $H^*$ and hence central. 
\end{proof}

The following result combines Theorem~\ref{CCEGE} and Lemma~\ref{subnormal}.

\begin{thm}\label{HopfCEG}
    Let $A$ be a finite-dimensional Hopf algebra. Suppose for any group homomorphism $\brho: C_2\to \ost(A)$, $H^n(C_2,Z_G(A))$ is trivial for $n=2,3$, where the $C_2$-module action on $Z_G(A)$ is given by $\brho$. Then the following are in one-to-one correspondence:
\begin{enumerate}
        \item nontrivial $C_2$-actions on $\C{A}$ up to equivalence;
        \item nontrivial extensions of $\k[C_2]$ by $A$ up to equivalence;
\item order two stable twists in $\ost(A)$.     
\end{enumerate}
\end{thm}

We conclude this subsection with a structural theorem on $C_2$-extensions, under hypotheses on $\ost(A)$ that we verify for specific pointed Hopf algebras later.

\begin{thm}\label{HopfCEGC2}
   Let $A$ be a finite-dimensional Hopf algebra with $\ost(A)\cong C_2$. Suppose $\ost(A)$ is generated by the stable twist $(\varphi,J)$ and there exists an element $z\in \A$ such that $(\varphi,J)^2=(I_z,\partial_z)^{-1}$ with $\varphi(z)=z$. Then, there are $|Z_G(A)/2Z_G(A)|+|Z_G(A)^\varphi/(1+\varphi)Z_G(A)|$ distinct extensions of $\k[C_2]$ by $A$ up to equivalence, which are given by
   \[
   A[u^{\pm 1};\varphi,J,zg,2]\quad  \text{and}\quad  A[u^{\pm 1};\id,1\otimes 1,h,2]\, ,
   \]
   where $g,h\in Z_G(A)$ run through a set of representatives for the quotient groups $Z_G(A)^\varphi/(1+\varphi)Z_G(A)$ and $Z_G(A)/2Z_G(A)$, respectively.
\end{thm}
\begin{proof}
Under our assumption, there are exactly two group homomorphisms $\brho: C_2 \to \ost(A)$: the trivial homomorphism and an isomorphism that maps the generator of $C_2$ to $\overline{(\varphi,J)}$. Suppose $\brho$ is trivial. Thus, the $C_2$-action on $Z_G(A)$ is trivial, which implies that $H^2(G, Z_G(A)) = Z_G(A)/2Z_G(A)$. Observe that every extension of $C_2$ is cocentral by Lemma \ref{subnormal}. Consequently, Theorem \ref{2dimE}, together with Corollary \ref{cor:anotEG}, implies that the corresponding extensions of $\k[C_2]$ by $A$ are given by $A[u^{\pm 1}; \id, 1\otimes 1, h, 2]$, where $h \in Z_G(A)$ ranges over a set of representatives of the quotient $Z_G(A)/2Z_G(A)$. The proof in the case $\brho: C_2\cong \ost(A)$ proceeds in the same way.
\end{proof}

We illustrate the framework on the pointed fusion category $\operatorname{Vec}_G^\omega$ in the cyclic case $G=\Z/p$, building on Example~\ref{ex:pointed-SES} (which records the general short exact sequence \eqref{eq:VecGom-SES-intro} for arbitrary $G$) to construct an explicit $C_2$-equivariantization. 

\begin{example}\label{ex:quasi-Hopf-pointed}
Let $p$ be an odd prime and $G=\Z/p=\langle g\rangle$, with $[\omega]\in H^3(G,\k^\times)$ nontrivial. The quasi-Hopf algebra
\[
A=\bigl(\k^G,\,\Delta,\,\ve,\,\Phi_\omega,\,S,\,1,\,\beta_\omega\bigr),
\]
with
\[
\k^G=\bigoplus_{h\in G}\k\,\delta_h,\quad \Delta(\delta_h)=\sum_{ab=h}\delta_a\o\delta_b,\quad \ve(\delta_h)=\delta_{h,1},\quad S(\delta_h)=\delta_{h\inv},
\]
\[
\Phi_\omega=\sum_{a,b,c}\omega(a,b,c)\,\delta_a\o\delta_b\o\delta_c,\quad \alpha=1,\quad \beta_\omega=\sum_h\omega(h,h\inv,h)\inv\,\delta_h,
\]
has representation category $\C{A}\simeq\operatorname{Vec}_G^\omega$, the pointed fusion category. Czenky \cite[Corollary~6.8]{Czenky2024} shows that $\mathrm{Aut}(G;[\omega])\cong\Z/2$, generated by $\sigma_{-1}(g)=g^{-1}$; we write $\vp_{-1}\in\operatorname{Aut}(\k^G)$ for the dual algebra automorphism $\vp_{-1}(\delta_h)=\delta_{h\inv}$. Since $H^2(\Z/p,\k^\times)=0$, the sequence \eqref{eq:VecGom-SES-intro} collapses to
\[
\ost(A)\;=\;\operatorname{Aut}(G;[\omega])\;\cong\;\Z/2\, .
\]
Since $[\sigma_{-1}^\ast\omega]=[\omega]$,
there exists a normalized $2$-cochain
$J_0\colon G\times G\to\k^\times$, unique up to $2$-coboundary,
satisfying
\begin{equation}\label{E:J0-cobdy}
\frac{J_0(b,c)\,J_0(a,bc)}{J_0(ab,c)\,J_0(a,b)}
=\frac{\omega(a\inv,b\inv,c\inv)}{\omega(a,b,c)}
\qquad\text{for all }a,b,c\in G.
\end{equation}
As $H^2(G,\k^\times)=0$, the pair $(\vp_{-1},J_0)$ represents
the unique nontrivial class in $\ost(A)\cong\Z/2$,
independently of the choice of $J_0$.

The central grouplike elements of $A=\k^G$ are the characters of $G$:
$Z_G(A)=\hat G=\operatorname{Hom}(G,\k^\times)\cong\Z/p$.
The nontrivial element of $\ost(A)$ acts on $\hat G$ by inversion
$\chi\mapsto\chi\inv$, and
\[
H^n(\Z/2,\Z/p)=0\quad\text{for all $n\ge 1$}
\]
since $p$ is odd.

Since $\vp_{-1}^2=\id$,
$(\vp_{-1},J_0)^2=(\id,\,\vp_{-1}(J_0)\cdot J_0)$.
The $2$-cochain $K\colon G\times G\to\k^\times$ defined by
\[
K(g^i,g^j)=J_0(g^{-i},g^{-j})\,J_0(g^i,g^j)
\]
is a $2$-cocycle (as $d^2(K)=1$), and $H^2(G,\k^\times)=0$ implies $K=\partial_z$
for some normalized 1-cochain $z: G\to \k^\times$.  We identify $J_0,K$ with the elements $\sum_{a,b\in G}J_0(a,b)\,\delta_a\o\delta_b$ and $\sum_{a,b\in G}K(a,b)\,\delta_a\o \delta_b$ in $(\k^G)^{\o 2}$, and also treat $z$ as $\sum_{a\in G} z(a)\delta_a\in (\k^G)^\times_1$. This gives the relation $K=(z\otimes z)\Delta(z)^{-1}$ with $(\varphi_{-1}\otimes \varphi_{-1})(K)=K$. Since $\Delta(\varphi_{-1}(z))=(\varphi_{-1}\otimes \varphi_{-1})\Delta(z)$, a direct check gives $\Delta(z/\varphi_{-1}(z))=z/\varphi_{-1}(z)\otimes z/\varphi_{-1}(z)$, so $z/\varphi_{-1}(z)$ is a character on $\Z/p$.  Since the squaring map on
$\hat G\cong\Z/p$ is surjective (as $p$ is odd), we may adjust $z$ by
a character to arrange $\vp_{-1}(z)=z$, i.e., $z(h\inv)=z(h)$
for all $h\in G$.
Setting $a=z\inv$, we have $(\vp_{-1},J_0)^2=(I_a,\partial_a)\inv$
with $\vp_{-1}(a)=a$.
Concretely, $a\in(\k^G)^\times$ is the unique element satisfying
$a(1)=1$, $a(h\inv)=a(h)$ for all $h\in G$, and
\begin{equation}\label{E:a-cobdy}
\frac{a(g^ig^j)}{a(g^i)\,a(g^j)}
=J_0(g^{-i},g^{-j})\,J_0(g^i,g^j)
\qquad\text{for all }0\le i,j\le p-1.
\end{equation}

By Lemma~\ref{lem-Caction}, the obstruction
$[\brho]\in H^3(C_2,\hat G)$ vanishes, and the lifting of the
nontrivial homomorphism $\brho\colon C_2\to\ost(A)$ is unique up to
equivalence (since $H^2(C_2,\hat G)=0$).
The construction preceding Corollary~\ref{cor:anotEG} then yields a
quasi-Hopf algebra
\[
B=A[u^{\pm 1};\vp_{-1},J_0,a,2],
\]
which, as a $\k$-algebra, is generated by $\k^G$ and $u^{\pm 1}$
subject to the relations
\begin{gather*}
\delta_h\cdot u=u\cdot\delta_{h\inv}\quad\text{for all }h\in G,
\qquad u^2=a.
\end{gather*}
The quasi-Hopf structure is inherited from $A$ together with
\begin{gather*}
\Delta(u)=(u\o u)J_0,\qquad \ve(u)=1,\qquad
S(u)=\nu\inv u\inv,
\end{gather*}
where $\nu\in(\k^G)^\times$ is determined by the condition
$\vp_{-1}(\alpha)=\nu\,\alpha_{J_0}$ (cf.\ the construction
preceding Corollary~\ref{cor:anotEG}).
By \eqref{ABJ} with $\alpha=1$, we compute
\[
\alpha_{J_0}
=\sum_{s,t\in G}J_0(s,t)\inv\,S(\delta_s)\,\delta_t
=\sum_{s\in G}J_0(s,s\inv)\inv\,\delta_{s\inv}
=\sum_{h\in G}J_0(h\inv,h)\inv\,\delta_h.
\]
Since $\vp_{-1}(\alpha)=\vp_{-1}(1)=1$, we obtain
$\nu=\alpha_{J_0}\inv$, i.e.,
\[
\nu=\sum_{h\in G}J_0(h\inv,h)\,\delta_h.
\]
The associator of $B$ is $\Phi_\omega\in A^{\o 3}\subset B^{\o 3}$
(unchanged from~$A$), and
$\beta_B=\beta=\sum_{h\in G}\omega(h,h\inv,h)\inv\,\delta_h$.
The algebra $B$ has dimension $2p$, and
\[
\C{B}\simeq(\operatorname{Vec}_G^\omega)^{C_2}.
\]

In particular, fix a primitive $p$-th root of unity $\zeta\in\k$ and take the
standard normalized $3$-cocycle 
\[
\omega(g^i,g^j,g^k)=\zeta^{\,i\lfloor(j+k)/p\rfloor},
\qquad 0\le i,j,k\le p-1,
\]
whose class generates $H^3(G,\k^\times)$. The parameters of the quasi-Hopf algebra
$B=A[u^{\pm 1};\vp_{-1},J_0,a,2]$ are then given as follows.
\begin{align*}
J_0(g^i,g^j)&=
\begin{cases}
1 & \text{if }i=0\text{ or }j=0,\\
\zeta^{\,i-1} & \text{if }1\le i,j\le p-1
  \text{ and }i+j\le p,\\
\zeta^{\,i} & \text{if }1\le i,j\le p-1
  \text{ and }i+j>p.
\end{cases}\\
a&=\delta_1+\zeta\inv\sum_{i=1}^{p-1}\delta_{g^i}.\\
\nu&=\delta_1+\sum_{i=1}^{p-1}\zeta^{-(i+1)}\,\delta_{g^i}.
\end{align*}

For the trivial homomorphism $\brho=1$, the unique extension is
$A[u^{\pm 1};\id,1\o 1,1,2]\cong A\o\k[C_2]$.
Therefore there are exactly two cocentral cleft extensions
of $\k[C_2]$ by $A$ up to equivalence: the trivial extension
$A\o\k[C_2]$ and the nontrivial extension~$B$.
\end{example}

\subsection{The braided counterpart}\label{SS:braided-equiv}
We now collect the braided analogues of \S\S\ref{SS:Gaction}--\ref{SS:Examples}. Throughout, $(A,\Delta,\ve,\Phi,S,\a,\b,R)$ is a finite-dimensional QT quasi-Hopf algebra over $\kk$ and $\C{A}$ the corresponding braided integral tensor category; the braided sub-$(3,1)$-category $\3ust^{\rm br}\subseteq\3ust$, the set $\st^{\rm br}(A,R)\subseteq\st(A)$, and its image $\ost^{\rm br}(A,R)\subseteq\ost(A)$ are those of \S\ref{sss:braided}. The categorical motivation is standard: when $G$ acts on a braided tensor category $\CC$ by \emph{braided} tensor autoequivalences (monoidal autoequivalences intertwining the braiding, $\zeta^g_{Y,X}\circ c_{X,Y}=c_{\zeta^g X,\zeta^g Y}\circ\zeta^g_{X,Y}$), the equivariantization $\CC^G$ inherits the braiding $c_{(X,p),(Y,q)}:=c_{X,Y}$ of $\CC$. Specializing to $\CC=\C{A}$ via Corollary~\ref{c:peqv-br}, a braided $G$-action is a monoidal functor $\ul G\to\2ust^{\rm br}(A,R)$, and the equivariantization $A\#_{(\rho,\sigma)}G$ acquires a canonical QT structure (Definition~\ref{DefGE-br}). Each block below mirrors one of the preceding subsections, with $\st,\ost,\2ust$ replaced by their braided counterparts; the proofs amount to checking that the braided condition \eqref{eq:R-cond-multi} is preserved at each step.

\subsubsection{Braided Postnikov classification}\label{sss:braided-postnikov}
The strict $2$-group $\2ust^{\rm br}(A,R)$ of braided stable twists is the braided automorphism $2$-group of $(A,R)$. As in the non-braided case, we read off the homotopy invariants of $B\2ust^{\rm br}(A,R)$ directly:
\begin{equation}\label{eq:pi-2ust-br}
\pi_1\bigl(B\2ust^{\rm br}(A,R)\bigr)=\ost^{\rm br}(A,R),\qquad \pi_2\bigl(B\2ust^{\rm br}(A,R)\bigr)=Z_G(A).
\end{equation}
Here $\pi_1=\ost^{\rm br}(A,R)$ by definition and $\pi_2=Z_G(A)$ exactly as in the non-braided computation (Proposition~\ref{prop:pin}, \eqref{eq:pi-2ust}).

\begin{thm}\label{LiftGactionTw-br}
A homomorphism $\brho^{\rm br}\colon G\to\ost^{\rm br}(A,R)$ lifts to a $G$-action on $\C{A}$ by \emph{braided} tensor autoequivalences if and only if its obstruction class $[\brho^{\rm br}]\in H^3(G,Z_G(A))$ is trivial; when it vanishes, equivalence classes of braided lifts form an $H^2(G,Z_G(A))$-torsor.
\end{thm}
\begin{proof}
The proof of Theorem~\ref{LiftGactionTw} carries over with $\st(A)$ and $\ost(A)$ replaced by $\st^{\rm br}(A,R)$ and $\ost^{\rm br}(A,R)$, using \eqref{eq:pi-2ust-br}. Lemma~\ref{lem:braided-closed} ensures that the cocycle $\sigma$ and its deviation $\tau$ from a $2$-cocycle remain inside the braided sub-$2$-group, so no obstruction beyond $[\tau]\in H^3(G,Z_G(A))$ is introduced by requiring the braiding to be preserved.
\end{proof}

Theorem~\ref{LiftGactionTw-br} specializes for $A=D(H)$ \textup{(}$H$ any finite-dimensional quasi-Hopf algebra\textup{)} to the algebraic version of \cite[Theorem~1.1]{ENO2010} at the level of $\pi_1$, in parallel with the fusion-categorical Picard formalism of \cite[\S 7]{ENO2010}.

\subsubsection{Braided $G$-equivariantization}\label{sss:braided-equiv-qhopf}
Let $(\rho,\sigma)\colon\ul G\to\2ust^{\rm br}(A,R)$ be a monoidal functor, i.e., a $G$-action on $\C{A}$ by braided tensor autoequivalences.

\begin{deff}\label{DefGE-br}
The \emph{braided $G$-equivariantization} of $(A,R)$ is the quasi-Hopf algebra $A\#_{(\rho,\sigma)}G$ of Definition~\ref{DefGE}, equipped with the $R$-matrix
\begin{equation}\label{eq:R-sharp}
R^{\#}\;:=\;R\;\in\;A\otimes A\;\hookrightarrow\;(A\#_{(\rho,\sigma)}G)\otimes(A\#_{(\rho,\sigma)}G).
\end{equation}
\end{deff}

The QT axioms \eqref{eq:QT1}--\eqref{eq:QT2} for $(A\#_{(\rho,\sigma)}G,R^{\#})$ are inherited from those of $(A,R)$ since $R^{\#}\in A^{\otimes 2}$ and $\Phi\in A^{\otimes 3}$. The remaining axiom~\eqref{eq:QT3} on a generator $h=u^g$ reduces, via $\Delta(u^g)=(J^g)^{-1}(u^g\otimes u^g)$, to $(\vp^g\otimes\vp^g)(R)=J^g_{21}\,R\,(J^g)^{-1}$, which is exactly \eqref{eq:R-cond-multi} for $(\vp^g,J^g)\in\st^{\rm br}(A,R)$. Hence $(A\#_{(\rho,\sigma)}G,R^{\#})$ is a QT quasi-Hopf algebra.

\begin{thm}\label{HRepGE-br}
The tensor equivalence of Theorem~\ref{HRepGE} upgrades to a natural braided tensor equivalence
\[
(\C{A})^G\;\stackrel{\otimes,\,{\rm br}}{\cong}\;\C{A\#_{(\rho,\sigma)}G},
\]
where the right-hand side carries the QT structure $R^{\#}=R$ of \eqref{eq:R-sharp}. In particular, $(\C{A})^G$ is a braided integral tensor category.
\end{thm}
\begin{proof}
The underlying tensor equivalence $F\colon(\C{A})^G\xrightarrow{\sim}\C{A\#_{(\rho,\sigma)}G}$ is Theorem~\ref{HRepGE}. Both braidings come from the same $R$-matrix: on $\C{A\#_{(\rho,\sigma)}G}$ via $R^{\#}=R$ through the inclusion $A\hookrightarrow A\#_{(\rho,\sigma)}G$, on $(\C{A})^G$ via the underlying braiding of $\C{A}$ restricted to $G$-equivariant objects. The restriction is well-defined precisely because each $\rho(g)\in\st^{\rm br}(A,R)$ preserves $R$ up to gauge (Lemma~\ref{lem:braided-closed}), so $F$ intertwines braidings.
\end{proof}

\subsubsection{Braided cocentral cleft extensions}\label{sss:braided-cce}
A cocentral cleft extension $1\to A\xrightarrow{\iota} H\xrightarrow{\pi}\k[G]\to 1$ is called \emph{braided} \textup{(}of $\k[G]$ by $(A,R)$\textup{)} if $H$ admits a quasitriangular structure $R^H$ such that $\iota$ is a morphism of QT quasi-Hopf algebras, i.e.,
\[
R^H\;=\;(\iota\otimes\iota)(R)\;\in\;H\otimes H;
\]
equivalently, the braiding of $(A,R)$ extends along $\iota$ to a braiding of $H$. Two braided cocentral cleft extensions are equivalent if their underlying cleft extensions are equivalent via an isomorphism intertwining the $R$-matrices. Theorems~\ref{2dimE} and~\ref{CCEGE} then upgrade to:

\begin{thm}\label{CCEGE-br}
Let $(A,R)$ be a finite-dimensional QT quasi-Hopf algebra and $G$ a finite group such that $H^2(G,Z_G(A))$ is trivial for every homomorphism $\brho^{\rm br}\colon G\to\ost^{\rm br}(A,R)$. The following are in one-to-one correspondence:
\begin{enumerate}
\item braided $G$-actions on $\C{A}$, up to equivalence;
\item monoidal functors $\ul G\to\2ust^{\rm br}(A,R)$, up to isomorphism;
\item group homomorphisms $\brho^{\rm br}\colon G\to\ost^{\rm br}(A,R)$ with $[\brho^{\rm br}]=0$ in $H^3(G,Z_G(A))$;
\item braided cocentral cleft extensions of $\k[G]$ by $(A,R)$, up to equivalence.
\end{enumerate}
\end{thm}
\begin{proof}
The bijections (i)$\leftrightarrow$(ii)$\leftrightarrow$(iii) follow from Theorem~\ref{LiftGactionTw-br}. For (i)$\leftrightarrow$(iv): under Theorem~\ref{2dimE}, a $G$-action $(\rho,\sigma)$ lifts to $\2ust^{\rm br}(A,R)$ iff each $(\vp^g,J^g)$ satisfies \eqref{eq:R-cond-multi}, which by Definition~\ref{DefGE-br} is exactly the condition that the resulting extension $A\#_{(\rho,\sigma)}G$ be braided with $R^H=(\iota\otimes\iota)(R)$. Equivalences of extensions intertwining the $R$-matrices correspond, under \eqref{prof.natco}, to monoidal natural isomorphisms of the corresponding braided monoidal functors.
\end{proof}

\subsubsection{Braided cyclic example}\label{sss:braided-cyclic}
Suppose $(\vp,J)\in\st^{\rm br}(A,R)$ satisfies $(\vp,J)^n=(I_a,\partial_a)^{-1}$ with $\vp(a)=a$. Then the cocentral cleft extension $A[u^{\pm 1};\vp,J,a,n]$ of Corollary~\ref{cor:anotEG} is quasitriangular, with $R^H=(\iota\otimes\iota)(R)$: the verification of \eqref{eq:QT3} on the generator $u$ is, as in Definition~\ref{DefGE-br}, exactly the braided condition $(\vp\o\vp)(R)=J_{21}RJ^{-1}$. Combined with Theorem~\ref{HopfCEG}:

\begin{thm}\label{HopfCEG-br}
Let $(A,R)$ be a finite-dimensional QT Hopf algebra and suppose $H^n(C_2,Z_G(A))$ is trivial for $n=2,3$ for every $\brho^{\rm br}\colon C_2\to\ost^{\rm br}(A,R)$. Then the following are in bijection:
\begin{enumerate}
\item nontrivial braided $C_2$-actions on $\C{A}$, up to equivalence;
\item nontrivial braided cocentral Hopf extensions of $\k[C_2]$ by $(A,R)$, up to equivalence;
\item order-two elements of $\ost^{\rm br}(A,R)$.
\end{enumerate}
\end{thm}

\section{The groups of stable twists of some pointed Hopf algebras}\label{S:StGrH}

This section is devoted to the explicit computation of the group of stable twists of the pointed Hopf algebras $A_\delta(n,\xi)$ (cf.~\cite{Rad1975}), which are
central extensions of the Taft algebras by a cyclic group algebra. For any Taft algebra and any cyclic group, there are two inequivalent extensions which are characterized by $\delta\in\{0,1\}$. The case $\delta=0$ \textup{(}including all Taft algebras\textup{)} admits a complete explicit classification of stable twists; the case $\delta=1$ with $d\ge 2$ involves a substantial chain of coefficient identities, normalizations, and a delicate existence argument. We collect the main statements first so that a reader interested only in conclusions can skip the proofs.

\subsection{Summary of main results}\label{SS:A-summary}
Let $\kk$ be an algebraically closed field of characteristic zero. For integers $m,d\ge 1$, set $n=md$, and fix a primitive $n$-th root of unity $\theta\in\kk^\times$. Set $\xi:=\theta^d$ (a primitive $m$-th root of unity). Let $\mathbb G_m\subseteq\kk^\times$ denote the group of $m$-th roots of unity, and for $i\ge 0$ write $[i]_\xi:=1+\xi+\cdots+\xi^{i-1}$ and $[i]_\xi!:=[1]_\xi\cdots[i]_\xi$.

For $\delta\in\{0,1\}$ (with $\delta=0$ enforced when $d=1$), define \textit{the pointed Hopf algebras $A_\delta(n,\xi)$:}
\[
A_\delta(n,\xi)\;:=\;\kk\langle x,g\rangle\,\big/\,\bigl\langle\,x^m-\delta(1-g^m),\;g^n-1,\;gx-\xi xg\,\bigr\rangle,
\]
with Hopf structure
\begin{align*}
\Delta(g)&=g\otimes g, & \Delta(x)&=x\otimes g+1\otimes x,\\
\ve(g)&=1, & \ve(x)&=0,\\
S(g)&=g^{-1}, & S(x)&=-xg^{-1}.
\end{align*}
Then $\dim A_\delta(n,\xi)=m^2d$, with $\kk$-basis $\{x^ig^j:0\le i\le m-1,\,0\le j\le n-1\}$. When $d=1$, $A_0(n,\xi)\cong T_{n^2}(\xi)$ is the Taft algebra of parameter $\xi$ \cite{Taft}. When $d>1$ and $\delta=1$, $A_1(n,\xi)$ sits in a non-split Hopf extension
\[
0\to\kk\langle g^m\rangle\to A_1(n,\xi)\to T_{m^2}(\xi)\to 0.
\]

We next record, for general $\delta$, some notation in the commutative Hopf subalgebra $\kg{g}\subseteq A_\delta(n,\xi)$. It contains the idempotents $e_i:=\frac1n\sum_{j\in\Z_n}\theta^{-ij}g^j$ satisfying
\begin{equation}\label{eq:com_rel}
g e_i=\theta^i e_i,\qquad e_i x=xe_{i-d}
\end{equation}
for $i\in\Z_n$. They form a complete set of primitive orthogonal idempotents of $\kg{g}$ and give the projective $A$-submodule decomposition $A=\bigoplus_{i\in\Z_n}A\,e_i$; the $Ae_i$ are indecomposable, so the $e_i$ are also primitive in $A_\delta(n,\xi)$, and by the second relation in \eqref{eq:com_rel} they are not central. Writing $\oi$ for the image of $i$ under $\Z_n\twoheadrightarrow\Z_d\cong\Z_n/\Z_n d$, the sums
\[
e_\oi:=\sum_{j\in\oi}e_j=e_i+\cdots+e_{i+d(m-1)}
\]
are central idempotents satisfying $g^m e_\oi=\theta^{mi}e_\oi$, hence
\begin{equation}\label{eq:gm}
e_\ok=\frac1d\sum_{j=0}^{d-1}\theta^{-mjk}g^{mj}\quad\text{for }\ok\in\Z_d.
\end{equation}
For $\delta=1$ we fix an $m$-th root $\g_\oi$ of $1-\theta^{mu}$ for each $u\in\oi$, $\oi\ne\ol0$, set $\g_{\ol0}:=0$, and write $\g_u:=\g_\ou$ for $u\in\Z_n$. Finally, since $A_\delta(n,\xi)^{\otimes 2}$ is free over $\kg{g}^{\otimes2}$ on the monomials $\{x^i\otimes x^j\}_{0\le i,j\le m-1}$, every $J\in A_\delta(n,\xi)^{\otimes2}$ expands uniquely as
\begin{equation}\label{Stwist}
J=\sum_{0\le i,j\le m-1}(x^i\otimes x^j)\,\J{}{i,j},\qquad \J{}{i,j}\in\kg{g}^{\otimes2},
\end{equation}
and we call $\JC{}{0}:=\J{}{0,0}$ its \emph{constant term}. In particular, write $\JC{}{i}:=\J{}{i,m-i}$.

The technical core of our computation is to show that every stable twist of $A_\delta(n,\xi)$ is equivalent to a standard normalized one; the classification is then carried out on these representatives.

\begin{deff}\label{d:sf}
A stable twist $(\vp,J)$ of $A_\delta(n,\xi)$ is \emph{standard} if
\begin{equation}\label{eq:simplerJ}
\vp(g)=g^l,\qquad \vp(x)=xa,\qquad J=\JC{}{0}+\sum_{i=1}^{m-1}(x^i\otimes x^{m-i})\,\JC{}{i},
\end{equation}
where $l$ is an integer coprime to $n$ with $l\equiv1\pmod m$, the element $a$ lies in $\k^\times$ if $\delta=0$ and equals $\sum_{\oi\in\Z_d}a_\oi e_\oi\in\kk\langle g^m\rangle$ with $\ao\ne0$ and $a_\oi^m=\g_{\oi l}^m/\g_\oi^m$ $(\oi\in\ZD)$ if $\delta=1$, and $\JC{}{0}(\eo\otimes\eo)=\eo\otimes\eo$. It is \emph{normalized} if in addition $\JC{}{0}=1\otimes1$.
\end{deff}

We first classify all stable twists of $A_0(n,\xi)$ explicitly.

\begin{thm}\label{STH0}
Every stable twist of $A_0(n,\xi)$ is equivalent to $(\vp^\lambda,J^\alpha)$ for some $\lambda\in\kk^\times$ and $\alpha\in\kk$, where
\begin{gather*}
\vp^\lambda(g)=g,\quad \vp^\lambda(x)=\lambda x,\\
J^\alpha\;:=\;1\otimes 1\;+\;\alpha\sum_{1\le i\le m-1}\frac{x^i}{[i]_\xi!}\otimes\frac{x^{m-i}g^i}{[m-i]_\xi!},
\end{gather*}
and $\alpha=0$ if $m<n$. In particular, $[J^\alpha,\Delta(A_0(n,\xi))]=0$.
\end{thm}

\begin{cor}\label{c:all_twists}
Every $2$-cocycle of the Taft algebra $T_{m^2}(\xi)$ is equivalent to some $J^\alpha$ up to coboundary.
\end{cor}

In the next result, we denote by $\ost_t(A)\subseteq\ost(A)$ the subgroup consisting of the classes of stable twists $(\varphi,J)$ with $J$ a coboundary. The Taft case recovers and re-derives the bi-Galois classification of Masuoka \cite{Ma94} and Schauenburg \cite{Schauenburg2000} via a different route.

\begin{cor}\label{StableTwTaft}
We have $\ost_t(A_0(n,\xi))\cong\kk^\times/\mathbb G_m$, and
\[
\ost(A_0(n,\xi))\;\cong\;\begin{cases}
\kk^\times/\mathbb G_m & \text{if } m<n,\\
\kk\rtimes(\kk^\times/\mathbb G_m) & \text{if } m=n,
\end{cases}
\]
where in the Taft case $\kk^\times/\mathbb G_m$ acts on $(\kk,+)$ by $[\lambda]\cdot\alpha=\lambda^m\alpha$.
\end{cor}

Now, we focus on the stable twists of $A_1(md,\xi)$ with $d\ge 2$. This is the technical heart of the section, computed by combining the general exact sequence $1\to H^2_\ell(A^*)\to\ost(A)\xrightarrow{\Pi}\OutTw(A)\to 1$ of Theorem~\ref{t:general-exact-stable} with a classification of stable twists in a convenient normal form.

\begin{thm}\label{inverT}
Let $(\vp,J)$ be a standard normalized stable twist of $A_1(n,\xi)$ with $n=md$, $d>1$. Set $\delta_m:=(3+(-1)^m)/2$. Then:
\begin{enumerate}
\item[\textup{(i)}] $\vp(g)=g^l$ for some integer $l$ with $l\equiv\pm 1\pmod{\delta_m d}$.
\item[\textup{(ii)}] If $l\equiv 1\pmod{\delta_m d}$ \textup{(}and $d>2$ when $\delta_m=1$\textup{)}, then $J=1\otimes 1$ and $(\vp,J)\in\st_0(A)$.
\item[\textup{(iii)}] If $l\equiv-1\pmod{\delta_m d}$, then $\gcd(m,d)=1$, and $(\vp,J)$ is equivalent to the unique nontrivial standard normalized stable twist $(\tilde\vp,\tilde J)$ with
\[
\tilde\vp(g)=g^l,\qquad \tilde\vp(x)=\sqrt[m]{-1}\,x\,g^{-r},
\]
\begin{equation}\label{eq:tJ}
\tilde J\;=\;1\otimes 1-\sum_{i=1}^{m-1}(x^iy\otimes x^{m-i}g^{(1-r)i}y)\!\left(\prod_{j=i+1}^{m-1}(g^r-\xi^j)\otimes\prod_{j=m-i+1}^{m-1}(g^r-\xi^j)\right),
\end{equation}
where $1<r<n$ is determined by $r\equiv 0\pmod m$ and $r\equiv 1\pmod d$, and $y=\sum_{k=0}^{r/m-1}g^{mk}+\frac{1-r}{n}\sum_{k=0}^{d-1}g^{mk}$.
\end{enumerate}
\end{thm}

\begin{thm}\label{StableTwTaft2d}
The group $H^2_\ell(A_1(md,\xi)^\ast)$ is trivial, and moreover
\[
\ost(A_1(md,\xi))\;\cong\;\OutTw(A_1(md,\xi))\;\cong\;\begin{cases} \Z/2\Z & \text{if }\gcd(m,d)=1,\\ 1 & \text{otherwise.}\end{cases}
\]
\end{thm}

When $\gcd(m,d)=1$ and $d$ is odd, the unique nontrivial class in $\ost(A_1(md,\xi))\cong \Z/2\Z$ lifts to an explicit $C_2$-equivariantization on $\C{A_1(md,\xi)}$, realized by the following new family of Hopf algebras $\N(md,\xi)$.

\begin{deff}\label{def:N-family}
Let $d,m>1$ be coprime integers and $\xi$ a primitive $m$-th root of unity. Choose integers $l,r,s\in\Z/md\Z$ with
\begin{align*}
l&\equiv 1,& r&\equiv 0,& s&\equiv\tfrac{(-1)^m-3}{2} & \pmod m\\
l&\equiv-1,& r&\equiv 1,& 2s&\equiv 1-m & \pmod d.
\end{align*}
The Hopf algebra $\N(md,\xi)$ is the $\kk$-algebra on generators $g,x,u$ with relations
\begin{align*}
x^m&=1-g^m, & g^{md}&=1, & gx&=\xi xg,\\
gu&=ug^l, & xu&=\sqrt[m]{-1}\,uxg^{-r}, & u^2&=z,
\end{align*}
where $z=\tfrac{1}{m}g^s\prod_{i=1}^{m-1}(g^r-\xi^i)$, equipped with the coalgebra structure
\begin{align*}
\Delta(g)&=g\otimes g, & \Delta(x)&=x\otimes g+1\otimes x, & \Delta(u)&=(u\otimes u)J,\\
\ve(g)&=1, & \ve(x)&=0, & \ve(u)&=1,
\end{align*}
where, assuming $1<r<md$ and writing $y=\sum_{k=0}^{r/m-1}g^{mk}+\frac{1-r}{n}\sum_{k=0}^{d-1}g^{mk}$,
\[
J=1\otimes 1-\sum_{i=1}^{m-1}(x^iy\otimes x^{m-i}g^{(1-r)i}y)\!\left(\prod_{j=i+1}^{m-1}(g^r-\xi^j)\otimes\prod_{j=m-i+1}^{m-1}(g^r-\xi^j)\right),
\]
and antipode $S(g)=g^{-1}$, $S(x)=-xg^{-1}$, $S(u)=ug^{-s}$.
\end{deff}

\begin{remark}\label{rem:N-field}
$\N(md,\xi)$ is well-defined over any field $\kk$ that contains $\BQ(\sqrt[m]{-1})$. The parameters $l,r,s$ are uniquely fixed modulo $md$ when $d$ is odd; when $d$ is even there are two distinct choices of $s$, giving two non-isomorphic Hopf algebras.
\end{remark}

The explicit presentations in the smallest cases are recorded as follows.

\begin{example}\label{ex:8d}
Suppose $m=2$ and $d>1$ is odd. Let $\ve=\jacd$ be the Jacobi symbol of $-1$ modulo $d$. Then $l=2d-1$, $r=d+1$, and
\[
s=\begin{cases}\tfrac{d-1}{2}+d & \text{if }\ve=1,\\[2pt]\tfrac{d-1}{2} & \text{if }\ve=-1.\end{cases}
\]
With $y=\tfrac{1}{2}\bigl(1+\sum_{i=1}^{(d-1)/2}(g^{2i}-g^{-2i})\bigr)$ and $z=\tfrac{1}{2}(g^{(d+\ve)/2}+g^{-(d+\ve)/2})$, the Hopf algebra $\N(2d,-1)$ has generators $g,x,u$, defining relations
\begin{align*}
x^2&=1-g^2, & g^{2d}&=1, & gx&=-xg,\\
gu&=ug^{-1}, & xu&=\sqrt{-1}\,ug^{d-1}x, & u^2&=z,
\end{align*}
and Hopf structure
\begin{align*}
\Delta(g)&=g\otimes g, & \Delta(x)&=x\otimes g+1\otimes x, & \Delta(u)&=u\otimes u+uxy\otimes uS(xy),\\
\ve(g)&=1, & \ve(x)&=0, & \ve(u)&=1,\\
S(g)&=g^{-1}, & S(x)&=-xg^{-1}, & S(u)&=ug^{(\ve d+1)/2}.
\end{align*}
The family $\N(2p,-1)$ for $p$ an odd prime is used in \S\ref{S:clas} to complete the classification of $8p$-dimensional nonsemisimple Hopf algebras.
\end{example}

Combining the previous results yields a complete picture of nontrivial cocentral cleft Hopf extensions of $\kk[C_2]$ by $A_1(md,\xi)$:

\begin{thm}\label{th:extC2}
Let $A=A_1(md,\xi)$ with integers $m,d\ge 2$. The nontrivial cocentral cleft Hopf extensions of $\kk[C_2]$ by $A$ \textup{(}equivalently, nontrivial $C_2$-actions on $\C{A}$ up to equivalence\textup{)} are classified as follows.
\begin{enumerate}
\item[\textup{(i)}] If $\gcd(m,d)>1$ and $d$ is odd, there is no nontrivial extension.
\item[\textup{(ii)}] If $\gcd(m,d)>1$ and $d$ is even, there is a unique nontrivial extension, given by $A[u^{\pm 1};\id,1\otimes 1,g^m]$.
\item[\textup{(iii)}] If $\gcd(m,d)=1$ and $d$ is odd, there is a unique nontrivial extension, given by $\N(md,\xi)$.
\item[\textup{(iv)}] If $\gcd(m,d)=1$ and $d$ is even, there are three nontrivial extensions, given by $\N(md,\xi)$ and $A[u^{\pm 1};\id,1\otimes 1,g^m]$.
\end{enumerate}
\end{thm}

The remainder of \S\ref{S:StGrH} is devoted to the proofs and supporting computations. A reader content with the statements above can proceed directly to \S\ref{S:StGrQH}.

\subsection{Standard stable twists of $A_\delta(n,\xi)$}\label{sss:Adelta-standard}
Throughout, we write $A=A_\delta(n,\xi)$. We record a structural fact used repeatedly below. For $\delta = 0$, the Jacobson radical of $A$ is $Ax$, of nilpotency index $m$. For $n > m$ with $\delta =1$, $x^m$ acts on $Ae_\ok$ as the scalar $1-\theta^{mk} \ne 0$ for $\ok \ne \ol{0}$, while $x^m$ acts trivially on $A \eo$: indeed $A \eo \cong A/A x^m$ is isomorphic, as an algebra, to the Taft algebra $T_{m^2}(\xi)$, whose Jacobson radical is $Ax\eo$, of nilpotency index $m$. 

The expansion \eqref{Stwist} of $J\in A^{\otimes2}$ extends to higher tensor powers: for any positive integer $r$, the tensor power $A^{\o r}$ is free over $\kg{g}^{\o r}$ with basis $\{x^{k_1} \o \cdots \o x^{k_r}\mid 0\le k_1, \dots, k_r \le m-1\}$, so every $Z \in A^{\o r}$ is uniquely
$$
Z = Z^{\bo}+\sum_{k_1 + \cdots + k_r >0} (x^{k_1} \o \cdots \o x^{k_r})Z^\bk,\qquad Z^{\bo}, Z^{\bk} \in \kg{g}^{\o r},
$$
where $\bk=[k_1, \ldots, k_r] \in \Z_{\ge 0}^r$ and $\bo$ is the zero vector; as in the case $r=2$ of \eqref{Stwist}, we call $Z^{\bo}$ the \emph{constant term} of $Z$. Refining the coefficients of a stable twist $(\vp, J)$ of $A$,
$$
\J{}{i,j} = \sum_{u,v \in \Z_n} \J{u,v}{i,j}\, e_u \o e_v
$$
for some $\J{u,v}{i,j}\in \k$. Since $(\ve \o \id)(J)= (\id \o \ve)(J)=1$, we find
$$
\JC{u,0}{0}=1 = \JC{0,v}{0}
$$
for all $u,v\in \mathbb Z_n$.

\begin{lem} \label{l:vg}
Let $(\vp, J)$ be a stable twist of $A_\delta(n,\xi)$ with $\vp(g)=g^l$ for some integer $l$ coprime to $n$. Then $l\equiv 1 \pmod m$, $\vp(x) = x a$ for some nonzero $a\in \kg{g}$, and $\J{}{i,j}=0$ unless $i+j\in \{0,m\}$. If in addition $\JC{}{0}=1\otimes 1$, then $(\vp,J)$ is a normalized standard stable twist. 
\end{lem}
\begin{proof}
Since $g x g^{-1} = \xi x$, $g^l \vp(x) g^{-l} = \xi \vp(x)$ and so $\vp(x) = x^\ell a$ for some $a \in \kg{g}$ and positive integer $\ell \le m-1$ such that $l\ell \equiv 1 \mod m$. When $\delta=1$, we have $\vp(\eo) = \eo$. As noted above, $Ax\eo$ is the Jacobson radical of $A\eo$, of nilpotency index $m$, and the same holds for $\vp(A x\eo) = A x^\ell a e_0 = A x^\ell e_0$. If $\ell > 1$, the nilpotency index of $A x^\ell \eo$ is strictly less than $m$. Therefore, $\ell =1$ and $l \equiv 1 \mod m$. The argument for $\delta=0$ is similar. 

     Writing $J$ as in \eqref{Stwist}, since $J \Delta(\vp(g)) = ((\vp \o \vp)\Delta(g)) J$, we find $J (g^l \o g^l) = (g^l \o g^l) J$.  By comparing the coefficients of $x^i \o x^j$ on both sides, we find
     $$
     \xi^{l(i+j)}  \J{}{i,j}\,(g^l \o g^l) = \J{}{i,j}\, (g^l \o g^l)\, ,
     $$
     which implies $\J{}{i,j} = 0$ unless $i+j \in\{0,m\}$, so that $J$ takes the reduced form \eqref{eq:simplerJ}.

Now, suppose $\JC{}{0}=1\otimes 1$. By comparing the coefficients of $1\otimes x$ on both sides of $J\Delta(\varphi(x))=(\varphi\otimes \varphi)\Delta(x)J$, we obtain that 
\begin{equation*}
    (1\otimes x)\Delta(a)+(x^{m-1}\otimes x)\J{}{m-1,1}(x\otimes g)=1\otimes (xa)+(xa\otimes g^l)(x^{m-1}\otimes x)\J{}{m-1,1}.
\end{equation*}
When $\delta=0$, we get $\Delta(a)=1\otimes a$, which implies that $a\in \k^\times$. When $\delta=1$, we consider the integral $\Lambda=\sum_{i=1}^{d-1}g^{mi}$ of the group algebra $\k\langle g^m\rangle$. By multiplying $\Lambda \otimes 1$ on both sides of the above equation,  we get $(\Lambda\otimes 1)\Delta(a)=\Lambda \otimes a$ since $x^m\Lambda=(1-g^m)\Lambda=0$. This implies that $a = \sum_{\oi \in \Z_d} a_\oi e_\oi \in \kg{g^m}$ and
\begin{align*}
    \varphi(x^m)&=\sum_{\oi \in \Z_d}a_{\oi}^mx^me_{\oi}=\sum_{\oi \in \Z_d}a_{\oi}^m(1-g^m)e_{\oi}=\sum_{\oi \in \Z_d}a_{\oi}^m\gamma_\oi^me_{\oi}\\
    &=\varphi(1-g^m)=1-g^{lm}=\sum_{\oi \in \Z_d} \gamma_{l\oi}^me_\oi.
\end{align*}
Consequently, $a^m_\oi = \frac{\g^m_{\oi l}}{\g^m_\oi}$ for $\oi \in \ZD$ and $a_{\overline{0}}\neq 0$ is obvious. 
\end{proof}

\begin{lem} \label{l:cocycle}
    Let $J$ be a 2-cocycle of $A = A_\delta(n,\xi)$ as given in \eqref{eq:simplerJ}. The following equality holds
    \begin{align}\label{eq:2coJ}
     \JC{v,w}{0}\JC{u, v+w}{0} = \JC{u,v}{0}\JC{u+v, w}{0} 
    \end{align}
    for all $u,v \in \Z_n$. If, in addition, $\JC{u,v}{0} \ne 0$ for all $u,v \in \Z_n$, there exists $z \in \kg{g}^\times_1$ such that $(J_z)^{\bo} = 1 \o 1$.
\end{lem}
\begin{proof}
The first statement of the lemma can be derived from equating the constant terms on both sides of the identity $(1 \o J)(\id \o \Delta)(J) = (J \o 1)(\Delta \o \id)(J)$. 
     
     If $\JC{u,v}{0} \ne 0$ for all $u,v \in \Z_n$, then $J^{\bo}: \Z_n\times \Z_n\to \k^\times$ defines a normalized 2-cocycle. Since $H^2(\Z_n, \k^\times)$ is trivial, there exists a function $f:\Z_n \to \k^\times$ with $f(0)=1$ such that
    $$
  \JC{u,v}{0} =   \frac{f(u+v)}{f(u)f(v)}
    $$
    for all $u,v \in \Z_n$. Now we let $z=\sum_{u \in \Z_n} f(u) e_u$. Then, it is direct to check that $z \in \kg{g}^\times_1$, and $(J_z)^{\bo}= 1\o 1$.
\end{proof}
\begin{remark}\label{rmk:inverA}
    It can be seen immediately that $\JC{u,v}{0} \ne 0$ for $A_0(n,\xi)$ from the invertibility of $J$. However, it is not the case for  $A_1(n,\xi)$.
\end{remark}

\subsection{Stable twists of $A_0(n,\xi)$}\label{sss:A0-stable} We now compute all stable twists of $A=A_0(n,\xi)$ explicitly. Theorem~\ref{STH0} shows that every stable twist is trivial when $n>m$, while for the Taft algebra $n=m$ the nontrivial invariant twists are exactly the family $J^\alpha$ from the standard form above. Corollary~\ref{StableTwTaft} then identifies $\ost(A_0(n,\xi))$ explicitly.

\begin{lem}\label{bgT}
     Let $(\varphi, J)$ be a stable twist of $A_0(n,\xi)$. Then there exists $z \in \A$ such that $\vp_z(g)=g^l$ for some integer $l \equiv 1\mod m$ coprime to $n$, and
     $$
     (J_z)^{\bo} = 1 \o 1\,.
     $$
\end{lem}
\begin{proof}

The Jacobson radical of $A$ is the ideal $I=Ax$ generated by $x$, which is a Hopf ideal.  Then $\vp(I) = I$ and so $\Delta_J(\varphi(I)) = (\varphi \o \varphi)\Delta(I) \subset A \o I + I \o A$. Thus, $I$ is a Hopf ideal of $A$ and $A_J$. Therefore,  $\varphi$ induces a Hopf algebra isomorphism $\ol \varphi: A/I \to A_J/I$. 
Since $A/I \cong A_J/I \cong \k[C_n]$ as Hopf algebras and $g$ is a generator of $C_n$, $\ol\vp(g)$ must be a grouplike element generating $C_n$. Therefore, $\vp(g)=g^l+a$ for some integer $l$ coprime to $n$ and $a\in I$.  

We consider the linear operator $D_a: I\to I$ such that $D_a(b)=[g^l,b]-ba$ for any $b\in I$. Clearly, $D_a(I^i) \subseteq I^i$ for any positive integer $i$. We have the induced operator $\overline{D}_a(-)=[g^l,-]$ on each factor $I^i/I^{i+1}$ for $1\le i\le m-1$. It is clear that $\overline{D}_a$ is bijective since $\overline{D}_a(x^i h)=(\xi^{li}-1)x^ig^lh$ for any $h \in \k[g]$. Hence, $D_a$ is bijective and there exists $b\in I$ such that $D_a(b)=a$. Let $z=1+b$. Then, we have $z \in A^\times_1$ and 
\[
z\varphi(g)-g^lz=(1+b)(g^l+a)-g^l(1+b)=a-D_a(b)=0,
\]
which implies that $\vp_{z}(g)=z\vp(g)z^{-1}=g^l$. By Lemma \ref{l:vg}, $l \equiv 1 \mod m$.

Now we consider the stable twist $(\vp_{z}, J_{z})$. Then $J_z$ and its inverse $\ol{J_z}$ are in the form of \eqref{eq:simplerJ}. Thus, the constant term of $J_{z} \ol J_{z}$ is $J_{z}^{\bo} \ol{J}_{z}^{\bo}$ and so $J_{z}^{\bo} \ol{J}_{z}^{\bo} = 1 \o 1$. Therefore, $\JC{u,v}{0} \ne 0$ for all $u,v \in \Z_n$. By Lemma \ref{l:cocycle}, there exists $z' \in \kg{g}^\times_1$ such that $(J_{z'z})^{\bo} = 1\o 1$ and $\vp_{z'z}(g) = \vp_{z'}(g)=g^l$. 
\end{proof}

Recall the $q$-integers $[n]_q=\frac{q^n-1}{q-1}$ and $q$-factorial $[0]_q!=1$, $[n]_q!=\prod_{1\le i\le n} [i]_q$ and $\left[\begin{matrix}n\\m\end{matrix}\right]_q=\frac{[n]_q!}{[m]_q! [n-m]_q!}$.
We thank Fedor Petrov for informing us that the following identity of $q$-binomial coefficients is an immediate consequence of the $q$-analogue of the Chu--Vandermonde identity.
\begin{lem}
    For positive integers $k \le n$,  we have
    \begin{equation}    \label{eq:qidentity}
    \sum_{i=0}^n(-1)^i\left[\begin{matrix}n\\i\end{matrix}\right]_q\,\left[\begin{matrix}k+i-1\\k-1\end{matrix}\right]_q\, q^{\frac{i(i+1-2k)}{2}}=0\,.
\end{equation}
\end{lem}
\begin{proof}
By the $q$-Chu-Vandermonde identity 
$$
\qbin{a+b}{n} = \sum_{i=0}^n \qbin{a}{n-i} \qbin{b}{i}  q^{i(a-n+i)} 
$$
with $a=n, b=-k$, we find
$$
0=\qbin{n-k}{n} = \sum_{i=0}^{n}q^{i^2}\qbin{n}{n-i} \qbin{-k}{i} \,.
$$
Now, \eqref{eq:qidentity} follows from the equalities 
$$
\qbin{-k}{i} = (-1)^i q^{-\binom{i}{2}-ki}\qbin{k-1+i}{i} = (-1)^i q^{-\binom{i}{2}-ki}\qbin{k-1+i}{k-1}.
$$
\qedhere
\end{proof}

\begin{proof}[Proof of Theorem~\ref{STH0}]
By Lemmas \ref{l:vg} and \ref{bgT}, we can assume $\varphi(g)=g^l$, $\varphi(x)=\ld x$ and 
\[
 J = 1\o 1+ \sum_{i=1}^{m-1} (x^i \o x^{m-i})\JC{}{i}
\]
for some $l\equiv 1 \mod m$ with $l$ coprime to $n$, $\ld\in \k^\times$ and $\JC{}{i}\in \k[g]\otimes \k[g]$. Consider the 2-cocycle  equation $(1 \o J)(\id \o \Delta)(J)=(J \o 1)(\Delta \o \id)(J)$. The coefficient of $1\otimes x^i\otimes x^{m-i}$ in $(1 \o J)(\id \o \Delta)(J)$ is   $1 \o \JC{}{i}$ while that of $(J \o 1)(\Delta \o \id)(J)$ is $(\Delta \o \id)(\JC{}{i})$ for $1\le i\le m-1$. This implies $\JC{}{i} = 1 \o b_i$ for some $b_i \in \k[g]$.
Similarly, comparing the coefficients of $x^i\otimes x^{m-i}\otimes 1$ yields $\sum_{(b_i)}1\o {b_i}_{(1)} \o g^{m-i} {b_i}_{(2)} = 1 \o b_i\o 1$, which implies $b_i = \a_i g^{i-m}$ for some $\a_i \in \k$ and $i =1,\dots, m-1$. Hence, 
\[
J=1\otimes 1+\sum_{1\le i\le m-1} \a_i x^i\otimes x^{m-i}g^{i-m}\, .
\]

Suppose $n > m$. We further compare the coefficients of $x^i\otimes 1 \otimes x^{m-i}$ in the 2-cocycle equation to obtain 
$$
\a_i 1\o g^{i-m}\o g^{i-m} = \a_i 1\o  g^i \o g^{i-m},
$$
which implies $\a_i =0$ since $g^i \ne g^{i-m}$. Hence $J = 1\o 1$, $A_J =A$ and $\varphi$ is a Hopf algebra automorphism of $A$. This implies $l=1$ and the statement follows. 

Suppose $n=m$. Then the equation $J\Delta(\varphi(x))=(\varphi\otimes \varphi)\Delta(x)J$ yields 
\begin{align*}
  \sum_{i=1}^{m-2} & \a_i x^{i+1}\otimes x^{m-i}g^{i-m+1} + \sum_{i=2}^{m-1} \a_i\xi^{i} x^{i}\otimes x^{m-i+1}g^{i-m} \\
 =&  \sum_{i=1}^{m-2}  \a_i \xi^{-i} x^{i+1}\otimes x^{m-i}g^{i-m+1}+ \sum_{i=2}^{m-1}  \a_i x^i\otimes x^{m-i+1}g^{i-m}
\end{align*}
which is equivalent to
$$
\sum_{i=2}^{m-1} (\a_{i-1}+\xi^{i}\a_i)x^i\otimes x^{m-i+1}g^{i-m}=\sum_{i=2}^{m-1} (\xi^{-i+1}\alpha_{i-1}+\a_i)x^i\otimes x^{m-i+1}g^{i-m}.
$$
Thus, $\alpha_i=\frac{\xi^{-i+1}-1}{\xi^i-1}\,\alpha_{i-1}=\frac{[m-i+1]_\xi}{[i]_\xi}\alpha_{i-1}$ and so $\a_i = \frac{[m-1]_\xi!}{[i]_\xi! [m-i]_\xi!} \a_1$. Thus, $J$ is of the stated form by setting $\a = [m-1]_\xi ! \a_1$. In particular, we find $[J^\a, \Delta(x)] =[J^\a, \Delta(g)]=0$, which implies that $[J^\alpha,\Delta(A)]=0$.  Then it is clear that $\varphi(x)=\lambda x$ belongs to the $(g^l,1)$ primitive space of $A$. Hence, $l=1$ and the statement follows as well.  
\end{proof}

The cleft extensions and bi-Galois objects of Taft algebras are computed in \cite{Ma94, Schauenburg2000}. Here, we list all the 2-cocycles of Taft algebras up to coboundary. 
\begin{proof}[Proof of Corollary~\ref{c:all_twists}]
It suffices to show that $A\cong A_J$ as bialgebras for any 2-cocycle $J$ of $A=T_{m^2}(\xi)$. Recall that $A \cong A_J$ as $\k$-algebras and so $A^* \cong (A_J)^*$ as coalgebras. Since $A \cong A^*$ as Hopf algebras, $A^*$ is pointed with a nontrivial $(\a, 1)$ skew primitive element $y$ and the grouplike element $\a$ has order $m$. By \cite[Proposition 1.8]{AN01}, $(A_J)^*$ is isomorphic to some Taft algebra $T_{m^2}(\xi')$, and so is $A_J$. In particular, $T_{m^2}(\xi')$ and $T_{m^2}(\xi)$ are gauge equivalent and so  $\xi=\xi'$ by \cite[Cor 2.4]{Schauenburg2000} (or \cite[Cor. 3.3]{KMN2012}).
\end{proof}

\begin{proof}[Proof of Corollary~\ref{StableTwTaft}] In light of Theorem \ref{STH0}, every equivalence class in $\ost(A)$ contains a stable twist $(\vp^\ld, J^\a)$ for some $\a, \ld \in \k$ with $\ld \ne 0$. If $(\vp^\ld, J^\a)$ is equivalent to $(\vp^{\ld'}, J^{\a'})$, then there exists  $a\in A_1^\times$ such that 
\begin{equation}\label{eq:equiv}
    (a\otimes a)J^{\a}\Delta(a)^{-1}=J^{\a'} \quad\text{and}\quad I_a \circ \vp^\ld=\vp^{\ld'}. 
\end{equation}
The latter implies $a g a^{-1} = g$ by Theorem \ref{STH0}, and so $a \in \kg{g}$.
Since $[J^\a,\Delta(a)]=0$, the first equation of \eqref{eq:equiv} implies  $(a\otimes a)\Delta(a)^{-1}=1\otimes 1$, and hence $a$ is a group-like element and $\a = \a'$. So, $a=g^i$ for some integer $i$, and $a\vp^\ld(x) a\inv = \xi^i\ld x=\ld' x$. Thus, $\ld'/\ld \in \BG_m$. Conversely, if   $\ld'/\ld = \xi^i$, then $\cb{a}(\varphi^\ld,J^\a) =(\varphi^{\ld'},J^{\a'})$ where $a=g^i \in \A$. This proves that $(\vp^\ld, J^\a)$ is equivalent to $(\vp^{\ld'}, J^{\a'})$ if and only if $\ld'/\ld \in \BG_m$ and $\alpha=\alpha'$. 

In particular, $\ost_t(A)$ consists of the equivalence classes of the form $\ol{(\vp^\ld,1\o 1)}$, and we simply denote its equivalence class by $\widetilde{\ld}$. Then the assignment $\widetilde{\cdot}: \k^\times \to \ost_t(A)$, $\ld \mapsto \widetilde{\ld}$ defines a surjective group homomorphism, and its kernel is $\BG_m$. Therefore,  $\ost_t(A) \cong \k^\times/\BG_m$. 

For $m\neq n$, by Theorem \ref{STH0}, every stable twist of $A$ is equivalent to a trivial one. Thus, $\ost(A) \cong \ost_t(A) \cong \k^\times/\BG_m$.

Now we consider the case $m=n$. Again by Theorem \ref{STH0}, 
$$
\ost(A) =\{\ol{(\vp^\ld, J^\a)}\mid \ld, \a \in \k \text{ and }\ld \ne 0\}.
$$
The assignment $\widehat{\cdot}:(\k,+) \to \ost(A)$, $\a \mapsto \widehat{\a}=\ol{(\vp^1, J^\a)}$ clearly defines a group homomorphism. It is injective, as discussed at the beginning of this proof. Obviously, $\widehat{\a} \widetilde{\ld} = \ol{(\vp^\ld, J^\a)}$ and so $\widehat{\k} \cdot\widetilde{\k^\times} =\ost(A)$ by Theorem \ref{STH0}. Moreover, $\widehat{\k} \cap \widetilde{\k^\times} = \{\widehat{0}\cdot \widetilde{1}\}$, the trivial subgroup. To complete the proof, it suffices to show that $\widehat{\k}$ is normal in $\ost(A)$.
Let  $\Omega = J^1 - 1\o1$. For any $\a \in \k$ and $\ld \in \k^\times$, we have
\begin{align*}
\widetilde{\ld} \widehat{\a} \widetilde{\ld}\inv &=\, \ol{(\vp^\ld, 1\o 1 )}\cdot \ol{(\varphi^1,J^\alpha)}\cdot  \ol{(\vp^{\ld\inv},1\otimes 1)}=\ol{(\vp^\ld, 1\o 1 )}\cdot \ol{(\vp^{\ld\inv}, J^\a)}\\
&=\,\ol{\left(\vp^\ld\vp^{\ld\inv},\, (\vp^\ld\otimes \vp^\ld)J^{\a} \right)}=\ol{\left(\vp^{1}, \, 1\otimes 1+\a(\vp^\ld\otimes \vp^\ld)(\Omega)\right)}\\
&=\, \ol{\left(\vp^{1},\, 1\otimes 1+\lambda^m\a\Omega)\right)}=\ol{\left(\vp^{1},\, J^{\lambda^m\a}\right)} = \widehat{\ld^m\a}. \qedhere
\end{align*}
\end{proof}

\begin{example}[Sweedler's $H_4$ and a braided refinement]\label{ex:H4-braided}
For Sweedler's four-dimensional Hopf algebra $A=H_4=A_0(2,-1)$, Corollary~\ref{StableTwTaft} gives the stable twist group $\ost(H_4)\cong(\k,+)\rtimes(\k^\times/\{\pm 1\})$. Radford \cite{Radford1993} classified the quasitriangular structures on $H_4$ up to gauge equivalence as the $1$-parameter family $(t\in\k)$
\[
R_t\;=\;\tfrac12\bigl(1\otimes 1+g\otimes 1+1\otimes g-g\otimes g\bigr)\,+\,\tfrac{t}{2}\bigl(x\otimes x+x\otimes gx-gx\otimes x+gx\otimes gx\bigr).
\]
For $H_4$, the standard form of \S\ref{sss:Adelta-standard} specializes to $J^\a=1\o 1+\a(x\o xg)$. Two direct computations, using $x^2=0$, $gx+xg=0$, and $g^2=1$, give
\begin{equation}\label{eq:H4-J-actions}
(\vp^\ld\o\vp^\ld)(R_t)=R_{\ld^2 t},\qquad (J^\a)_{21}\,R_t\,(J^\a)^{-1}=R_{t+2\a}\, . 
\end{equation}
The braided condition \eqref{eq:R-cond-multi} for $(\vp^\ld,J^\a)$ on $(H_4,R_t)$ thus reduces to $R_{\ld^2 t}=R_{t+2\a}$, equivalently
\[
\a\;=\;\tfrac{1}{2}(\ld^2-1)\,t.
\]
For $t=0$ this forces $\a=0$, so the braided representatives are $(\vp^\ld,1\o 1)$ with $\ld\in\k^\times$; for $t\ne 0$ the parameter $\a$ is no longer free but uniquely determined by $\ld$, and the braided representatives $(\vp^\ld,J^{(\ld^2-1)t/2})$ have a generically nontrivial gauge component. In either case the inner gauge by $g\in G(H_4)$ identifies $\ld$ with $-\ld$, and we obtain
\begin{equation}\label{eq:Twbr-H4}
\ost^{\rm br}(H_4,R_t)\;\cong\;\k^\times/\{\pm 1\}\quad\text{for every } t\in\k.
\end{equation}
The braiding thus locks $\a$ to $\ld$, eliminating it as a free parameter without forcing $J=1\o 1$. As an abstract group $\ost^{\rm br}(H_4,R_t)\cong\k^\times/\{\pm 1\}$ uniformly in $t$, but its embedding in $\ost(H_4)\cong(\k,+)\rtimes(\k^\times/\{\pm 1\})$ depends on $t$:
\[
\ost^{\rm br}(H_4,R_t)\;=\;\bigl\{\bigl(\tfrac{1}{2}(\ld^2-1)t,\;[\ld]\bigr):[\ld]\in\k^\times/\{\pm 1\}\bigr\}.
\]
For $t=0$ this is the algebra-automorphism factor $\{0\}\rtimes(\k^\times/\{\pm 1\})$; for $t\ne 0$ the $(\k,+)$-component $\tfrac{1}{2}(\ld^2-1)t$ is generically nonzero.
\end{example}

\subsection{Stable twists of $A_1(n,\xi)$}\label{SS:A1}
Having computed $\ost(A_0(n,\xi))$ in \S\ref{sss:A0-stable}, we now turn to the more involved case $A=A_1(n,\xi)$ with $n=md$ and $d>1$. By Theorem~\ref{t:general-exact-stable} there is an exact sequence
\begin{equation}\label{eq:A1-SES}
1\longrightarrow H^2_\ell(A^\ast)\longrightarrow \ost(A)\xrightarrow{\Pi} \OutTw(A)\longrightarrow 1,
\end{equation}
where $\Pi\colon\ost(A)\to \operatorname{Aut}_{\rm alg}(A)/\operatorname{Inn}(A)$, $[(\vp,J)]\mapsto[\vp]$. We compute $\OutTw(A)$ and $\ker(\Pi)$ separately, working throughout in the standard normal form of Definition~\ref{d:sf}. We first set up the matrix-algebra decomposition of $A$ and reduce every stable twist to standard form (\S\ref{SS:A1-framework}), derive the resulting coefficient identities (\S\ref{SS:A1-coefficients}), normalize the representatives (\S\ref{SS:A1-normalization}), and assemble the classification (\S\ref{SS:A1-classification}); a final subsubsection realizes the unique nontrivial class as a Hopf extension (\S\ref{SS:A1-C2action}).

\subsubsection{The matrix-algebra decomposition and reduction to standard form}\label{SS:A1-framework}
Throughout this and the subsequent subsubsections we use the indecomposable matrix-algebra decomposition of $A_1(md,\xi)$ developed next, culminating in Remark~\ref{r:heisenberg}.

Since $x^m=1-g^m$ is central in $A$, $Ax^m$ is a Hopf ideal, and as observed in \S\ref{sss:Adelta-standard} we have $A\eo\cong A/Ax^m\cong T_{m^2}(\xi)$ as $\k$-algebras, so $A\eo$ is an indecomposable ideal. For $\oi\ne\ol 0\in\Z_d$ and $u\in\oi$, $x^m e_u=(1-\theta^{mu})e_u$, the scalar being independent of the representative. Define $x_{\ol 0}:=xe_0$, satisfying $x_{\ol 0}^m=0$, and $x_\oi:=\g_\oi^{-1}xe_\oi$ for $\oi\ne\ol 0$.

Similarly, we define $g_{\oi}:= \theta^{-\oi} g e_\oi$ for any $\oi \in \Z_d$, where $\theta^{-\oi} = \theta^{-i_0}$ where $i_0$ is the least nonnegative residue of $i$ modulo $d$. For $\oi \ne \ol{0}$, it is immediate to verify that $x_{\oi}^m= g_\oi^m =e_{\oi}$ and $g_\oi x_\oi = \xi x_\oi g_\oi$. 

For any $1\le i\le d-1$, we set
\[
E_{p,q}^{\oi}~:=~
x_{\oi}^{p-q}\,e_{i+qd} \in A e_{\oi} \text{\quad for\quad}  p,q \in \Z_m.
\]
One can verify directly that 
$$
E_{p,q}^{\ou} E_{r,s}^{\ov}~=~ \delta_{\ou, \ov } \delta_{q,r}\, E_{p,s}^{\ou} 
$$
for any $1\leq u,v\leq d-1$ and $p,q,r,s \in \Z_m$. Thus, one can identify $A e_\oi$ with the matrix algebra $M_m(\k)$ indexed by $\Z_m$ via the isomorphism of algebras defined by the assignments $E_{p,q}^{\oi} \mapsto \se_{p,q}$ for $p, q \in \Z_m$.  In particular, $Ae_\oi$ is a simple ideal of $A$ for $\oi \ne \ol0$. Hence, $\{e_\oi\,|\,\oi  \in \Z_d\}$ forms
a complete set of indecomposable central idempotents of $A$ and $A = \bigoplus_{\oi \in \Z_d}A e_\oi$ is the indecomposable ideal decomposition of $A$.

\begin{remark}\label{r:heisenberg}
 Under the preceding identification of $A e_\oi$ and $M_m(\k)$  for $\oi\neq \ol0$, one can see that 
\begin{equation}\label{eq:heisenberg}
    x_\oi \mapsto \begin{bmatrix}
    0& 0 &\cdots& 0& 1\\
    1 & 0& \cdots&0& 0\\
    0 & 1& \cdots & 0&0\\
    \vdots&\vdots&\cdots&\vdots&\vdots\\
    0&0& \cdots& 1&0
\end{bmatrix} \quad \text{and}\quad g_\oi  \mapsto \zeta_{\oi}\begin{bmatrix}
    1& 0 &\cdots&  0\\
    0 & \xi& \cdots& 0\\
    \vdots&\vdots&\ddots&\vdots\\
    0&0& \cdots& \xi^{m-1}
\end{bmatrix} 
\end{equation}
for any $m$-th root of unity $\zeta_{\oi}$. Therefore, under this identification, $x_\oi, g_\oi$ determine an irreducible representation of the Heisenberg group 
$$
\HH_m:=\left\{\begin{bmatrix}
    1& a & c\\
    0&1&b\\
    0&0&1
\end{bmatrix}\, \Bigg|\, a, b, c\in  \Z_m \right\}
$$
which has the presentation 
$$
\HH_m=\langle A, B,C\mid A^m, B^m, C^m, ABA\inv B\inv C\inv,   ACA\inv C\inv, BCB\inv C\inv\rangle.
$$
The group homomorphism $\pi_\oi: \HH_m \to (A e_\oi)^\times$ defined by $\pi_\oi: A \mapsto x_\oi$, $B \mapsto \zeta_{\oi} g_\oi$ and $C \mapsto \xi^{-1} {e_\oi}$ is an irreducible representation of $\HH_m$, and they are isomorphic representations of $\HH_m$. 
\end{remark}

With the matrix-algebra decomposition in hand, we obtain a manageable normal form for representatives of stable twists: the next lemma shows that every stable twist of $A$ is equivalent to a \emph{standard} one, so the image computation can be carried out inside a controlled family of algebra automorphisms.

\begin{lem}\label{bgT2} Let $m,d \ge 2$ be integers and $n=md$. Then every stable twist of $A_1(n,\xi)$ is equivalent to a standard one in the sense of Definition~\ref{d:sf}. In particular, $\Jt{u,v}{0} = 1$ for all $u, v \in d\,\Z_n$.
     \end{lem}
\begin{proof}
Let $(\vp, J)$ be any stable twist of $A$. Since $g^m$ is central group-like, so is $\varphi(g^m)$. Therefore $\varphi(g^m)=g^{ml}$ for some integer $l$ coprime to $d$, and since $\Z_n^\times\twoheadrightarrow\Z_d^\times$ is surjective, we may assume $l$ is coprime to $n$. By the definition~\eqref{eq:gm} of $e_\oi$ we have $\vp(e_\oi) = e_{\ell \oi}$, where $l\ell \equiv 1\pmod n$, and hence $\vp(Ae_\oi) = A e_{\ell \oi}$ for $\oi \in \Z_d$. By Remark \ref{r:heisenberg}, $\pi_\oi: \HH_m \to (A e_\oi)^\times$ is an irreducible representation of the Heisenberg group for each $\oi \ne \ol 0$ by setting the $m$-th root of unity $\zeta_\oi=\theta^{l\ol{\ell i}-\oi}$, and they are all isomorphic as representations of $\HH_m$. Hence, there exists  $z_{\oi} \in (Ae_{\oi})^\times$ for $\oi \ne \ol0$ such that 
$$
z_{\ell\oi}\, \vp(x_\oi) \,z_{\ell\oi}\inv = x_{\ell \oi } \quad\text{ and }\quad z_{\ell\oi}\, \vp(g_\oi) \,z_{\ell\oi}\inv = \zeta_\oi g_{\ell \oi }
$$
for each $\oi \ne \ol 0$.

Now, we need to define $z_{\ol0} \in (Ae_{\ol 0})^\times$ which satisfies the same relations.
Since  $\vp(g^m) =g^{m l}$, $\vp(I) = I$ where $I = A(1-g^m)$ is a Hopf ideal of $A$. Therefore, $(\vp, J)$ can be descended to a stable twist $(\ol{\vp}, \ol{J})$ of $\ol{A}=A/I \cong T_{m^2}(\xi)$. It follows from Theorem \ref{STH0} that there exists  $\ol z \in \ol{A}^\times_1$  which satisfies $\ol z\, \ol{\vp}(\ol g)\, \ol z\inv = \ol g$, $\ol z\, \ol{\vp}(\ol x)\, \ol z\inv = \ld \ol x$ for some $\ld \in \k^\times$. Here $\ol a \in \ol A$ denotes the image of $a \in A$. Moreover, by Theorem \ref{STH0}, $J(\eo \o \eo)=\eo \o \eo + \a \Omega(\eo\o\eo)$,  where $\Omega = \sum_{i=1}^{m-1}\frac{x^i}{[i]_\xi!} \o \frac{x^{m-i}g^{i-m}}{[m-i]_\xi!}$. So, we define $z_{\ol 0} \in A \eo$ such that $\ol z_{\ol 0} = \ol z \in \ol A$. In particular, $\ve(z_{\ol0})=1$ and $z_{\ol0}$ is invertible in $A \eo$.  

Let $z = \sum_{\oi \in \Z_d} z_{\oi}$. Then we have
\begin{align*}
    \vp_z(g) =&\,\, z \vp(g) z\inv =  g_{\ol0}+\sum_{\oi \in \ZD} \theta^{\oi} z_{\ell \oi} \vp(g_\oi)\, z_{\ell \oi}\inv =g_{\ol0}+\sum_{\oi \in \ZD} \theta^{\oi} \zeta_\oi g_{\ell\oi} \\
=&\, g_{\ol0}+\sum_{\oi \in \ZD} \theta^{l \ol{\ell i}} g_{\ell\oi} = \, g_{\ol0}+\sum_{\oi \in \ZD} \theta^{l \oi} g_{\oi} =ge_{\ol 0}+\sum_{\oi\in \ZD} g^le_{\oi}.  
\end{align*}
Similarly, 
\begin{align*}
 \vp_z(x)&=\, z_{\ol0} \vp(x_{\ol0}) z_{\ol0}\inv+ \sum_{\oi \in \Z_d\setminus \{\ol0\}} \g_\oi\,z_{\ell\oi}\vp(x_\oi) z_{\ell\oi}\inv = \ld x_{\ol0}+\sum_{\oi\in \Z_d\setminus \{\ol0\}}\g_{\oi}\,x_{\ell\oi } \\
&=\, \ld x_{\ol0}+ \sum_{\oi\in \ZD} \g_{l\oi}\,x_{\oi}=x\,\left(\lambda e_{\ol0}+\sum_{\oi\in \ZD} \frac{\gamma_{l \oi}}{\gamma_{\oi}}e_{\oi}\right)\, ,
\end{align*}
where we denote $a=\sum_{\oi \in \Z_d} a_\oi e_\oi\in \k\langle g^m\rangle$ with $a_{\ol 0}=\lambda$ and $a_{\oi}=\frac{\gamma_{l \oi}}{\gamma_{\oi}}$ for $\oi\in \ZD$.

Moreover, since $\varphi_z$ is an algebra automorphism of $A$, we must have $\varphi_z(g)\varphi_z(x)=\xi \varphi_z(x)\varphi_z(g)$, which implies that $(\xi^l-\xi)a_{\oi}xg^le_{\ol i}=0$ for any $\oi\in \ZD$. Hence, $\xi^l=\xi$ and so $l\equiv 1\pmod m$. Consequently, $\varphi_z(g)=g^l+(g-g^l)e_{\ol 0}=g^l$. 
\end{proof}

Recall from Definition~\ref{d:sf} that such a standard twist is moreover \emph{normalized} when $\JC{}{0}=1\o1$; the next subsubsections will reduce to this case. The following proposition highlights the rigidity of standard stable twists.
\begin{prop} \label{p:uniqueness}
    Suppose  $(\vp,J)$ and $(\vp',J')$ are standard stable twists of $A_1(n,\xi)$ satisfying $\vp(g)=\vp'(g)$. There exists $z \in \kgt{g}$ such that $\vp_z(g)=\vp'(g)$ and $\vp_z(x) - \vp'(x) \in A \eo$. In addition, if $\vp'(x)\eo  =\xi^k\vp(x)\eo$ for some integer $k$, then there exists $z \in \kgt{g}$ such that $\vp_z(x) = \vp'(x)$. Moreover, $(\vp_z, J_z)$ is a standard stable twist.
\end{prop}
\begin{proof}
 Let $\vp(g) = g^l$, $\vp(x) = \sum_{\ou \in \Z_d} \au x\eu$ and  $\vp'(x) = \sum_{\ou \in \Z_d} a'_\ou x\eu$ for some $\au, a'_{\ou} \in \k^\times$ such that $(a'_{\ou})^m=(\au)^m = \frac{\g_{\ou l}^m}{\g_{\ou}^m}$ for $\ou \ne \ol0$. Then, for each $\ou \in \ZD$, there exists an $m$-th root of unity $\zeta_\ou = \xi^{k_\ou}$ such that $a'_\ou = \zeta_\ou \au$. Let $z = \eo+\sum_{\ou \in \ZD} g^{k_{\ou}} \eu$. Then $z \in \kgt{g}$ and
 $$
 z \vp(x)\eu z\inv = \au g^{k_{\ou}}xg^{-k_{\ou}} \eu =\zeta_\ou \au x \eu = a'_{\ou} x \eu.
 $$
 Thus,  $\vp_z(x) - \vp'(x) \in A \eo$, and obviously $\vp_z(g)=\vp(g)=\vp'(g)$. If, in addition, $\vp'(x)\eo = \xi^{k_{\ol0}} \vp(x) \eo$, then we define $z = \sum_{\ou \in \Z_d} g^{k_\ou}\eu$. Now, we have $\vp_z(x) = \vp'(x)$. It is immediate to see that $(J_z)^{\pmb{0}} = (z \o z) \JC{}{0} \Delta(z)^{-1} =  \JC{}{0} \partial_z$. Since $\partial_z (\eo \o \eo) =\eo \o \eo$, $(J_z)^{\pmb{0}}(\eo \o \eo) = \eo \o \eo$, and so $(\vp_z, J_z)$ is a standard stable twist.
\end{proof}

In contrast to the Taft algebras, the constant term $\JC{}{0}$ of a stable twist $(\vp, J)$ in standard form need not yield a $2$-cocycle of $\Z_n$.
\begin{example}
    Consider the element $z = \eo + \frac{1}{\sqrt{2}} x e_{\ol1} \in A_1(4,-1)^\times$.   Then, $z\inv = z$, $I_z(g) = g^3$, $I_z(x) = x$, and 
    $$
    \partial_z = \eo \o \eo + \eo\o e_{\ol1} +  e_{\ol1} \o g\eo+ \frac{1}{2}(x \eo \o x ge_{\ol1}  +  x e_{\ol1} \o x\eo+  x e_{\ol1} \o x e_{\ol1}).
    $$
    So, $\cb{z}$ is a standard stable twist but $\partial_z^{\hspace{1pt}\bo}$ is not invertible as the coefficient of $e_{\ol1}\o e_{\ol 1}$ is 0.
\end{example}

\subsubsection{Coefficient identities for standard stable twists}\label{SS:A1-coefficients}
We now derive the coefficient identities needed to classify which algebra automorphisms can actually occur in the image of $\Pi$. These calculations are not the conceptual conclusion of the section; rather, they provide the technical input for the eventual description of $\OutTw(A)$ in the exact-sequence argument. Suppose $(\vp, J) \in \st(A)$ is standard. Then $\vp(g) = g^l$ for some integer  $l$ coprime to $n=md$ with $l \equiv 1 \pmod m$. It follows from Lemma \ref{l:vg} that  $J$ has the form \eqref{eq:simplerJ} such that 
\begin{equation}\label{eq:standard}
    J=\sum_{u,v\in \mathbb Z_n}\,\left(\JC{u,v}{0}\, e_u \o e_v \,+\,\sum_{1\le i\le m-1}
    \JC{u,v}{i}\, x^i e_u \o x^{m-i} e_v\right)
\end{equation}
for coefficients $\JC{u,v}{0},\JC{u,v}{i}\in \kk$. Moreover, for any $u, v \in \ol0$,  $\JC{u,v}{0} =1$. We first obtain all the relations on $\JC{u,v}{0}$ and $\JC{u,v}{i}$  from the 2-cocycle condition of $J$. 
It follows from the $q$-binomial theorem that
$$
\Delta(x^i) = \sum_{j=0}^i \qqbin{i}{j}{\xi} x^{j} \o x^{i-j} g^{j}
$$
for any integer $i \ge 0$.
For any 2-cocycle $J$ in the form of \eqref{eq:standard}, we have
\begin{equation}\label{eq:1oJ}
    \begin{aligned}
  &\, (1 \o J)(\id \o \Delta)(J) =  \sum_{u,v, w\in \mathbb Z_n}\Big(\JC{u, v+w}{0}\JC{v,w}{0} e_u \o e_v \o e_w\\
  &\, + \sum_{\stackrel{j+k=m}{1\leq j,k\leq m-1}}\JC{u, v+w}{0}\JC{v,w}{j} e_u \o x^je_v \o x^ke_w \\
&\, +\sum_{\substack{i+j+k=m\\0\leq j,k\le m-1\\1\leq i\leq m-1}}\left[\begin{matrix}j+k\\j\end{matrix}\right]_\xi \JC{u,v+w}{i}\JC{v+jd,w+kd}{0} \theta^{jw} x^ie_u \o x^je_v \o  x^k e_w\\ 
&\, +\sum_{\substack{i+j+k=m\\j'+k'=m\\0\leq j,k\le m-1\\1\leq i, j',k'\leq m-1}}\left[\begin{matrix}j+k\\j\end{matrix}\right]_\xi \JC{u,v+w}{i}\JC{v+jd,w+kd}{j'} \theta^{jw} x^ie_u \o x^{j+j'}e_v \o  x^{k+k'} e_w\Big)
\end{aligned}
\end{equation}
and 
\begin{equation}\label{eq:Jo1}
    \begin{aligned}
 &\, (J \o 1)(\Delta \o \id)(J) =  \sum_{u,v, w \in \mathbb Z_n}\Big(\JC{u+v, w}{0}\JC{u,v}{0} e_u \o e_v \o e_w \\
 &\, +\sum_{\stackrel{i+j=m}{1\le i,j\le m-1}} \JC{u+v,w}{0}\JC{u,v}{i} x^ie_u \o x^je_v \o e_w \\
&\, +\sum_{\substack{i+j+k=m\\0\leq i,j\le m-1\\1\leq k\leq m-1}}\left[\begin{matrix}i+j\\i\end{matrix}\right]_\xi \JC{u+v,w}{i+j}\JC{u+id,v+jd}{0} \theta^{iv} x^ie_u \o x^je_v \o  x^k e_w\\
&\, +\sum_{\substack{i+j+k=m\\i'+j'=m\\0\leq i,j\le m-1\\1\leq k,i',j'\leq m-1}}\left[\begin{matrix}i+j\\i\end{matrix}\right]_\xi \JC{u+v,w}{i+j}\JC{u+id,v+jd}{i'} \theta^{iv} x^{i+i'}e_u \o x^{j+j'}e_v \o  x^{k} e_w\Big)\,.  
\end{aligned}
\end{equation}

\begin{lem} \label{l:2-cyc_2}
Let $u,v,w\in \mathbb Z_n$. We have the following equations 
\begin{equation}\label{e2}
\JC{v,w}{i}\JC{u,v+w}{0}-\JC{u,v+id}{0}\JC{u+v,w}{i}=\sum_{j=1}^{i}\left[\begin{matrix}i\\j\end{matrix}\right]_\xi\,\JC{u+jd,v+(i-j)d}{m-j}\JC{u+v,w}{i}\,\theta^{vj}\gamma_{\overline{u}}^m\, ,
\end{equation}
\begin{equation}\label{e3}
\JC{u+v,w}{0}\JC{u,v}{m-i}-\JC{u,v+w}{m-i}\JC{v+id,w}{0}\theta^{iw}=\sum_{j=0}^{i-1}\left[\begin{matrix}i\\j\end{matrix}\right]_\xi \JC{v+jd,w+(i-j)d}{i-j}\JC{u,v+w}{m-i}\theta^{wj}\gamma_{\overline{w}}^m\,
\end{equation}
for all $1\le i\le m-1$ and 
\begin{equation}\label{e4}
\begin{split}
\JC{u+id,v}{0}\JC{u+v,w}{i}\theta^{vi}+\sum_{j=0}^{i-1}\left[\begin{matrix}i\\j\end{matrix}\right]_\xi \JC{u+jd,v+(i-j)d}{i-j}\JC{u+v,w}{i}\theta^{vj}\gamma_{\overline{v}}^m\\
=\JC{v,w+kd}{0}\JC{u,v+w}{i}+\sum_{j=1}^{k}\left[\begin{matrix}k\\j\end{matrix}\right]_\xi \JC{v+jd,w+(k-j)d}{m-j}\JC{u,v+w}i\theta^{wj}\gamma_{\overline{v}}^m
\end{split}
\end{equation}
for all $1\le i,k\le m-1$ satisfying $i+k=m$.
\end{lem}
\begin{proof}
By comparing the coefficients of $e_u \o x^j e_v \o x^{k} e_w$, $x^ie_u \o x^j e_v \o e_w$, $x^i e_u \o e_v \o x^k e_w$ of \eqref{eq:1oJ} and \eqref{eq:Jo1}, we obtain \eqref{e2}, \eqref{e3} and \eqref{e4} respectively.
\end{proof}

Next, denote $\vp(x) = xa$, where $a = \sum_{\ou \in \Z_d} a_\ou  e_\ou \in \kg{g^m}^{\times}$. Since $\vp: A \to A_J$ is a Hopf algebra isomorphism, we have $J\Delta(\varphi(x))=(\varphi\otimes \varphi)\Delta(x)J$ where 
\begin{equation}\label{eq:rvpx}
 \begin{aligned}
    J\Delta(\varphi(x))=&\,J(x\otimes  g+1\otimes x)\Delta(a)\\
    =&\,a_{\overline{u+v}}(\JC{u+d,v}{0}\theta^v+\JC{u,v+d}{1}\gamma_{\overline{v}}^m)xe_u\otimes e_v\\
    &\,+a_{\overline{u+v}}(\JC{u,v+d}{0}+\JC{u+d,v}{m-1}\gamma_{\overline{u}}^m\theta^v)e_u\otimes xe_v\\
       &\,+\sum_{2\le i\le m-1} a_{\overline{u+v}}(\JC{u+d,v}{i-1}\theta^v+\JC{u,v+d}{i})x^ie_u\otimes x^{m+1-i}e_v\, ,
\end{aligned}   
\end{equation}
and
\begin{equation}\label{eq:lvpx}
 \begin{aligned}
    (\varphi\otimes \varphi)\Delta(x)J&=(xa\otimes  g^l+1\otimes xa)J\\
   &=(a_{\ol u}\JC{u,v}{0}\theta^{lv}+a_{\ol v}\JC{u,v}{1}\gamma_{\ol v}^m)xe_u\otimes e_v\\
   &\,+(a_{\ol v}\JC{u,v}{0}+a_{\ol u}\JC{u,v}{m-1}\gamma_{\ol u}^m\theta^{l(v+d)})e_u\otimes xe_v\\
       &\, +\sum_{2\le i\le m-1} (\theta^{l(v-(i-1)d)}\JC{u,v}{i-1}a_{\overline{u}}+\JC{u,v}{i}a_{\overline{v}})x^ie_u\otimes x^{m+1-i}e_v\,.
\end{aligned}   
\end{equation}

\begin{lem} \label{l:x}
    Let $u, v \in \Z_n$. We have the following equations
 \begin{eqnarray}
  a_{\overline{u+v}}(\JC{u,v+d}{0}+\JC{u+d,v}{m-1}\gamma_{\overline{u}}^m\theta^v)&=&\JC{u,v}{0}a_{\overline{v}}+\JC{u,v}{m-1}\gamma_{\overline{u}}^m\theta^{l(v+d)}a_{\overline{u}}\label{eq:x2}\, , \\
    a_{\overline{u+v}}(\JC{u+d,v}{0}\theta^v+\JC{u,v+d}{1}\gamma_{\overline{v}}^m)&=&\theta^{lv}\JC{u,v}{0}a_{\overline{u}}+\JC{u,v}{1}\gamma_{\overline{v}}^ma_{\overline{v}}\label{eq:x3}\,
\end{eqnarray} 
and for all $2\le i\le m-1$\
\begin{equation}
a_{\overline{u+v}}(\JC{u+d,v}{i-1}\theta^v+\JC{u,v+d}{i})=\theta^{l(v-(i-1)d)}\JC{u,v}{i-1}a_{\overline{u}}+\JC{u,v}{i}a_{\overline{v}}\, . \label{eq:x4}
\end{equation}
 In particular, for any integer $s$ and $u, v \in \Z_n$, we have
\begin{equation}\label{eq:recursion}
    \JC{ds, v+d}{0} = \JC{ds, v}{0} \quad \text{and}\quad   \JC{u+d, ds}{0} = \JC{u, ds}{0}\, .
\end{equation}
\end{lem}
\begin{proof}
The equations \eqref{eq:x2}, \eqref{eq:x3} and \eqref{eq:x4} follow immediately by comparing the coefficients of  $e_u \o  xe_v$, $xe_u \o  e_v$ and $x^ie_u\otimes x^{m+1-i}e_v$. At last, the two equations of \eqref{eq:recursion} follow from setting $u=ds$ in \eqref{eq:x2} and $v= ds$ in \eqref{eq:x3}. 
\end{proof}

Finally, we want to derive more relations on $\JC{}{0}$ in addition to \eqref{eq:2coJ}. Let $\ol{J}$ be the inverse of $J$. Since $g \o g$ commutes with $\ol{J}$, we find $\ol{J}$ also has the form:
$$
\ol{J} = \sum_{u,v \in \Z_n}\, \left( \JI{u,v}{0} e_u \o e_v +\sum_{i=1}^{m-1} \JI{u,v}{i} x^i e_u \o x^{m-i} e_v\right)\, . 
$$

Since $ \ol{J}J = 1 \o 1$, we have $(\ol{J}J )^{\pmb{0}} = 1 \o 1$. Thus, 
\begin{align*}
  1 \o 1 & = ( \ol J J)^{\bo} \\
  &= \sum_{u,v,j,k\in \mathbb Z_n} \Bigl(\JI{j,k}{0} \JC{u,v}{0}  e_j e_u  \o e_ke_v  \nonumber\\
  &\qquad\qquad\qquad\qquad +\sum_{i=1}^{m-1}\JI{j,k}{i} \JC{u,v}{m-i}  x^ie_jx^{m-i}e_u \o x^{m-i}e_k x^i e_v\Bigr)  \nonumber\\
  &= \sum_{u,v\in \mathbb Z_n} \left( \JI{u,v}{0}\JC{u,v}{0} e_u  \o e_v  +\sum_{i=1}^{m-1} \JI{u+(m-i)d,v+id}{i} \JC{u,v}{m-i} x^m e_{u}  \o x^m e_{v}\right) \nonumber\\
   &= \sum_{u,v\in \mathbb Z_n}\left( \JI{u,v}{0}\JC{u,v}{0}  + \sum_{i=1}^{m-1}\g_\ou^m \g_\ov^m\JI{u+(m-i)d,v+id}{i} \JC{u,v}{m-i}\right)  e_u  \o  e_v\, , 
\end{align*}
and hence for all $u,v\in \mathbb Z_d$, we have  
\begin{align}
 \JI{u,v}{0}\JC{u,v}{0}  + \sum_{i=1}^{m-1}\g_\ou^m \g_\ov^m\JI{u+id,v+(m-i)d}{m-i} \JC{u,v}{i}=1\, .  \label{eq:JJinv}
\end{align}

\begin{lem} \label{l:Jd}
    For any integer $s$, $u,v \in \Z$, 
    \begin{equation}\label{eq:J0x1}
     \JC{u,sd}{0} = (\JC{u, d}{0})^s\quad \text{ and }\quad\JC{sd, v}{0} = (\JC{d, v}{0})^s\,.  
 \end{equation}
    Moreover, $(\JC{u, v}{0})^m =1$ if $\ou =\ol0$ or $\ov=\ol0$.
     Furthermore, for $i=0, \dots, m-1$, we have
    \begin{equation}\label{eq:J0x2}
\JC{u+d, v}{i}  = \JC{u,v}{i} \left(\frac{\JC{d, u+v}{0}}{\JC{d, u}{0}}\right), \quad 
\JC{u, v+d}{i} = \JC{u,v}{i} \left(\frac{\JC{u+v, d}{0}}{\JC{v, d}{0}}\right) \xi^i\,.
 \end{equation}
\end{lem}
\begin{proof}
  We know $\JC{u,v}{0}=1$ whenever $\ou=\ov=\ol 0$. Now, we apply Lemma \ref{l:cocycle} and \eqref{eq:recursion} to obtain
$$
\JC{u, sd}{0} \JC{u,d}{0} =\JC{u, sd}{0} \JC{u+sd,d}{0}=\JC{u, (s+1)d}{0} \JC{sd, d}{0}=\JC{u, (s+1)d}{0}\,.
$$
Thus,  $\JC{u,sd}{0}=(\JC{u,d}{0})^s$ and so $(\JC{u,sd}{0})^m=1$ for any integer $s$. Therefore, $(\JC{u,v}{0})^m=1$  for $\ov = \ol0$. Similarly, $\JC{sd,v}{0}=(\JC{d, v}{0})^s$ and so $(\JC{u,v}{0})^m = 1$ whenever $\ou=\ol0$.
By Lemma \ref{l:cocycle} again, we have
$$
\JC{u,v}{0} \JC{d, u+v}{0} = \JC{u+d, v}{0}  \JC{d, u}{0}, 
$$
which implies
$$
\JC{u+d, v}{0}  = \JC{u,v}{0} \left(\frac{\JC{d, u+v}{0}}{\JC{d, u}{0}}\right)\,. 
$$
This gives the first equality of \eqref{eq:J0x2} for $i=0$; the second follows similarly for $i=0$. For $i =1, \ldots, m-1$, set $u = d$ in \eqref{e2} to obtain
    $$
    \JC{v,w}{i}\JC{d,v+w}{0}= \JC{d,v+id}{0}\JC{d+v,w}{i} = \JC{d,v}{0}\JC{d+v,w}{i},  
    $$
    which implies the first equation of \eqref{eq:J0x2}.
 By letting $w=d$ in \eqref{e3}, we also have
$$
\JC{u+v,d}{0}\JC{u,v}{m-i}=\JC{u,v+d}{m-i}\JC{v+id,d}{0}\theta^{id} = \JC{u,v+d}{m-i}\JC{v,d}{0}\xi^{i}\,,
$$
and so $$
\JC{u+v,d}{0}\JC{u,v}{i}= \JC{u,v+d}{i}\JC{v,d}{0}\xi^{-i}\,. 
$$
Thus, the second equation of \eqref{eq:J0x2} follows.
\end{proof}

\subsubsection{Normalization of stable twists}\label{SS:A1-normalization}
The exact-sequence framework is most effective on \emph{normalized} representatives. Our goal here is therefore to show, in Proposition~\ref{p:invertible}, that every standard stable twist of $A$ is equivalent to one whose constant term is invertible, and hence, after a further normalization, to a representative with $\Jt{}{0}=1\otimes 1$.

We begin with the definition  
 $$
 T_J~:=~\{u \in \Z_n\mid \JC{u,-u}{0} \ne 0\}
$$
for any 2-cocycle $J$ of $A$ given by \eqref{eq:standard}.

\begin{lem} \label{l:TJ} The following hold:
\begin{enumerate}[label=\rm{(\roman*)}]
    \item if $\JC{u,v}{0} = 0$ for some $u,v \in \Z_n$, then $u,v \notin T_J$;
    \item $T_J$ is a subgroup of $\Z_n$;
    \item if $\JC{u,v}{0}=0$, then $\JC{u+ds,v+dt}{0}=0$ for all integers $s,t$;
    \item $\JC{}{0}$ is invertible if and only if $T_J=\Z_n$, equivalently $1 \in T_J$.
\end{enumerate}
\end{lem}

\begin{proof}
    By the 2-cocycle condition in Lemma \ref{l:cocycle},  $\JC{u,-u}{0}\JC{0, u}{0}=\JC{-u,u}{0} \JC{u,0}{0}$. Thus,  if $u \in T_J$, then so does $-u$ since $\JC{0, u}{0}=\JC{u,0}{0}=1$. Note that if $u \in T_J$, then 
    $\JC{u, v}{0} \JC{-u, u+v}{0} = \JC{0, v}{0} \JC{-u, u}{0} \ne 0$, and so $\JC{u, v}{0} \ne 0$ for all $v$. Similarly $\JC{v, u}{0} \ne 0$ for $v \in \Z_n$. Thus, if $\JC{u,v}{0} = 0$, then $u,v \notin T_J$, proving statement~(i).
    
If $u, v \in T_J$, then
\[
\JC{u+v,-u-v}{0}\JC{-v,-u}{0}=\JC{u+v,-v}{0} \JC{u,-u}{0} \ne 0,
\]
and hence $\JC{u+v,-u-v}{0} \ne 0$, proving (ii). Statement (iii) follows from \eqref{eq:J0x2}, and (iv) is immediate from (i) and (ii).
\end{proof}

Note that $\JC{ds, dt}{0}=1$ for all integers $s,t$, so $d\in T_J$ always.

\begin{lem}\label{l:sum}
    If $1\not\in T_J$, then $l\equiv 1 \pmod d$. Moreover, $(\varphi,J)$ is equivalent to some standard stable twist $(\widetilde{\varphi},\widetilde{J})$ with $\widetilde{\varphi}=\id$.  
\end{lem}
\begin{proof}
 Suppose $\JC{u,v}{0}=0$ for some $\ou,\ov\in \ZD$. By Lemma \ref{l:TJ} (iii), we have $\JC{u+ds, v+dt}{0} = 0$ for any integers $s,t$. By \eqref{eq:J0x2} and the fact that $l\equiv 1\pmod m$, we derive the following equations from \eqref{eq:x2}, \eqref{eq:x3} and \eqref{eq:x4}:
\begin{equation}\label{eq:J1}
     \JC{u,v}{1}(\frac{\av}{a_{\ol{u+v}}}-\xi \,\frac{\JC{u+v,d}{0}}{\JC{v,d}{0}})=0\, ,   
\end{equation}
\begin{equation}\label{eq:J2}
      \JC{u,v}{m-1}(\theta^v \,\frac{\JC{d,u+v}{0}}{\,\JC{d,u}{0}}-
    \xi\theta^{lv}\frac{\au}{a_{\ol{u+v}}})=0\,  
\end{equation}
and for $2\le i\le m-1$
 \begin{equation}\label{eq:J3}
\left(\theta^v\,\frac{\JC{d,u+v}{0}}{\JC{d,u}{0}}-\xi^{-(i-1)}\theta^{lv}\frac{\au}{a_{\ol{u+v}}}\right)\JC{u,v}{i-1}=\left(\frac{\av}{a_{\ol{u+v}}}-\xi^i\,\frac{\JC{u+v,d}{0}}{\JC{v,d}{0}}\right)\JC{u,v}{i}\, .
\end{equation}
By denoting 
\begin{equation*}
 \alpha_i=\frac{\av}{a_{\ol{u+v}}}-\xi^i\,\frac{\JC{u+v,d}{0}}{\JC{v,d}{0}}\quad  \text{and}\quad   \beta_i=\theta^v \,\frac{\JC{d,u+v}{0}}{\,\JC{d,u}{0}}-
    \xi^{-i}\theta^{lv}\frac{\au}{a_{\ol{u+v}}}
\end{equation*}
for all $1\le i\le m-1$, the above equations \eqref{eq:J1}-\eqref{eq:J3} can be rewritten as  
\begin{equation}\label{eq:xJ}
\begin{aligned}
      \alpha_1\, \JC{u,v}{1}=0,\quad \beta_{m-1}\,\JC{u,v}{m-1}=0,\\
        \beta_{i-1}\JC{u,v}{i-1}=\alpha_i\JC{u,v}{i}\, \text{  for all $2\le i\le m-1$}\, . \end{aligned}
        \end{equation}
If $\alpha_i\neq 0$ for all $1\le i\le m-1$, then \eqref{eq:xJ} would imply that $\JC{u,v}{i}=0$ for all $0\le i\le m-1$, a contradiction to \eqref{eq:JJinv}. Thus, there exists $1\le i\le m-1$ such that
\begin{equation*}
 0=\a_i= \frac{\av}{a_{\ol{u+v}}}-\xi^i\,\frac{\JC{u+v,d}{0}}{\JC{v,d}{0}}\, .  
 \end{equation*}
 According to Lemma \ref{l:Jd}, $\left(\frac{\av}{a_{\ol{u+v}}}\right)^m=\left(\xi^i\,\frac{\JC{u+v,d}{0}}{\JC{v,d}{0}}\right)^m=1$.
Similarly, we conclude that $\b_i=0$ for some $1\le i\le m-1$ and $\theta^{mlv}\,\left(\frac{\au}{a_{\ol{u+v}}}\right)^m=\theta^{mv}$ which implies that $a_{\ol u}^m=\frac{\theta^{mv}}{\theta^{mlv}}\av^m=a_{-\ol{v}}^m$. By the 2-cocycle condition, we find $\JC{1,1}{0}\JC{-1,2}{0}=\JC{0,1}{0}\JC{-1,1}{0}$. Since $1 \not\in T_J$, $\JC{-1,1}{0}=0$ and hence $\JC{1,1}{0}=0$ or $\JC{-1,2}{0}=0$. By taking $u=1,v=1$ or $u=-1,v=2$ in  $a_{\ol u}^m=a_{-\ol{v}}^m$, we obtain 
\[
\frac{1-\theta^{ml}}{1-\theta^{m}}=\frac{1-\theta^{-ml}}{1-\theta^{-m}}\quad \text{or}\quad \frac{1-\theta^{-ml}}{1-\theta^{-m}}=\frac{1-\theta^{-2ml}}{1-\theta^{-2m}}\, .
\]
A straightforward computation yields that $\theta^{m(l-1)}=1$ or $l\equiv 1\pmod d$. Consequently, $a_{\ol w}^m=\frac{1-\theta^{lmw}}{1-\theta^{mw}}=1$ for all $\ow\in \Z_d\setminus\{\ol 0\}$. Moreover, $a_{\ol 0}^m=1$ by letting $u=1,v=-1$ in $\left(\frac{\av}{a_{\ol{u+v}}}\right)^m=1$.

Let $1\le \ell\le n-1$ such that $l\ell\equiv 1 \pmod{md}$. It is clear that $\ell\equiv 1 \pmod m$ and $\pmod d$. Set $z=e_{\ol 0}+\sum_{\ou\in \ZD} (x_{\ou})^{\frac{1-\ell}{d}u}$ where $x_\ou=\frac{1}{\g_\ou}{xe_\ou}$ for all $\ou\in \ZD$. Then $z\in \A$ with $z^m=1$, and a direct computation gives 
\begin{align}\label{eq:Iz}
I_z(g)&=zgz^{-1}=g\eo+\sum_{\ou \in \ZD}x_\ou^{\frac{1-\ell}{d}u}g x_\ou^{-\frac{1-\ell}{d}u}=g\eo+\sum_{\ou \in \ZD} \xi^{\frac{\ell-1}{d}u} g\eu\\
&=\sum_{\ou \in \Z_d} \theta^{(\ell-1)u}g \eu=g^{\ell}\notag
\end{align}
and $I_z(x)=x$. By our previous discussion, $(\vp, J)$ is standard such that $\vp(g) = g^l$, and $\vp(x) = \sum_{\ou \in \Z_d} a_\ou x \eu$, where $l\equiv 1 \pmod m$, $l\equiv 1\pmod d$ and $a_\ou^m=1$ for $\ou \in \mathbb Z_d$. From \eqref{eq:Iz}, $\vp_z(g) =g^{l\ell}=g$. Since $I_z(g^m) = g^{m\ell}=g^m$, $\vp_z(x) = \vp(x)$. Moreover, $(J_z)^{\pmb{0}}(e_{\ol 0}\otimes e_{\ol 0})=\JC{}{0}(e_{\ol 0}\otimes e_{\ol 0}) =e_{\ol 0}\otimes e_{\ol 0}$, and hence $(\vp_z, J_z)$ is also standard. By Proposition \ref{p:uniqueness}, $(\varphi_z,J_z)$ is equivalent to $(\widetilde{\vp},\widetilde{J})$ with $\widetilde{\vp}=\id$ as  $a_{\ol 0}^m=1$. 
\end{proof}

In what follows, we simply write $\ZC{u,v}{i}=\JC{u,v}{i}\theta^{-iv}$ for all $u,v\in \mathbb Z$ and $0\le i\le m-1$. According to \eqref{eq:J0x2}, we have
\begin{equation}\label{eq:boldJ}
    \ZC{u+td,v+sd}{i}=\ZC{u,v}{i}\left(\frac{\ZC{d,u+v}{0}}{\ZC{d,u}{0}}\right)^t\left(\frac{\ZC{u+v,d}{0}}{\ZC{v,d}{0}}\right)^s
\end{equation}
for any $s,t\in\mathbb Z$. 

\begin{lem}
   Suppose $(\id, J)$ is a stable twist of $A$. Then for any $u,v,w\in \mathbb Z_n$:
      \begin{equation}\label{eq:JJ1}
       \ZC{u,v}{0}\,\left(\frac{\ZC{d,u+v}{0}}{\ZC{d,u}{0}}-1\right)=\delta_{\ov,\ol0}\,\left(1-\xi\,\frac{\ZC{u+v,d}{0}}{\ZC{v,d}{0}}\right)\,\ZC{u,v}{1}\, ,
   \end{equation}
   \begin{equation}\label{eq:JJ2}
       \ZC{u,v}{0}\,\left(1-\frac{\ZC{u+v,d}{0}}{\ZC{v,d}{0}}\right)=\delta_{\ou,\ol0}\theta^{mv}\,\left(\frac{\ZC{d,u+v}{0}}{\ZC{d,u}{0}}-\xi\right)\,\ZC{u,v}{m-1}\, ,
   \end{equation} 
   and for all $2\le i\le m-1$:
    \begin{equation}\label{eq:JJ3}
\left(\frac{\ZC{d,u+v}{0}}{\ZC{d,u}{0}}-\xi^{-(i-1)}\right)\ZC{u,v}{i-1}=\left(1-\xi^i\,\frac{\ZC{u+v,d}{0}}{\ZC{v,d}{0}}\right)\ZC{u,v}{i}\, .
\end{equation}
Moreover, we have 
  \begin{equation}\label{eq:JJ4}
  \begin{aligned}
        &\ZC{u,v+w}{0}\ZC{v,w}{i}  =\ZC{u+v,w}{i}\left(\ZC{u,v}{0}\left(\frac{\ZC{u+v,d}{0}}{\ZC{v,d}{0}}\right)^i\right.\\
&\quad \left.+\sum_{j=1}^{i}\left[\begin{matrix}i\\j\end{matrix}\right]_\xi \ZC{u,v}{m-j}\left(\frac{\ZC{d,u+v}{0}}{\ZC{d,u}{0}}\right)^j\left(\frac{\ZC{u+v,d}{0}}{\ZC{v,d}{0}}\right)^{i-j} \theta^{mv}\xi^{-j(i-j)} \delta_{\ou,\ol 0} \right),
         \end{aligned}
    \end{equation}
    
     \begin{equation}\label{eq:JJ5}
     \begin{aligned}
  &\ZC{u+v,w}{0} \ZC{u,v}{m-i} = \ZC{u, v+w}{m-i} \theta^{mw}\left(\ZC{v,w}{0}\left(\frac{\ZC{d,v+w}{0}}{\ZC{d,v}{0}}\right)^i\right.\\
  &\quad \left.+\sum_{j=0}^{i-1}\left[\begin{matrix}i\\j\end{matrix}\right]_\xi\ZC{v,w}{i-j}\left(\frac{\ZC{d,v+w}{0}}{\ZC{d,v}{0}}\right)^j\left(\frac{\ZC{v+w,d}{0}}{\ZC{w,d}{0}}\right)^{i-j}\xi^{(i-j)^2}\delta_{\ol w,\ol 0}\right), 
  \end{aligned}
 \end{equation}

  for all $1\le i\le m-1$. In particular, for any $u,v\notin d\Z_n$ and $1\le i\le m-1$:
     \begin{equation} \label{eq:JJ7}
     \ZC{0,v}{i}=\frac{\prod_{j=0}^{i-1}(\ZC{d,v}{0}-\xi^{-j})}{\prod_{j=1}^i(1-\xi^j)},  \quad\quad 
            \ZC{u,0}{i}=\frac{\prod_{j=i+1}^m(1-\xi^j\ZC{u,d}{0})}{\prod_{j=i}^{m-1}(1-\xi^{(m-j)})}\, , 
 \end{equation}
\begin{equation} \label{eq:JJ8}
    (\ZC{u,d}{0})^i+\sum_{j=1}^{i}\left[\begin{matrix}i\\j\end{matrix}\right]_\xi \ZC{u,0}{m-j}(\ZC{u,d}{0})^{i-j} \xi^{-j(i-j)}=\frac{\prod_{j=1}^i(1-\xi^j\ZC{u,d}{0})}{\prod_{j=1}^i(1-\xi^j)}\, , 
\end{equation}

\begin{equation} \label{eq:JJ9}
(\ZC{d,v}{0})^i+\sum_{j=0}^{i-1}\left[\begin{matrix}i\\j\end{matrix}\right]_\xi\ZC{0,v}{i-j}(\ZC{d,v}{0})^j\xi^{(i-j)^2}=\frac{\prod_{j=1}^i(\ZC{d,v}{0}-\xi^j)}{\prod_{j=1}^i(1-\xi^j)}, 
\end{equation}

\begin{equation}\label{eq:JJ10}
    \ZC{0,0}{i}=\ZC{u,0}{i}\left((\ZC{u,d}{0})^i+\sum_{j=1}^{i}\left[\begin{matrix}i\\j\end{matrix}\right]_\xi \ZC{u,0}{m-j}(\ZC{u,d}{0})^{i-j} \xi^{-j(i-j)}\right)\, ,
\end{equation}

\begin{equation}\label{eq:JJ11}
\ZC{0,0}{m-i}=\ZC{0,v}{m-i}\theta^{mv}\,\left((\ZC{d,v}{0})^i+\sum_{j=0}^{i-1}\left[\begin{matrix}i\\j\end{matrix}\right]_\xi\ZC{0,v}{i-j}(\ZC{d,v}{0})^j\xi^{(i-j)^2}\right)\, ,
\end{equation}

\begin{equation}\label{eq:JJ12}
\ZC{0,0}{i} =0\,, 
\end{equation}

\begin{equation}\label{eq:JJ13}
    \ZC{u,v}{0}\,\prod_{j=0}^{i-1}(\ZC{d,v}{0}-\xi^{-j})=\ZC{u,v}{i}\,\prod_{j=1}^i(1-\xi^j\ZC{u,d}{0})\, , 
\end{equation}

\begin{equation}\label{eq:JJ14}
    \ZC{u,v}{0}\,\prod_{j=i+1}^{m}(1-\xi^j\ZC{u,d}{0})=\ZC{u,v}{i}\theta^{mv}\,\prod_{j=1}^{m-i}(\ZC{d,v}{0}-\xi^j)\, .
\end{equation}

  \end{lem}
\begin{proof}
   We observe that $\g_{\ou}^m=\delta_{\ou,\ol 0}$ since $\varphi=\id$. By applying \eqref{eq:boldJ}, one obtains \eqref{eq:JJ1}-\eqref{eq:JJ3} through \eqref{eq:x2}-\eqref{eq:x4}. Moreover, \eqref{eq:JJ4}-\eqref{eq:JJ5} are gained from \eqref{e2}-\eqref{e3}.

 The first equation of \eqref{eq:JJ7} follows by letting $u=0$ and $v\not\in d\mathbb Z_n$ in \eqref{eq:JJ1}, and then applying \eqref{eq:JJ3} inductively. Similarly, by letting $v=0$ and $u\not\in d\mathbb Z_n$ in \eqref{eq:JJ2} and using \eqref{eq:JJ3} inductively, one can show the second equation of \eqref{eq:JJ7}. 
 
For \eqref{eq:JJ8}, we know $\delta_{\ou,\ol 0}\,\ZC{u,0}{m-j}=\frac{\prod_{s=0}^{j-1}(1-\xi^{-s}\ZC{u,d}{0})}{\prod_{s=1}^j(1-\xi^s)}$ by \eqref{eq:JJ7}.  For $1\le j\le m-1$, let $\alpha_{j}(x)=\frac{\prod_{s=0}^{j-1}(1-\xi^{-s}x)}{\prod_{s=1}^j(1-\xi^s)}$ and $\alpha_0(x)=1$. It suffices to prove the following equality between polynomials of degree $i$ in $\kk[x]$, 
\[
\sum_{j=0}^{i}\qqbin{i}{j}{\xi}\, \xi^{-j(i-j)}\, \alpha_{j}(x)x^{i-j} =\frac{\prod_{s=1}^i(1-\xi^sx)}{\prod_{s=1}^i(1-\xi^s)},
\]
by showing that this equality holds for  $x=\xi^{-k}$ for $0\le k\le i$. When $x=1$, both sides equal $1$ since $\alpha_j(1)=0$ unless $j=0$. For $\xi^{-k} \ne 1$, then the right-hand side is zero, and the left-hand side is
\begin{align*}
& \sum_{j=0}^{i} \qqbin{i}{j}{\xi}\, \xi^{-j(i-j)}\, \alpha_{j}(\xi^{-k})\xi^{-k(i-j)}=\xi^{-ik}\left(\sum_{j=0}^{i}\qqbin{i}{j}{\xi^{-1}} \alpha_{j}(\xi^{-k})\xi^{kj} \right) \\ 
=&\quad \xi^{-ik}\left(\sum_{j=0}^{i}(-1)^j\left[\begin{matrix}i\\j\end{matrix}\right]_{\xi^{-1}}\, \left[\begin{matrix}k-1+j\\j\end{matrix}\right]_{\xi^{-1}}\xi^{-\frac{j(j+1-2k)}{2}} \right) =0    
\end{align*}
by \eqref{eq:qidentity}. We can prove \eqref{eq:JJ9} in a similar fashion.

Next, \eqref{eq:JJ10} follows from \eqref{eq:JJ4} by letting $v=w=0$ and \eqref{eq:JJ11} follows from \eqref{eq:JJ5} by letting $u=v=0$.

By \eqref{eq:JJ7} and \eqref{eq:JJ8}, \eqref{eq:JJ10} becomes
\begin{equation*}
    \ZC{0,0}{i}=\frac{\prod_{j=i+1}^m(1-\xi^j\ZC{u,d}{0})}{\prod_{j=i}^{m-1}(1-\xi^{-j})}\frac{\prod_{j=1}^i(1-\xi^j\ZC{u,d}{0})}{\prod_{j=1}^i(1-\xi^j)} =\frac{1-(\ZC{u,d}{0})^m}{\prod_{j=i}^{m-1}(1-\xi^{-j})\prod_{j=1}^i(1-\xi^j)}=0
\end{equation*}
since $\ZC{u,d}{0}$ is an $m$-th root of unity. This proves \eqref{eq:JJ12}.

Finally, \eqref{eq:JJ13} follows from \eqref{eq:JJ4} by letting $v=0$ and using \eqref{eq:JJ8}. And \eqref{eq:JJ14} follows from \eqref{eq:JJ5} by letting $v=0$ and using \eqref{eq:JJ9}.

\end{proof}

Using the relations above for $\varphi=\id$, we now establish criteria for $v\in T_J$. For any $u,v\in\Z_n$, define the index set
\[
\SJ{u,v}:=\{i\in\{0,1,\ldots,m-1\}\mid\ZC{u,v}{i}\ne 0\}\,.
\]
By \eqref{eq:JJinv}, $\SJ{u,v}\ne\emptyset$ for any $u,v\in\Z_n$.

\begin{lem} \label{l:T1}  Suppose $(\id, J)$ is a stable twist of $A$. For any $v \in \Z_n$, we have
\begin{enumerate}
    \item $\ZC{v,d}{0} \ne 1$ iff  $\ZC{-v, v}{0} = 0$ iff $\ZC{\pm v,d}{0} \ne 1$ and  $\ZC{d, \pm v}{0} \ne 1$. 
\item  If $\ZC{v,d}{0} = \xi^a$ for some nonnegative integer $a \le m-1$, then $\ZC{d,-v}{0} = \xi^a$ and $\ZC{-v,d}{0} = \ZC{d,v}{0} = \xi^{-a}$. Moreover,
$$
\SJ{-v,v} = \{a\}\quad\text{and}\quad \SJ{v,-v} = \{\ol{-a}\},
$$
where $\ol{-a}$ denote the least nonnegative residue of $-a$ modulo $m$.
\end{enumerate}
\end{lem}
\begin{proof}
Note that for $v \in d \Z_n$, $\ZC{v,d}{0} = 1$ and $\ZC{-v,v}{0}=1$. So, we may assume $v \not\in d\Z_n$. 
By setting $u=-v$ and $w=0$ in \eqref{eq:JJ4} and by \eqref{eq:JJ12}, we find $\ZC{-v,v}{0} \ZC{v,0}{i} = 0$ for  $1\le i\le m-1$. If $\ZC{v,d}{0} \ne 1$, then $\ZC{v,0}{m-1} \ne 0$ by \eqref{eq:JJ7}. Thus, $\ZC{-v,v}{0}=0$.

Conversely, if $\ZC{-v,v}{0} = 0$,  then  $\ZC{v,-v}{0} = 0$ since $T_J$ is a subgroup of $\Z_n$.  There exists a positive integer $i \le m-1$ such that $\ZC{v,-v}{i}\ne 0$ by \eqref{eq:JJinv}. It follows from \eqref{eq:JJ13} that $\ZC{v,d}{0}=\xi^{-a}$ with $1 \le a \le i$, and from \eqref{eq:JJ14} that $\ZC{d,-v}{0}= \xi^{b}$ with $1 \le b \le m-i$.  In particular, both $\ZC{v,d}{0}$ and $\ZC{d,-v}{0}$ are not equal to $1$.  Since $\ZC{-v,v}{0} = 0$, we also have   $\ZC{d, v}{0}, \ZC{-v,d}{0} \ne 1$. This proves statement (i).

We prove (ii) in the following two steps. Let $a=\min \SJ{-v,v}$ and $b=\max \SJ{-v,v}$. We claim that $\ZC{v,d}{0}=\xi^{a}$ and $\ZC{d, -v}{0} = \xi^{b}$.
If $a=0$, then $\ZC{v,d}{0}=1$ by (i). If $a>0$, then $\ZC{-v,v}{0}=0$ and so $\ZC{v,d}{0} \ne 1$  by (i).  It follows by \eqref{eq:JJ1} and \eqref{eq:JJ3} that $1-\xi^a/\ZC{v,d}{0}=0$ or $\ZC{v,d}{0}=\xi^a$.  Similarly, if $b=0$, then $v \in T_J$ and hence $\ZC{d,-v}{0}=1 =\xi^b$ by (i). If $b > 0$,   it follows by \eqref{eq:JJ2} and \eqref{eq:JJ3} that $1/\ZC{d, -v}{0}-\xi^{-b}=0$ or $\ZC{d, -v}{0} = \xi^{b}$. This proves the claim.

 Let $a'=\min \SJ{v,-v}$ and $b'=\max \SJ{v,-v}$. By the preceding paragraph, we have $\ZC{-v,d}{0}=\xi^{a'}$ and $\ZC{d,v}{0}=\xi^{b'}$. If $a=0$, then $b=a'=b'=0$ by (i), and so the statement (ii) holds in this case. Now, we assume $a >0$ or $v \not\in T_J$.  Now take $u=-v$ and $i=a$ in \eqref{eq:JJ13}, which implies that \[\ZC{-v,v}{a}\prod_{j=1}^a(1-\xi^{j+a'})=0.\]
Since $\ZC{-v,v}{a}\neq 0$, we have $a+a'\ge m$. Similarly, by taking $u=-v$ and $i=b$ in \eqref{eq:JJ14}, we get $b+b'\le m$. Hence, $b+a'\le b+b'\le m\le a+a'\le b+a'$. This implies that $b+b'=a+a'=b+a'=m$. Consequently,  $a'=b'$, $a=b$ and $a+b'=m$.  This completes the proof of (ii). 
\end{proof}

\begin{cor}\label{l:Ichar}
    Let  $(\id, J)$ be a standard stable twist of $A$.  Then $\psi_J(v):= \ZC{v,d}{0}$ defines a character of $\Z_n$ and  $\psi_J^{\gcd(d,m)}=1$.
     In particular, if $\gcd(d,m)=1$, then $\JC{}{0}$ is invertible. 
\end{cor}
\begin{proof}
We first need to show that 
\begin{align}\label{eq:cPsi}
\ZC{v+w,d}{0}=\ZC{v,d}{0}\ZC{w,d}{0}\quad \text{for any $v,w\in \mathbb Z_n$.} 
\end{align}  
By Lemma \ref{l:T1}(i), we know $\ZC{v,d}{0}=1$ if and only if $v\in T_J$.  Hence, \eqref{eq:cPsi} holds if $v, w \in T_J$. Now, we assume  $v \not\in T_J$. By Lemma \ref{l:T1}(ii), we have  $\ZC{v,d}{0}=\ZC{d,-v}{0}=\xi^a$ and  $\ZC{w, d}{0}=\xi^{a'}$ where $a = \min\SJ{-v,v}$ and $a'=\min\SJ{-w,w}$.

If $m-a' \ge m-a$, then we consider $u=-v$ and $i=m-a$ in \eqref{eq:JJ4}. Since $\SJ{-v,v}=\{a\}$ by Lemma \ref{l:T1}(ii), the sum from $j=1$ to $i=m-a$ in \eqref{eq:JJ4} only involves one term as $j=m-a$ and is nonzero. Moreover, by \eqref{eq:JJ7}, $\ZC{0,w}{m-a}\neq 0$. So we find 
\begin{equation}\label{eq:ub1}
\ZC{-v, v+w}{0}\ZC{v,w}{m-a} \ne 0\,.
\end{equation}
By the 2-cocycle condition on $\JC{}{0}$, $\JC{-v,v}{0}=0$ implies $\ZC{v,w}{0} = 0$. Let $i=\SJ{v,w}$. By \eqref{eq:ub1},  we find  $i \le m-a$. By \eqref{eq:JJ13}, $i \ge m-a$ and so $i=m-a$. It follows from \eqref{eq:JJ1} and \eqref{eq:JJ3} that $1-\xi^{m-a} \ZC{v+w, d}{0}/\ZC{w,d}{0} = 0$, and so $\ZC{v+w, d}{0}/\ZC{w,d}{0}=\xi^a=\ZC{v,d}{0}$. Therefore, \eqref{eq:cPsi} holds.

Now we assume $m-a'< m-a$. Then $\ZC{0,-v}{m-a'} \ne 0$ by \eqref{eq:JJ7} and $a' > 0$. So, $w \not\in T_J$ by Lemma \ref{l:T1}(i). Consider the substitutions $u:=w$, $v:=-w$ and $w:=-v$ with $i:=m-a'$ in \eqref{eq:JJ4}. Since $\SJ{w,-w} =\{m-a'\}$ by Lemma \ref{l:T1}(ii), the sum in \eqref{eq:JJ4} involves only one nonzero term and hence we have
\begin{equation}\label{eq:ub2}
\ZC{w, -v-w}{0}\ZC{-w,-v}{m-a'} \ne 0\,.
\end{equation}
By the 2-cocycle condition on $\JC{}{0}$, $\JC{w,-w}{0}=0$ implies $\ZC{-w,-v}{0} = 0$. Let $i=\min \SJ{-w,-v}$. Then $i \le m-a'$ by $\eqref{eq:ub2}$, and \eqref{eq:JJ13} implies $i \ge m-a'$. Therefore, $i=m-a'$.
It follows from \eqref{eq:JJ1} and \eqref{eq:JJ3} that \[1-\xi^{m-a'} \ZC{-w-v, d}{0}/\ZC{-v,d}{0}=1-\frac{\ZC{-w-v, d}{0}}{\ZC{-w,d}{0}\ZC{-v,d}{0}}=1-\frac{\ZC{w,d}{0}\ZC{v,d}{0}}{\ZC{w+v, d}{0}}=0.\] Thus, \eqref{eq:cPsi} holds again.

Now we have $(\ZC{1,d}{0})^d=\ZC{d,d}{0}=1$. By Lemma \ref{l:Jd} we have $(\ZC{1,d}{0})^m=1$, and so $\ZC{\gcd(d,m),d}{0}=(\ZC{1,d}{0})^{\gcd(d,m)}=1$. Thus, $\psi_J^{\gcd(d,m)}=1$ and $\gcd(d,m) \in T_J$. In particular, if $\gcd(d,m)=1$,  then $1\in T_J$ and so $\ZC{}{0}$ is invertible by Lemma \ref{l:TJ}(iv). 
\end{proof}

\begin{prop} \label{p:invertible}
    Every standard stable twist of $A_1(n,\xi)$ is equivalent to a standard stable twist $(\tilde{\vp}, \tilde{J})$ with $\tilde{J}^{\bo}$ being invertible.
\end{prop}
\begin{proof}
    Let $(\vp, J) \in \st(A)$ be in standard form.  It suffices to show that if $\JC{}{0}$ is not invertible, then there exists $z \in \A$ such that $(\vp_z, J_z)$ is standard and $(J_z)^\bo$ is invertible.

      By Lemma \ref{l:sum}, we now assume $\varphi=\id$ and $\JC{}{0}$ is not invertible. By Lemma \ref{l:TJ}(iv), $\ZC{-1,1}{0}=0$, and so $\ZC{1,d}{0}=\xi^a$ where $\{a\}= \SJ{-1,1}$ by Lemma \ref{l:T1} (ii). Moreover, Corollary \ref{l:Ichar} implies that $\gcd(d,m)>1$ and $(\ZC{1,d}{0})^d=\ZC{d,d}{0}=1$. Thus, $1=\xi^{ad}$ and so $\frac{m}{\gcd(d,m)}\mid a$. 
     
     Now set $\ell=1-ad$ which satisfies $\ell\equiv 1 \pmod d$ and $\ell\equiv 1 \pmod m$. Consider the trivial stable twist $\cb{z}$ where $z=e_{\ol 0}+\sum_{\ou\in \ZD} (x_{\ou})^{\frac{1-\ell}{d}u}$ with $x_\ou = \frac{1}{\g_u} x \eu$ for $\ou \in \ZD$. As shown in \eqref{eq:Iz} and the details below, we know $(\tilde{\vp}, \tilde{J})=(\varphi_z,J_z)$ is standard. Moreover, the coefficient of $e_{-1} \o e_{1}$ in $\tilde{J}=J_z=(z \o z) J \Delta(z\inv)$ is obtained from
\begin{align*}
    \JC{-1,1}{a} (x_{\ol{-1}})^{-a} x^a e_{-1} e_{\ol{-1}} \o (x_{\ol {1}})^{a} x^{m-a} e_{1} e_{\ol {1}} =\g_{\ol{-1}}^{a}\g_{\ol 1}^{m-a}\, \JC{-1,1}{a}e_{-1}\otimes e_1\neq 0.
\end{align*}     
By Lemma \ref{l:TJ} (iv),  $\Jt{}{0}$ is invertible.
\end{proof}

\subsubsection{Classification of normalized stable twists}\label{SS:A1-classification}
With the normalization in place (Proposition~\ref{p:invertible}), we now classify all standard normalized stable twists of $A$. Lemma~\ref{l:J0=1} shows that we may assume $\JC{}{0}=1\otimes 1$; Corollary~\ref{c:trivial} characterizes the trivial class. We are ready to prove Theorem~\ref{inverT} and assemble the final exact sequence.

By Proposition~\ref{p:invertible}, we may restrict to standard stable twists $(\vp, J)$ of $A$ with $\JC{}{0}\in \kg{g}^\times$, equivalently $\JC{u,v}{0} \ne 0$ for all $u,v \in \Z_n$.
\begin{lem} \label{l:J0=1}
    Let $(\vp, J)$ be any standard stable twist of $A$ with $\JC{}{0}$ being  invertible. Then there exists $z \in \kgt{g}$ such that $(\vp_z, J_z)$ is a  standard normalized stable twist.
\end{lem}
\begin{proof}
    By Lemma \ref{l:cocycle}, there exists $z = \sum_{u \in\Z_n} f(u) e_u \in\kgt{g}$ for some map $f: \mathbb Z_n\to \kk^\times$ with $f(0)=1$ such that  $J_z^\bo=1\otimes 1$ which is equivalent to $\JC{u,v}{0} = \frac{f(u+v)}{f(u)f(v)}$ for all $u, v \in \Z_n$. Obviously, $\vp_z(g) = z\vp(g) z\inv =\vp(g)$ since $\vp(g)=g^l$ for some integer $l$ coprime to $n$. Moreover, $\JC{}{0}(e_{\ol 0}\otimes e_{\ol 0})=(J_z)^{\bo}\,(e_{\ol 0}\otimes e_{\ol 0})$. By Lemma \ref{l:vg}, $(\vp_z, J_z)$ is a normalized standard stable twist.
    \end{proof}

By Lemma~\ref{l:J0=1}, we may assume henceforth that $(\vp, J)$ is a standard \emph{normalized} stable twist.
\begin{lem}\label{l:Jshift}
Let $(\vp,J)$ be a standard normalized stable twist of $A$. Then
\begin{equation} \label{eq:shift}
    \JC{u+ds,v+dt}{i} = \theta^{idt} \JC{u,v}{i}
\end{equation}
for all $s,t \in \Z_n$ and $1\le i\le m-1$. This implies $\JC{}{i} \in \kg{g^m}\o \kg{g^m}g^i$.
\end{lem}
\begin{proof}
 Since $(\vp, J)$ is normalized, $\JC{u,v}{0} = 1$ for all $u, v \in \Z_n$. Recall that $\g_{\ol0} = 0$. By \eqref{e2}, we find
 $$
 \JC{v,w}{i} = \JC{ds+v,w}{i}
$$
for any $s \in \Z$, and $v, w \in \Z_n$. Similarly, by \eqref{e3}, we have
\begin{equation}
      \JC{u,v}{m-i}~= ~\JC{u, v+dt}{m-i}\, \theta^{-(m-i)dt} \,.
 \end{equation}
 Equation \eqref{eq:shift} follows by combining the above equalities.
  
  Let $Z^i = \JC{}{i}(1\o g^{-i})= \sum_{u,v \in \Z_n} \JC{u,v}{i} \theta^{-iv} e_u \o e_v$. It suffices to prove that $Z^i \in \kg{g^m} \o \kg{g^m}$ which is equivalent to 
  $$
  \JC{u,v}{i} \theta^{-iv}  ~= ~\JC{u+ds,v+dt}{i}\, \theta^{-i(v+dt)}
  $$
  for any $u, v, s,t \in \Z_n$. However, this follows directly from \eqref{eq:shift}.
\end{proof}

In view of Lemma \ref{l:Jshift}, the quantity $\ZC{u,v}{i}=\JC{u,v}{i}\theta^{-iv}$ introduced above satisfies $\ZC{u+ds, v+dt}{i} = \ZC{u,v}{i}$ for any $u,v,s,t \in \Z_n$ and $1\le i\le m-1$. Thus $\ZC{u,v}{i}$ depends only on the cosets $\ou,\ov\in\Z_n/\Z_n d\cong \Z_d$, and
\begin{equation}\label{eq:NewJ}
\begin{aligned}
     J -1\o 1&=   \sum_{1\le i\le m-1}\sum_{u,v \in \Z_n } \JC{u,v}{i}\,x^i e_{u}\o x^{m-i}  e_v\\
     &=  \sum_{1\le i\le m-1}\sum_{u,v \in \Z_n } \ZC{u,v}{i}\,x^i e_{u}\o x^{m-i} g^i e_v\,.   
\end{aligned}
\end{equation}

The next lemma expresses $\ZC{u,v}{i}$ in terms of the coefficients of $\varphi(x) = \sum_{\ou \in \Z_d} a_{\ou} x e_{\ou}$.

\begin{lem} \label{l:Z} Let $(\vp,J)$ be a standard normalized stable twist of $A$. Suppose $\vp(x) = xa$ where $a = \sum_{\ou \in \Z_d} a_\ou e_\ou \in \kg{g^m
}^\times$. Then for any $u, v, w \in \Z_n$,
 \begin{equation}\label{eq:z2}
    \ZC{u,v}{1} \g_\ov^m(a_\ov -\xi a_{\ol{u+v}}) =  a_{\ol{u+v}} - a_\ou \theta^{(l-1)v} ,  
 \end{equation}
  \begin{equation}\label{eq:z3}
 \ZC{u,v}{m-1} \g_\ou^m \theta^{mv} (a_{\ol{u+v}} -\xi^l a_\ou   \theta^{(l-1)v }) =   a_\ov - a_{\ol{u+v}},
  \end{equation}
    \begin{equation}\label{eq:z4}
(a_{\ol{u+v}}-\xi^{-l(i-1)}a_{\ol u}\theta^{(l-1)v})\ZC{u,v}{i-1} =  (a_\ov -\xi^i a_{\ol{u+v}})\ZC{u,v}{i} \text{\ for all $2\le i\le m-1$},
  \end{equation}
  \begin{equation}
        \ZC{v,w}{i}  =\ZC{u+v,w}{i}(1+\sum_{j=1}^{i}\left[\begin{matrix}i\\j\end{matrix}\right]_\xi \ZC{u,v}{m-j} \theta^{mv}\xi^{-j(i-j)} \g_{\ou}^m ),\label{e5}
    \end{equation}
 \begin{equation}
   \ZC{u,v}{m-i} = \ZC{u, v+w}{m-i} \theta^{mw}(1+\sum_{j=0}^{i-1}\left[\begin{matrix}i\\j\end{matrix}\right]_\xi\ZC{v,w}{i-j} \xi^{(i-j)^2}\g_{\ow}^m),  \label{e6}
 \end{equation}
  \begin{equation} 
  \begin{split}
       \ZC{u, v+w}{i}(1+\sum_{j=1}^{m-i}\left[\begin{matrix}m-i\\j\end{matrix}\right]_\xi\ZC{v,w}{m-j} \theta^{mw}\xi^{j(i+j)} \g_{\ov}^m)\\
       =\ZC{u+v,w}{i}( 1 +\sum_{j=0}^{i-1}\left[\begin{matrix}i\\j\end{matrix}\right]_\xi\ZC{u,v}{i-j}\xi^{(i-j)^2} \g_\ov^m)\,,\label{e7}
  \end{split}
       \end{equation}
     for all $1\le i\le m-1$. Moreover, for any $\ou, \ov \ne \ol0$ and $1\le i\le m-1$,  we have 
  \begin{equation} \label{eq:Z01}
           \ZC{0,v}{i}=\frac{\prod_{j=0}^{i-1}(1-\xi^{-j}\frac{a_{\ol 0}}{a_{\ol v}}\theta^{(l-1)v})}{\g_{\ol v}^m\prod_{j=1}^i(1-\xi^j)},  \quad\quad 
             \ZC{u,0}{i}=\frac{\prod_{j=i+1}^m(\frac{a_{\ol 0}}{a_{\ol u}}-\xi^j)}{\g_{\ol u}^m\prod_{j=i}^{m-1}(1-\xi^{(m-j)})},
 \end{equation}

\begin{equation}\label{eq:qJrelation}
\begin{split}
      1+\sum_{j=0}^{i-1}\left[\begin{matrix}i\\j\end{matrix}\right]_\xi\ZC{0,v}{i-j}\xi^{(i-j)^2}\g_\ov^m&=\frac{\prod_{j=1}^i(1-\xi^j\frac{a_{\ol{0}}}{a_\ov}\theta^{(l-1)v})}{\prod_{j=1}^i(1-\xi^j)}\, ,  \\
1+\sum_{j=1}^i\left[\begin{matrix}i\\j\end{matrix}\right]_\xi\ZC{u,0}{m-j}\xi^{-j(i-j)}\g_{\ou}^m&=  \frac{\prod_{j=1}^i(\frac{a_{\ol{0}}}{a_\ou}-\xi^j)}{\prod_{j=1}^i(1-\xi^j)}\, ,   
\end{split}
\end{equation}

  \begin{equation} \label{eq:quadratic}
   \begin{split}
      \ZC{0,0}{i} &= \ZC{0,v}{i}\theta^{mv}(1+\sum_{j=0}^{m-i-1}\left[\begin{matrix}m-i\\j\end{matrix}\right]_\xi\ZC{0,v}{m-i-j}\xi^{(m-i-j)^2}\g_\ov^m)\\
      &=  \ZC{u,0}{i}(1+\sum_{j=1}^i\left[\begin{matrix}i\\j\end{matrix}\right]_\xi\ZC{u,0}{m-j}\xi^{-j(i-j)}\g_{\ou}^m),
   \end{split}
 \end{equation} 
 
 \begin{equation}\label{e8}
   \ZC{0,0}{i} = \frac{(-1)^{i-1}}{\xi^{\frac{i(i-1)}{2}}}\frac{(\frac{a_{\ol0}}{a_{\ou}})^m-1}{m(1-\xi^i)\g_\ou^m} =  \frac{(-1)^{i-1}}{\xi^{\frac{i(i-1)}{2}}}\theta^{mv}\frac{1-(\frac{a_{\ol0}}{a_{\ov}}\theta^{(l-1)v})^m}{m(1-\xi^i)\,\g_\ov^m}\,, 
 \end{equation} 

  \begin{equation}\label{e9}
  \begin{split}
   \g_\ov^m\ZC{u,v}{i}\prod_{j=1}^i(\frac{a_{\ol0}}{a_{\ou}}-\xi^j)&=\prod_{j=0}^{i-1}(1-\xi^{-j}\frac{a_{\ol 0}}{a_{\ol v}}\theta^{(l-1)v})\, ,\\
\g_\ou^m\ZC{u,v}{i}\theta^{mv}\prod_{j=i}^{m-1}(1-\xi^{-j}\frac{a_{\ol0}}{a_{\ov}}\theta^{(l-1)v})&=\prod_{j={i+1}}^{m}(\frac{a_{\ol 0}}{a_{\ol u}}-\xi^j)\, ,   
  \end{split}
 \end{equation} 
  \begin{equation}\label{e10}
 \ZC{u,v}{i}\ZC{0,0}{i}=\ZC{u,0}{i}\ZC{0,v}{i}\,.
 \end{equation} 
 
\end{lem}
\begin{proof}
We apply $\JC{u,v}{0} = 1$ and $\ZC{u,v}{i} \theta^{iv} = \JC{u,v}{i}$ to the equations \eqref{eq:x2}-\eqref{eq:x4} to obtain
  \begin{align}
       a_{\ol{u+v}}(1  +\ZC{u,v}{m-1}\theta^{mv} \g_\ou^m )& = a_{\ov}+\xi^l\,a_{\ou}\ZC{u,v}{m-1}\theta^{(m-1+l)v}\g_\ou^m \label{eq:x21}\\
    a_{\ol{u+v}}(  \theta^{v} +\ZC{u,v}{1}\theta^{v+d} \g_\ov^m) &= a_{\ou}\theta^{lv}+ a_{\ov}\ZC{u,v}{1}\theta^v\g_\ov^m \label{eq:x31}\\
    a_{\overline{u+v}}(\ZC{u,v}{i-1}\theta^{iv}
    +\ZC{u,v}{i}\theta^{i(v+d)})&=\xi^{-l(i-1)}\theta^{(l+i-1)v}\ZC{u,v}{i-1}a_{\overline{u}}+\ZC{u,v}{i}a_{\overline{v}}\theta^{iv}\label{eq:x41}  
 \end{align}   
for all $2\le i\le m-1$. One can simplify \eqref{eq:x31} to \eqref{eq:z2}, \eqref{eq:x21} to \eqref{eq:z3}, and \eqref{eq:x41} to \eqref{eq:z4}. Similarly, applying the same substitutions to the 2-cocycle equations in Lemma \ref{l:2-cyc_2}, we find the equations \eqref{e5}, \eqref{e6}, and \eqref{e7}.

Now, apply $\ov  \ne \ol0$ and $\ou = \ol0$ to \eqref{eq:z2} and \eqref{eq:z4}, we find  for all $2\le i\le m-1$
\begin{equation}
\ZC{0,v}{1}=\frac{1-\frac{a_{\ol{0}}}{a_{\ol{v}}}\theta^{(l-1)v}}{\g_\ov^m(1-\xi)},\quad \ZC{0,v}{i}=\frac{1-\xi^{-(i-1)}\frac{a_{\ol{0}}}{a_{\ov}}\theta^{(l-1)v}}{1-\xi^i}\,\ZC{0,v}{i-1},  \label{J001} 
\end{equation}
and the first equation of \eqref{eq:Z01} inductively. Similarly,  by applying $\ou  \ne \ol0$ and $\ov = \ol0$ to \eqref{eq:z3} and \eqref{eq:z4}, we get 
 \begin{equation}
\ZC{u,0}{m-1}=\frac{\frac{a_{\ol{0}}}{a_{\ol{u}}}-1}{\g_\ou^m(1-\xi)},\quad \ZC{u,0}{i-1}=\frac{\frac{a_{\ol{0}}}{a_{\ol{u}}}-\xi^i}{1-\xi^{-(i-1)}}\,\ZC{u,0}{i},\label{J002} 
\end{equation}
for all $2\le i\le m-1$. Hence, the second equation of \eqref{eq:Z01} follows inductively. 

Regarding the first equation of \eqref{eq:qJrelation}, we consider $\ZC{0,v}{i}\g_\ov^m$ as a polynomial in terms of the variable $\frac{a_{\ol{0}}}{a_\ov}\theta^{(l-1)v}$ for $1\le i\le m-1$, which we denote by 
\begin{equation*}
\alpha_i(x)=\frac{\prod_{j=0}^{i-1}(1-\xi^{-j}x)}{\prod_{j=1}^i(1-\xi^j)}\quad  \text{with\quad  $\alpha_0(x)=1$}
\end{equation*}
according to \eqref{eq:Z01}. Therefore, it suffices to show that the following polynomial identity holds
\begin{align*}
\sum_{j=0}^{i}\left[\begin{matrix}i\\j\end{matrix}\right]_\xi\xi^{(i-j)^2}\,\alpha_{i-j}(x)=\frac{\prod_{j=1}^i(1-\xi^jx)}{\prod_{j=1}^i(1-\xi^j)}.
\end{align*}
Since both sides are polynomials of degree $i$, it remains to show that they agree on $i+1$ values of $x$. It is clear that both equal $1$ when $x=1$ and for any $x=\xi^{-k}$ with $1\le k\le i$, we get
\begin{align*}
\sum_{j=0}^{i}\left[\begin{matrix}i\\j\end{matrix}\right]_\xi\xi^{(i-j)^2}\,\alpha_{i-j}(\xi^{-k})&=\sum_{j=0}^{i}\left[\begin{matrix}i\\j\end{matrix}\right]_\xi\xi^{j^2}\,\alpha_{j}(\xi^{-k})\\
&=\sum_{j=0}^i(-1)^j\left[\begin{matrix}i\\j\end{matrix}\right]_\xi\,\left[\begin{matrix}k+j-1\\j\end{matrix}\right]_\xi\, \xi^{\frac{j(j+1-2k)}{2}}=0,
\end{align*}
where we apply \eqref{eq:qidentity}. This gives the first equation of \eqref{eq:qJrelation} and the second equation of \eqref{eq:qJrelation} follows similarly. 

The first equality of \eqref{eq:quadratic} is obtained by letting $\ou, \ov = \ol 0$ in \eqref{e6}, and the second one is similarly obtained by setting $\ov, \ol w =\ol 0$ in \eqref{e5}. In view of \eqref{eq:qJrelation}, these relations then imply \eqref{e8} such that 
\begin{align*}
\ZC{0,0}{i}&=\ZC{u,0}{i}\,\Big(1+\sum_{j=1}^i  \left[\begin{matrix}i\\j\end{matrix}\right]_\xi  \ZC{u,0}{m-j}\g_\ou^m\xi^{-j(i-j)}\Big)\\
&=\frac{(\frac{a_{\ol{0}}}{a_\ou}-\xi^{i+1})(\frac{a_{\ol{0}}}{a_\ou}-\xi^{i+2})\cdots (\frac{a_{\ol{0}}}{a_\ou}-\xi^m)}{\g_\ou^m(1-\xi)(1-\xi^2)\cdots (1-\xi^{m-i})}\,\frac{(\frac{a_{\ol{0}}}{a_\ou}-\xi)(\frac{a_{\ol{0}}}{a_\ou}-\xi^2)\cdots (\frac{a_{\ol{0}}}{a_\ou}-\xi^i)}{(1-\xi)\cdots (1-\xi^{i-1})(1-\xi^i)}\\
&=\frac{(-1)^{i-1}}{\xi^{\frac{i(i-1)}{2}}}\frac{\Big(\frac{a_{\ol{0}}}{a_\ou}\Big)^m-1}{\g_\ou^m(1-\xi^i)(1-\xi)\cdots (1-\xi^{m-i})(1-\xi^{m-i+1})\cdots (1-\xi^{m-1})}\\
&=\frac{(-1)^{i-1}}{\xi^{\frac{i(i-1)}{2}}}\frac{\Big(\frac{a_{\ol{0}}}{a_\ou}\Big)^m-1}{\g_\ou^m\,m(1-\xi^i)}.
\end{align*}
The second equality of \eqref{e8} can be obtained similarly.

The first equation of \eqref{e9} comes from letting $\ov=0$ in \eqref{e5} and applying \eqref{eq:Z01} and \eqref{eq:qJrelation} and the second equation of \eqref{e9} can be proved similarly. Finally, \eqref{e10} is obtained by multiplying both sides of \eqref{e5} by $\ZC{u,0}{i}$ after setting $\ov = 0$, and then applying \eqref{eq:quadratic}. \qedhere
\end{proof}

\begin{cor} \label{c:trivial}
    Let $(\vp, J)$  be a standard normalized stable twist of $A$. Then the following are equivalent.
    \begin{itemize}
        \item[(i)]   $J=1\otimes 1$.
        \item[(ii)] $\left(\frac{a_{\ol0}}{a_{\ol s}}\right)^m=1$ for some $\ol s \in \ZD$.
        \item[(iii)] $\left(\frac{a_{\ol0}}{a_{\ol s}}\right)^m=\theta^{-m(l-1)s}$ for some $\ol s\in \ZD$.
        \item[(iv)] $(\varphi,J)=(I_z,\partial_z)$ for some grouplike element $z\in G(A)$.
    \end{itemize}
    In particular, $J$ is uniquely determined by $\varphi$ and $\st_t(A)=\st_0(A)$.
\end{cor}
\begin{proof}
   (i) $\Rightarrow$ (ii), (iii) is from \eqref{e8}. 
   
   (ii) $\Rightarrow$ (i).  We first show that $\frac{a_{\ol 0}}{a_{\ol s}}=1$. On the contrary, assume $\frac{a_{\ol0}}{a_{\ol s}}=\xi^a$ for some $1\le a\le m-1$. By \eqref{e9}, we have $\prod_{j=0}^{i-1}(1-\xi^{-j}\frac{a_{\ol0}}{a_{\ol s}}\theta^{(l-1)s})=0$ for all $a\le i\le m-1$. Hence, $\ZC{0,s}{i}=0$ for all $a\le i\le m-1$ by \eqref{eq:Z01}. Setting $u=-v$, $w=s$ in \eqref{e5} yields that $\ZC{v,s}{i}=0$ for all $v\in \mathbb Z_n$ and $a\le i\le m-1$. Therefore,  $\ZC{u,0}{i}=0$ for all $v\in \mathbb Z_n$ and $a\le i\le m-1$ by setting $v=0$ and $w=s$ in \eqref{e6}. In particular, $\ZC{s,0}{m-1}=\frac{\frac{a_{\ol0}}{a_{\ol s}}-1}{\g_{\ol{s}}^m(1-\xi)}=0$. A contradiction to our assumption!  
   
Hence, we have $\frac{a_{\ol 0}}{a_{\ol s}}=1$. By \eqref{eq:Z01}, $\ZC{s,0}{i}=0$ for all $1\le i\le m-1$. Setting $u=s,w=-v$ in \eqref{e6} yields $\ZC{s,v}{i}=0$ for all $v\in \mathbb Z_d$. Hence, $\ZC{0,w}{i}=0$ for all $w\in \mathbb Z_d$ by letting $u=s,v=0$ in \eqref{e5}. Therefore, $\ZC{v,w}{i}=0$ for all $v,w\in \mathbb Z_d$ and $1\le i\le m-1$ by setting $u=-v$ in \eqref{e5}.
   
   (iii) $\Rightarrow$ (i) can be proved similarly.  (iv) $\Rightarrow$ (i) is trivial. 
   
   (i) $\Rightarrow$ (iv) By our previous discussion, we know $\frac{a_{\ol0}}{a_{\ov}}=1=\theta^{-(l-1)v}$ for all $v\in \ZD$. Consequently, $l\equiv 1 \pmod n$ and $\varphi(x)=xa$ such that $a\in \k^\times$. From the relation $x^m=1-g^m$, it is easy to deduce that $a=\xi^i$ and we can take $z=g^i$. 

In summary, equations \eqref{eq:Z01} and \eqref{e8} imply that $\ZC{u,0}{i},\ZC{0,v}{i},\ZC{0,0}{i}$ for any $u,v\in \ZD$ and $1\le i\le m-1$ can be expressed in terms of the parameters $l$ and $\frac{\ao}{\aw}$ for $\ow \in \ZD$. Furthermore if $\ZC{0,0}{i}\neq 0$, so can $\ZC{u,v}{i}$ by \eqref{e10}. Otherwise if $\ZC{0,0}{i}=0$, then $(\frac{a_{\ol0}}{a_{\ol s}})^m=1$ by \eqref{e8} and in this case $\ZC{u,v}{i}=0$ for all $u,v\in \mathbb Z_d$ by our previous discussion.

Finally, suppose there is some stable twist $(\varphi',J')\in \st_t(A)$. Without loss of generality, we can assume $J'=1\otimes 1$. Now, since $\varphi': A\to A_{J'}=A$ is a Hopf algebra automorphism, we must have $\varphi'(g)=g^{l'}$ for some integer $l'$ coprime to $n$. By Lemma \ref{l:vg}, $(\varphi',J')$ is a standard normalized stable twist of $A$. So we know it belongs to $\st_0(A)$.  
\end{proof}

\begin{proof}[Proof of Theorem~\ref{inverT}] Since $(\vp, J)$ is standard, $\vp(g)=g^l$ for some integer $l$ coprime to $n$. By setting $u, w=\ol 0,\ov\neq \ol0$ and $i=m-1$ in \eqref{e7}, we find the equation
  \[
   \ZC{0, v}{m-1}\left(1+\ZC{v,0}{m-1}\g_\ov^m\right) =  \ZC{v, 0}{m-1}\left(1+\sum_{j=0}^{m-2}\left[\begin{matrix}m-1\\j\end{matrix}\right]_\xi\,\ZC{0,v}{m-1-j}\xi^{(m-1-j)^2}\g_\ov^m\right)\,. 
  \]
Now, we apply \eqref{eq:Z01} and \eqref{eq:qJrelation} to obtain
\begin{align*}
  &\frac{(1-\frac{a_{\ol{0}}}{a_\ov}\theta^{(l-1)v})(1-\xi^{-1}\frac{a_{\ol{0}}}{a_\ov}\theta^{(l-1)v})\cdots (1-\xi^{-(m-2)}\frac{a_{\ol{0}}}{a_\ov}\theta^{(l-1)v})}{\g_\ov^m(1-\xi)(1-\xi^2)\cdots (1-\xi^{m-1})}\,\frac{\frac{a_{\ol{0}}}{a_\ov}-\xi}{(1-\xi)} \\
   &=\frac{\frac{a_{\ol{0}}}{a_\ov}-1}{\g_\ov^m(1-\xi)}\,\frac{(1-\xi\frac{a_{\ol{0}}}{a_\ov}\theta^{(l-1)v})(1-\xi^2\frac{a_{\ol{0}}}{a_\ov}\theta^{(l-1)v})\cdots (1-\xi^{(m-1)}\frac{a_{\ol{0}}}{a_\ov}\theta^{(l-1)v})}{(1-\xi)(1-\xi^2)\cdots (1-\xi^{m-1})}\,,
\end{align*}
which can be further simplified into 
\begin{equation}\label{eq:as}
 \Big(\frac{\ao}{\av}\Big)^2 = \frac{1}{\theta^{(l-1)v}}.
\end{equation}
By \eqref{e8} with $\ol u=\ol v\neq \ol 0$, we have
$(\frac{\ao}{\av})^m-1=\theta^{mv}(1-(\frac{\ao}{\av}\theta^{(l-1)v})^m)$, which implies that 
 \begin{equation*}
\Big(\frac{\ao}{\av}\Big)^m(1+\theta^{lvm})=1+\theta^{vm}.
 \end{equation*} 
Suppose $m$ is even. Then, we have $(1-\theta^{\frac{m}{2}(l-1)})(1-\theta^{\frac{m}{2}(l+1)})=0$ by further setting $\ov=\ol 1$ above and using \eqref{eq:as}. This shows that $l\equiv \pm 1\pmod {2d}$. If $m$ is odd, similarly we get 
\begin{equation}
  \frac{\ao}{\av}(1+\theta^{lvm})~=~\theta^{\frac{m-1}{2}(l-1)v}  (1+\theta^{vm})\,.\label{eq:as2}
\end{equation}
By applying \eqref{eq:as} again, we deduce that $(1-\theta^{m(l-1)})(1-\theta^{m(l+1)})=0$, which implies that $l\equiv \pm 1\pmod d$. This completes the proof of statement (i).

Now, we first consider the case  $l\equiv 1\pmod{\delta_m\,d}$. Then, by \eqref{eq:as}, we have if $m$ is even
$$
\left(\frac{\ao}{\av}\right)^m~=~\theta^{-\frac{m}{2}(l-1)v}=1\quad \text{ for all }\ov \ne \ol0\,.
$$
Suppose $m$ is odd and $d>2$ ($1+\theta^m\neq 0$). By letting $\ov=\ol{1}$ in \eqref{eq:as2}, we get $\frac{\ao}{\av}=\theta^{\frac{m-1}{2}(l-1)}$ and hence
$$
\left(\frac{\ao}{\av}\right)^m~=~\theta^{\frac{m-1}{2}(l-1)m}=1\quad \text{ for }\ov = \ol1\,.
$$
It follows from Corollary  \ref{c:trivial} that $(\varphi,J)$ is equivalent to $(\id_A, 1 \o 1)$, and hence $(\vp, J) \in \st_0(A)$. This proves assertion (ii).

Next, we assume $l\equiv -1\pmod{\delta_m\, d}$. Since $l\equiv 1 \pmod m$ by Lemma \ref{bgT2}, it is easy to deduce that $\gcd(m,d)\mid 2$. Suppose $\gcd(m,d)=2$. Let $\ov=\ol{d/2}\in \mathbb Z_d$. Then we can write $l=1+km$ with $k$ odd, since $l\equiv -1 \pmod {4}$. Then,  \eqref{eq:as} implies that $(\frac{\ao}{\av})^2=\theta^{-(l-1)v}=-1$. Note that $l\neq 1$ and Corollary \ref{c:trivial} requires that $(\frac{\ao}{\av})^m\neq 1$. Further by applying \eqref{e9} to \eqref{e5} while setting $\ou=\ov=\ol{d/2}$, $\ol w=\ol 0$ and $i=1$, we obtain 
\begin{align*}
  \ZC{v,0}{1}=\ZC{0,0}{1}(1+\ZC{v,v}{m-1}\theta^{mv}\g_\ov^m)=\ZC{0,0}{1}(1+\frac{\frac{\ao}{\av}-1}{1-\xi \frac{\ao}{\av}\theta^{(l-1)v}})=\ZC{0,0}{1}\frac{\frac{\ao}{\av}(1+\xi)}{1+\xi \frac{\ao}{\av}}\, .
\end{align*}
In view of \eqref{eq:Z01} and \eqref{e8}, the above equation implies that 
\[
\left(\frac{\ao}{\av}-1\right)\left((1+\xi)\frac{\ao}{\av}-(\xi-1)\right)=0\,.
\]
Since $(\frac{\ao}{\av})^2=-1$, we must have $\xi^2=-1$ and $m=4$. It is a  contradiction to Corollary \ref{c:trivial} since $(\frac{\ao}{\av})^4\neq 1$. 

In the remaining, we assume $\gcd(m,d)=1$. There is a unique integer $1<r<n$ such that $r=0\pmod m$ and $r=1 \pmod d$. We set up $\frac{\ao}{\au}=\theta^{r u}$ for any $u\in \mathbb Z_d$. Then, it is direct to check that \eqref{eq:as} holds and $\frac{\ao}{\au} \theta^{(l-1)v}=\theta^{-r u}$ for any $u\in \mathbb Z_d$. In particular, $a_\ou^m = \frac{\g^m_{\ol{lu}}}{\g^m_{\ou}}=-\theta^{-mu}$ for $\ou\in \ZD$ and $a_0^m=\theta^{mr  u}a_\ou^m=-1$. 
Now, we can construct a standard normalized stable twist $(\tilde\vp, \tilde J)$ of $A$ in the following way: as an algebra automorphism of $A$, the map $\tilde\vp: A \to A$ is given by $\tilde\vp(g) = g^l$ and 
\[
\tilde\vp(x) = x\left(\sum_{u\in \mathbb Z_d}a_\ou e_\ou\right)=\ao x\left(\sum_{u\in \mathbb Z_d} \theta^{-r\,u} e_\ou\right)=\sqrt[m]{-1}\,xg^{-r}\, .
\]
The 2-cocycle $\tilde J$ is given by the coefficients $\Zt{u,v}{i}$ according to \eqref{eq:NewJ} via the formulae derived in Lemma \ref{l:Z}, which can be summarized in a single formula,
\begin{equation}\label{closeform}
    \Zt{u,v}{i}:=\frac{-\theta^{-r v i}}{\left(\prod_{j=1}^i(\theta^{r u}-\xi^j)\right)\left(\prod_{j=1}^{m-i}(\theta^{r v}-\xi^j)\right)},
\end{equation}
for any $u, v \in \Z_n$ and $1\le i\le m-1$. In details, $\tilde J=1\o 1 + \sum_{i=1}^{m-1} (x^i \o x^{m-i} g^i)\Zt{}{i}$ such that 
\begin{align*}
  \Zt{}{i}&=\sum_{\ou,\ov\in \mathbb Z_d}   \Zt{u,v}{i} e_\ou\otimes e_\ov=\sum_{\ou,\ov\in \mathbb Z_d}\frac{-\theta^{-r v i}}{\left(\prod_{j=1}^i(\theta^{r u}-\xi^j)\right)\left(\prod_{j=1}^{m-i}(\theta^{r v}-\xi^j)\right)}\,e_\ou\otimes e_\ov\\
  &=-\sum_{\ou,\ov\in \mathbb Z_d}\frac{1}{\prod_{j=1}^i(\theta^{r u}-\xi^j)}\,e_\ou\otimes \frac{1}{\prod_{j=1}^{m-i}(\theta^{r v}-\xi^j)}g^{-r i}e_\ov\\
  &=-\sum_{\ou,\ov\in \mathbb Z_d}(y\otimes y)\left(\prod_{j=i+1}^{m-1}(g^{r}-\xi^j)\otimes g^{-r i}\prod_{j=m-i+1}^{m-1}(g^{r}-\xi^j)\right)e_\ou\otimes e_\ov\, ,
\end{align*}
where 
\begin{equation}\label{yinJ}
 y=\frac{1}{m}e_{\ol{0}}+\sum_{\ou\in \ZD} \frac{\theta^{r u}-1}{\theta^{ m u}-1}\,e_\ou=\sum_{k=0}^{\frac{r}{m}-1}g^{mk}+\frac{1-r}{n}\sum_{k=0}^{d-1}g^{mk}\,.   
\end{equation}
Then $(\tilde\vp, \tilde J)$ is a standard normalized stable twist of $A$: it satisfies the compatibility conditions \eqref{eq:1oJ}--\eqref{eq:Jo1} and \eqref{eq:x2}--\eqref{eq:x4}, as one checks by direct computation.

Since $\ao = \sqrt[m]{-1}$,  the stable twist of the Taft algebra $T_{m^2}(\xi)$ descended from $(\tilde\vp, \tilde J)$  via the surjection $\pi: A \to A/A x^m \cong T_{m^2}(\xi)$ is non-trivial by Corollary \ref{StableTwTaft}. Therefore, the stable twist $(\tilde\vp, \tilde J)$ is nontrivial in $\ost(A)$.  

Finally, we will show that $(\tilde\vp, \tilde J)$ is the only nontrivial standard normalized stable twist of $A$ up to equivalence. Let $(\vp,J)$ be a nontrivial standard normalized stable twist of $A$. By the preceding discussion, we have $\gcd(m,d)=1$ and $l$ is the unique integer $\pmod {n=md}$ determined by $l=-1 \pmod d$ and $l=1 \pmod m$. Moreover, we have $\vp(g)=g^{l}$  and $\vp(x) = \sum_{\ou \in \Z_d} \au x \eu$ in which $\au \in \k^\times$ satisfying $\au^m = \frac{\g_{-\ou}^m}{\g_{\ou}^m}=-\theta^{-mu}$ for all $\ou \in \ZD$, and $\ao^m=-1$. In particular, $\vp(x)\eo = \xi^k\tilde\vp(x)\eo$ for some integer $k$. By Proposition~\ref{p:uniqueness}, there exists $z \in \kgt{g}$ with $\vp_z = \tilde\vp$. By Corollary~\ref{c:trivial}, $J_z$ is determined by $\vp_z(x) = \tilde\vp(x)$, so $J_z=\tilde J$. This proves statement (iii).
\end{proof}

\begin{proof}[Proof of Theorem~\ref{StableTwTaft2d}]
We compute the kernel and image of $\Pi$ in the exact sequence \eqref{eq:A1-SES}.

Let $[(\id_A,K)]\in \ker(\Pi)$. We show that it is equivalent to some normalized standard stable twist $(\varphi,J)$ where $\varphi(g)=g^l$ with $l\equiv 1\pmod d$. Let $K^\bo$ be the constant term of $K$. If $K^\bo$ is invertible, by Lemma \ref{l:cocycle} there is some $z\in (\k\langle g\rangle)^\times_1$ such that $(I_z,K_z)$ is standard and normalized. Then it is clear that $I_z(g)=g$. When $K^\bo$ is not invertible, by the proof of Proposition \ref{p:invertible}, there is some $z\in \A$ such that $(K_z)^\bo$ is invertible and $I_z(g)=g^l$ satisfies $l\equiv 1 \pmod d$ according to \eqref{eq:Iz}. Applying Lemma \ref{l:cocycle} to $(I_z,K_z)$ again, we reach the same conclusion. In both cases $l=1$, whence Theorem \ref{inverT}\textup{(ii)} gives $(\id_A,K)\in \st_0(A)$. Thus $(\id_A,K)$ is trivial in $\ost(A)$, so $\ker(\Pi)=1$; equivalently $H^2_\ell(A^\ast)=1$.

Let $[(\vp,J)]\in \ost(A)$. By Lemma \ref{bgT2} its class has a standard representative with $\vp(g)=g^l$, $\vp(x)=xa$ and $l\equiv 1 \pmod m$, so every class in $\OutTw(A)$ is represented by an automorphism of this form. By Theorem \ref{inverT}, a standard normalized stable twist is either trivial or, when $\gcd(m,d)=1$, equivalent to the unique nontrivial twist $(\tvp,\tJ)$ of Theorem \ref{inverT}\textup{(iii)}. Hence $\OutTw(A)$ is trivial when $\gcd(m,d)>1$, and has exactly two elements, represented by $\id_A$ and $\tvp$, when $\gcd(m,d)=1$.

Since $\ker(\Pi)=1$, the exact sequence \eqref{eq:A1-SES} gives $\ost(A)\cong\OutTw(A)$, which is $\Z/2\Z$ when $\gcd(m,d)=1$ and trivial otherwise.
\end{proof}

\subsubsection{\texorpdfstring{$C_2$-actions and the associated Hopf extensions}{C2-actions and the associated Hopf extensions}}\label{SS:A1-C2action}
Consider $A = A_1(md, \xi)$ with $\gcd(d,m)=1$. We have already established that the stable twist group satisfies $\ost(A)\cong C_2$. Our next goal is to prove that $\ost(A)$ lifts to a nontrivial $C_2$-action on $\C{A}$.

Recall the unique integer $1<r<md$ such that $r\equiv 0 \pmod m$ and $r\equiv 1 \pmod d$. Also let $1<s<md$ be an integer such that $s\equiv \delta_m-3 \pmod m$ and $2s\equiv 1-m \pmod d$, where $\delta_m=
\frac{3+(-1)^m}{2}$ as in Theorem~\ref{inverT}. Note that $s$ is not unique if $d$ is even.  

According to Lemma \ref{lem-Caction}, one needs to find an element $z \in \A$ such that  $(\tvp, \tJ)^2 =\cb{z}^{-1}$ and $\tvp(z)=z$. 

\begin{lem}\label{l:C2action}
   Let $z = \frac{1}{m} g^s\prod_{i=1}^{m-1}(g^{r}-\xi^i)$. Then $(\tvp, \tJ)^2 =\cb{z}^{-1}$ and $\tvp(z)=z$. Moreover, we have $\alpha_{\tJ}=\frac{1}{m}g^{-r(m-1)}\,\prod_{j=1}^{m-1}(g^r-\xi^j)$. 
\end{lem}
\begin{proof} 
We can write $z= \sum_{u \in \Z_n} z_u e_u$, where
    $$
    z_u = \frac{1}{m}\theta^{s u}\prod_{i=1}^{m-1}(\theta^{r u}-\xi^i)\ne 0 \quad  \text{for } u \in \Z_n\, .
    $$
   Thus $z \in \A$. Moreover,  
\begin{align*}
    \tvp(z)&=\frac{1}{m}g^{s l}\prod_{i=1}^{m-1}(g^{-r}-\xi^i)=\frac{1}{m}g^{s l-(m-1)r}(-1)^{m-1}\xi^{\frac{m(m-1)}{2}}\prod_{i=1}^{m-1}(g^r-\xi^i)\\
    &=\frac{1}{m} g^s\prod_{i=1}^{m-1}(g^{r}-\xi^i)=z.
\end{align*}
It remains to show that 
   \begin{equation} \label{eq:cob}
       \cb{z}(\tvp, \tJ)^2 = (\id_A, 1\o 1)\quad \text{or}\quad (\tvp, \tJ)^2 = \cb{z^{-1}}\, .
   \end{equation}
    Note that $(\tvp, \tJ)^2 =(\tvp^2, \tvp(\tJ)\tJ)$. Therefore, one can compute that $\tvp^2(g)=g^{l^2}=g$ and $\tvp^2(x)=(\sqrt[m]{-1})^2xg^{-r(l+1)}=\xi^{3-\delta_m}x$. Since $r\equiv 0\pmod m$ and $s\equiv \delta_m-3 \pmod m$, we have $x z = \xi^{3-\delta_m}zx$ or $I_{z^{-1}}(x)=\xi^{3-\delta_m}x$. Moreover, $I_{z^{-1}}(g)=g$ as $z \in \kg{g}^\times$. Therefore,  $I_z \tvp^2 = \id_A$. Next, by direct computation from \eqref{eq:NewJ} using \eqref{closeform}, we find
    \begin{align*}
        \tvp(\tJ)\tJ  &=1\otimes 1-\sum_{i=1}^{m-1}\sum_{u,v\in \mathbb Z_d}\ZC{-u,-v}{i}\ZC{u,v}{m-i}\g_\ou^m\g_\ov^m\xi^{i^2}\theta^{-ir(u+v)}e_u\otimes e_v\\
       &=1\otimes 1+\sum_{i=1}^{m-1}\left(\sum_{\ou,\ov\in \mathbb Z_d}\frac{\g_\ou^m}{(\theta^{r u}-\xi^{m-i})\left(\prod_{j=1}^{m-1}(\theta^{r u}-\xi^j)\right)}\,e_\ou\otimes\right.\\
      &\quad \quad\quad\quad \left. \frac{\g_\ov^m}{(\theta^{r v}-\xi^i)\left(\prod_{j=1}^{m-1}(\theta^{r v}-\xi^j)\right)}\, e_\ov\right)\\
 &=\sum_{\ou,\ov\in \mathbb Z_d}\left(1+\sum_{i=1}^{m-1}\frac{1+\xi^{m-i}\theta^{r u}+\cdots +(\xi^{m-i}\theta^{r u})^{m-1}}{1+\theta^{r u}+\cdots+ \theta^{(m-1)r u}}\,e_\ou\otimes \right.\\
 &\quad\quad\quad\quad \left.\frac{1+\xi^{i}\theta^{r v}+\cdots +(\xi^{i}\theta^{r v})^{m-1}}{1+\theta^{r v}+\cdots+ \theta^{(m-1)r v}}\,e_\ov\right)\\
            &=m\!\sum_{\ou,\ov\in\mathbb Z_d}\!\!\left(\!\frac{1+\theta^{r(u+v)}+\cdots+\theta^{(m-1)r(u+v)}}{(1+\theta^{r u}+\cdots+\theta^{(m-1)r u})(1+\theta^{r v}+\cdots+\theta^{(m-1)r v})}\!\right)\!e_\ou\!\otimes\! e_\ov\\
      &= \sum_{u,v\in \Z_d} \frac{z_{u+v}}{z_u z_v} e_\ou\o e_\ov = \partial_{z^{-1}}\, .
    \end{align*}
   This completes the proof of \eqref{eq:cob}. 

Finally, by applying \eqref{closeform}, we obtain that
\begin{align*}
    \beta_{\tJ}&=1+\sum_{i=1}^{m-1} \left(\sum_{u,v\in \mathbb Z_d} \ZC{u,v}{i}\, x^ie_\ou S\left(x^{m-i}g^ie_\ov\right)\right)\\
    &=1+\sum_{i=1}^{m-1} (-1)^{i+1}\xi^{\frac{i(i-1)}{2}}\left(\sum_{u\in \mathbb Z_d} \ZC{u,-u}{i}\g_\ou^m \theta^{-mu}\, e_\ou\right)\\
    &=1+\sum_{i=1}^{m-1} \left(\sum_{u\in \mathbb Z_d} \frac{\theta^{r u}-1}{\theta^{r u}-\xi^i}e_\ou\right)\\
    &=1+\sum_{i=1}^{m-1}\frac{g^r-1}{g^{r}-\xi^i}\\
    &=\frac{m\,g^{r(m-1)}}{\prod_{j=1}^{m-1}(g^r-\xi^j)}.
\end{align*}
Therefore, $\alpha_\tJ=(\beta_\tJ)^{-1}=\frac{1}{m}g^{-r(m-1)}\,\prod_{j=1}^{m-1}(g^r-\xi^j)$.
\end{proof}

The construction $\N(md,\xi)=A[u^{\pm 1};\tilde\varphi,\tilde J,z]$ of Definition~\ref{def:N-family} now feeds into the comprehensive classification of $C_2$-extensions.

\begin{proof}[Proof of Theorem~\ref{th:extC2}]
Note that Theorem \ref{StableTwTaft2d} and Lemma \ref{l:C2action} guarantee that $A$ satisfies all the assumptions stated in Theorem \ref{HopfCEGC2}, and our result follows. 
\end{proof}

\section{Central cocycle modifications and braided stable twists}\label{S:StGrQH}

This section collects applications of the stable-twist formalism to central modifications of (quasitriangular) quasi-Hopf algebras. \S\ref{SS:general-setup} fixes the setup: a quasi-Hopf algebra $A$ with a central sub-bialgebra $Z=\k^G$ (automatically a Hopf algebra by Lemma~\ref{l:central-is-Hopf}). \S\ref{SS:reduction-general} shows that modifying the associator by a central normalized $3$-cocycle on $Z$ leaves $\ost(A)$ unchanged under natural hypotheses (Theorem~\ref{t:reduction-general}). \S\ref{SS:Adelta-examples} works this out for the pointed Hopf algebras $A_\delta(n,\xi)$ of \S\ref{S:StGrH} (Proposition~\ref{c:A1-corollary}) and the new quasi-Hopf $C_2$-equivariantization $\N^\omega(md,\xi)$. \S\ref{SS:braided-reduction} establishes the quasitriangular analogue, in which both the associator and the $R$-matrix are modified by an Eilenberg--Mac~Lane $2$-abelian $3$-cocycle pair $(\omega,c)$ on $Z$. Finally, \S\ref{SS:soft} relates the kernel $H^2_{\ell,{\rm br}}(A^*)$ of the braided exact sequence (Theorem~\ref{t:braided-SES}) to two notions from the literature: for $A=D(G)$ it recovers, in explicit Hopf-algebraic form, Davydov's identification of the soft symmetries of \cite{SoftSym2025} with the Bogomolov multiplier $B_0(G)$ \cite{Davydov2014} (Theorem~\ref{t:B0=lazy-br}), and through the Etingof--Nikshych--Ostrik correspondence it computes Brauer--Picard groups (Proposition~\ref{p:Tw-BrPic}).

\subsection{General setup}\label{SS:general-setup}

Let $(A, \Delta, \ve, \Phi, S, \a, \beta)$ be a finite-dimensional quasi-Hopf algebra over an algebraically closed field $\kk$ of characteristic zero. Throughout, $Z$ denotes a \emph{central sub-bialgebra} of $A$: a subalgebra $Z \subseteq Z(A)$ satisfying $\Delta(Z) \subseteq Z \otimes Z$. 

\begin{lem}\label{l:central-is-Hopf}
Every central sub-bialgebra $Z$ of a finite-dimensional quasi-Hopf algebra $A$ is a Hopf algebra under the inherited structure maps $(m, u, \Delta|_Z, \ve|_Z, S|_Z)$. In particular, $S(Z) \subseteq Z$.
\end{lem}
\begin{proof}
We first show $Z$ is a well-defined bialgebra. Since $Z \subseteq Z(A)$, for any $z \in Z$ the element $(\Delta \o \id)(\Delta(z)) \in Z^{\otimes 3}$ is central in $A^{\otimes 3}$. The quasi-coassociativity $(\id \o \Delta)(\Delta(z)) = \Phi \cdot (\Delta \o \id)(\Delta(z)) \cdot \Phi^{-1}$ reduces to $(\id \o \Delta)(\Delta(z)) = (\Delta \o \id)(\Delta(z))$. So $\Delta|_Z$ is genuinely coassociative, making $(Z, m, u, \Delta|_Z, \ve|_Z)$ a bialgebra.

The quasi-antipode axiom $\sum S(z_1)\a z_2 = \ve(z)\a$ and centrality of $z_2$ give 
\[\bigl(\sum S(z_1)z_2\bigr)\a = \ve(z)\a.\] Multiplying into $\sum \Phi^1 \beta S(\Phi^2)\a\Phi^3 = 1$ (cf.~\eqref{EAssociator}) and using centrality of $\sum S(z_1)z_2 \in Z(A)$:
\[
\textstyle\sum S(z_1)z_2 \;=\; \bigl(\sum S(z_1)z_2\bigr)\cdot 1 \;=\; \sum \Phi^1\beta S(\Phi^2)\bigl(\sum S(z_1)z_2\bigr)\a\,\Phi^3 \;=\; \ve(z)\cdot 1\,.
\]
Similarly $\sum z_1 S(z_2) = \ve(z)$.

It remains to show $S(Z)\subseteq Z$. The convolution algebra $\operatorname{End}_\kk(Z)$ embeds into $\operatorname{Hom}_\kk(Z,A)$ via post-composition with $\iota\colon Z\hookrightarrow A$. Left convolution by $\operatorname{id}_Z$ defines $\lambda\colon\operatorname{End}_\kk(Z)\to\operatorname{End}_\kk(Z)$, $f\mapsto\operatorname{id}_Z*f$. The corresponding map $\Lambda$ on $\operatorname{Hom}_\kk(Z,A)$ is injective: if $\iota*f=0$, then $f=(S|_Z*\iota)*f=S|_Z*(\iota*f)=0$. Since $\lambda$ is the restriction of $\Lambda$, it is injective, hence surjective by finite-dimensionality. In particular $\operatorname{id}_Z$ has a right convolution inverse $\bar S\in\operatorname{End}_\kk(Z)$; the analogous argument with right convolution gives a left inverse, and by associativity the two agree. Uniqueness of the convolution inverse in $\operatorname{Hom}_\kk(Z,A)$ gives $S|_Z=\iota\circ\bar S$, whence $S(Z)\subseteq Z$ and $\bar S$ is the antipode of $Z$.
\end{proof}

Since $Z$ is commutative and $\kk$ is algebraically closed of characteristic zero, $Z\cong\kk^G$ for some finite group $G$, with $Z^*\cong\kk[G]$; in most applications $Z=\kk[K]$ for some central subgroup $K\le Z_G(A)$, with $G=\hat K$.

\begin{deff}\label{d:normalized-Phi}
An element $\phi \in Z^{\otimes 3}$ is a \emph{normalized $3$-cocycle on $Z$} if
\begin{enumerate}
\item $\phi$ is invertible in $Z^{\otimes 3}$;
\item $(\id \o \id \o \Delta)(\phi)\cdot(\Delta \o \id \o \id)(\phi) = (1 \o \phi)\cdot(\id \o \Delta \o \id)(\phi)\cdot(\phi \o 1)$;
\item $(\id \o \ve \o \id)(\phi) = 1 \o 1$.
\end{enumerate}
\end{deff}
Concretely, when $Z=\kk^G$ as above, every normalized $3$-cocycle on $Z$ has the form
\[
\phi_\omega \;:=\; \sum_{a,b,c \in G}\omega(a,b,c)\,\delta_a \o \delta_b \o \delta_c\qquad(\omega\in Z^3(G,\kk^\times)\text{ normalized})
\]
for the basis $\{\delta_g\}_{g\in G}$ of orthogonal idempotents of $\kk^G$. Beyond \textup{(iii)}, such a $\phi_\omega$ automatically satisfies the further normalizations $(\ve\o\id\o\id)(\phi)=1\o1=(\id\o\id\o\ve)(\phi)$.

This construction generalizes \cite[Example~10.6, Proposition~10.1]{BCPVO} --- which treats the Hopf case $\Phi=1$, where the modification comes purely from the abelian cohomology of $G$ --- to a possibly nontrivial associator $\Phi$.

\begin{prop}\label{p:quasi-Hopf-general}
For any normalized $3$-cocycle $\phi=\sum \phi^1\otimes \phi^2\otimes \phi^3 \in Z^{\otimes 3}$, the tuple $(A, \Delta, \ve, \Phi \cdot \phi, S, \a', \beta')$ is a quasi-Hopf algebra with $\a'=\a$ and $\beta'=\beta (\sum S(\phi^1)\phi^2 S(\phi^3))$.
\end{prop}
\begin{proof}
Since $\phi \in Z^{\otimes 3} \subseteq Z(A^{\otimes 3})$ is central, the modified associator $\Phi \cdot \phi$ satisfies the pentagon (the central factor $\phi$ commutes with all other factors, reducing the verification to the individual pentagons for $\Phi$ and $\phi$). Quasi-coassociativity for $\Phi \cdot \phi$ follows from that for $\Phi$ (conjugation by a central element).  Applying ${\rm mult}^{(4)}(S\otimes \id \otimes S\otimes \id): A^{\otimes 4}\to A$ to Definition \ref{d:normalized-Phi}(ii) and using the normalizations of $\phi$, we get $(\sum S(\phi^1)\,\phi^2\,S(\phi^3))^{-1}=\sum \phi^{1}S(\phi^{2})\phi^{3}$. Then antipode axioms \eqref{EquasiS} and \eqref{EAssociator} are verified directly.
\end{proof}

\subsection{Central modification for quasi-Hopf algebras}\label{SS:reduction-general}\label{SS:cocycle-general}

We recall the notion of twist-liftable outer automorphisms of a finite-dimensional quasi-Hopf algebra $(A,\Phi)$ as
\[
\OutTw(A,\Phi)\;:=\;\operatorname{Im}\bigl(\Pi_A\colon\ost(A,\Phi)\to\operatorname{Out}_{\rm alg}(A)\bigr)\;\subseteq\;\operatorname{Out}_{\rm alg}(A).
\]
We denote by $\st^Z(A,\Phi)\subseteq \st(A,\Phi)$ the subgroup consisting of those stable twists $(\varphi,J)$ where $\varphi(Z)=Z$. It is clear that $\st_0(A)\subseteq \st^Z(A,\Phi)$ and we write $\ost^Z(A,\Phi)\subseteq \ost(A,\Phi)$ and $\OutTwZ(A,\Phi)\subseteq \OutTw(A,\Phi)$ as the corresponding subgroups. By definition, $\OutTwZ(A,\Phi)$ acts on $H^3(G,\k^\times)$ with $Z=\k^G$. In terms of any normalized 3-cocycle $\phi_\omega$ on $Z$, one can check that $\varphi^{\otimes 3}(\phi_\omega)$ is again a normalized 3-cocycle on $Z$ for any $\overline{\varphi}\in \OutTwZ(A,\Phi)$. Set $\Phi_\omega:=\Phi\cdot \phi_\omega$. Moreover, define the \emph{cocycle ratio} $C_\vp(\phi_\omega)$ and, for invertible $K\in A^{\otimes 2}$, the obstruction $d^2(K)$:
\begin{align*}
C_\vp(\phi_\omega)&\;:=\; \vp^{\otimes 3}(\phi_\omega)\cdot\phi_\omega^{-1}\;\in\;Z^{\otimes 3},\\
d^2(K) &\;:=\; (1 \o K)\,(\id \o \Delta)(K)\,(\Delta \o \id)(K^{-1})\,(K^{-1} \o 1)\;\in\;A^{\otimes 3}.
\end{align*}

\begin{deff}\label{d:CT}
Let $\overline{\varphi}\in \OutTwZ(A,\Phi\omega)$ so that $\varphi(Z)=Z$ and $\varphi$ induces an automorphism $\varphi^*$ of $H^3(G,\k^\times)$. We say $\overline{\varphi}$ \emph{acts trivially on $H^3(G,\k^\times)$} if $\varphi^*=\id$; equivalently, for every normalized $3$-cocycle $\phi$ on $Z$ the cocycle ratio $C_\vp(\phi)=\vp^{\o 3}(\phi)\cdot(\phi)\inv$ is a $3$-coboundary on $Z$, i.e.\ there exists an invertible $J^0_\vp\in Z^{\o 2}$ with
\begin{equation}\label{eq:J0-normalized-gen}
(\ve \o \id)(J^0_\vp) = (\id \o \ve)(J^0_\vp) = 1\,,\qquad d^2(J^0_\vp) = C_\vp(\phi).
\end{equation}
We call such $J^0_\vp$ a \emph{primitive} for $(\varphi,\phi)$.
\end{deff}

For any stable twist $(\varphi,J)\in \ost^Z(A,\Phi_\omega)$, the associator condition $\vp^{\otimes 3}(\Phi_\omega)=(\Phi_\omega)_J$ together with centrality of $\phi_\omega$ in $A^{\otimes 3}$ unwinds to
\begin{equation}\label{eq:d2J-C-general}
\vp^{\otimes 3}(\Phi) \cdot C_\vp(\phi_\omega) \;=\; (\Phi)_J,
\end{equation}
and in particular, if $\Phi=1\otimes 1\otimes 1$, then $C_\vp(\phi_\omega)=d^2(J)$.

\begin{thm}\label{t:reduction-general}\label{t:main-general}
Let $(A,\Phi)$ be a finite-dimensional quasi-Hopf algebra with a central sub-bialgebra $Z=\k^G\subseteq Z(A)$, and let $\omega\in Z^3(G,\k^\times)$. Assume:
\begin{itemize}
\item[(i)] $H^2(G, \kk^\times) = 1$;
\item[(ii)] $\OutTwZ(A,\Phi)=\OutTw(A,\Phi)$ and $\OutTwZ(A,\Phi_\omega)=\OutTw(A,\Phi_\omega)$;
\item[(iii)] both $\OutTw(A,\Phi)$ and $\OutTw(A,\Phi_\omega)$ act trivially on $H^3(G,\kk^\times)$.
\end{itemize}
Then the map
\[
\Psi \colon \ost(A, \Phi_\omega) \longrightarrow \ost(A,\Phi)\,, \qquad [(\vp, J)] \longmapsto [(\vp,\; (J^0_\vp)^{-1} J)]\,,
\]
where $J^0_\vp$ is any primitive for $(\varphi,\phi_\omega)$, is a well-defined group isomorphism.
\end{thm}
\begin{proof}
By hypothesis (ii), it suffices to show that
\[
\Psi \colon \ost^{Z}(A, \Phi_\omega) \longrightarrow \ost^{Z}(A,\Phi)\,, \qquad [(\vp, J)] \longmapsto [(\vp,\; (J^0_\vp)^{-1} J)]\,,
\]
where $J^0_\vp$ is any primitive for $(\varphi,\phi_\omega)$, is a well-defined group isomorphism. We break this into four steps.

\medskip
\noindent\textit{Step 1: $(\vp, J'):=(\vp, (J^0_\vp)^{-1} J)$ is a stable twist of $(A,\Phi)$ fixing $Z$.}

Write $J^0:=J^0_\vp$ (which exists by~(iii)) and $J':=(J^0)^{-1}J$; as $J^0\in Z^{\otimes 2}\subseteq Z(A^{\otimes 2})$, it is central. The intertwiner property follows from this centrality:
\begin{align*}
J'\,\Delta(\vp(a))\,J'^{-1}
&= (J^0)^{-1}\,J\,\Delta(\vp(a))\,J^{-1}\,J^0
= (J^0)^{-1}\,(\vp\o\vp)(\Delta(a))\,J^0\\
&= (\vp\o\vp)(\Delta(a)).
\end{align*}
For the associator, $\vp^{\otimes 3}(\Phi_\omega)=(\Phi_\omega)_J$ together with \eqref{eq:d2J-C-general} and $C_{\vp}(\phi_\omega)=d^2(J^0)$ gives $\vp^{\otimes 3}(\Phi)=d^2(J^0)^{-1}(\Phi)_J=(\Phi)_{J'}$. Finally $(\ve\o\id)(J')=(\id\o\ve)(J')=1$, since $J^0$ and $J$ are both counit-normalized. Thus $(\vp,J')\in\st^Z(A,\Phi)$.

\medskip
\noindent\textit{Step 2: $\Psi$ is independent of the choice of $J^0_\vp$ and of the representative of $[(\vp,J)]$.}

If $\tilde J^0$ is another counit-normalized primitive, then $b:=(J^0)^{-1}\tilde J^0\in(Z\otimes Z)^\times$ is a normalized $2$-cocycle ($d^2(b)=1$), so $b=(z\o z)\Delta(z)^{-1}$ for some $z\in Z_1^\times$ by $H^2(G,\kk^\times)=1$; hence $(\vp,(\tilde J^0)^{-1}J)\sim(\vp,(J^0)^{-1}J)$ in $\ost(A,\Phi)$ via $z$. Next, suppose $(\vp,J)\sim(\tvp,\tJ)$ in $\st(A,\Phi_\omega)$ via $u\in A_1^\times$, so $\tvp=I_u\circ\vp$ and $\tJ=(u\o u)J\Delta(u)^{-1}$. Since $I_u$ is inner and $\phi_\omega$ is central, $C_{\tvp}(\phi_\omega)=C_\vp(\phi_\omega)$, so we may take $J^0_{\tvp}=J^0_\vp$; centrality of $J^0_\vp$ then gives $(J^0_{\tvp})^{-1}\tJ=(u\o u)J'\Delta(u)^{-1}$, i.e. $(\vp,J')\sim(\tvp,(J^0_{\tvp})^{-1}\tJ)$ in $\st(A,\Phi)$ via $u$.

\medskip
\noindent\textit{Step 3: $\Psi$ is a group homomorphism.}

Let $(\vp_1, J_1), (\vp_2, J_2) \in \st^Z(A, \Phi_\omega)$. By~\eqref{GroupTW},
\[
(\vp_1, J_1)\otimes(\vp_2, J_2) = (\vp_1 \vp_2,\ \vp_1(J_2)\,J_1);
\]
write $C_\vp:=C_\vp(\phi_\omega)$. Since $\phi_\omega$ is central and the $\vp_i$ are algebra automorphisms,
\[
C_{\vp_1\vp_2} = \vp_1^{\otimes 3}\bigl(\vp_2^{\otimes 3}(\phi_\omega)\bigr)\phi_\omega^{-1}=\vp_1^{\otimes3}(C_{\vp_2})\cdot C_{\vp_1}.
\]
Choosing primitives $J^0_{\vp_i}$ with $d^2(J^0_{\vp_i})=C_{\vp_i}$, the element $K:=J^0_{\vp_1}\cdot\vp_1^{\otimes 2}(J^0_{\vp_2})\in Z^{\otimes 2}$ is counit-normalized and satisfies $d^2(K)=C_{\vp_1}\cdot\vp_1^{\otimes3}(C_{\vp_2})=C_{\vp_1\vp_2}$ (using $d^2(\vp_1^{\otimes2}(J^0_{\vp_2}))=\vp_1^{\otimes3}(C_{\vp_2})$), so $K$ is a valid primitive for $(\vp_1\vp_2,\phi_\omega)$. Writing $J_i':=(J^0_{\vp_i})^{-1}J_i$ and using centrality of the $J^0$'s,
\begin{align*}
\Psi\bigl([(\vp_1\vp_2,\vp_1(J_2)J_1)]\bigr)
&= [(\vp_1\vp_2,\ K^{-1}\vp_1(J_2)J_1)]
= [(\vp_1\vp_2,\ \vp_1(J_2')\,J_1')]\\
&= \Psi([(\vp_1,J_1)])\cdot\Psi([(\vp_2,J_2)]),
\end{align*}
the last equality by~\eqref{GroupTW} in $\ost(A,\Phi)$.

\medskip
\noindent\textit{Step 4: $\Psi$ is an isomorphism.}

Exchanging the roles of $\Phi$ and $\Phi_\omega$, Steps~1--3 furnish a well-defined homomorphism
\[
\Psi'\colon \ost^{Z}(A,\Phi)\to\ost^{Z}(A,\Phi_\omega),\qquad [(\vp,J')]\mapsto\bigl[(\vp,(K^0_\vp)^{-1}J')\bigr],
\]
where $K^0_\vp$ is any primitive for $(\vp,\phi_\omega^{-1})$, that is, $d^2(K^0_\vp)=C_\vp(\phi_\omega^{-1})=C_\vp(\phi_\omega)^{-1}$. Since $\Psi'$ does not depend on this choice (Step~2) and $(J^0_\vp)^{-1}$ is such a primitive, we may take $K^0_\vp=(J^0_\vp)^{-1}$, so that
\[
\Psi'\bigl(\Psi([(\vp,J)])\bigr)=\bigl[(\vp,\,J^0_\vp\,(J^0_\vp)^{-1}J)\bigr]=[(\vp,J)],
\]
and symmetrically $\Psi\circ\Psi'=\id$. Hence $\Psi$ is a group isomorphism with inverse $\Psi'$, which is the asserted isomorphism $\ost(A,\Phi_\omega)\xrightarrow{\sim}\ost(A,\Phi)$.
\end{proof}

The following result is probably well known in group cohomology; we include a proof for the reader's convenience.

\begin{lem}\label{l:canonical-J0-A1}
Let $C_n=\langle g\rangle$ be the cyclic group of order  $n\ge 1$ and $\omega \in Z^3(C_n, \kk^\times)$. Then for any group automorphism $\sigma_k(g)=g^k$ for any $k\ge 1$ coprime to $n$, we have $[\sigma_k^*(\omega)] = [\omega^{k^2}]$ in $H^3(C_n, \kk^\times)$.
\end{lem}
\begin{proof}
We may assume $n\ge 2$ and regard $\kk^\times$ as a trivial $C_n$-module. The standard periodic free resolution of $\Z$ over $\Z C_n$,
\[
\cdots \xrightarrow{N} \Z C_n \xrightarrow{g-1} \Z C_n \xrightarrow{N} \Z C_n \xrightarrow{g-1} \Z C_n \xrightarrow{\ve} \Z \to 0\qquad\bigl(N:=\textstyle\sum_{i=0}^{n-1}g^i\bigr),
\]
computes $H^*(C_n,\kk^\times)$ and yields the standard identification $H^3(C_n,\kk^\times)\xrightarrow{\sim}\mu_n(\kk)$, $[\omega]\mapsto\xi_\omega:=\omega(1)$. The automorphism $\sigma_k\colon g\mapsto g^k$ lifts to a $\sigma_k$-semilinear chain endomorphism $\tilde\sigma_\bullet$ of this resolution; writing $N_k:=1+g+\cdots+g^{k-1}$ and using $g^k-1=(g-1)N_k$ and $N\cdot N_k=k\,N$, one may take
\[
\tilde\sigma_0(1)=1,\quad \tilde\sigma_1(1)=N_k,\quad \tilde\sigma_2(1)=k,\quad \tilde\sigma_3(1)=k\,N_k.
\]
Since the action on $H^3$ is precomposition with $\tilde\sigma_3$, we get $(\sigma_k^*\omega)(1)=\omega(\tilde\sigma_3(1))=\xi_\omega^{\,k^2}$, i.e.\ $[\sigma_k^*\omega]=[\omega^{k^2}]$ in $H^3(C_n,\kk^\times)$.
\end{proof}

\begin{example}\label{ex:central-C2-group}
Let $(A,\Phi)$ be a finite-dimensional quasi-Hopf algebra such that $Z_G(A)=\langle u\rangle\cong C_2$. In this case, $H^3(C_2,\k^\times)=\mathbb Z/2\mathbb Z$, whose unique nontrivial cohomology class is presented by the normalized 3-cocycle $\omega: C_2\times C_2\times C_2\to \k^\times$
\begin{align*}
    \omega(x,y,z)=
    \begin{cases}
        -1 & \text{if } x=y=z=u\\
        1 & \text{otherwise}
    \end{cases}\, .
\end{align*}
We write $(A,\Phi_\omega)$ for the corresponding modified quasi-Hopf algebra, $\Phi_\omega:=\Phi\,\phi_\omega$. Since stable twists fix central grouplike elements and $\operatorname{Aut}(C_2)=1$ acts trivially on $H^3(C_2,\kk^\times)$, the hypotheses of Theorem~\ref{t:main-general} hold and it gives
\[
\ost(A)\;\cong\;\ost(A,\Phi_\omega)\,.
\]
This applies in particular to group algebras $A=\kk[G]$ whenever $Z(G)\cong C_2$ --- e.g.\ the dihedral $D_{4n}$, dicyclic $Q_{4n}$, semidihedral $SD_{2^n}$, extraspecial $2$-groups $2^{1+2n}_\pm$, $\operatorname{SL}_2(\F_q)$ for $q$ odd, and the binary polyhedral groups $2T,2O,2I$ --- and to non-group examples such as the Kac algebra $K_8$.

The three dimension-$8$ Hopf algebras with $|Z_G|=2$ illustrate this. Their representation categories are Tambara--Yamagami categories $\mathrm{TY}(\Z_2\times\Z_2,\chi,\tau_\pm)$, with $\chi$ the hyperbolic bicharacter for $\kk[D_8],\kk[Q_8]$ and the diagonal one for $K_8,(K_8,\Phi_\omega)$. Since $\kk[D_8]$ is gauge equivalent to $(\kk[Q_8],\Phi_\omega)$ \cite{NgSchaunburgCentral} and $\oAut(\mathrm{TY}(A,\chi,\tau))\cong\operatorname{Aut}(A,\chi)$ \cite[Prop.~1]{Tambara2000}, together with $\operatorname{Aut}(\Z_2\times\Z_2,\chi_{\rm hyp})\cong GL_2(\F_2)\cong S_3$ and $\operatorname{Aut}(\Z_2\times\Z_2,\chi_{\rm diag})\cong\Z_2$, we obtain
\[
\ost(\kk[D_8])\,\cong\,\ost(\kk[Q_8])\,\cong\,S_3,\qquad\ost(K_8)\,\cong\,\ost(K_8,\Phi_\omega)\,\cong\, \Z_2.
\]
\end{example}

\subsection{Examples for $A_\delta(n,\xi)$}\label{SS:Adelta-examples}
We apply the central-modification machinery of \S\ref{SS:reduction-general} to the pointed Hopf algebras $A_\delta(n,\xi)$ of \S\ref{S:StGrH}.

\subsubsection{Stable twists of the quasi-Hopf algebra $(A_\delta(md,\xi),\Phi_\omega)$}\label{SS:Adelta-qHopf-twists}
We recall the pointed Hopf algebra $A=A_\delta(md,\xi)$ with $Z=\k[Z_G(A)]=\k\langle g^m\rangle\cong \k[C_d]$. We reserve $\omega$ for a normalized 3-cocycle in $Z^3(C_d,\k^\times)$, with corresponding associator $\Phi_\omega\in Z^{\otimes 3}$.
The standard form of Lemmas~\ref{l:vg}, \ref{bgT}, and \ref{bgT2} carries over verbatim to $(A,\Phi_\omega)$, so every stable twist $(\varphi,J)$ is equivalent to a standard one in the sense of Definition~\ref{d:sf}. As there, when $\delta=0$ the constant term $J^\bo\colon\Z_n\times\Z_n\to\k^\times$ is invertible, while for $\delta=1$ we may assume $a=\sum_{\ou\in\Z_d}a_\ou e_\ou\in \kg{g^m}^\times$. Since $\Phi_\omega$ is central, the stable-twist condition $\varphi(\Phi_\omega)=(\Phi_{\omega})_J$ upgrades Lemma \ref{l:cocycle} to the \emph{twisted} $2$-cocycle equation
\begin{equation}\label{eq:twisted-cocycle-A1-qH}
J^\bo(v,w)\,J^\bo(u, v{+}w)\;=\; C(\ou,\ov,\ow)\,J^\bo(u,v)\,J^\bo(u{+}v, w)\, ,
\end{equation}
where $C(\ou,\ov,\ow):=(\sigma^*_l(\omega)/\omega)(\ou,\ov,\ow)$ for all $u,v,w\in \Z_n$.

We record the normalization lemma that replaces Lemma~\ref{l:cocycle} in the quasi-Hopf setting.

\begin{lem}\label{l:J0omega-A0}
Let $A=A_\delta(md,\xi)$ with $\gcd(m,d)=1$. Set $Z=\k\langle g^m\rangle$, and let $\Phi_\omega\in Z^{\otimes 3}$ be a normalized 3-cocycle. Suppose  $(\varphi,J)$ is a standard stable twist of $(A,\Phi_\omega)$ with $\varphi(g)=g^l$ and $J^\bo\in \k\langle g\rangle^{\otimes 2}$ invertible. Then:
\begin{enumerate}[label=\textup{(\roman*)}]
    \item $\sigma_l^*(\omega)/\omega$ is a $3$-coboundary on $C_d$;
    \item for any normalized $2$-cochain $J^\bo_\omega: Z_d\times Z_d\to \k^\times$ with $d^2(J^\bo_\omega)=\sigma_l^*(\omega)/\omega$, there exists $z\in\kg{g}^\times_1$ such that the gauged stable twist $(\varphi_z,J_z)$ is standard and satisfies $(J_z)^{\bo}=\sum_{\ou,\ov\in \Z_d}J^\bo_\omega(\ou,\ov)\,e_{\ou}\otimes e_{\ov}\in Z^{\otimes 2}$;
    \item when $\delta=0$ after this gauge $\varphi_z(x)=\lambda x$ for some $\lambda\in \k^\times$.
\end{enumerate}
\end{lem}
\begin{proof}
The existence of an invertible $J^\bo$ in \eqref{eq:twisted-cocycle-A1-qH} forces the pullback $\pi^*C$ along $\pi\colon\Z_n\twoheadrightarrow\Z_d$ to be a $3$-coboundary on $\Z_n$. Under the identifications $H^3(C_d,\kk^\times)\cong\Z/d\Z$ and $H^3(C_n,\kk^\times)\cong\Z/n\Z$ of Lemma~\ref{l:canonical-J0-A1}, the inflation map is multiplication by $m^2$ from $\Z/d$ to $\Z/n$. Since $\gcd(d,m)=1$, this inflation map is injective, so $[\pi^*C]=0$ in $H^3(C_n,\kk^\times)$ forces $[C]=0$ in $H^3(C_d,\kk^\times)$.

Now, fix any normalized 2-cochain $J^\bo_\omega: \Z_d\times \Z_d\to \k^\times$ with $d^2(J^\bo_\omega)=C$, pulled back to $\Z_n$ along $\pi$, and set $\tilde J^\bo(u,v):=J^\bo(u,v)/J^\bo_\omega(\bar u,\bar v)$. Since $J^\bo_\omega$ itself satisfies \eqref{eq:twisted-cocycle-A1-qH} with the same defect $C$, the ratio $\tilde J^\bo$ satisfies the \emph{genuine} normalized $2$-cocycle equation on $\Z_n$,
\[
\tilde J^\bo(v,w)\,\tilde J^\bo(u,v+w)\;=\;\tilde J^\bo(u,v)\,\tilde J^\bo(u+v,w).
\]
Since $H^2(\Z_n,\kk^\times)=1$, there exists $h\colon\Z_n\to\kk^\times$ with $h(0)=1$ such that $\tilde J^\bo(u,v)=h(u)h(v)/h(u+v)$. Setting $z=\sum_{u\in\Z_n}h(u)^{-1}e_u\in\kg{g}^\times_1$, the standard gauge calculation (cf.\ the proof of Lemma~\ref{l:cocycle}) yields
\[
(J_z)^{\bo}(u,v)\;=\;J^\bo(u,v)\cdot\frac{h(u+v)}{h(u)h(v)}\;=\;{J^\bo_\omega}(\bar u,\bar v),
\]
so $(J_z)^{\bo}$ depends only on $\bar u,\bar v\in\Z_d$ and equals $J^\bo_\omega\in Z^{\otimes 2}$. Because $z\in\kg{g}^\times$ commutes with $g$, $\varphi_z(g)=\varphi(g)=g^l$ and $(\varphi_z,J_z)$ remains standard by Lemma~\ref{l:vg}.

For the final assertion, write $\varphi_z(x)=xa$ with $a\in\kg{g}$ (Lemma~\ref{l:vg}). Since $\delta=0$, the $1\otimes x$ block of the stable-twist identity $J_z\Delta(\varphi_z(x))=(\varphi_z\otimes\varphi_z)\Delta(x)J_z$ reduces to
\[
(J_z)^{\bo}\,(1\otimes x)\Delta(a)\;=\;(1\otimes xa)\,(J_z)^{\bo}.
\]
Since $(J_z)^{\bo}$ commutes with both $1\otimes xa$ and $\Delta(a)\in\kg{g}^{\otimes 2}$, it cancels from both sides. The resulting equation $(1\otimes x)\Delta(a)=1\otimes xa$ forces $\Delta(a)=1\otimes a$, so $a\in\kk^\times$; setting $\lambda:=a$ gives $\varphi_z(x)=\lambda x$.
\end{proof}

The following result computes the stable twist groups of the pointed Hopf algebras $A_\delta(n,\xi)$ investigated in Section~\ref{S:StGrH} modified by central $3$-cocycle associators. 

\begin{prop}\label{c:A1-corollary}
Let $A = A_\delta(md, \xi)$ with $\gcd(m,d)=1$, and $Z =\kk[Z_G(A)]\cong \kk[C_d]$. Let $\omega \in Z^3(C_d, \kk^\times)$ be a normalized $3$-cocycle, with associator $\Phi_\omega\in Z^{\otimes 3}$. Then
\[
\ost(A, \Phi_\omega) \;\cong\; \ost(A)\, .
\]
\end{prop}
\begin{proof}
In all cases, let $(\varphi,J)$ be a standard stable twist in $\st(A,\Phi_\omega)$. The definition of stable twist shows that $\varphi$ preserves central grouplike elements $Z_G(A)=\langle g^m\rangle\cong C_d$. We will show that $\varphi(g^m)=g^{\pm m}$, which implies that $\varphi=\sigma_{\pm 1}$ acts trivially on $H^3(C_d,\kk^\times)$ by Lemma~\ref{l:canonical-J0-A1}. Therefore, the statement follows from Theorem~\ref{t:main-general}.

When $\delta=0$, we proceed as follows. If $n=m$ (Taft case), then $d=1$, $Z=\kk$, and $\omega$ is trivial. Assume $n>m$, so $d\ge 2$. Apply Lemma~\ref{l:J0omega-A0}: replace $(\varphi,J)$ by an equivalent representative with $\varphi(x)=\lambda x$ for some $\lambda\in\kk^\times$ and $J=J^\bo_\omega+\sum_{1\le i\le m-1}(x^i\otimes x^{m-i})\JC{}{i}$ for some normalized 2-cochain $J^\bo_\omega\in Z^{\otimes 2}$ satisfying $d^2(J^\bo_\omega)=\vp^{\otimes 3}(\Phi_\omega)/\Phi_\omega=C_\vp(\Phi_\omega)$. We claim $J^{i}=0$ for all $1\le i\le m-1$. Similar to the Hopf counterpart in Theorem~\ref{STH0}, we substitute such $J$ into the cocycle equation
\[
C_\vp(\Phi_\omega)(J\otimes 1)(\Delta\otimes 1)(J)\, =\, (1\otimes J)(1\otimes \Delta)(J)\,.
\]
First, extracting the coefficients of $1\otimes x^i\otimes x^{m-i}$ on both sides forces
\[
C_\vp(\Phi_\omega)(J^\bo_\omega\otimes 1)(\Delta \otimes 1)(\JC{}{i})\, =\, (1\otimes \JC{}{i})(1\otimes \Delta)(J^\bo_\omega)\, .
\]
After applying $\id \otimes \ve \otimes \id $ above and using the normalized condition of $J^\bo_\omega$ and $\Phi_\omega$, we get $\JC{}{i}=(1\otimes b_i)J^\bo_\omega$ for some $b_i\in\kk\langle g\rangle$. Then, comparing the coefficients of $x^i\otimes x^{m-i}\otimes 1$ further reduces $b_i=\alpha_i g^{i-m}$ for some $\alpha_i\in\kk$. Finally,  investigating $x^i\otimes 1\otimes x^{m-i}$ yields $\alpha_i(1\otimes g^{i-m}\otimes g^{i-m})=\alpha_i(1\otimes g^i\otimes g^{i-m})$, hence $\alpha_i=0$ since $g^i\ne g^{i-m}$ for $d\ge 2$. Thus $\JC{}{i}=0$ for all $1\le i\le m-1$. Since $J=J^\bo_\omega\in Z^{\otimes 2}$ is central in $A^{\otimes 2}$, it commutes with $\Delta(A)$ on the nose; in particular, $[J,\Delta(g)]=[J,\Delta(x)]=0$. It follows that $\varphi\in\operatorname{Aut}_{\mathrm{Hopf}}(A_0(md,\xi))$, so $\varphi(x)=\lambda x$ must lie in the $(g^l,1)$-skew-primitive space of $A_0(md,\xi)$. This implies that $l=1$ and closes the case $\delta=0$.

Now suppose $\delta=1$, so that $A=A_1(md,\xi)$ with $Z_G(A)=\langle g^m\rangle\cong C_d$. By Lemma~\ref{bgT2} we may assume $(\varphi,J)$ is standard, so $\varphi(g)=g^l$ for some integer $l$ coprime to $n=md$ with $l\equiv 1\pmod m$, and $\varphi(x)=xa$ with $a=\sum_{\ou\in\Z_d}a_\ou e_\ou\in \kg{g^m}^\times$. We will show $l\equiv \pm 1\pmod d$, i.e.,\ $\varphi$ acts on $Z_G(A)\cong C_d$ as $\sigma_{\pm 1}$. 

We make two general observations that control the role of $C$ in \eqref{eq:twisted-cocycle-A1-qH} throughout the analysis:
\begin{itemize}
\item[(a)] Because $C$ is a ratio of values of the unit-valued cocycle $\omega$, it is \emph{nowhere zero}; assertions about \emph{vanishing} of $J^\bo$-coefficients therefore transfer unchanged to $(A,\Phi_\omega)$.
\item[(b)] Because $\omega$ is normalized, $C(\ou,\ov,\ow)=1$ whenever any of $\ou,\ov,\ow$ equals $\ol 0\in\Z_d$. Hence identities derived by substitutions of this boundary type are \emph{numerically unchanged}.
\end{itemize}
Moreover, the compatibility relations \eqref{eq:x2}--\eqref{eq:x4}, equivalently \eqref{eq:z2}--\eqref{eq:z4}, come from $\Delta_J\circ\varphi=(\varphi\otimes\varphi)\circ\Delta$ applied to the skew-primitive $x$. They involve only the coalgebra structure of $A$ and the gauge $J$, not the associator, so are \emph{independent of $\Phi_\omega$} and carry over verbatim. Then, there are two cases to consider regarding $T_J=\{u\in Z_n\,|\, \JC{u,-u}{0}\neq 0\}$. 

\emph{Case $1\notin T_J$ \textup{(}equivalently $J^\bo$ not invertible\textup{)}.} The proof of Lemma~\ref{l:sum} carries over verbatim: its identities are derived from the compatibility relations \eqref{eq:x2}--\eqref{eq:x4} together with cocycle relations on $J^\bo$ that involve only boundary substitutions \textup{(}observation~(b) gives $C\equiv 1$ there\textup{)}. Thus Lemma~\ref{l:sum} produces $z\in \A$ with $\varphi_z=\id$. In particular $l=1$ here, so $l\equiv\pm 1\pmod d$ holds trivially.

\emph{Case $1\in T_J$ \textup{(}equivalently $J^\bo$ invertible\textup{)}.} Apply Lemma~\ref{l:J0omega-A0}\textup{(}i\textup{)}--\textup{(}ii\textup{)}: gauge $(\varphi,J)$ to a representative with $(J)^\bo=J^\bo_\omega\in Z^{\otimes 2}$ for a chosen normalized $2$-cochain $J^\bo_\omega$ with $d^2(J^\bo_\omega)=\sigma_l^*(\omega)/\omega$. As a consequence, all formulas \eqref{eq:Z01}, \eqref{eq:qJrelation}, \eqref{e7}, \eqref{e8} of \S\ref{S:StGrH} hold when evaluated at boundary substitutions ($\bar u=\ol 0$, $\bar v=\ol 0$, or $\bar w=\ol 0$), since $J^\bo_\omega$ then reduces to $1$ on the relevant slot.

With these identities in hand, set $\ou=\ow=\ol 0$ and $\ov\ne\ol 0$ with $i=m-1$ in \eqref{e7}. Since $C(\ol 0,\ov,\ol 0)=1$ by observation~(b), the resulting equation
\[
\ZC{0, v}{m-1}\bigl(1+\ZC{v,0}{m-1}\g_\ov^m\bigr) =  \ZC{v, 0}{m-1}\Bigl(1+\sum_{j=0}^{m-2}\qqbin{m-1}{j}{\xi}\,\ZC{0,v}{m-1-j}\xi^{(m-1-j)^2}\g_\ov^m\Bigr)
\]
is identical to the one displayed in the proof of Theorem~\ref{inverT}(i). Substituting the explicit formulae for $\ZC{0,v}{i}$ and $\ZC{v,0}{i}$ from \eqref{eq:Z01} together with \eqref{eq:qJrelation} and simplifying (a polynomial manipulation in $a_{\ol 0}/a_\ov$ that does not invoke $\omega$) yields
\begin{equation}\label{eq:as-qH}
\Bigl(\frac{\ao}{\av}\Bigr)^{\!2}\;=\;\frac{1}{\theta^{(l-1)v}}\qquad\text{for all }\ov\in\Z_d\setminus\{\ol 0\}.
\end{equation}
Identically to \eqref{eq:as}, this is the master constraint on $l$. Likewise, the boundary identity \eqref{e8} with $\ou=\ov$ gives 
\begin{equation}\label{eq:as-qH2}
\Bigl(\frac{\ao}{\av}\Bigr)^m=\frac{1+\theta^{vm}}{1+\theta^{lvm}}\, .
\end{equation}
This identity involves only the internal $n$-th root of unity $\theta$ (recall $\xi=\theta^d$) and is unaffected by the associator $\Phi_\omega$. 
\begin{itemize}
\item If $m$ is even, setting $\ov=\ol 1$, taking $m/2$-th-power of the square identity \eqref{eq:as-qH} and inserting into \eqref{eq:as-qH2} gives
\[
\bigl(1-\theta^{(m/2)(l-1)}\bigr)\bigl(1-\theta^{(m/2)(l+1)}\bigr)=0,
\]
hence $l\equiv\pm 1\pmod{2d}$, and in particular $l\equiv\pm 1\pmod d$.
\item If $m$ is odd, a similar substitution yields \eqref{eq:as2} and further combined with \eqref{eq:as-qH} again gives
\[
\bigl(1-\theta^{m(l-1)}\bigr)\bigl(1-\theta^{m(l+1)}\bigr)=0,
\]
so $l\equiv\pm 1\pmod d$.
\end{itemize}
This is precisely the conclusion of Theorem~\ref{inverT}(i), now established for the quasi-Hopf algebra $(A,\Phi_\omega)$. Hence $\varphi(g^m)=g^{ml}=g^{\pm m}$ on $Z_G(A)$, i.e., $\varphi$ restricts to $\sigma_{\pm 1}$ on $C_d$.
\end{proof}

\subsubsection{The quasi-Hopf $C_2$-equivariantization $\protect\N^\omega(md,\xi)$}\label{SS:Nomega-family}

Assume $\gcd(m,d)=1$ and $d$ is odd. We construct the quasi-Hopf algebra $\N^\omega(md,\xi)$, the unique nontrivial $C_2$-equivariantization of the tensor category $\C{(A_1(md,\xi),\Phi_\omega)}$. The construction generalizes two threads now carried out in the quasi-Hopf setting: the general cyclic-equivariantization scheme of \S\ref{S:Equiv}, and the explicit computation behind the Hopf family $\N(md,\xi)$ of Definition~\ref{def:N-family}, particularly Lemma~\ref{l:C2action}. Throughout, $A=A_1(md,\xi)$ and $Z=\kk\langle g^m\rangle\cong\kk[C_d]$; the data $l,r,s,\tilde\varphi,\tilde J,z,y$ is as in Definition~\ref{def:N-family}.

\medskip\noindent\textit{Step 1: the canonical $2$-cochain $J^\bo_\omega$ trivializing $\sigma_l^*(\omega)/\omega$.} By Proposition~\ref{c:A1-corollary}, $\tilde\varphi$ restricts to $\sigma_l$ on $Z_G(A)\cong C_d$ with $l\equiv-1\pmod d$. By Lemma~\ref{l:canonical-J0-A1}, $[\sigma_l^*\omega]=[\omega]^{l^2}=[\omega]$ in $H^3(C_d,\kk^\times)$, so the $3$-cocycle $\sigma_l^*(\omega)/\omega$ is a coboundary. Fix a normalized $\kk^\times$-valued $2$-cochain
\[
J^\bo_\omega\;\in\;Z^{\otimes 2}=\kk\langle g^m\rangle^{\otimes 2}\qquad\text{with}\qquad d^2(J^\bo_\omega)\;=\;\frac{\sigma_l^*(\omega)}{\omega}\,.
\]
The choice is unique up to multiplication by a $2$-coboundary $\partial(y_0)$ with $y_0\in Z^\times_1$, which corresponds to a gauge transformation of the resulting extension (Corollary~\ref{cor:anotEG}(i)), so $\N^\omega(md,\xi)$ is well-defined up to equivalence; the canonical such $J^\bo_\omega$ is the cochain $J_0$ of Example~\ref{ex:quasi-Hopf-pointed} (eq.~\eqref{E:J0-cobdy}) specialized to $C_d=\langle g^m\rangle$.

\medskip\noindent\textit{Step 2: the quasi-Hopf stable twist $(\tilde\varphi,\tilde J^\omega)\in\st(A,\Phi_\omega)$.} Set
\[
\tilde J^\omega \;:=\; J^\bo_\omega\cdot\tilde J\;\in\; A^{\otimes 2}.
\]
We verify $(\tilde\varphi,\tilde J^\omega)$ is a stable twist of $(A,\Phi_\omega)$. Centrality of $J^\bo_\omega\in Z^{\otimes 2}$ and $d^2(\tilde J)=1$ (since $\tilde J$ is a Hopf $2$-cocycle) yield
\[
d^2(\tilde J^\omega)\;=\;d^2(J^\bo_\omega)\cdot d^2(\tilde J)\;=\;\sigma_l^*(\omega)/\omega.
\]
Since $\Phi_\omega\in Z^{\otimes 3}$ is central and commutes with $\tilde J^\omega$, this gives
\[
(\Phi_\omega)_{\tilde J^\omega}\;=\;d^2(\tilde J^\omega)\,\Phi_\omega\;=\;\sigma_l^*(\omega)/\omega\cdot\Phi_\omega\;=\;\sigma_l^*(\Phi_\omega)\;=\;\tilde\varphi(\Phi_\omega).
\]
So the associator condition $\tilde\varphi(\Phi_\omega)=(\Phi_\omega)_{\tilde J^\omega}$ holds; combined with the fact that $\tilde\varphi\colon A\to A_{\tilde J}$ is a bialgebra isomorphism (Theorem~\ref{inverT}(iii)) and that $\Delta_{\tilde J^\omega}=\Delta_{\tilde J}$ (since $\Delta_{J^\bo_\omega}=\Delta$ by centrality), $(\tilde\varphi,\tilde J^\omega)$ is a stable twist of $(A,\Phi_\omega)$.

\medskip\noindent\textit{Step 3: the squaring computation --- quasi-Hopf analogue of Lemma~\ref{l:C2action}.} Choose $\delta_\omega\in Z^\times_1$ satisfying \emph{both}
\[
\partial(\delta_\omega)\;=\;(\delta_\omega\otimes\delta_\omega)\Delta(\delta_\omega)^{-1}\;=\;\sigma_l^*(J^\bo_\omega)\cdot J^\bo_\omega
\qquad\text{and}\qquad \sigma_l^*(\delta_\omega)=\delta_\omega.
\]
Such a $\delta_\omega$ exists. First, the right-hand side is a $2$-cocycle: $d^2\bigl(\sigma_l^*(J^\bo_\omega)\cdot J^\bo_\omega\bigr)=\sigma_l^*(d^2(J^\bo_\omega))\cdot d^2(J^\bo_\omega)=\bigl(\omega/\sigma_l^*(\omega)\bigr)\bigl(\sigma_l^*(\omega)/\omega\bigr)=1$ (using $\sigma_l^2=\id$), hence a $2$-coboundary since $H^2(C_d,\kk^\times)=0$ ($d$ odd); pick any $\delta_\omega$ with this coboundary. The valid choices form a torsor under the characters $\chi\in\widehat{C_d}=Z_G(A)$, and since $\sigma_l=\sigma_{-1}$ gives $\sigma_l^*(\chi)=\chi^{-1}$, replacing $\delta_\omega$ by $\delta_\omega\chi$ with $\chi^2=\sigma_l^*(\delta_\omega)\,\delta_\omega^{-1}$ --- possible because squaring is surjective on $\widehat{C_d}\cong C_d$ for $d$ odd --- enforces $\sigma_l^*(\delta_\omega)=\delta_\omega$. 
Define
\[
z^\omega\;:=\;z\cdot\delta_\omega^{-1}\;\in\;A^\times_1.
\]
Since $\delta_\omega\in Z$ is central in $A$, multiplying $z$ by $\delta_\omega^{-1}$ does not change the induced inner automorphism: $I_{z^\omega}=I_z$.
We claim
\begin{equation}\label{eq:Nomega-square}
(\tilde\varphi,\tilde J^\omega)^2 \;=\; (I_{z^\omega},\,\partial_{z^\omega})^{-1}\qquad\text{and}\qquad \tilde\varphi(z^\omega)\;=\;z^\omega.
\end{equation}

\emph{Verifying \eqref{eq:Nomega-square}.} The square is
\[
(\tilde\varphi,\tilde J^\omega)^2 \;=\; \bigl(\tilde\varphi^2,\;\tilde\varphi(\tilde J^\omega)\,\tilde J^\omega\bigr) \;=\; \bigl(\tilde\varphi^2,\;\tilde\varphi(J^\bo_\omega)\,\tilde\varphi(\tilde J)\,J^\bo_\omega\,\tilde J\bigr)=\bigl(\tilde\varphi^2,\;\sigma_l^*(J^\bo_\omega)\cdot J^\bo_\omega\cdot \tilde\varphi(\tilde J)\,\tilde J\bigr).
\]
By Lemma~\ref{l:C2action}, $\tilde\varphi^2=I_{z^{-1}}$ and $\tilde\varphi(\tilde J)\,\tilde J=\partial_{z^{-1}}$; substituting these and using $\sigma_l^*(J^\bo_\omega)\cdot J^\bo_\omega=\partial(\delta_\omega)$ gives
\[
(\tilde\varphi,\tilde J^\omega)^2\;=\;\bigl(I_{z^{-1}},\,\partial(\delta_\omega)\cdot\partial_{z^{-1}}\bigr)\;=\;\bigl(I_{z^{-1}},\,\partial(\delta_\omega\,z^{-1})\bigr)\;=\;\bigl(I_{(z^\omega)^{-1}},\,\partial_{(z^\omega)^{-1}}\bigr).
\]
This verifies the first half of \eqref{eq:Nomega-square}. Since $\tilde\varphi$ restricts to $\sigma_l$ on $Z$, the invariance $\sigma_l^*(\delta_\omega)=\delta_\omega$ built into the choice of $\delta_\omega$ in Step~3 gives $\tilde\varphi(\delta_\omega)=\delta_\omega$. Combined with $\tilde\varphi(z)=z$, this yields
\[
\tilde\varphi(z^\omega)=\tilde\varphi(z)\,\tilde\varphi(\delta_\omega)^{-1}=z\,\delta_\omega^{-1}=z^\omega.
\]

\medskip\noindent\textit{Step 4: the distinguished elements $\alpha_{\tilde J^\omega}$, $\beta_{\tilde J^\omega}$, $\nu^\omega$.} Modifying the associator to $\Phi_\omega$ changes the distinguished elements: by the formula of Example~\ref{ex:quasi-Hopf-pointed}, applied to the central group $C_d=\langle g^m\rangle$, the quasi-Hopf algebra $(A,\Phi_\omega)$ has
\[
\alpha=1,\qquad \beta_\omega:=\sum_{\ou\in\Z_d}\omega(g^{mu},g^{-mu},g^{mu})^{-1}\,e_\ou\in Z,
\]
where $e_\ou\in Z$ is the central idempotent with $g^m e_\ou=\theta^{mu}e_\ou$. These must be carried through formula~\eqref{ABJ}: writing $\tilde J^\omega=\sum_i f_i\otimes g_i$ and $(\tilde J^\omega)^{-1}=\sum_i\ol f_i\otimes\ol g_i$,
\[
\alpha_{\tilde J^\omega}=\sum_i S(\ol f_i)\,\alpha\,\ol g_i=\sum_i S(\ol f_i)\,\ol g_i,\qquad \beta_{\tilde J^\omega}=\sum_i f_i\,\beta_\omega\,S(g_i).
\]
Since $J^\bo_\omega\in Z^{\otimes2}$ is central and $\tilde J^\omega=J^\bo_\omega\tilde J$, these factor through the Hopf distinguished elements $\alpha_{\tilde J}=\frac{1}{m}g^{-r(m-1)}\prod_{j=1}^{m-1}(g^r-\xi^j)$ and $\beta_{\tilde J}=\alpha_{\tilde J}^{-1}$ of Lemma~\ref{l:C2action}. Writing $J^\bo_\omega=\sum_i r_i\otimes s_i$ and $(J^\bo_\omega)^{-1}=\sum_i\ol r_i\otimes\ol s_i$, the $\alpha$-factor is built from $(J^\bo_\omega)^{-1}$ and the $\beta$-factor from $J^\bo_\omega$, in accordance with~\eqref{ABJ}:
\[
\alpha_{\tilde J^\omega}=\Bigl(\textstyle\sum_i S(\ol r_i)\,\ol s_i\Bigr)\,\alpha_{\tilde J},\qquad \beta_{\tilde J^\omega}=\beta_\omega\,\Bigl(\textstyle\sum_i r_i\,S(s_i)\Bigr)\,\beta_{\tilde J}.
\]
Finally, $\nu^\omega\in A^\times_1$ is determined by uniqueness \eqref{nuS} applied to $(\tilde\varphi,\tilde J^\omega)$ on $(A,\Phi_\omega)$:
\[
\tilde\varphi\circ S\;=\;I_{\nu^\omega}\circ S\circ\tilde\varphi\,,\qquad \tilde\varphi(1)\;=\;\nu^\omega\cdot\alpha_{\tilde J^\omega}\,,\qquad \tilde\varphi(\beta_\omega)\;=\;\beta_{\tilde J^\omega}\cdot(\nu^\omega)^{-1}.
\]
Since $\tilde\varphi(\beta_\omega)=\beta_\omega^{-1}$, the above relations give $\nu^\omega=\alpha_{\tilde J^\omega}^{-1}=\beta_{\tilde J^\omega}\,\beta_\omega$.

\medskip\noindent\textit{Step 5: the construction $\N^\omega(md,\xi)$.}
\begin{deff}\label{def:N-omega-family}
With $d>1$ odd, $m>1$ coprime to $d$, $\xi$ a primitive $m$-th root of unity, $\omega\in Z^3(C_d,\kk^\times)$, and $J^\bo_\omega,\,\tilde J^\omega=J^\bo_\omega\tilde J,\,z^\omega=z\delta_\omega^{-1}$ as in Steps 1--3 above, define the quasi-Hopf algebra
\[
\N^\omega(md,\xi)\;:=\;(A,\Phi_\omega)[u^{\pm 1};\tilde\varphi,\tilde J^\omega,z^\omega,2]
\]
via the general $C_2$-equivariantization construction of \S\ref{S:Equiv} (preceding Corollary~\ref{cor:anotEG}). Explicitly, $\N^\omega(md,\xi)$ is the $\kk$-algebra generated by $g,x,u$ subject to the relations
\begin{align*}
 x^m&=1-g^m, & g^{md}&=1, &  gx&=\xi xg,\\
 gu&=ug^{l},& xu&=\sqrt[m]{-1}\, uxg^{-r},& u^2&=z^\omega,
\end{align*}
with $l,r$ as in Definition~\ref{def:N-family}. Its coalgebra structure is
\begin{gather*}
\Delta(g)=g\otimes g,\quad \Delta(x)=x\otimes g+1\otimes x,\\
\Delta(u)=(u\otimes u)\,\tilde J^\omega\;=\;(u\otimes u)\,J^\bo_\omega\,\tilde J,\\
\ve(g)=1,\quad \ve(x)=0,\quad \ve(u)=1,
\end{gather*}
where $\tilde J=1\otimes 1-\sum_{i=1}^{m-1}(x^iy\otimes x^{m-i}g^{(1-r)i}y)\bigl(\prod_{j=i+1}^{m-1}(g^r-\xi^j)\otimes\prod_{j=m-i+1}^{m-1}(g^r-\xi^j)\bigr)$ and $y=\sum_{k=0}^{r/m-1}g^{mk}+\frac{1-r}{md}\sum_{k=0}^{d-1}g^{mk}$. Its associator is $\Phi_\omega\in Z^{\otimes 3}\subseteq\N^\omega(md,\xi)^{\otimes 3}$; its distinguished elements are inherited from $(A,\Phi_\omega)$, namely $\alpha_{\N^\omega}=1$ and $\beta_{\N^\omega}=\beta_\omega$; and its antipode is
\[
S(g)=g^{-1},\quad S(x)=-xg^{-1},\quad S(u)=(\nu^\omega)^{-1}u^{-1},
\]
where $\nu^\omega$ is the unique element determined in Step 4.
\end{deff}

\begin{remark}\label{rem:Nomega-Hopf-specialization}
When $[\omega]=1$, we may take $\omega=1\in Z^3(C_d,\kk^\times)$, $J^\bo_\omega=1\otimes 1$, hence $\tilde J^\omega=\tilde J$, $\delta_\omega=1$, $z^\omega=z$, $\beta_\omega=1$, $\nu^\omega=\beta_{\tilde J}=\alpha_{\tilde J}^{-1}$, and $\Phi_\omega=1\otimes 1\otimes 1$. Then $\N^\omega(md,\xi)$ specializes to the genuine Hopf algebra $\N(md,\xi)$ of Definition~\ref{def:N-family}. For $[\omega]\ne 1\in H^3(C_d,\kk^\times)\cong\Z/d$, $\N^\omega(md,\xi)$ is genuinely quasi-Hopf: its associator $\Phi_\omega$ is not gauge-equivalent to $1\otimes 1\otimes 1$.
\end{remark}

\begin{thm}\label{th:Nomega-extC2}
Let $m,d\ge 2$ be coprime with $d$ odd, $\xi$ a primitive $m$-th root of unity, and $\omega\in Z^3(C_d,\kk^\times)$ normalized. Then:
\begin{enumerate}
\item[(i)] $\ost(A_1(md,\xi),\Phi_\omega)\cong\ost(A_1(md,\xi))\cong C_2$, generated by $[(\tilde\varphi,\tilde J^\omega)]$.
\item[(ii)] Up to equivalence, $\N^\omega(md,\xi)$ is the unique nontrivial cocentral cleft quasi-Hopf extension of $\kk[C_2]$ by $(A_1(md,\xi),\Phi_\omega)$, and the unique nontrivial $C_2$-equivariantization of $\C{(A,\Phi_\omega)}$.
\item[(iii)] $\dim_\kk\N^\omega(md,\xi)=2m^2d$, and $\N^{\omega=1}(md,\xi)\cong\N(md,\xi)$ as Hopf algebras.
\end{enumerate}
\end{thm}
\begin{proof}
Part~(i) is Proposition~\ref{c:A1-corollary} together with Theorem~\ref{StableTwTaft2d}, with the explicit generator from Steps 1--2. For (ii), $\gcd(|C_2|,|Z_G(A)|)=\gcd(2,d)=1$ since $d$ is odd, so $H^i(C_2,Z_G(A))=0$ for all $i\ge 1$. Theorem~\ref{CCEGE} (the quasi-Hopf form of Theorem~\ref{HopfCEGC2}) gives a bijection between nontrivial cocentral cleft quasi-Hopf extensions of $\kk[C_2]$ by $(A,\Phi_\omega)$ and order-$2$ elements of $\ost(A,\Phi_\omega)\cong C_2$; the unique such class is $[(\tilde\varphi,\tilde J^\omega)]$, whose corresponding extension is $\N^\omega(md,\xi)$ by Corollary~\ref{cor:anotEG} together with the squaring relation~\eqref{eq:Nomega-square}. Part~(iii): $\dim\N^\omega=|C_2|\cdot\dim A=2m^2d$, and Remark~\ref{rem:Nomega-Hopf-specialization} gives the specialization.
\end{proof}


\subsection{Braided central modification for QT quasi-Hopf algebras}\label{SS:braided-reduction}

We now develop the quasitriangular counterpart of the central-modification theorem of \S\ref{SS:reduction-general}. Starting from a QT quasi-Hopf algebra $(A,\Phi,R)$ with a central sub-bialgebra $Z=\kk^G\subseteq Z(A)$, we modify \emph{both} the associator $\Phi$ and the $R$-matrix $R$ by central cocycle data $(\phi_\omega,r_c)$ drawn from $Z$, and show that the braided stable twist group is unchanged under hypotheses parallel to those of Theorem~\ref{t:reduction-general}. Throughout this subsection, $G$ denotes a finite \emph{abelian} group and $Z=\kk^G$.

We recall the classification of braidings on pointed fusion categories, with coefficients in $\kk^\times$ written multiplicatively. A function $q\colon G\to\kk^\times$ is a \emph{quadratic form} if $q(a^{-1})=q(a)$ for all $a\in G$ and $b_q(a_1,a_2):=q(a_1a_2)\,q(a_1)^{-1}\,q(a_2)^{-1}$ is bimultiplicative; we write $\operatorname{Quad}(G,\kk^\times)$ for the group of quadratic forms.

A \emph{$2$-abelian $3$-cocycle} on $G$ is a pair $(\omega,c)$ with $\omega\in Z^3(G,\kk^\times)$ a normalized $3$-cocycle and $c\colon G\times G\to\kk^\times$ a function satisfying
\begin{align}
\omega(y,z,x)\,c(x,yz)\,\omega(x,y,z)&=c(x,z)\,\omega(y,x,z)\,c(x,y),\label{eq:2ab1}\\
\omega(z,x,y)^{-1}\,c(xy,z)\,\omega(x,y,z)^{-1}&=c(x,z)\,\omega(x,z,y)^{-1}\,c(y,z);\label{eq:2ab2}
\end{align}
it is a \emph{$3$-coboundary} if $\omega=\partial^3(f)$ and $c(a_1,a_2)=f(a_2,a_1)\,f(a_1,a_2)^{-1}$ for some $f\colon G\times G\to\kk^\times$. Writing $H^3_{2\text{-ab}}(G,\kk^\times)$ for the group of $2$-abelian $3$-cocycles modulo coboundaries, the \emph{trace map} $\operatorname{Tr}([(\omega,c)])(x)=c(x,x)$ is an isomorphism $H^3_{2\text{-ab}}(G,\kk^\times)\xrightarrow{\sim}\operatorname{Quad}(G,\kk^\times)$ \textup{(}the Eilenberg--Mac~Lane theorem \cite[Theorem~12]{EM53}; see \cite{Galindo2025} for a recent elementary proof\textup{)}.

To a $2$-abelian $3$-cocycle $(\omega,c)$ on $G$ with values in $\kk^\times$ we associate the central modification data $(\phi_\omega,r_c)$, where
\[
\phi_\omega\;:=\;\sum_{x,y,z\in G}\omega(x,y,z)\,\delta_x\o\delta_y\o\delta_z\in Z^{\o 3},\qquad r_c:=\sum_{g,h\in G}c(g,h)\,\delta_g\o\delta_h\in Z^{\o 2}.
\]
We use $(\phi_\omega,r_c)$ to modify $(A,\Phi,R)$, setting
\[
\Phi_\omega\;:=\;\Phi\cdot\phi_\omega\in A^{\o 3},\qquad R_c:=R\cdot r_c\in A^{\o 2}.
\]

\begin{prop}\label{p:central-QT}
The tuple $(A,\Phi_\omega,R_c)$ is a QT quasi-Hopf algebra.
\end{prop}
\begin{proof}
$(A,\Phi_\omega)$ is quasi-Hopf by Proposition~\ref{p:quasi-Hopf-general}. Since all $Z$-factors are central, the QT axioms \eqref{eq:QT1}--\eqref{eq:QT3} for $R_c=R\cdot r_c$ relative to $\Phi_\omega=\Phi\cdot\phi_\omega$ split into a $(\Phi,R)$-part, which holds as $(A,\Phi,R)$ is QT, and a $(\phi_\omega,r_c)$-part, which is the QT quasi-Hopf structure on $\kk^G$ of \cite[Example~10.6, Proposition~10.1]{BCPVO}.
\end{proof}

Let $(\vp,J)\in \st^Z(A,\Phi)$; then $\vp$ induces an automorphism $\sigma\in\operatorname{Aut}(G)$. The \emph{cocycle ratio} on the central modification data $(\phi_\omega,r_c)$ now has two components: the associator ratio $C_\vp=\vp^{\o 3}(\phi_\omega)\phi_\omega\inv\in Z^{\o 3}$ (as in \S\ref{SS:cocycle-general}) and the \emph{R-matrix ratio} $D_\vp:=\vp^{\o 2}(r_c)(r_c)\inv\in Z^{\o 2}$. In terms of group cochains, $C_\vp=\sigma^*(\omega)/\omega$ and $D_\vp(\delta_g\o\delta_h)=c(\sigma(g),\sigma(h))/c(g,h)$.

A \emph{braided primitive} for $(\vp,(\omega,c))$ is a function $k\colon G\times G\to\kk^\times$ such that
\begin{equation}\label{eq:braided-primitive}
\partial(k)=C_\vp=\sigma^*(\omega)/\omega,\qquad \operatorname{Alt}(k):=k(\cdot,\cdot)^{\rm flip}/k(\cdot,\cdot)=D_\vp=\sigma^*(c)/c.
\end{equation}
A braided primitive exists if and only if $\sigma^*[(\omega,c)]=[(\omega,c)]$ in $H^3_{2\text{-ab}}(G,\kk^\times)$, i.e., $q\circ\sigma=q$ where $q=\operatorname{Tr}[(\omega,c)]$.

Paralleling \S\ref{SS:reduction-general}, set \[\OutTwBr(A,\Phi_\omega,R_c)\, :=\, \operatorname{Im}\bigl(\Pi_A\colon\ost^{\rm br}(A,\Phi_\omega,R_c)\to\operatorname{Out}_{\rm alg}(A)\bigr)\, ,\] and write $\ost^{{\rm br},Z}(A,\Phi_\omega,R_c)$ and $\OutTwBrZ(A,\Phi_\omega,R_c)\subseteq\OutTwBr(A,\Phi_\omega,R_c)$ for the subgroups arising from braided stable twists $(\vp,J)$ with $\vp(Z)=Z$. 

\begin{deff}\label{d:BCT}
Let $\overline{\varphi}\in\OutTwBrZ(A,\Phi_\omega,R_c)$, inducing $\sigma\in\operatorname{Aut}(G)$. We say $\overline{\varphi}$ \emph{acts trivially on $\operatorname{Quad}(G,\kk^\times)\cong H^3_{2\text{-ab}}(G,\kk^\times)$} if the induced action $\sigma^*$ is the identity; equivalently $q\circ\sigma=q$ for every quadratic form $q\colon G\to\kk^\times$, equivalently a braided primitive \eqref{eq:braided-primitive} for $(\vp,(\omega',c'))$ exists for every $2$-abelian $3$-cocycle $(\omega',c')$.
\end{deff}

Since the forgetful map $H^3_{2\text{-ab}}(G,\kk^\times)\to H^3(G,\kk^\times)$, $[(\omega,c)]\mapsto [\omega]$, is $\operatorname{Aut}(G)$-equivariant, any $(\vp,J)$ acting trivially on $H^3_{2\text{-ab}}(G,\kk^\times)$ acts trivially on $H^3(G,\kk^\times)$ (Definition~\ref{d:CT}).

\begin{thm}\label{t:braided-reduction}
Let $(A,\Phi,R)$ be a finite-dimensional QT quasi-Hopf algebra with a central sub-bialgebra $Z=\kk^G\subseteq Z(A)$, $G$ abelian, and let $(\omega,c)$ be a $2$-abelian $3$-cocycle on $G$. Assume:
\begin{itemize}
\item[\textup{(i)}] $H^2(G,\kk^\times)=1$;
\item[\textup{(ii)}] $\OutTwBrZ(A,\Phi,R)=\OutTwBr(A,\Phi,R)$, and similarly for $(A,\Phi_\omega,R_c)$;
\item[\textup{(iii)}] both $\OutTwBr(A,\Phi,R)$ and $\OutTwBr(A,\Phi_\omega,R_c)$ act trivially on $\operatorname{Quad}(G,\kk^\times)$.
\end{itemize}
Then there is a group isomorphism
\[
\ost^{\rm br}(A,\Phi_\omega,R_c)\;\xrightarrow{\,\sim\,}\;\ost^{\rm br}(A,\Phi,R)\,.
\]
\end{thm}
\begin{proof}
The argument parallels Theorem~\ref{t:reduction-general}, now tracking the $R$-matrix as well. For $(\vp,J)\in\st^{\rm br}(A,\Phi_\omega,R_c)$, let $k$ be a braided primitive~\eqref{eq:braided-primitive} (which exists by~(iii)) and set $J^0:=\sum_{g,h\in G}k(g,h)\,\delta_g\o\delta_h\in Z^{\o 2}$ and $J':=(J^0)^{-1}J$. Since $(\vp,J)$ is braided for $R_c$ and $J^0,r_c\in Z^{\o 2}$ are central, using $J=J^0J'$, $R_c=R\,r_c$, $\vp^{\o 2}(r_c)=D_\vp\,r_c$, and $\operatorname{Alt}(k)=J^0_{21}(J^0)^{-1}=D_\vp$:
\[
D_\vp\,r_c\cdot(\vp\o\vp)(R)\;=\;(\vp\o\vp)(R_c)\;=\;J_{21}R_cJ^{-1}\;=\;D_\vp\,r_c\cdot J'_{21}RJ'^{-1}.
\]
Cancelling the central invertible factor $D_\vp\,r_c$ gives $(\vp\o\vp)(R)=J'_{21}RJ'^{-1}$, so $(\vp,J')\in\st^{\rm br}(A,\Phi,R)$.
\end{proof}

\subsection{Soft symmetries and the Brauer--Picard group}\label{SS:soft}

We close by recording how the braided exact sequence of \S\ref{SS:braided-SES} connects to two notions already studied in the literature: the soft symmetries of topological orders and the Brauer--Picard group.

\subsubsection{Soft symmetries}
The emergent symmetry group of a $(2{+}1)$D topological order modelled by an MTC $\CC$ is $\oAut^{\rm br}(\CC)$; a \emph{soft symmetry} \cite{SoftSym2025} is a braided autoequivalence fixing every simple object \textup{(}anyon\textup{)} yet not naturally isomorphic to the identity. By Corollary~\ref{c:peqv-br} and Theorem~\ref{t:braided-SES}, the label-preserving braided autoequivalences of $\C{A}$ are exactly the kernel $H^2_{\ell,{\rm br}}(A^*)$, realized by invariant twists $(\id,K)$ with $K_{21}RK^{-1}=R$ \textup{(}equivalently, quantum-commutative bi-Galois objects \cite{ZhangZhu2015}\textup{)}. For the Drinfeld double this kernel was computed by Davydov.

\begin{thm}[{\cite[Cor.~2.14]{Davydov2014}}]\label{t:B0=lazy-br}
For a finite group $G$ and the Drinfeld double $D(G)$ with standard $R$-matrix $R=\sum_{g}(p_g\o 1)\o(1\o g)$,
\[
H^2_{\ell,{\rm br}}(D(G)^*)\;\cong\;B_0(G):=\ker\Bigl(H^2(G,\kk^\times)\to\textstyle\prod_{A\le G\,{\rm ab.}}H^2(A,\kk^\times)\Bigr),
\]
the Bogomolov multiplier. Concretely, $J=\sum_{g,h}\sigma(g,h)\,g\o h\in\kk[G]^{\o 2}\subseteq D(G)^{\o 2}$ represents a class in the kernel if and only if $\sigma\in Z^2(G,\kk^\times)$ restricts to a symmetric cocycle on every abelian subgroup of $G$.
\end{thm}

In our language the soft symmetries of $\C{D(G)}\simeq Z(\operatorname{Vec}_G)$ are the invariant twists indexed by $B_0(G)$; Davydov's decomposition of $\oAut^{\rm br}(Z(\operatorname{Vec}_G))$ in terms of $B_0(G)$ and the double class-preserving automorphisms \cite[Thm.~2.12, Cor.~2.13]{Davydov2014}, and the vanishing of $B_0(G)$ for $|G|<64$, for symmetric groups, and for simple groups of Lie type in odd characteristic \cite{Bogomolov1988,Kunyavskii2010}, transfer verbatim.

\subsubsection{The Brauer--Picard group}
For a finite-dimensional quasi-Hopf algebra $A$, the double $D(A)$ is a QT quasi-Hopf algebra with $\C{D(A)}\simeq Z(\C{A})$ \cite[Ch.~8]{BCPVO}, and the Etingof--Nikshych--Ostrik isomorphism $\operatorname{BrPic}(\CC)\cong\oAut^{\rm br}(Z(\CC))$ \cite{ENO2010} applies.

\begin{prop}\label{p:Tw-BrPic}
For any finite-dimensional quasi-Hopf algebra $A$,
\[
\ost^{\rm br}(D(A),R)\;\cong\;\operatorname{BrPic}(\C{A}).
\]
\end{prop}

Thus the braided stable twist group computes the Brauer--Picard group. For $A=\kk[G]$, Theorem~\ref{t:B0=lazy-br} identifies the kernel of $\operatorname{BrPic}(\C{\kk[G]})\twoheadrightarrow\OutTwBr(D(G),R)$ with $B_0(G)$, recovering the orthogonal-group description for abelian $G$ \cite{NaiduNikshych2008} and the Brauer--Picard decomposition of \cite{LentnerPriel2017}.

\section{Nonsemisimple Hopf algebras of dimension $8p$}\label{S:clas}
We work over an algebraically closed field $\k$ of characteristic zero and apply our earlier results to the classification of nonsemisimple Hopf algebras of dimension $8p$, $p$ an odd prime. By Example~\ref{ex:8d}, for every odd integer $d\ge 3$ the Hopf algebra $\N(2d,-1)$ of dimension $8d$ is a nontrivial $C_2$-extension of the pointed Hopf algebra $A_1(2d,-1)$; we conjecture that $\N(2p,-1)$ and its dual are the only $8p$-dimensional Hopf algebras lacking both the Chevalley and dual Chevalley properties.

A finite tensor category $\CC$ has the \emph{Chevalley property} if the full subcategory $\CC_{\rm ss}\subseteq\CC$ of semisimple objects is closed under the tensor product. For $\CC=\C{H}$ with $H$ a finite-dimensional Hopf algebra, this is equivalent to the coradical $H_0$ being a Hopf subalgebra, and we then simply say $H$ has the Chevalley property. We say $H$ has the \emph{dual Chevalley property} if $H^*$ has the Chevalley property, equivalently, its Jacobson radical $J_H$ is a Hopf ideal.

The next subsection collects the known $8p$-dimensional nonsemisimple Hopf algebras with the Chevalley property that are non-pointed.

\subsection{Nonsemisimple Hopf algebras of dimension $8p$ with Chevalley property}
The $8p$-dimensional nonsemisimple Hopf algebras $H$ with the Chevalley property, which are neither pointed nor copointed, were classified in \cite[Thm. 3.15]{BGN}. These are two completely distinct families, denoted by $\HH_{8p}$ and $R_{-1}\# A_{4p}$, which are described in detail in the following examples.

\begin{example} The Hopf algebra $\HH_{8p}$ was constructed in \cite{CDMM},  and its coradical was proved to be isomorphic to $\k^{\Dic_{4p}}$ in \cite[Thm. 4.3]{CDMM},  where $\Dic_{4p}$ is the dicyclic group of order $4p$. The underlying $\k$-algebra of $\HH_{8p}$ is generated by the elements $a,x,z$ subject to the relations 
    \begin{gather*}
        a^2=1,\quad  x^{2p}=1,\quad  z^2=a-1,\quad ax=xa,\quad az=za,\quad xz=-zx.
    \end{gather*}
    It is a Hopf algebra with coalgebra and antipode structures given by
    \begin{gather*}
           \Delta(a)=a\otimes a,\quad \Delta(x)=x\otimes e_0x+ax^{-1}\otimes e_1x,\quad \Delta(z)=z\otimes 1+g\otimes z\\
     \ve(a)=\ve(x)=1, \quad \ve(z)=0\\
    S(a)=a,\quad S(x)=e_0x^{-1}+e_1x,\quad S(z)=-g^{-1}x,
    \end{gather*}
where $e_0=(1+a)/2$, $e_1=(1-a)/2$, and $g=(e_0+\sqrt{-1}e_1)x^p$.  By \cite[Prop. 3.14]{BGN}, $\HH_{8p}$ is not basic as  $(\mathcal H_{8p}^*)_0=\k C_{2p}\oplus \mathcal M^*(2,\k)^p$.
    \end{example}

    \begin{example}
       Consider the semisimple Hopf algebra $A_{4p}$ constructed in \cite{Ge1997}. As $\k$-algebra, $A_{4p}$ is generated by $a,s_{\pm}$ subject to the relations 
       \[
       a^2=1,\quad as_{\pm}=s_{\pm} a,\quad s_{\pm}^2=1,\quad (s_+s_-)^p=1
       \]
       together with Hopf algebra structures 
\begin{gather*}
    \Delta(a)=a\otimes a,\quad \Delta(s_{\pm})=s_{\pm}\otimes e_0s_{\pm}+s_{\mp}  \otimes e_1s_{\pm}, \\
    \ve(a)=\ve(s_{\pm})=1,\\
    S(a)=a,\quad S(s_\pm)=e_0s_\pm+e_1s_\mp,
\end{gather*}
where $e_0=(1+a)/2$ and $e_1=(1-a)/2$. For any positive integer $i$, let 
$$
s_+(i):=\underbrace{s_+s_-\cdots}_{i \text{ terms}}\,. 
$$
Then $a, s_+(p)$ generate $G(A_{4p})$, which is isomorphic to $\Z_2 \times \Z_2$.

Let $R_{-1}=\k[y]/(y^2)$ be the graded algebra in $\!_{A_{4p}}^{A_{4p}}\mathcal YD$ where the $A_{4p}$-module structure is given by $a\cdot y=-y$ and $s_\pm \cdot y=y$ and the $A_{4p}$-comodule structure is given by $\delta(y)=s_+(p)\otimes y$.
Moreover, $R_{-1}$ is a Hopf algebra in $\!_{A_{4p}}^{A_{4p}}\mathcal YD$ with   $\Delta(y)=y\otimes 1+1\otimes y$. Thus, its bosonization $H=R_{-1}\# A_{4p}$  is an $8p$-dimensional Hopf algebra with $H_0 = A_{4p}$. Since $H$ is self-dual (cf. \cite[Thm. 4.3]{CDMM}), $(H^*)_0 \cong A_{4p}$.
    \end{example}

Now, we turn our attention to nonsemisimple pointed Hopf algebras of dimension $8p$. We will apply the Lifting Method of Andruskiewitsch and Schneider by first computing the corresponding graded Hopf algebras ${\rm gr} H \cong R \# \k[G]$, where $G$ is a finite group whose order divides $8p$, and $R$ is a nontrivial Nichols algebra in $\YD{\k[G]}$ with $\dim R=8p/|G|$.
These examples exhaust all $8p$-dimensional Hopf algebras with the Chevalley property that are neither pointed nor basic.

\begin{thm}\cite[Theorem 3.16]{BGN}\label{CP8p}
Let $H$ be a nonsemisimple Hopf algebra of dimension $8p$ with the Chevalley property.
\begin{itemize}
    \item[(i)] Assume $H$ is pointed. Then $H$ is isomorphic to one of the Hopf algebras constructed in \cite{Gra2000}.
    \item[(ii)] Assume neither $H$ nor $H^*$
is pointed. Then $H$ is isomorphic either to $\mathcal{H}_{8p}$ or $R_{-1}\# A_{4p}$.
\end{itemize}
\end{thm}

\subsection{Nonsemisimple Hopf algebras of dimension $8p$ with $p\,|\, |G(H)|$}\label{SS:8p-pdivG}
We start by investigating the (dual) Chevalley property for finite-dimensional Hopf algebras within the setting of $G$-equivariantizations.

\begin{prop}\label{CP}
Let $A$ be a finite-dimensional Hopf algebra,  $G$ a finite group, and $H$ a Hopf algebra extension of $\k[G]$ by $A$. Then $H$ has the dual Chevalley property if and only if $A$ has it.
\end{prop}
  \begin{proof}
By \cite[Prop. 1.7]{Ng09}, the exact sequence of Hopf algebras
$$
1\to A\to H\to \k[G]\to 1
$$
implies $J_H=J_AH=HJ_A$. If $J_A$ is a Hopf ideal of $A$, then $J_H$ is a Hopf ideal of $H$. Conversely, if $J_H$ is a Hopf ideal, the proof of \cite[Prop. 1.7]{Ng09} shows that $J_A=J_H\cap A$. Then $\Delta(J_A)\subset \Delta(J_H)\cap \Delta(A)\subset (J_H\otimes H+H\otimes J_H)\cap (A\otimes A)=J_A\otimes A+A\otimes J_A$. Hence, $J_A$ is also a Hopf ideal.
 \end{proof}

Let $A$ be a finite-dimensional Hopf algebra together with a finite group $G$-action on $\C{A}$ via some monoidal functor $(\rho,\sigma): \underline{G} \to\2ust(A)$. We denote by $H$ the corresponding cocentral extension of $\k[G]$ by $A$ according to Theorem \ref{2dimE}. It is clear that if $H$ has the Chevalley property, then so does $A$. However, the converse is not valid due to the following result.

\begin{prop}\label{Prop:dualCP}
 If $A$ has the Chevalley property and $\ol\st(A_0)=\ol\st_t(A_0)$, then $H$ has the Chevalley property if and only if the induced group homomorphism $\overline{\rho}: G\to \ost(A)$ factors through $\ost_t(A)$. In particular, if $A$ is pointed with cyclic $G(A)$ and $\overline{\rho}$ does not factor through $\ost_t(A)$, then $H$ does not have the Chevalley property.
\end{prop}
\begin{proof}
  Suppose $\overline{\rho}: G\to \ost_t(A)$.
  As Lemma \ref{lem:3co}, the 3-cocycle $[\brho]$ is trivial in $H^3(G,Z_G(A))$. Moreover, by the definition of $\ost_t(A)$, in the lift $(\rho,\sigma): \underline{G} \to\2ust(A)$, we can take $\rho(g)=(\varphi^g,1\otimes 1)$ for all $g\in G$. Then by the construction of $H=A\#_{(\rho,\sigma)}G$, it is clear that $H\cong A\otimes \k[G]$ as coalgebras. Thus, $H_0\cong A_0\otimes \k[G]$, which is a Hopf subalgebra of $H$ since $A_0$ is a Hopf subalgebra.

   On the contrary,  suppose $\rho(g)=(\varphi^g, J^g)$ is not trivial for some $g\in G$ but the coradical $H_0$ of $H$ is a Hopf subalgebra. Since we have a coalgebra decomposition $H=\bigoplus_{h\in G} Au^h$, the coradical $H_0=\bigoplus_{h\in G} (H_0)_h$ where $(H_0)_h=(Au^h)_0$ is the coradical of each subcoalgebra $Au^h$. Moreover, it is direct to check that $\Delta((H_0)_h)\subset (H_0)_h\otimes (H_0)_h$ and $(H_0)_h(H_0)_f\subseteq (H_0)_{hf}$ for any $h,f\in G$. Consequently, $H_0$ is a cocentral extension of $\k[G]$ by $A_0=(Hu^e)_0=(H_0)_e$ (cf. \cite[Lemma 1.3]{BNY}). Applying Theorem \ref{2dimE} and its proof, there exists a stable twist $(I_{bu^g}|_{A_0},J)=(I_b\circ \varphi|_{A_0},J)$ of $A_0$, where $J=(bu^g\otimes bu^g)\Delta(bu^g)^{-1}\in (A_0\otimes A_0)^\times$. This implies $(J^g)^{-1}(u^g\otimes u^g)=\Delta(u^g)=\Delta(b)^{-1}J^{-1}(b\otimes b)(u^g\otimes u^g)$. So $J=(b\otimes b)J^g\Delta(b)^{-1}$ and hence $J,J^g$ are equivalent. But $(I_{bu^g}|_{A_0},J)$ is trivial due to $\ol\st(A_0)=\ol\st_t(A_0)$. Hence, $(\varphi^g,J^g)$ must be trivial, a contradiction!

   Finally, assume $A$ is pointed with $G(A)=\mathbb Z_n$. In this case, any 2-cocycle $J$ for $A_0$ defines a normalized 2-cocycle $J: \mathbb Z_n\times \mathbb Z_n\to \k^\times$. Since $H^2(\mathbb Z_n,\k^\times)$ is trivial, $J$ is a 2-coboundary and hence $\ol\st(A_0)=\ol\st_t(A_0)$. Thus, the result follows from the previous discussion.
\end{proof}

We are now ready to collect some fundamental properties of the family of Hopf algebras $\N(md,\xi)$ for every odd integer $d>1$.

\begin{thm}\label{L:ProN8d} Let $m,d>1$ be two integers such that $\gcd(m,d)=1$ and $d$ is odd and let $\xi$ be a $m$-th root of unity. The Hopf algebra $N=\N(md,\xi)$ is, up to equivalence, the unique nontrivial extension of $A=A_1(md,\xi)$ by $\k[C_2]$, and does not have the Chevalley or the dual Chevalley property.  Moreover,  for any nontrivial $C_2$-action on $\C{A}$, there is an equivalence of finite tensor categories
  \[
    \C{N}~\overset{\otimes}{\cong}~(\C{A})^{C_2}.
  \]
\end{thm}
\begin{proof}
By Theorem \ref{th:extC2}, we know $N$ is the only nontrivial extension of $A$ by $\k[C_2]$ up to equivalence when $d$ is odd. Since $A$ does not have the dual Chevalley property, neither does $N$ by Proposition \ref{CP}. Note that the corresponding group homomorphism $\brho: G\to \ost(A)\cong C_2$ is surjective and does not factor through $\ost_t(A)$, which is trivial. Moreover, $A$ is pointed with $G(A)\cong C_{2d}$ such that $H^2(C_{2d},\k^\times)$ is trivial.  By Proposition \ref{Prop:dualCP}, $N$ does not have the Chevalley property. The last statement follows from Theorem \ref{HRepGE}.
\end{proof}

\begin{rem}\label{r:8p-ftc}
The last statement of Theorem~\ref{L:ProN8d} is intrinsically
categorical: for any nontrivial $C_2$-action on $\C{A}$, the
equivariantization $(\C{A})^{C_2}$ is tensor equivalent to $\C{N}$.
By contrast, the classification results in the rest of this section use
Hopf-algebraic input such as Hopf subalgebras, Hopf quotients,
skew-primitive elements, and dimension arguments, so they do not
immediately extend to arbitrary finite tensor categories.
\end{rem}

We denote by $\irr(H)$ a complete set of nonisomorphic simple $H$-modules. For any $V\in \irr(H)$, the {\it projective cover} of $V$ is denoted by $P(V)$. Since $H$ is a Frobenius algebra \cite{LG}, we have (see \cite[61.13]{CR62})
\begin{equation}\label{dimH}
    \dim H~=\sum_{V\in \irr(H)} \dim V\cdot \dim P(V).
\end{equation}
For any $H$-module $M$ and $V\in \text{Irr}(H)$, we define $[M:V]$ to be the multiplicity of $V$ appearing as a composition factor of $M$. Let $V,W\in \text{Irr}(H)$. Since $V^*\otimes P(W)$ is projective, it is a direct sum of indecomposable projective $H$-modules. More precisely,

\begin{align}\label{PM}
V^*\otimes P(W)= & \bigoplus_{U\in \mathrm{Irr}(H)} \dim \left(\Hom_H(V^*\otimes P(W),U)\right)\cdot P(U) \nonumber\\
=&\bigoplus_{U\in \mathrm{Irr}(H)} [V\otimes U:W]\cdot P(U).
\end{align}

\begin{lem}\label{Dihe0} 
   Let $p$ be an odd prime, and $H$ a nonsemisimple Hopf algebra of dimension $8p$. If $H$ has a Hopf subalgebra $K$, which is either $2p$-dimensional  or semisimple of dimension $4p$,  then $H$ has a nontrivial skew-primitive element. If further $K \not\cong \k[C_{2p}]$, then $H$ has the Chevalley property with $\dim H_0 = 4p$.
\end{lem}
\begin{proof} Assume $H$ has no nontrivial skew primitive elements.  Then, for any 1-dimensional $H^*$-module $V$, $JP(V)/J^2 P(V)$ must have at least one simple constituent of dimension $>1$, where $J$ denotes the Jacobson radical of $H^*$. In particular, $\dim P(V) \ge 4$. We proceed to show that none of these Hopf subalgebras of $H$ exists.

If $H$ admits a semisimple Hopf subalgebra $K$ of dimension $4p$, then $K = H_0$ by \cite[Lem. 3.11]{BGN}. Since $J=H_0^\perp$,  $J^2=0$. In particular, $JP(\1)$ is the socle of $P(\1)$ which is also $1$-dimensional (cf. \cite[Lem. 1.1]{Ng08}). Thus, $\dim P(\1) =2$, a contradiction.

Now we assume $H$ admits a $2p$-dimensional Hopf subalgebra $K$. By \cite{Ng2005}, $K \cong \k[C_{2p}],\, \k[D_{2p}]$ or $\k^{D_{2p}}$. In particular, $K$ is semisimple. Since there is a Hopf surjection $H^*\to K^*$, the fusion category $\C{K^*}$ is a tensor subcategory of $\C{H^*}$.  

We first show that $K^*$ cannot be commutative. Otherwise, $\C{K^*}$ is a pointed fusion subcategory with a complete set of nonisomorphic 1-dimensional $K^*$-modules $V_1, \cdots, V_{2p}$. By the first paragraph, $\dim P(V_i) \ge 4$  and so
$$
\dim H \ge \sum_{i=1}^{2p} \dim P(V_i) \ge 4\cdot 2p = 8p\,.
$$
This implies $\{V_1,\ldots, V_{2p}\}$ is a complete set of nonisomorphic simple $H^*$-modules, which contradicts that there must be simple $H^*$-module of dimension $> 1$. 
    
    Now we have $K\cong \k^{D_{2p}}$ or $K^*\cong \k[D_{2p}]$. 
    Recall that $\k[D_{2p}]$ has two one-dimensional simple modules $\k,\chi$ and $(p-1)/2$ two-dimensional self-dual simple modules $V_1,\cdots, V_{\frac{p-1}{2}}$. By the first paragraph,  $\dim P(\k)=\dim P(\chi)\ge 4$. 
    
   It follows from the fusion rules of $\C{\k[D_{2p}]}$, $V_i\otimes V_i \cong \k\oplus \chi \oplus V_j$ for some  $V_j$. Reorder $V_i$, if necessary, so that $\dim P(V_1)\le \dim P(V_i)$ for any $i$. Thus,
\begin{align*}
   2\dim P(V_1)&\,=\dim (P(V_1)\otimes V_1)\\
   &\,\ge \dim P(\k)+\dim P(\chi)+\dim P(V_j) \text{ for some }j\\
   &\,\ge 4+4+\dim P(V_1). 
\end{align*}
This implies that $\dim P(V_1)\ge 8$ and so $\dim P(V_i)\ge 8$ for all $i$. Thus,
\begin{align*}
8p&\,=\dim (H^*)=\sum_{V\in {\rm Irr}(H^*)} (\dim V)(\dim P(V))\\
&\, \ge \dim P(\k)+\dim P(\chi)+\sum_{1\le i\le (p-1)/2} (\dim V_i)(\dim P(V_i))\\
&\,\ge 4+4+2\cdot 8\cdot \frac{(p-1)}{2}=8p. 
\end{align*}
Therefore, we have $\{\k,\chi, V_1,\ldots,V_{\frac{p-1}{2}}\}$ a complete set of nonisomorphic simple $H^*$-modules. Hence, $H^*/J \cong \k[D_{2p}]$. By duality, it implies that $H_0=K$, and so the graded Hopf algebra,  obtained from the radical filtration, is given by the bosonization $\mbox{gr} H=R \# \k^{D_{2p}}$ for some connected graded Hopf algebra $R$ in $\YD{\k^{D_{2p}}}\cong \YD{\k[D_{2p}]}$ with $\dim R=4<\infty$. 
However, this cannot occur, since if $p > 3$, \cite{carnovale2024twist} proves that $\dim R = 1$, while if $p = 3$, \cite{andruskiewitsch2010nichols} shows that $\dim R$ can only be $1$ or $\ge 12$.

Finally, assume that $K \cong \k [D_{2p}]$ or $\k^{D_{2p}}$. We consider the iterated wedge product such that $\wedge^0 K=\{0\}$, $\wedge^1 K=K$, and $\wedge^i K=(\wedge ^{i-1} K)\wedge K$ for all $i\ge 1$. As in \cite[Lemma 5.2.8]{Mont93}, one can similarly show that $L=\cup_{n\ge 0} \wedge^n K$ is a Hopf subalgebra of $H$ where $\wedge^1K\subseteq \wedge^2 K\subseteq \wedge^3 K\subseteq \cdots $ is an exhaustive Hopf algebra filtration on $L$. Hence, the associated graded Hopf algebra ${\rm gr}L\cong R\ \# K$ with $R$ a connected graded braided Hopf algebra in $\!_K^K\mathcal{YD}$ with $\dim(R) \le 4$. As we have seen, $\dim R = 1$ and $\wedge^2 K = K$. In particular, $H_0 \neq K$ because $H = \bigcup_{i \ge 0} \wedge^i H_0$ and $H$ is nonsemisimple. Because $H$ contains a nontrivial $(g,1)$-primitive element $x$ for some $g\in G(H)$, it follows that $g\notin K$; indeed, if $g$ were in $K$, then $x$ would lie in $\wedge^2 K\setminus K$. Therefore, $K,g$ generates a semisimple Hopf subalgebra of dimension $>2p$, which must be of dimension $4p$ since $H$ is nonsemisimple. Consequently, $\dim H_0=4p$ and $H$ has the Chevalley property by \cite[Lem. 3.11]{BGN}.
\end{proof}

We denote by $\N_{8p}=\N(2p,-1)$ the $8p$-dimensional nonsemisimple Hopf algebra that lacks both the Chevalley and dual Chevalley properties for an odd prime $p$. The classification of $8p$-dimensional nonsemisimple Hopf algebras for an odd prime $p$ starts in \cite{BG2013, BGN}, where they show that if $|G(H)|=2p$ and assuming $H$ and $H^*$ are nonpointed, then $H$ fits into an exact sequence of Hopf algebras $1\to A_\delta(2p,-1)\to H\to \k[C_2]\to 1$. Now, we are able to complete this case. 

\begin{thm}\label{T:ChavCCE}
      Let $p$ be an odd prime, and $H$ a nonsemisimple Hopf algebra of dimension $8p$. Suppose $H$ contains a Hopf subalgebra of dimension $p$ or $2p$ and neither $H$ nor $H^*$ has the Chevalley property. Then $H\cong \mathcal N_{8p}$. In particular, $G(H)\cong C_{2p}$.
\end{thm}
\begin{proof}
    
We first assert that $p$ divides $|G(H)|$. Indeed, suppose $H$ contains a Hopf subalgebra $K$ of dimension $p$. Then, by Zhu's result \cite{Zhu94}, we have $K\cong \k[C_p]$. If instead $K$ has dimension $2p$, Lemma \ref{Dihe0} forces $K\cong \k[C_{2p}]$. Hence, in all cases, the claim is established.
    
    Now, we can apply \cite[Prop. 4.2]{BG2013} to get $|G(H)|=2p$. Moreover,  \cite[Lem. 3.3]{BGN} shows that there exists an exact sequence of Hopf algebras,
    $$
    0\to A_\delta(2p,-1)\to H\to \k[C_2]\to 0\, .
    $$
  If $\delta=0$, then $A_0(2p,-1)$ satisfies the dual Chevalley property. Consequently, $H$ also has this property by Proposition \ref{CP}, which contradicts our assumption. Therefore, $\delta=1$. Thus, the result follows from Theorem \ref{L:ProN8d}.
\end{proof}

\begin{cor}\label{Dihe1} Let $p$ be an odd prime.
    If $H$ is an $8p$-dimensional nonsemisimple Hopf algebra such that $H^*$ does not have any nontrivial skew-primitive element, then there does not exist any $2$-dimensional self-dual simple $H$-module. If there exists a $2$-dimensional simple $H$-module $V$ such that $V^{**} \cong V$, then $p > 5$ and 
    $$
    V^* \o V = \left\{\begin{array}{ll}
        \sum_{\b \in G'} \k_\b  &  \text{for some order 4 subgroup $G'$ of $G(H^*)$ or}  \\\\
         \1 \oplus W & \text{ for some nonprojective simple self-dual $H$-module $W$.  }
    \end{array}\right.
    $$
\end{cor}
\begin{proof}
(i) Suppose $H$ admits a 2-dimensional self-dual simple module $V$. Then  $H^*$ admits a simple 4-dimensional coalgebra $C$ such that $S(C)=C$. By \cite{Na}, the Hopf subalgebra $K_C$ generated by $C$ is either semisimple or an extension of a nonsemisimple basic Hopf algebra $A$ by $\k^G$ for some finite group $G$, i.e., we have an exact sequence of Hopf algebras
    \begin{equation} \label{eq:4pexact}
         1 \to \k^G \to K_C \to A \to 1\,.
    \end{equation}
    If $K_C$ is semisimple, then $\dim K_C\in \{8, 2p, 4p\}$. This implies $H^*$ has the Chevalley property and $H^*$ has a nontrivial skew-primitive element, which is a contradiction. Therefore, $K_C$ must be nonsemisimple.

    Since any 8-dimensional nonsemisimple Hopf algebra has a nontrivial skew-primitive, $\dim K_C\in \{2p,4p,8p\}$. By \cite{Ng2005}, $2p$-dimensional Hopf algebras are semisimple. Moreover, if $K_C=H$, then $G=C_2$ and we have an exact sequence $1\to A^*\to H\to \k[C_2]\to 1$, where $A^*$ is a pointed Hopf algebra of dimension $4p$. We know either $H$ has the dual Chevalley property or $H\cong \N_{8p}$, where in both cases $H^*$ has a nontrivial skew-primitive element. Thus, $\dim K_C = 4p$ and $\dim A = 4$ or $4p$. If $\dim A = 4p$, then $K_C = A$ and so $K_C = A_1(2p,-1)^*$ which has a skew-primitive, a contradiction. However, if $\dim A = 4$, then $G$ is a cyclic group of order $p$. By \cite{YN2011}, $K_C \cong A_{\delta}(2p, -1)$, which must have a nontrivial skew-primitive, contradiction. 

    This proves $H$ does not have 2-dimensional self-dual simple module.

(ii) Suppose $V$ a simple $H$-module such that $\dim V =2$ and $V^{**} \cong V$. Let $\phi: V \to V^{**}$ be an $H$-module isomorphism. Then $\tr(\phi) \ne 0$. Otherwise,
$$
0 = (\1 \xrightarrow{\coev} V\o V^* \xrightarrow{\phi \o \id} V^{**}\o V^* \xrightarrow{\ev} \1)
$$
and so $V \o V^*$ is indecomposable with the radical series $\1, W, \1$, where $W$ is a simple self-dual $H$-module of dimension 2. This contradicts (i).

Since $\tr(\phi) \ne 0$, the injective map $\coev: \1 \to V \o V^*$ has a section and so $V \o V^* = \1 \oplus W$ for some 3-dimensional $H$-module $W$. If $W = \k_\b \oplus U$ for some simple $H$-module $U$ of dimension 2 and $\b \in G(H^*)$. Then $\b^{-1} = \b$ and $U^*=U$ since $V \o V^*$ is self-dual, but this contradicts (i). Therefore, either $W$ is irreducible or $W$ is a sum of 1-dimensional simple $H$-modules. Thus, $V \o V^*$ has only two possible decompositions
$$
V \o V^* = \1 \oplus W \quad \text{ or }  \quad V \o V^* = \sum_{\b \in G'} \k_\b 
$$
Let $\phi: W \to W^{**}$ be any isomorphism of $H$-modules. We claim that $\tr(\phi) \ne 0$ if $W$ is irreducible. We can assume $S^{32} = \id$. By rescaling $\phi$, we can further assume that $\phi^{16} = \id: W \to W$. Therefore, $\tr(\phi) \neq 0$ since $\dim W = 3$. Thus, $W$ is not projective and $\dim P(W) \ge 6$. Therefore, either
$$
\dim H \ge 4 \dim P(\1) + 2 \times 4 \dim P(\1) = 12 \dim P(\1) \ge 48
$$
or 
$$
\dim H \ge \dim P(\1) + 2 \times (\dim P(\1)+\dim P(W))+ 3\dim P(W) = 12 + 5 \times 6 =42
$$
In particular, $p > 5$.
\end{proof}

\subsection{Nonsemisimple Hopf algebras of dimension $24$}
In the final part of our work, we finish the classification of all nonsemisimple Hopf algebras of dimension $24$ over an algebraically closed field of characteristic zero, initiated in \cite{BG2013,BGN}, by proving that $\mathcal{N}_{24}$ and $\mathcal{N}_{24}^*$ are the only examples that do not satisfy the Chevalley property.

\begin{lem}\label{divHopfsub}
   Let $\pi: H \to L$ be a surjective Hopf algebra map and $K$ a Hopf subalgebra of $H$. Then $H^{\co L}$ is a free left (and right) $K^{\co L}$-module, where both $K$ and $H$ are right $L$-comodules via $\pi$.
\end{lem}
\begin{proof}
    Since $K^{\co L} \subset H^{{\co} L}$ are coinvariant subalgebras via the surjective Hopf algebra map $\pi$, they are left coideal subalgebras of $H$. According to \cite{FMS}, they are Frobenius algebras. Moreover, $H^{{\co} L}$ is an object in the Hopf module category $\!^H_{K^{{\co}L}}\mathcal{M}$ and $^H\!\mathcal{M}_{K^{{\co}L}}$. By \cite{Ma92},  $H^{\co \pi}$ is a free left and a free right $K^{\co \pi}$-module. 
\end{proof}

\begin{lem}\label{BonFac}
    Let $H$ be a finite-dimensional Hopf algebra of dimension $mn$ for relatively prime integers $m,n$. Suppose $H$ contains a Hopf subalgebra $K$ and admits a quotient Hopf algebra $L$  such that $\dim K=\dim L=n$. Then $K\cong L$ and $H$ is the bosonization $H^{\co L}\times K$. 
\end{lem}
\begin{proof}
    Let $\pi: H\to L$ be a surjective Hopf algebra map. By the preceding lemma, $H^{\co L}$ is a free left $K^{\co L}$-module, and so $\dim K^{\co L}\mid \dim H^{\co L}$. Note that 
    $$
    \dim K^{{\rm co}L} = \frac{\dim K}{\dim \pi(K)} \,\Big|\, n \quad\text{and}\quad \dim H^{\co L}=\frac{\dim H}{\dim L}=m.$$ 
    Since $m,n$ are coprime, the divisibility $\dim K^{{\co}L}|\dim H^{{\co}L}$ implies $\dim K^{\co L}=1$. Thus,  $\pi|_K$ is injective, and hence $L = \pi(K)$ or $\pi: K \to L$ is a Hopf algebra isomorphism. In particular, $\pi: H \to K$ has a Hopf section. It follows from \cite{Rad1985} that $H$ is the Radford biproduct (or bosonization) $H^{\co L} \times K$.   
\end{proof}

\begin{lem}\label{L:InnS}
Let $A$ be a finite-dimensional semisimple algebra. Let $\alpha$ be an algebra automorphism of $A$ with finite order $m$, and let $\delta$ be an $(\alpha,1)$-skew derivation such that $\alpha\delta = \lambda\,\delta\alpha$ for some nonzero scalar $\lambda$. Then there exists an element $y \in A$ with $\alpha(y) = \lambda y$ such that, for every $a \in A$, we have $\delta(a) = ya - \alpha(a)y$. 
\end{lem}
\begin{proof}
    From $\alpha\delta=\lambda \delta\alpha $, we have $\delta=\alpha^m\delta=\lambda^m \delta\alpha^m=\lambda^m\delta$. If $\lambda^m\neq 1$, $\delta=0$ and we can take $y=0$. Assume $\lambda^m=1$. Let $\xi$ be a primitive $m$-th root of unity and $A=\bigoplus\limits_{0\le i\le m-1} A_i$ where $A_i$ is the eigenspace corresponding to the $\alpha$-eigenvalue $\xi^i$. Then one can compute for any $a\in A_i$ such that 
    \[\alpha(\delta (a))=\lambda \delta(\alpha(a))=\lambda \xi^{i}\delta(a)= \xi^{i+j}\delta(a),
\] 
where we write $\lambda=\xi^j$. Hence, $\delta$ is homogeneous by sending $A_i$ to $A_{i+j}$. By \cite[Lemma 6.3]{EN2022}, we get  $\delta$ is inner. Moreover, there is some element $y\in A_j$ that is $\alpha(y)=\xi^j y=\lambda y$ such that $\delta(a)=ya-\alpha(a)y$.  
\end{proof}

We provide an alternative proof concerning the actions of pointed Hopf algebras on semisimple commutative algebras, as studied in \cite{YN2011, BG2013}.

\begin{lem}\label{ChQBon}
    Let $K$ be a finite-dimensional pointed Hopf algebra with abelian $G(K)$, and $R$ a braided Hopf algebra in $\!_{K}^{K}\mathcal{YD}$. Suppose $R$ is semisimple commutative. Then for any nontrivial skew-primitive element $x$ of $K$ we have $x.R=0$. Consequently, the bosonization $R\times K$ has a semisimple Hopf quotient of dimension $\dim (K/I)\dim R$, where $I$ is the Hopf ideal in $K$ generated by all nontrivial primitive elements. 
\end{lem}
\begin{proof} 
Recall from the classification of finite-dimensional pointed Hopf algebras whose grouplike elements form abelian groups, $K$ is generated by $G(K)$ and a set of primitive elements $\{x\}$ such that $\Delta(x)=x\otimes 1+g\otimes x$ and $gx=\theta xg$ with $g\in G(K)$ and some root of unity $\theta\neq 1$. Let $\ord(g)=m>1$. Note that $\theta^m=1$. So we can write $\theta=\xi^j$ for some primitive $m$-th root of unity $\xi$ with $j\neq 0$. Since $R$ is a left $K$-module algebra, $g$ acts as an algebra automorphism, and $x$ acts as a $(g,1)$ skew-derivation on $R$. 

Write $R=\bigoplus\limits_{0\le i\le m-1} R_i$ where $R_i$ is the eigenspace with $g$-eigenvalue $\xi^i$. By Lemma \ref{L:InnS}, there is some element $y\in R_j$ such that $x.r=yr-g(r)y$ for all $r\in R$. Take the right integral $\Lambda$ of $R^*$. By \cite[Eq. (3.4)]{YN2011}, we have
    \[
   0=\ve(x)\Lambda(r)=\Lambda(x.r)=\Lambda(yr-g(r)y).
    \]
Thus, $\Lambda(ry)=\Lambda(yr)=\Lambda(g(r)y)$ for all $r\in R$. We claim that $\Lambda (Ry)=0$. Let $r\in R_i$. We get $\Lambda(ry)=\Lambda(g(r)y)=\xi^i\Lambda(ry)$. So $\Lambda(ry)=0$ for $i\neq 0$. If $r\in R_0$, we get \[\Lambda(ry)=\ve(g)\Lambda(ry)=\Lambda(g(ry))=\xi^j\Lambda(ry)\]
since $ry\in R_0R_j\subseteq R_j$. Again $\Lambda(ry)=0$ since $j\neq 0$. This proves our claim. By \cite[Remark 5.9]{FMS}, $\Lambda$ induces a nondegenerate associative bilinear form on $R$ via $(a,b)=\Lambda(ab)$ making it Frobenius. Hence, $\Lambda$ cannot vanish on any nonzero left ideals on $R$. Consequently, $y=0$ and $x.R=0$. 

Let $H=R\times K$ be the bosonization and $I$ the ideal of $K$ generated by all its nontrivial skew primitive elements. Again, by the classification result, we know $I$ is a Hopf ideal of $K$ and $K/I$ is semisimple. Our previous discussion shows that $IR=RI$ and $HI=IH$ is a Hopf ideal of $H$ such that $H/HI\cong R\# (K/I)$ is semisimple.
\end{proof}

In the remainder, we let $H$ be a nonsemisimple Hopf algebra of dimension $24$ over an algebraically closed field $\k$ of characteristic zero. According to \cite{St}, every nonsemisimple Hopf algebra of dimension $8$ is isomorphic to $A_\delta(4,-1)$, its dual, or one of the two Hopf algebras listed below. 
\begin{gather*}
B:= \k\langle g,x,y\,|\, g^2-1=x^2 =y^2 =gx+xg =gy+yg =xy+yx=0\rangle\\
\Delta(g) = g\otimes g,\quad  \Delta(x)=x\otimes 1+g\otimes x,\quad  \Delta(y)=y\otimes 1+g\otimes y;
\end{gather*}
\begin{gather*}
      C:=\k \langle g,h,x\,|\, g^2=h^2 = 1, x^2 = gx+xg = hx+xh=gh-hg=0\rangle,\\ 
  \Delta(g) = g\otimes g,\quad  \Delta(h)=h\otimes h,\quad  \Delta(x)=x\otimes 1+g\otimes x.
 \end{gather*}
 We point out some important facts about 8-dimensional nonsemisimple Hopf algebras. First, $A_\delta(4,-1)$ does not contain the 4-dimensional Taft algebra as a Hopf subalgebra. Secondly, $B$ is the only 8-dimensional pointed Hopf algebra such that $|G(B)|=2$. Finally, $A_1(4,-1)^*$ is the only nonsemisimple non-pointed 8-dimensional Hopf algebra generated by its unique 4-dimensional simple subcoalgebra.

Our next result is part of the main results in \cite[Theorem B]{BG2013}. We provide an alternative proof for the reader's convenience. 

  \begin{lem}\label{ordH}
 Suppose $H$ and $H^*$ do not have the Chevalley property. If $3\nmid |G(H\otimes H^*)|$, then $H$ and $H^*$ cannot simultaneously contain Hopf subalgebras of dimension $8$. In particular, $(|G(H)|,|G(H^*)|)\in \{(2,4),(4,2),(2,2)\}$.
\end{lem}
\begin{proof}
On the contrary, assume $H$ and $H^*$ have Hopf subalgebras of dimension $8$. Then Lemma \ref{BonFac} implies that there is some $8$-dimensional Hopf algebra $K$ such that $H=R\times K$ is the bosonization for some braided Hopf algebra $R$ in $\!_{K}^{K}\mathcal{YD}$. Since $K$ cannot be semisimple according to \cite[Theorem 3.5]{BGN}, we can assume $K$ to be pointed by duality. Now by Radford $S^4$-formula, we get $S^{32}=\id$ as $3\nmid |G(H\otimes H^*)|$. So \cite[Lemma 1.4]{Ng2005} implies that ${\rm tr}(S^2|_{R})\neq 0$. Hence, $R$ is semisimple by \cite[Theorem 7.3]{AS1998} and commutative as $\dim R=3$. Then Lemma \ref{ChQBon} shows that $H$ has a semisimple Hopf quotient $L$ of dimension either $6$ or $12$. Finally, $\dim L=12$ is impossible by \cite[Lemma 3.11]{BGN} and $\dim L=6$ is also impossible since otherwise Theorem \ref{T:ChavCCE} would imply that $H^*\cong \mathcal N_{24}$ with $|G(H^*)|=6$.

By \cite[Lemma 4.1]{BGN}, we know both $H$ and $H^*$ have nontrivial skew-primitive elements. Consequently, $|G(H)|,|G(H^*)|>1$. Again \cite[Theorem 3.5]{BGN} implies that $|G(H)|, |G(H^*)|\neq 8$. Now suppose $|G(H)| = |G(H^*)| = 4$. By the same reasoning, both $H$ and $H^*$ would possess a Hopf subalgebra of dimension $8$. This contradicts our earlier discussion. Therefore, $(|G(H)|, |G(H^*)|) \in \{(2,4), (4,2), (2,2)\}$.
\end{proof}

\begin{lem}\label{Nodual2}
    Suppose $H$ and $H^*$ do not have the Chevalley property. If $3\nmid |G(H\otimes H^*)|$, then neither $H$ nor $H^*$ has a self-dual two-dimensional simple module.
\end{lem}
\begin{proof}
    Suppose $H^*$ has a self-dual two-dimensional simple. By duality, we know $H$ must contain a $4$-dimensional simple subcoalgebra $C$ stabilized by its antipode \cite[Lemma 2.6]{YN2011}. Let $K$ be the Hopf subalgebra of $H$ generated by $C$. Since $\dim K \mid \dim H = 8p$, we have $\dim K\neq 1,2,4,p$ by the classification results. Moreover, $\dim K \neq 2p$ by Proposition \ref{T:ChavCCE} and our assumption. Hence, $\dim K\in \{8,4p,8p\}$. It follows that $K$ is nonsemisimple and by \cite[Proposition 1.3]{Na}, there exists an exact sequence of Hopf algebras
    \[1\to \k^G\to K\to A\to 1\]
    where $A^*$ is nonsemisimple pointed and $G$ is a finite group. 

    Case (1): $\dim K=8p$. Hence, $|G|\neq 4,8,p,2p,4p,8p$. By our assumption, we must have $|G|=2$. Thus, $A^*$ is nonsemisimple pointed of dimension $4p=12$ with $|G(A^*)|=2,4$. This is impossible by the classification of pointed $4p$-dimensional Hopf algebras (cf. \cite[Theorem II]{YN2011}).

    Case (2): $\dim K=4p$. A similar argument indicates that $|G|=1$ and $K^*=A^*$ is pointed. By our assumption, $K^*$ has a $2$-dimensional self-dual simple. According to \cite{YN2011}, the only $4p$-dimensional pointed Hopf algebra that has a $2$-dimensional simple is $A_1(2p,-1)$. As discussed in Section \ref{S:StGrH}, there are $p-1$ two-dimensional simples in $A_1(2p,-1)$, denoted by $V_1,\ldots, V_{p-1}$. We still denote by $e_{\oi}$ for $\oi\in \mathbb Z_d\setminus \{\overline{0}\}$ the indecomposable central idempotents in $A_1(2p,-1)$ corresponding to $V_i$ such that $g^me_\oi=\theta^{im} e_\oi$ with $\theta$ a $2p$-th root of unity. Then it is easy to see that $V_i^*=V_{-i}\not\cong V_i$. So this case is impossible.

    Case (3): $\dim K=8$. From our previous discussion regarding the classification of $8$-dimensional Hopf algebras, we conclude that $K \cong A_1(4,-1)^*$. By Lemma \ref{ordH}, $|G(H^*)| = 2$, since $|G(H^*)| = 4$ would imply that $H^*$ also contains a Hopf subalgebra of dimension $8$. Hence, $H^*$ contains the $4$-dimensional Taft algebra $T_4$. By a dimension argument and Lemma \ref{divHopfsub}, $T_4$ would have to be a Hopf subalgebra of $A_1(4,-1)$ via the surjection $H^* \to A_1(4,-1)$, which is impossible.
 \end{proof}

\begin{lem}\label{3dG}
Suppose $H$ and $H^*$ do not have the Chevalley property. Then we must have $3$ divide $|G(H\otimes H^*)|$.
\end{lem}
\begin{proof}
On the contrary, suppose $3\,\nmid\,|G(H\otimes H^*)|$. Without loss of generality, we can assume $H^*$ does not contain any $8$-dimensional Hopf subalgebra by Lemma \ref{ordH}. In particular, $|G(H^*)|=2$. Thus, $H$ has two one-dimensional simples denoted by $\k$ and $\chi$. Let $d=\max\limits_{V\in {\rm Irr}(H)} \dim V$ and choose a simple $V$ such that $\dim V=d$. By \eqref{dimH} and our assumption, one gets $2\le d\le 4$. Moreover, since $S^{16}=\id$, we can write 
 $\dim P(\k)=2m$ for some integer $m\ge 1$ by \cite[Corollary 1.5]{YN2011}. We claim that $m=2$ and the composition factors of $P(\k)$ are all 1-dimensional. 

Case (1): $d=4$. Then, $V$ is the only simple of dimension 4 that has to be projective. Consequently, $V^*\cong V$ and $\chi \otimes V\cong V\otimes \chi \cong V$. By \eqref{dimH}, we get $\dim P(\k)=\dim P(\chi)=4$. Moreover, all composition factors of $P(\k)$ are 1-dimensional.  

Case (2): $d=3$. We know $G(H^*)\cong C_2$ acts on ${\rm Irr}(H)$ by tensoring with $\chi$. If $\chi \otimes V\cong V$, then we know $2|\dim V$ by \cite[Lemma 3.2]{NW23}. This is absurd. Then, one can argue that $V,\chi\otimes V$ are the only two simples of dim $\ge 2$ that are also projective. Consequently, $2\dim P(\k)=\dim P(\k)+\dim P(\chi)=24-2(\dim V)^2=6$. This implies that $\dim P(\k)=3$, which is impossible.

Case (3): $d=2$.  Since $H$ has no self-dual simples by Lemma \ref{Nodual2}, we can write ${\rm Irr}(H)=\{\k,\chi,V_1, V_1^*, \ldots, V_n, V_n^*\}$ for some $n\le 2$ by \eqref{dimH}. If $m\ge 3$, this forces $n=1$. Since $4\dim P(V_1)=24-2\dim P(\k)=24-4m\leq 12$, we must have $m=4$ and $\dim P(V_1)=\dim V_1=2$. This implies that $V_1$ and $V_1^*$ are projective and hence $P(\k)$ is a direct summand of $V_1^*\otimes V_1$. A contradiction since $8=\dim P(\k)>\dim (V_1^*\otimes V_1)=4$. If $m=1$, then we must have $n=2$. Otherwise, we have $\dim P(\k) = \dim P(\chi) = 2$ and $\dim P(V_1) = \dim P(V_1^*) = 5$. This would imply that $JP(V_1)/J^2P(V_1)$ is isomorphic to either $\k$ or $\chi$, which leads to a contradiction, since neither $P(\k)$ nor $P(\chi)$ contains $V_1$ as a simple factor. By \eqref{dimH} again, we have $2\dim P(\k)+4\dim P(V_1)+4\dim P(V_2)=24$, that is, $\dim P(V_1)+\dim P(V_2)=5$. It is clear that there are no proper solutions. Thus, $\dim P(\k)=4$. If $P(\k)$ has a simple factor of dimension $2$, say $V_1$, then $\k$ will appear as a simple factor of $P(V_1)$. So $\dim P(V_1)\ge 5$, which contradicts \eqref{dimH} again. This proves our claim. 

Now, consider the subcategories $\CC$ of $H$-modules $V$ whose composition factors are 1-dimensional. It is clear that $\CC$ is a tensor subcategory of $\Mod{H}$, and hence $\CC= \Mod{L}$ for some proper Hopf quotient $L$ of $H$. Since $\dim P(\k)=4$, whose composition factors are 1-dimensional, we know $P(\k)\in \Mod{L}$ and is also the projective cover of $\k$ over $L$. Then it is clear that ${\rm Irr}(L)=\{\k,\chi\}$ and $\dim L=8$. Thus, $H^*$ must contain the 8-dimensional Hopf subalgebra $L^*$. A contradiction to our assumption!
\end{proof}

We conclude our paper by presenting a classification of nonsemisimple Hopf algebras of dimension 24.

\begin{thm}
      Let $H$ be a nonsemisimple Hopf algebra of dimension 24. Then one of the following holds.
  \begin{enumerate}
 \item $H$ or $H^*$ is pointed.
        \item Assume $H$ and $H^*$ are non-pointed. If either $H$ or $H^*$ has the Chevalley property, then $H$ is isomorphic to $R_{-1}\# A_{12}$ or $\mathcal H_{24}$ or their duals.
        \item Assume $H$ and $H^*$ do not have the Chevalley property. Then $H$ is isomorphic to $\mathcal N_{24}$ or its dual.     
  \end{enumerate}
\end{thm}
\begin{proof}
  (i) and (ii) follow from Theorem \ref{CP8p}.

  (iii) By Lemma \ref{3dG}, $3$ divides $|G(H\otimes H^*)|$. Hence, the classification follows from Theorem \ref{T:ChavCCE}.
\end{proof}

 \bibliographystyle{plain}
 
 \end{document}